\documentclass[reqno,10pt]{amsart}
\usepackage{xcolor}
\usepackage[top=2.0cm,bottom=2.0cm,left=3cm,right=3cm]{geometry}
\usepackage{amsthm,amsmath,amssymb,dsfont}
\usepackage{mathrsfs,amsfonts,functan,extarrows,mathtools}

\usepackage[colorlinks,
linkcolor=red,
citecolor=blue
]{hyperref}
\usepackage{xcolor}
\usepackage{indentfirst, latexsym, amssymb, enumerate,amsmath,graphicx}
\usepackage{float}
\usepackage{relsize}
\usepackage{marginnote}
\usepackage{stmaryrd}
\usepackage{esint}
\usepackage{graphicx}
\usepackage{bm}
\usepackage{caption}
\usepackage{subfigure}
\usepackage{cite}
\usepackage{color}
\usepackage{graphicx}
\usepackage{subfigure}

\usepackage{paralist}
\usepackage{indentfirst}
\allowdisplaybreaks 

\theoremstyle{plain}
\newtheorem{theorem}{Theorem}[section]

\newtheorem{lemma}{Lemma}[section]
\newtheorem{proposition}{Proposition}[section]
\newtheorem{remark}{Remark}[section]

\numberwithin{equation}{section}

\usepackage{appendix}
\usepackage{xcolor}

\begin{document}

\title[Time-periodic solutions of the
Boltzmann equation ]
{Time-periodic solutions of the Boltzmann equation with soft potentials in $\mathbb{R}^3$}


\author[Y.J.  Lei]{Yuanjie Lei}  
\address[YJL]{School of Mathematics and Statistics,  Huazhong University of Science and Technology,  and the Hubei Key Laboratory of Engineering Modeling and Scientific Computing,  Wuhan  430074, P. R. China}
 \email{leiyuanjie@hust.edu.cn}
 
 \author[S.Q. Liu]{Shuangqian Liu}
\address[SQL]{School of Mathematics and Statistics, Central China Normal University, Wuhan 430079, China; Hubei Key Laboratory of Mathematical Sciences, Central China Normal University, Wuhan 430079, China}
\email{sqliu@ccnu.edu.cn}

\author[J.Y.  Zhang]{Jiaying Zhang}   
\address[JYZ]{School of Mathematics and Statistics,  Huazhong University of Science and Technology,  and the Hubei Key Laboratory of Engineering Modeling and Scientific Computing,  Wuhan  430074, P. R. China}
\email{zjiaying@hust.edu.cn}

\author[H.J. Zhao]{Huijiang Zhao}
\address[HJZ]{School of Mathematics and Statistics, Wuhan University, Wuhan 430072, China; Computational Science Hubei Key Laboratory, Wuhan University, Wuhan 430072, China}
\email{hhjjzhao@whu.edu.cn}

\begin{abstract}
We establish the existence of time-periodic solutions to the Boltzmann equation for the full range of soft potentials and prove their global stability. The proof relies on refined energy estimates combining Besov and Sobolev regularity with velocity-weighted energy methods. The analysis is complicated by the lack of a spectral gap in the soft potential regime together with the absence of time integrability of the periodic forcing over $\mathbb{R}^+$.
 This provides a framework for the study of time-periodic behavior for the Boltzmann equation.
\end{abstract}



\maketitle

\setcounter{equation}{0}
 \indent \allowdisplaybreaks

\tableofcontents

\section{Introduction}\label{Sec:Introduction}

\subsection{The problem}
We consider the Boltzmann equation with an inhomogeneous forcing term:
\begin{equation}\label{BQ-1}
	\partial_t F+v\cdot\nabla_xF=\mathcal{Q}(F,F)+\mathcal{S},\ (t,x,v)\in\mathbb{R}^+\times\mathbb{R}^3\times\mathbb{R}^3.
\end{equation}
Here, $F=F(t,x,v)$ is the distribution function of the particles at time $t\geq 0$, located at $x = (x_1, x_2, x_3)\in \mathbb{R}^3
$ with the velocity $v = (v_1, v_2, v_3)\in \mathbb{R}^3$, the term $\mathcal{S}=\mathcal{S}(t,x,v)$ is a given external source, assumed to be time-periodic with period $T>0$, namely
\begin{equation}\label{H-condition-T}
	\mathcal{S}(t+T,x,v)=\mathcal{S}(t,x,v)\notag
\end{equation}
for all $(t,x,v)\in\mathbb{R}^+\times\mathbb{R}^3\times\mathbb{R}^3$.

The function $\mathcal{Q}(F,F)(v)$ is defined as follows:
\begin{equation*}
		\mathcal{Q}(F,F)(v)
		=\int_{\mathbb{R}^3\times \mathbb{S}^2}|u-v|^{\gamma}{\bf b}\left(\frac{\omega\cdot(v-u)}{|u-v|}\right)\left\{F(v')F(u')-F(v)F(u)\right\}d\omega du.
\end{equation*}
Here, $\omega\in \mathbb{S}^2$, and ${\bf b}$, the angular part of the collision kernel, satisfies the Grad angular cutoff assumption (cf. \cite{Grad-1958,Villani 2002}):
\begin{equation}\label{cutoff-assump}
	0\leq{\bf b}(\cos \theta)\leq C|\cos\theta|, \cos\theta =\omega\cdot\frac{v-u}{|v-u|}\notag
\end{equation}
for some positive constant $C>0$.

Additionally, we have the following velocity transformations after a collision of particles with velocities $v$ and $u$ before the collision:
\begin{equation*}
	v'=v-[(v-u)\cdot\omega]\omega,\quad u'=u+[(v-u)\cdot\omega]\omega.
\end{equation*}

The exponent $\gamma\in(-3,1]$ in the kinetic part of the collision kernel is determined by the intermolecular interaction potential. It is typically classified into two regimes: the soft potential case $-3<\gamma<0$,  and the hard potential case for $0\leq\gamma\leq 1$. The results in this paper cover the full range of soft potentials, namely $-3<\gamma<0$.

The time-periodic solution to the Boltzmann equation \eqref{BQ-1} refers to a solution that repeats itself over a fixed time interval. In 2006, Ukai \cite{Ukai-2006-DCDS} was the first to study time-periodic solutions of the Boltzmann equation \eqref{BQ-1}. Later, Ukai-Yang \cite{Ukai-Yang-2006-AA} extended this analysis to a broader solution space. It is worth noting that both works require the spatial dimension to be at least five when general inhomogeneous external forces are considered. In the three-dimensional case, an additional compatibility condition is imposed,
	$\int_{\mathbb{R}^3_x}{\bf P}[\mu^{-\frac12}\mathcal{S}]dx=0,$
where ${\bf P}$ is defined in \eqref{def-P}. In addition, Duan-Liu \cite{Duan-Liu-2015-Acta.Sci} addressed the Vlasov-Poisson-Fokker-Planck system in $\mathbb{R}^3_x$, demonstrating that the corresponding linear operator exhibits exponential decay.

Recently, Deguchi studied the stability of periodic and stationary solutions to fluid equations in the function space $\dot{B}^{\frac12}_{2,\infty}\cap \dot H^{N}$, see \cite{Deguchi-2025,Deguchi-2024-MathAnn}. In a closely related functional framework,
and more recently, at the kinetic level, Duan–Lei–Ni \cite{DLN-2026} constructed time-periodic solutions to the Boltzmann equation with external forcing in $\mathbb{R}^3$
by means of spectral and semigroup methods. The obtained solutions belong to $\dot{B}^{\frac12}_{2,\infty}(L^2_v)\cap \dot H^{N}_x(L^2_v)$.
In addition, Duan–Ukai–Yang–Zhao \cite{DUYZ-CMP-2008} and Duan–Ni \cite{DN-2026} studied time-periodic solutions $\mathbb{R}^5_x$ and $\mathbb{R}^3_x$, respectively, for the Boltzmann equation with inhomogeneous external forces.  
However, all existing results on time-periodic solutions for the Boltzmann equation are restricted to the hard potential case. A key reason is that their proofs rely on the spectral gap property of the linearized Boltzmann operator, which holds in the hard potential regime. In contrast, in the soft potential case, the linearized operator has no spectral gap, leading to significantly more delicate analytic difficulties.
This naturally raises the question of how to construct time-periodic solutions to the Boltzmann equation in the soft potential regime, which is the focus of the present work.

\subsection{The Perturbation}
We begin by introducing the global Maxwellian distribution, denoted as $\mu(v)$, defined as:
\[\mu=\mu(v)=(2\pi)^{-\frac32}e^{-\frac{|v|^2}2}.\]
We set the perturbation $f=f(t,x,v)$ as
\[F=\mu+\mu^{\frac12}f.\]
With this perturbation, equation \eqref{BQ-1} can be recast into the following form:
\begin{equation}\label{f-perturbation}
	\partial_t f +v\cdot \nabla_x f+\mathcal{L} f=\Gamma(f,f)+S,
\end{equation}
where
\[S=S(t,x,v)=\mu^{-\frac12}\mathcal{S}(t,x,v)\]

In this context, the linearized Boltzmann collision operator $\mathcal{L} $ and the nonlinear collision term ${\Gamma}(f,f)$ are defined as follows:
\begin{equation*}
	\mathcal{L} f = -{\bf \mu}^{-1/2}
	\left\{{\mathcal{Q}\left({\bf \mu},{\bf \mu}^{\frac12}f\right)+ \mathcal{Q}\left({\bf \mu}^{\frac12}f,{\bf \mu}\right)}\right\}
\end{equation*}
and
\begin{equation*}
	{\Gamma}(f,f) ={\bf \mu}^{-\frac12}\mathcal{Q}\left({\bf \mu}^{\frac12}f,{\bf \mu}^{\frac 12}f\right),
\end{equation*}
respectively.

The operator $\mathcal{L} $ can be decomposed as (cf. \cite{CIP-1994, Strain-Guo-2008-ARMA}):
\begin{align*} 
\mathcal{L} =\nu(v) - \mathcal{K},
\end{align*}
where $\nu(v)$ is a non-negative measurable function of $v$, which is given by
\begin{align}\label{def-nu}
\nu(v)=\int_{\mathbb{R}^3}\int_{\mathbb{S}^2}|v - u|^\gamma {\bf b}(\cos \theta)\mu(u){\rm d}\omega{\rm d}v_{*},\notag
\end{align}
and satisfies
\begin{align*}
\nu(v)\sim (1 + |v|)^{\gamma},\quad -3<\gamma<0.
\end{align*}
Here, the operator $\mathcal{K}$ is a self-adjoint compact operator on $L^2(\mathbb R^3_{v})$ with a real symmetric integral kernel $\mathcal{K}(v,v_{*})$. Moreover, 
\begin{align}
\mathcal{K}g(v):=&\int_{\mathbb R^3}\!\int_{\mathbb S^2} |v-u|^\gamma {\bf b}(\cos \theta)\sqrt{\mu(u)}\sqrt{\mu(v)}g(u){\rm d}\omega {\rm d}u\nonumber\\
&+\int_{\mathbb R^3}\!\int_{\mathbb S^2} |v-u|^\gamma {\bf b}(\cos \theta)\sqrt{\mu(u)}\Big\{\sqrt{\mu(u^\prime)}g(v^\prime)+\sqrt{\mu(v^\prime)}g(u^\prime)\Big\}{\rm d}\omega {\rm d}u.\nonumber
\end{align}

It is a well-known fact \cite{Guo-IUMJ-2004} that it is non-negative, and the null space $\mathcal{N}$ of ${\mathcal{L} }$ can be precisely defined as:
\begin{equation*}
	{\mathcal{ N}}={\textrm{Span}}\left\{{\bf \mu}^{\frac12}, ~ v_i{\bf \mu}^{\frac12} (1\leq i\leq3),\left(| v|^2-3\right){\bf \mu}^{\frac12}\right\}.
\end{equation*}

To provide a more comprehensive perspective, we define ${\bf P}$ as the orthogonal projection operator from $L^2({\mathbb{R}}^3_ v)$ to $\mathcal{N}$. Consequently, for any given function $f(t, x, v )\in L^2({\mathbb{R}}^3_ v)$, we can express it as follows:
\begin{eqnarray}
	{\bf P}f &=&{a^f(t, x){\bf \mu}^{\frac12}+b^f(t, x)\cdot v{\bf \mu}^{\frac12}+c^f(t, x)\left(| v|^2-3\right)}{\bf \mu}^{\frac12}\label{def-P},\\
	a^f&=&{\displaystyle\int_{{{\mathbb{R}}}^3_v}}{\bf \mu}^{\frac12}fd v={\displaystyle\int_{{{\mathbb{R}}}^3_v}}{\bf \mu}^{\frac12}{\bf P}f d v,\nonumber\\
	b^f_i&=&\displaystyle\int_{{\mathbb{R}}^3_v} v_ i {\bf \mu}^{\frac 12}fd v
	=\int_{{\mathbb{R}}^3_v} v_ i {\bf \mu}^{\frac12}{\bf P}fd v,\quad i =1,2,3,\nonumber\\
	c^f&=&\frac{1}{6}\int_{{\mathbb{R}}^3_v} \left(| v|^2-3\right){\bf \mu}^{\frac12}f d v
	=\frac{1}{6}{\displaystyle\int_{{{\mathbb{R}}}^3_v}}\left(| v|^2-3\right){\bf \mu}^{\frac12}{\bf P}f d v.\nonumber
\end{eqnarray}

Hence, we arrive at a macro-micro decomposition concerning the given global Maxwellian $\mu$:
\begin{equation}\label{macro-micro}
	f(t,x, v)={\bf P}f(t,x, v)+{\bf P^{\bot}}f(t, x, v).\notag
\end{equation}
In this context, ${\bf I}$ represents the identity operator, while ${\bf P}f$ and ${\bf P}^\bot f$ are referred to as the macroscopic and microscopic components of $f(t,x,v)$, respectively. For a corresponding macro-microscopic decomposition applied to the Boltzmann equation, please see \cite{Guo-IUMJ-2004, Liu-Yang-Yu-PD-2004}. Next, we introduce the splitting of the velocity space $\mathbb{R}^3_v$, i.e.,
\begin{align}V_0\oplus V_1\oplus V_2\oplus\cdots\oplus  V_j\oplus V_{j+1}\oplus\cdots=\mathbb{R}^3_v,\label{vel-dec}
\end{align}
where 
\[V_{j}=\left\{v\in\mathbb{R}^3 \mid 2^{j}\leq |v| < 2^{j+1},\ j\in\mathbb{Z}^+\right\},\]
and 
\[V_0=\left\{v\in\mathbb{R}^3 \mid |v|< 2\right\}.\]

\subsection{Main results}
Our first result pertains to the existence of a time-periodic solution for \eqref{BQ-1}. 

\begin{theorem}\label{theorem-1}
Assume $-3<\gamma<0$ and $S(t+T,x,v)=S(t,x,v)$ for $0<T<\infty$, take $N\ge5$ and $\ell\geq N-\frac\gamma 2$, suppose that for any $\varepsilon>0$
\begin{eqnarray}
\|S\|^2_{\mathcal{Z}}&\equiv&\sup_t\|S(t,x,v)\|^2_{\widetilde{L}_v^2(\dot{B}^{-\frac32}_{2,\infty}\cap\dot{B}^\frac12_{2,\infty})}+\sup_t\|S\|_{L_v^2(\dot{H}^1\cap\dot{H}^N)}^2\nonumber\\
&&+\sum_{|\alpha|+|\beta|\leq N}\sup_t\|\langle v\rangle^{\ell-(1+\frac{\varepsilon}{2})\gamma+\gamma|\beta|}\,\partial_x^\alpha\partial_v^\beta S\|^2\nonumber\\
&&+\sum_{j\in\mathbb{Z}^+}\sum_{|\alpha|+|\beta|\le N}\sup_t\|\nu^{-1}(v)\langle v\rangle^{\ell+\gamma|\beta|}\partial_x^\alpha\partial_v^\beta S\|^2_{L_v^2(V_j)L_x^2(\mathbb{R}^3)}\leq \epsilon_0,   \notag
\end{eqnarray} 
where $\epsilon_0>0$ is suitably small.
Then the equation \eqref{f-perturbation} in  $\mathbb{R}^3_x$ admits a unique time-periodic solution $f^{per}(t,x,v)$ with the same period $T$. This solution satisfies the following bounds:
\begin{eqnarray}
    &&\sup_t\|f^{per}(t,x,v)\|^2_{\widetilde{L}_v^2(\dot{B}^{\frac12}_{2,\infty}\cap\dot{B}^{\frac32}_{2,\infty})}+\sup_t\|f^{per}(t,x,v)\|^2_{L_v^2(\dot{H}^1_x\cap\dot{H}_x^N)}\nonumber\\
   && +\sup_t\sum_{|\alpha|+|\beta|\leq N}\|\langle v\rangle^{\ell+\gamma|\beta|}\partial^\alpha_x\partial^\beta_v{\bf P^{\bot}}f^{per}(t,x,v)\|^2
		\lesssim \|S\|^2_{\mathcal{Z}}.\notag
\end{eqnarray}
\end{theorem}

Our second result concerns the stability of the time-periodic solution established in Theorem \ref{theorem-1}. To this end, we introduce a perturbation $g$ around the time-periodic solution $f^{per}$, defined by
\begin{eqnarray}\label{def-f-bar}
	g(t)=f(t)-f^{per}(t),\notag
\end{eqnarray}
which satisfies the evolution equation:
\begin{eqnarray}\label{f-bar-equation}
\partial_t g(t) +v\cdot \nabla_x g(t)+\mathcal{L} g(t)
	=\Gamma(g(t), f^{per}(t))+\Gamma(f^{per}(t),g(t)).
\end{eqnarray}
We then consider the corresponding initial value problem with initial data
\begin{eqnarray}\label{f-bar-equation-intial}
	g(t)|_{t=0}=g(0)=f(0)-f^{per}(0).
\end{eqnarray}
\begin{theorem}\label{theorem-2}
Under the same assumptions as in Theorem \ref{theorem-1}, with $f^{per}(t)$ being the time-periodic solution obtained in Theorem \ref{theorem-1}, provided that the initial data 

\[\|g_{0}\|_{X}\equiv\|g(0,x,v)\|_{\widetilde{L}_v^2(\dot{B}^{\frac12}_{2,\infty}\cap\dot{B}^{\frac32}_{2,\infty})}+\|g(0,x,v)\|_{L_v^2(\dot{H}^1_x\cap\dot{H}^N_x)}+\sum_{|\alpha|+|\beta|\leq N}\|\langle v\rangle^{\ell+\gamma|\beta|}\partial^\alpha_x\partial^\beta_v{\bf P^{\bot}}g(0,x,v)\|\] is sufficiently small, there exists a unique solution $g(t)$ to \eqref{f-bar-equation}-\eqref{f-bar-equation-intial} satisfying

\begin{eqnarray}
	\sup_t\|g(t)\|_{X}&\equiv&\sup_t\|g(t)\|_{\widetilde{L}_v^2(\dot{B}^{\frac12}_{2,\infty}\cap\dot{B}^{\frac32}_{2,\infty})}+\sup_t\|g(t)\|_{L_v^2(\dot{H}^1_x\cap\dot{H}^N_x)}\nonumber\\
    &&\quad+\sup_t\sum_{|\alpha|+|\beta|\leq N}\|\langle v\rangle^{\ell+\gamma|\beta|}\partial^\alpha_x\partial^\beta_v{\bf P^{\bot}}g(t)\|\nonumber\\
    &\lesssim& \|g_{0}\|_{X}.\notag
\end{eqnarray}
Furthermore, we also obtain the time decay rates for $g(t,x,v)$ in the following:
\begin{eqnarray}\label{decay-g}
&&\sup_t\|g(t)\|_{\widetilde{L}_v^2(\dot{B}^{1}_{2,\infty})}+\sup_t\|g(t)\|_{L_v^2(\dot{H}_x^2\cap\dot{H}_x^{N-2})}\nonumber\\
&&+\sup_t\sum_{1\leq|\alpha|\leq N-2}\|\langle v\rangle^{\ell}\partial^\alpha_x{\bf P^{\bot}}g(t)\|\lesssim \|g_{0}\|_X(1+t)^{-\frac14}. \notag 
\end{eqnarray}

\end{theorem}
\begin{remark}
\begin{itemize} Two remarks are listed in the following:
	
	
\item [1)]Here we obtain a time decay rate $(1+t)^{-\frac14}$ for stability. If we impose the smallness
on \[\|g(0,x,v)\|_{\dot{B}^{s}_{2,\infty}L^2_v\bigcap\dot{H}^N_xL^2_v}, \ \ \ -\frac3{2}\leq s<\frac12,\]
instead of \[\|g(0,x,v)\|_{\dot{B}^{\frac12}_{2,\infty}L^2_v\bigcap\dot{H}^N_xL^2_v},\] we can achieve the better decay rate $(1+t)^{-\frac{1+s}{2}}$. To remain consistent with the function space of the time-periodic solution, we do not pursue a better decay.
\item [2)]Although we have only dealt with the case of soft potentials, our method is also applicable to the construction of time-periodic solutions for kinetic equations in the case of hard potentials.
\end{itemize}

\end{remark}
\subsection{Strategy of the proof}
We apply Serrin’s method \cite{Serrin-1959-ARMA} to prove the existence of time-periodic solutions to \eqref{f-perturbation}. The key step is to establish the global well-posedness of the Cauchy problem associated with \eqref{f-perturbation} and to construct a Cauchy sequence converging to the desired time-periodic solution.

As in the case of the compressible Navier–Stokes equations, the dissipative effect of the macroscopic component leads, in the three-dimensional whole space, to spatial decay of order $\frac{1}{|x|}$ as $|x|\rightarrow \infty $, indicating that the time-periodic solution does not belong to $L^2_x(\mathbb{R}^3)$.
To capture this slow decay, we work in the space $\dot{B}^{\frac12}_{2,\infty}$ introduced in \cite{Deguchi-2025}, which is suitable for handling such borderline spatial behavior. In addition, Tusda in  \cite{Tusda-2016-ARMA} deal with time-periodic solutions of the isentropic Navier-Stokes system. 

However, for the Boltzmann equation with soft potentials, the absence of a spectral gap prevents the direct use of spectral decomposition and semigroup techniques as in \cite{Deguchi-2025,DLN-2026} to obtain estimates in $\dot{B}^{\frac12}_{2,\infty}$. We therefore adopt the following approach:

\noindent$\bullet$ \textit{Fourier heat kernel inequality with velocity weight compensation.}

    We perform a Fourier transform in the spatial variable and then carry out energy estimates to obtain the estimate as
    \[\frac{{\rm d}}{{\rm d}t}\|\hat{f}\|_{L_v^2}^2+|\xi|^2\{\|\mathbf{P} \hat{f}\|^2+\|\mathbf{P}^{\bot}\hat{f}\,\langle v\rangle^{\frac{\gamma}{2}}\|_{L_v^2}^2\}\lesssim\cdots+\int_{\mathbb{R}_v^3}\hat{S}\hat{f}\,\mathrm{d}v\]
in the low frequency regime. In the case of soft potential, the dissipative part of the microscopic component suffers a loss of $\langle v\rangle^{\frac\gamma 2}$
  compared to the energy level. To compensate for this loss, we exploit a weighted interpolation of the form
  \[\|\mathbf{P}^{\bot}\hat{f}\|_{L_v^2}\le \|\mathbf{P}^{\bot}\hat{f}\,\langle v\rangle^{\frac{\gamma}{2}}\|_{L_v^2}^2+\eta\|\mathbf{P}^{\bot}f\,\langle v\rangle^{-\frac{\gamma}{2}}\|_{L_v^2}^2.\]
Consequently, we obtain
  \begin{align*}
    \frac{\rm d}{{\rm d}t}\|\hat{f}\|_{L_v^2}^2+|\xi|^2\|\hat{f}\|_{L_v^2}^2\le&\eta\,|\xi|^2\|\mathbf{P}^{\bot}f\langle v\rangle^{-\frac{\gamma}{2}}\|_{L_v^2}^2+\cdots,
\end{align*}
which further yields that
\begin{align}
\|\hat{f}\|_{L_v^2}^2\le\,&e^{-t|\xi|^2}\|\hat{f}(0,\xi,v)\|_{L_v^2}^2+\eta\int_0^t e^{-|\xi|^2(t-\tau)}|\xi|^2\|\mathbf{P}^{\bot}f\,\langle v\rangle^{-\frac{\gamma}{2}}\|_{L_v^2}^2\,\mathrm{d}\tau
\notag\\&+\int_0^t e^{-|\xi|^2(t-\tau)}\left|\hat{\Gamma}(f,f)\langle v\rangle^{-\frac{\gamma}{2}}\right|_{L_v^2}^2\,\mathrm{d}\tau
+\cdots.\label{low-heat}
\end{align} 

\noindent$\bullet$ \textit{Besov regularity for time integrability and decay.}

To handle the time convolution term 
$$\int_0^t e^{-|\xi|^2(t-\tau)}\left|\hat{\Gamma}(f,f)\langle v\rangle^{-\frac{\gamma}{2}}\right|_{L_v^2}^2\,\mathrm{d}\tau$$
in \eqref{low-heat}, we introduce the Besov space $\dot{B}^{\frac12}_{2,\infty}$, which allows us to recover the following bound:
 \begin{align*}
\sup_{j\in\mathbb{Z}}&\int_0^t \int_{\mathbb{R}_\xi^3}e^{-|\xi|^2(t-\tau)}\left|\hat{\Gamma}(f,f)\langle v\rangle^{-\frac{\gamma}{2}}\right|_{L_v^2}^2\varphi_j^22^j\,\mathrm{d}\xi\mathrm{d}\tau\\ &\lesssim
\sup_{j\in\mathbb{Z}}\int_0^t \int_{\mathbb{R}_\xi^3}e^{-|\xi|^2(t-\tau)}|\xi|^2\left|\hat{\Gamma}(f,f)\langle v\rangle^{-\frac{\gamma}{2}}\right|_{L_v^2}^2\varphi_j^22^{-j}\,\mathrm{d}\xi\mathrm{d}\tau\\
&\lesssim\sup_{j\in\mathbb{Z}}\int_0^t e^{-2^{2j}(t-\tau)}2^{2j}\,\mathrm{d}\tau\sup_t\sup_{j\in\mathbb{Z}}\int_{\mathbb{R}_\xi^3}\left|\left|f\right|_{H_v^2}\left|f\right|_{L_v^2}\right|^2\varphi_j^2\,2^{-j}\,\mathrm{d}\xi\\
     &\lesssim\sup_t\left|\left|f\right|_{H_v^2}\left|f\right|_{L_v^2}\right|_{\dot{B}_{2,\infty}^{-\frac{1}{2}}}^2
\lesssim\sup_t\|f\|_{\widetilde{H}_v^2(\dot{B}_{2,\infty}^{\frac{1}{2}})}^2\|f\|_{\widetilde{L}_v^2(\dot{B}_{2,\infty}^{\frac{1}{2}})}^2.
 \end{align*}   
Here $\widetilde{H}_v^2(\dot{B}_{2,\infty}^{\frac{1}{2}})$ denotes the corresponding Chemin-Lerner type norm, defined analogously to \eqref{cl-norm}. The choice of the Besov index $\frac 12$ is also dictated by this nonlinear estimate.
Moreover, the validity of the estimate \eqref{low-heat} requires controlling  the additional term
\[\eta\int_0^t e^{-|\xi|^2(t-\tau)}|\xi|^2\|\mathbf{P}^{\bot}f\,\langle v\rangle^{-\frac{\gamma}{2}}\|_{L_v^2}^2\,\mathrm{d}\tau,\] together with the corresponding weighted estimates for higher-order derivatives.

Once the global existence of the Cauchy problem for \eqref{f-perturbation} is established, the construction of time-periodic solutions further requires a Cauchy sequence in a larger functional space. To this end, we derive time decay estimates for the difference of two solutions to \eqref{f-perturbation}. Using Fourier-transform-based energy estimates, we obtain in the low-frequency regime the following bound:
 \begin{align}\label{decay-low-eqn}
\|g\|_{\widetilde{L}_v^2(\dot{B}_{2,\infty}^1)}^2&\lesssim \sup\limits_{j}\int_{\mathbb{R}_{\xi}^3} e^{-\lambda t|\xi|^2}|\xi|^2|\varphi_j\hat{g}_0|_2^2\,\mathrm{d}\xi+\eta\sup\limits_{j}\int_0^t \int_{\mathbb{R}_\xi^3}e^{-\lambda|\xi|^2(t-\tau)}|\xi|^2\left||\xi|\mathbf{P}^{\bot}\hat{g}\,\langle v\rangle^{-\frac{\gamma}{2}}2^j\right|_2^2\varphi_j^2\,\mathrm{d}\xi\mathrm{d}\tau\notag\\
    &+\sup\limits_{j}\int_0^t \int_{\mathbb{R}_\xi^3}e^{-\lambda|\xi|^2(t-\tau)}|\xi|^2\left|\hat{\Gamma}(g,f_1)\langle v\rangle^{-\frac{\gamma}{2}}\right|_2^2\varphi_j^2\,\mathrm{d}\xi\mathrm{d}\tau+\cdots.
\end{align}
To handle the nonlinear term 
\[\int_0^t \int_{\mathbb{R}_\xi^3}e^{-\lambda|\xi|^2(t-\tau)}|\xi|^2\left|\hat{\Gamma}(g,f_1)\langle v\rangle^{-\frac{\gamma}{2}}\right|_2^2\varphi_j^2\,\mathrm{d}\xi\mathrm{d}\tau,\]
we use the decay estimate
\begin{align}
    &\hspace{3mm}\int_0^te^{-2^{2j}(t-\tau)}2^{2j}(1+\tau)^{-\frac{1}{2}}\,\mathrm{d}\tau
    \lesssim(1+t)^{-\frac{1}{2}}\,,\notag
\end{align}
which follows from Lemma \ref{lem:max_estimate}.

Moreover, to control the second term on the right-hand side of \eqref{decay-low-eqn}, we establish decay estimates for the weighted energy. This is achieved by decomposing the linear operator $\mathcal{L}$ into \(\nu\) and \(K\), and exploiting the damping effect of \(\nu\). Together with Lemma \ref{lemma2.5}, and by allowing a controlled loss in velocity weights, we obtain the desired decay and thereby construct the Cauchy sequence required for the time-periodic solution.
It should be noted that, in this step, the mechanism for obtaining time decay differs from the spectral and semigroup-based approaches used in \cite{Deguchi-2025,DLN-2026,DUYZ-CMP-2008,Ukai-2006-DCDS,Ukai-Yang-2006-AA}.


\noindent$\bullet$ \textit{Annular decomposition in velocity space and collision-frequency compensation for time integrability.}

In the soft potential case, the weaker dissipation of the linearized operator introduces significant difficulties in the analysis. To address this issue, we decompose the linear operator as $\mathcal{L}=\nu(v)-\mathcal{K}$, following the classical framework in \cite{CIP-1994,Strain-Guo-2008-ARMA}. Exploiting the collision frequency $\nu(v)$, we obtain the weighted estimate such as
\[\left\|w_{\ell,0}\partial_x^{\alpha}f\right\|^2+\frac{1}{2}\int_0^t\int_{\mathbb{R}_{x,v}^6}e^{-\nu(v)(t-\tau)}\nu(v)\left(\partial_x^{\alpha}fw_{\ell,0}\right)^2 \, \mathrm{d}v\mathrm{d}x\mathrm{d}\tau\lesssim\cdots,\]
so that this integrating factor
 \(\int_0^te^{-\nu(v)(t-\tau)}\)
 can be used to handle the time-periodic term \(S\) which is non-integrable in time. 
 However, this decomposition requires re-estimating the operators \(\mathcal{K}\) and \(\Gamma\) under the weight \(e^{-\frac{\nu(v)}{2}(t-\tau)}\langle v\rangle^{\ell}\). The presence of the factor $e^{-\frac{\nu(v)}{2}(t-\tau)}$ prevents the direct application of the weighted exponential framework in \cite{Strain-Guo-2008-ARMA}, since the integral operator does not fully neutralize this decay for large $t-\tau$. To overcome this difficulty, we introduce an annular decomposition of the velocity space, separating the integration region so that the temporal decay can be exploited effectively. More precisely,
\begin{eqnarray*}
&& \int_0^t\int_{\mathbb{R}^6_{x,v}}e^{-\nu(v)(t-\tau)}\nu(v)\nu^{-1}(v)w^2_{\ell,0}\mathcal{K}\partial_x^\alpha f\partial_x^\alpha f\,{\rm d}v{\rm d}x{\rm d}\tau\nonumber\\
  &\lesssim & \sum_{j\in\mathbb{Z}^+}\int_0^t\int_{V_j\times\mathbb{R}_x^3}e^{-\nu(v)(t-\tau)}\nu(v)\nu^{-1}(v)w^2_{\ell,0}\mathcal{K}\partial_x^\alpha f\partial_x^\alpha f\,{\rm d}v{\rm d}x{\rm d}\tau+\cdots\nonumber\\
  &\lesssim&\sum_{j\in\mathbb{Z}^+}\int_0^t\int_{V_j\times\mathbb{R}_x^3}e^{-\tilde{c}_j(t-\tau)}c_j\nu^{-1}(v)w^2_{\ell,0}\mathcal{K}\partial_x^\alpha f\partial_x^\alpha f\,{\rm d}v{\rm d}x{\rm d}\tau+\cdots\\
    &\lesssim&\sum_{j\in\mathbb{Z}^+}\sup_t\int_{V_j\times\mathbb{R}_x^3}\nu^{-1}(v)w_{\ell,0}^2\mathcal{K}\partial_x^\alpha f\partial_x^\alpha f\,{\rm d}v{\rm d}x\sup_{j\in\mathbb{Z}^+}\int_0^t e^{-\tilde{c}_j(t-\tau)}c_j\,{\rm d}\tau+\cdots\\
    &\lesssim&\sum_{j\in\mathbb{Z}^+}\sup_t\int_{V_j\times\mathbb{R}_x^3}\nu^{-1}(v)w^2_{\ell,0}\mathcal{K}\partial_x^\alpha f\partial_x^\alpha f\,{\rm d}v{\rm d}x+\cdots
\end{eqnarray*}
 where $\tilde{c}_j\sim (1+2^{j+1})^{\gamma}$ and $c_j\sim (1+2^{j})^{\gamma}$, and the time integrability of $e^{-\tilde{c}_j(t-\tau)}c_j$ is used. Furthermore, combining this decomposition with the approach in \cite{Strain-Guo-2008-ARMA}, we obtain Lemmas \ref{lemma5.1} and \ref{lemma5.2} on each velocity region $V_j$. This leads to the necessity of controlling
 \[\sum_{j\in \mathbb{Z}^+}\sup_{t}\|w_{\ell,\beta}\partial^\alpha_x\partial_v^{\beta} {\bf P}^\bot  f\|_{L^2(V_j)L_x^2(\mathbb{R}^3)}^2\]
as carried out in Section \ref{Sect-weight} and Section \ref{Section-V-j}.



\subsection{Notation}
Before proceeding, we first introduce some basic notations used throughout the paper:
\begin{itemize}
	\item $C$ denotes some positive constant (generally large) and $\lambda$, $\epsilon$, $\kappa$, and $\delta$ stand for some positive constant (generally small). Note that all these constants may take different values in different places;
	\item $A\lesssim B$ means that there is a generic constant $C> 0$ such that $A \leq CB$. $A \sim B$ means $A\lesssim B$ and $B\lesssim A$. $A\gtrsim B$ can be defined similarly;
	\item The multi-indices $ \alpha= [\alpha_1,\alpha_2, \alpha_3]$ and $\beta = [\beta_1, \beta_2, \beta_3]$ will be used to record spatial and velocity derivatives, respectively. And $\partial^{\alpha}_x\partial^{\beta}_v=\partial^{\alpha_1}_{x_1}\partial^{\alpha_2}_{x_2}\partial^{\alpha_3}_{x_3} \partial^{\beta_1}_{ v_1}\partial^{\beta_2}_{ v_2}\partial^{\beta_3}_{ v_3}$. Similarly, the notation $\partial^{\alpha}$ will be used when $\beta=0$ and likewise for $\partial_{\beta}$. The length of $\alpha$ is denoted by $|\alpha|=\alpha_1 +\alpha_2 +\alpha_3$. $\alpha'\leq  \alpha$ means that no component of $\alpha'$ is greater than the corresponding component of $\alpha$, and $\alpha'<\alpha$ means that $\alpha'\leq  \alpha$ and $|\alpha'|<|\alpha|$. And it is convenient to write
	$$
	\left|\nabla_x^kf\right|\equiv\sqrt{\sum\limits_{|\alpha|=k}\left|\partial^{\alpha}
		f\right|^2};
	$$
     \item  $\mathbf \chi_{D}$ stands for the characteristic function of a set $D$;
	\item $\langle\cdot,\cdot\rangle$ is used to denote the ${L^2_{ v}}$ inner product in ${\mathbb{ R}}^3_{ v}$, with the ${L^2}$ norm $|\cdot|_{L_v^2}$, while $(\cdot, \cdot)$ denotes the ${L^2}$ inner product either in ${\mathbb{ R}}^3_{x}\times{\mathbb{ R}}^3_{ v }$ or in ${\mathbb{ R}}^3_{x}$ with the ${L^2}$ norm $\|\cdot\|$;
	\item For $p\geq 1, q\geq 1$, we also define the mixed velocity-space Lebesgue space $L_x^pL_v^q=L^p(\mathbb{R}_x^3;L^q(\mathbb{R}_v^3))$ with the norm
	$$
	\|f\|_{L^p_xL^q_v}=\left(\int_{\mathbb{R}^3_x}
	\left(\int_{\mathbb{R}^3_v}|f(x,v)|^qdv\right)^{\frac pq} dx \right)^{\frac 1p}
	$$
	for $f=f(x,v)\in L^p_xL^q_v$. For $p\geq 1$, $q\geq 1$, $\ell\in \mathbb{Z}^+$, $L_x^pL_v^q$, $L_v^pH_x^\ell$, etc. can be defined similarly;
   
    \item 
We will employ the notation $L_{v}^2(H^s)$ to represent the space $L^2(\mathbb R^3_v;H^s_x)$ with the norm defined as follows:
\begin{align*}
\|g\|_{L_v^2(H^s)}^2 = \int_{\mathbb R^3_v} \|g\|_{H^s}^2{\rm d}v.
\end{align*}
Similarly, we adopt the notations of $L_v^2(\dot H^s )$, $L_v^2( L^p )$ and $L ^p( L^2_v)$, etc;

\item  We introduce the following velocity-weighted norms that concerns the collision frequency:
\begin{align*}
\|g\|_{L_{x,v}^2}^2=\|g\|^2,\quad|g|_{\nu}^2:=|\sqrt{\nu}g|_{L_v^2}^2,\quad \|g\|_{\nu}^2=\|\sqrt{\nu}g\|^2.
\end{align*}
\item For a Banach space $X$ and an interval $I$ of $\mathbb{R}$, we define $\mathcal{C}(I;X)$ as the set of continuous functions on $I$ that take values in $X$, and $L^q(I;X)$ represents the set of measurable functions on $I$ that take values in $X$. When $I = (0,T)$, we denote its norm as   $\|\cdot\|_{L^q(0,T;X)}$.
\item Let $\mathscr{S}$ be the set of all Schwartz functions on $\mathbb{R}^3_x$, and let $\mathscr{S}'$ be the set of all tempered distributions on $\mathbb{R}^3_x$. For any $g(t,x,v) \in \mathscr{S}(\mathbb{R}^3_x)$, its Fourier transform is defined by
\begin{align*}
\hat{g}(t,\xi,v) = \mathcal{F}g(t,\xi,v) = \int_{\mathbb{R}^3_x} e^{-ix\cdot\xi}g(t,x,v) \mathrm{d}x, \quad x\cdot\xi = \sum_{j = 1}^3 x_j\xi_j,
\end{align*}
for all $\xi \in \mathbb{R}^3$, where $i = \sqrt{-1} \in \mathbb{C}$ denotes the imaginary unit, and $\mathcal{F}^{-1}g(t,\xi,v)$ denotes the inverse Fourier transform.
\item 
Let $\chi(\xi)$ be a smooth, radial, non-increasing function supported in the ball $B(0,\frac{4}{3})$ and satisfying $\chi(\xi)\equiv 1$ in the ball $B(0,\frac{3}{4})$.
Define $\varphi(\xi):=\chi(\frac{\xi}{2})-\chi(\xi)$. Then, $\varphi$ satisfies the following properties:
\begin{align*}
\sum_{j \in \mathbb{Z}} \varphi(2^{-j}\xi) = 1 \quad \text{for all } \xi \neq 0, \quad \text{and} \quad \mathrm{supp}\,\varphi \subset \left\{ \xi \in \mathbb{R}^3 \mid \frac{3}{4} \leq |\xi| \leq \frac{8}{3} \right\}.
\end{align*}
For any $j\in \mathbb{Z}$, we define the homogeneous dyadic blocks $\dot{\Delta}_j$ as follows:
\begin{align*}
\dot{\Delta}_j g := \mathcal{F}^{-1}\left( \varphi(2^{-j}\cdot)\mathcal{F}(g) \right) = 2^{3j} h(2^j\cdot) \ast g, \quad \text{where} \quad h := \mathcal{F}^{-1}\varphi.
\end{align*}
\item We define the homogeneous Besov space as in \cite{BCD-Book-2011,Dr-IM-2000}. Let $s\in \mathbb{R}$, $1\leq p,q\leq \infty$, and $g\in \mathscr{S}'(\mathbb{R}^3)$. Then, the homogeneous Besov space $\dot{B}^s_{p,q}(\mathbb{R}^3)$ is defined with the norm $\|\cdot\|_{\dot{B}^s_{p,q}}$ as follows:
\begin{align*}
  \|g\|_{\dot{B}_{p,q}^s}:= \bigg(\sum_{j\in\mathbb Z} 2^{jqs}\|\dot\Delta_j g\|_{L^p}^q  \bigg)^\frac{1}{q}. 
\end{align*}
In particular, when $r = \infty$, the norm is characterized by
\begin{align*}
\|g\|_{\dot{B}^s_{p,\infty}}:=\sup_{j\in\mathbb{Z}}2^{js}\|\dot{\Delta}_j g\|_{L^p}.   
\end{align*}
Here, we adopt the notation $L_v^\varrho(\dot{B}_{p,q}^s)$ to represent the space $L^\varrho(\mathbb{R}_v^3;\dot{B}_{p,q}^s(\mathbb{R}^3_x))$ equipped with the norm
\begin{align*}
\|g\|_{L_v^\varrho(\dot{B}_{p,q}^s)} = \left(\int_{\mathbb{R}^3} \|g\|_{\dot{B}_{p,q}^s}^\varrho\mathrm{d}v \right)^{\frac{1}{\varrho}}.
\end{align*}
Additionally, we know that $\dot{H}^s = \dot{B}_{2,2}^s$.
\item We introduce the Chemin–Lerner-type mixed space-velocity Besov spaces, as in \cite{dlx-16}. Let $s\in\mathbb{R}$ and $1\le\varrho, p,q\le\infty$, then the homogeneous Chemin-Lerner space $\widetilde{L}_v^\varrho(\dot{B}_{p,q}^s)$ is define with the norm $\|\cdot\|_{\widetilde{L}_v^\varrho(\dot{B}_{p,q}^s)}$ as follows:
\begin{align}
\|g\|_{\widetilde{L}_v^\varrho(\dot{B}_{p,q}^s)}:=\left(\sum_{j\in\mathbb{Z}}2^{jqs}\|\dot{\Delta}_jg\|_{L_v^\varrho L_x^p}^q\right)^\frac{1}{q}.\label{cl-norm}
\end{align}
In particular, when $q=\infty$, the norm is characterized by
\[\|g\|_{\widetilde{L}_v^\varrho(\dot{B}_{p,\infty}^s)}:=\sup_{j\in\mathbb{Z}}2^{js}\|\dot{\Delta}g\|_{L_v^\varrho L_x^p}\,.\]
It's obvious that $\|g\|_{\widetilde{L}_v^\varrho(\dot{B}_{2,\infty}^s)}\le\|g\|_{L_v^\varrho(\dot{B}_{2,\infty}^s)}$.


\end{itemize}
\subsection{Outline of the paper} The structure of the remaining sections in this paper is as follows:
\begin{itemize}

	\item In Section \ref{Sect-Pre}, We will present some useful inequalities that play an important role in the subsequent proofs.
	\item In Section \ref{Sect-Besov}–\ref{Section-V-j}, we focus on closing the a priori estimate and establishing the global existence for the Cauchy problem of \eqref{f-perturbation}.
	\item In Section \ref{Sect-time-devay}, we will prove time decay rates for the two solutions of \eqref{f-perturbation}, which plays an essential role in the Cauchy sequence argument.
    \item In Section \ref{Sect-time-periodic}, we will construct the time-periodic solution for \eqref{f-perturbation} and prove Theorem \ref{theorem-1}. Moreover, we will establish the global stability of the time-periodic solution for \eqref{f-perturbation} and prove Theorem \ref{theorem-2}.
\end{itemize}

\section{Preliminaries}\label{Sect-Pre}
In this section, we will present some basic properties of homogeneous Besov spaces, Sobolev spaces, and $L^p$ spaces, as well as the interpolation inequalities between them. 
\subsection{Basic properties of homogeneous Besov space}

\begin{proposition}{\rm(\!\!\cite[Chapter 2]{BCD-Book-2011})}\label{prop2.1}
The following properties hold:
\begin{itemize}
\item{} For $s\in\mathbb{R}$, $1\leq p_{1}\leq p_{2}\leq \infty$ and $1\leq q_{1}\leq q_{2}\leq \infty$, it holds that
\begin{equation}\label{embedding-1}
\begin{aligned}
\dot{B}^{s}_{p_{1},q_{1}}\hookrightarrow \dot{B}^{s-3 (\frac{1}{p_1} -\frac{1}{p_2})}_{p_{2},q_{2}}.
\end{aligned}
\end{equation}
\item{}For $1\leq p\le q\leq \infty$, it holds that
\begin{align}\label{embedding-2}
    \dot{B}^{3(\frac{1}{p}-\frac{1}{q})}_{p,1}\hookrightarrow L^q.
\end{align}
\item{}  For $1\leq p,q\leq \infty$, $s_1<s_2$, and $\theta\in (0,1)$, it holds that
\begin{align}\label{interpolation}
\|u\|_{\dot B_{p,q}^{\theta s_1+(1-\theta)s_2}}\lesssim\|u\|_{\dot B_{p,1}^{\theta s_1+(1-\theta)s_2}}\lesssim \frac{1}{\theta(1-\theta)(s_2-s_1)}\|u\|_{\dot B_{p,\infty}^{s_1}}^{\theta} \|u\|_{\dot B_{p,\infty}^{s_2}}^{1-\theta}  .  
\end{align}
\item{} For $1\leq p\leq q\leq\infty$, it holds that
\begin{equation}\nonumber
\begin{aligned}
  \dot{B}^{0}_{p,1}\hookrightarrow L^{p}\hookrightarrow \dot{B}^{0}_{p,\infty}\hookrightarrow \dot B_{q,\infty}^{\varsigma},\quad \varsigma=-3\Big(\frac{1}{p}-\frac{1}{q}\Big).
\end{aligned}
\end{equation}


    
\end{itemize}
\begin{itemize}
\item{} Let $s>0$, $1\leq  q\leq \infty$.  Then $\dot B^{s}_{2,q}\cap L^\infty$ is an algebra and
 \begin{align} \label{111}
\|gh\|_{\dot B_{2,q}^s}\lesssim \|g\|_{L^\infty}\|h\|_{\dot B_{2,q}^s}+\|h\|_{L^\infty}\|g\|_{\dot B_{2,q}^s},    
 \end{align}
which in particular leads to
\begin{align*}
\|gh\|_{\dot B_{2,1}^{\frac{3}{2}}}\lesssim \|g\|_{\dot B_{2,1}^{\frac{3}{2}}}\|h\|_{\dot B_{2,1}^{\frac{3}{2 }}} .    
\end{align*}

\item{} Let the real numbers $s_1$, $s_2$  and $q$ fulfill
\begin{align*}
 1\leq  q,q_1,q_2\leq\infty, \quad s_1<\frac{3}{2},\quad s_2<\frac{3}{2},\quad s_1+s_2>0,\quad \frac{1}{q}=\frac{1}{q_1}+\frac{1}{q_2}.
\end{align*}
  Then, it holds  that
 \begin{align} \label{newinterpolation}
\|gh\|_{\dot B_{2,q}^{s_1+s_2-\frac{3}{2}}}\lesssim \|g\|_{ \dot B_{2,q_1}^{s_1 }}\|h\|_{\dot B_{2,q_2}^{s_2}} .    
 \end{align}

\item{} Let  the real numbers $s_1$ and $s_2$ fulfill
\begin{align*}
 s_1\leq \frac{3}{2},\quad s_2< \frac{3}{2},\quad s_1+s_2\geq 0.   
\end{align*}
Then, we have
 \begin{align} \label{inter}
\|gh\|_{\dot B_{2,\infty}^{s_1+s_2-\frac{3}{2}}}\lesssim \|g\|_{ \dot B_{2,1}^{s_1 }}\|h\|_{\dot B_{2,\infty}^{s_2}}.    \notag
 \end{align}
    \end{itemize}
\end{proposition}

\subsection{Some interpolation inequalities between Besov space and other space}
\begin{lemma}\label{interpolation-besov-other}
    For all Schawartz functions $f$ on $\mathbb{R}^3$and any real numbers \(s\), the following inequalities hold:
    \begin{align}
       & \|f\|_{\dot{B}^{s+1}_{2,\infty}} \;\lesssim\;  \;\|f\|^{\frac12}_{\dot{H}^{s+2}}\;\|f\|^{\frac12}_{\dot{B}^{s}_{2,\infty}},\nonumber\\
       &\|\nabla f\|_{L^2} \;\lesssim\; \;\|f\|_{\dot{B}_{2,\infty}^{1/2}}^{2/3} \;\|f\|_{\dot{B}_{2,2}^{2}}^{1/3},\nonumber\\
       &\|f\|_{L^4} \lesssim \|f\|_{\dot{B}^{1/2}_{2,\infty}}^{1/2} \|\nabla f\|_{L^2}^{1/2},\nonumber\\
       &\|f\|_{L^\infty} \lesssim \|f\|_{\dot{B}_{2,\infty}^{1}}^{\frac12} \|\nabla_x^2 f\|_{L^2}^{\frac12}.\nonumber
    \end{align}
\end{lemma}
\begin{proof}
    For brevity, we just prove the last inequality as an example. And the rest inequalities can be obtained by a similar way. It can be deduced from \eqref{embedding-2} and \eqref{embedding-1} that $\dot{B}_{2,1}^\frac{3}{2}\hookrightarrow L^\infty$ and $\dot{B}_{2,2}^s\hookrightarrow\dot{B}_{2,\infty}^s$, which yield
    \[\|f\|_{L^\infty}\lesssim \|f\|_{\dot{B}_{2,1}^\frac{3}{2}},\;\|f\|_{\dot{B}_{2,\infty}^s}\lesssim\|f\|_{\dot{H}^s}.\]
    Moreover, \eqref{interpolation} implies that
    \[\|f\|_{\dot{B}_{2,1}^\frac{3}{2}}\lesssim\|f\|^\frac{1}{2}_{\dot{B}_{2,\infty}^1}\|f\|^\frac{1}{2}_{\dot{B}_{2,\infty}^2}\lesssim\|f\|^\frac{1}{2}_{\dot{B}_{2,\infty}^1}\|f\|_{\dot{H}^2}^\frac{1}{2}.\]
    Consequently,
    \[\|f\|_{L^\infty}\lesssim\|f\|^\frac{1}{2}_{\dot{B}_{2,\infty}^1}\|f\|_{\dot{H}^2}^\frac{1}{2}. \]
    This ends the proof of Lemma \ref{interpolation-besov-other}.
\end{proof}

\subsection{Some basic inequalities for time decay rates}
\begin{lemma}
\label{lem:max_estimate}
Let $\alpha \ge 0$, $c > 0$. Then for all $t> 0$,
\[
\sup_{x > 0} x^{\alpha} e^{-c x t} \le C (1+t)^{-\alpha},
\]
where the constant $C = C(\alpha, c) > 0$ depends only on $\alpha$ and $c$.
\end{lemma}

\begin{proof}

Define $f(x) = x^{\alpha} e^{-c x t}$ for $x>0$. Differentiating,
\[
f'(x) = \alpha x^{\alpha-1} e^{-c x t} - c t x^{\alpha} e^{-c x t}
      = x^{\alpha-1} e^{-c x t} (\alpha - c t x).
\]
Setting $f'(x)=0$ yields the unique critical point $x_0 = \frac{\alpha}{c t}$. Since $f(0^+)=0$ (if $\alpha>0$; if $\alpha=0$ then $f(0^+)=1$) and $f(\infty)=0$, $x_0$ is the global maximum. Substituting,
\[
f(x_0) = \left(\frac{\alpha}{c t}\right)^{\alpha} e^{-\alpha}
       = \left(\frac{\alpha}{c e}\right)^{\alpha} t^{-\alpha}.
\]
Thus
\[
\sup_{x>0} f(x) = C_0 t^{-\alpha}, \qquad C_0 := \left(\frac{\alpha}{c e}\right)^{\alpha}.
\]
For $t \ge 1$, $t^{-\alpha} \le (1+t)^{-\alpha}$; for $0 < t < 1$, $t^{-\alpha} \le (1+t)^{-\alpha} (1+t)^{\alpha} \le 2^{\alpha} (1+t)^{-\alpha}$. Hence there exists a constant $C = \max\{C_0, 2^{\alpha}C_0\}$ such that
\[
\sup_{x>0} x^{\alpha} e^{-c x t} \le C (1+t)^{-\alpha},\qquad \text{for all } t \ge 0.
\]
This completes the proof of Lemma \ref{lem:max_estimate}.
\end{proof}
\begin{lemma} \label{lemma2.5}
Let $-3<\gamma<0, t>0$, then it holds that
    \[\int_0^t\int_{\mathbb{R}_{v}^3} e^{-\nu(v)(t-\tau)}(1+\tau)^{-\frac{1}{2}}\mu^\delta\,\mathrm{d}v\mathrm{d}\tau\lesssim(1+t)^{-\frac{1}{2}}\,.\]
   where $\mu^{\delta}$ denotes a modified Maxwellian function which typically takes the form $\mu^{\frac{1}{2}}(v)$ multiplied by a sufficiently fast decaying weight in $|v|$.
\end{lemma}
\begin{proof}
   We divide the integral with respect to $t$ into two parts. And Lemma \ref{lem:max_estimate} implies that
   \begin{align*}
       &\hspace{3mm}\int_0^t\int_{\mathbb{R}_v^3}e^{-\nu(v)(t-\tau)}(1+\tau)^{-\frac{1}{2}}\mu^{\delta}\,{\rm d}v{\rm d}\tau\\
       &\lesssim\int_0^\frac{t}{2}\int_{\mathbb{R}_v^3}e^{-\nu(v)(t-\tau)}\nu(v)(1+\tau)^{-\frac{1}{2}}\nu(v)^{-1}\mu^{\delta}\,{\rm d}v{\rm d}\tau+\int_\frac{t}{2}^t\int_{\mathbb{R}_v^3}e^{-\nu(v)(t-\tau)}(1+\tau)^{-\frac{1}{2}}\mu^{\delta}\,{\rm d}v{\rm d}\tau\\
       &\lesssim\int_0^\frac{t}{2}\int_{\mathbb{R}_v^3}(1+t-\tau)^{-1}(1+\tau)^{-\frac{1}{2}}\mu^{\delta'}\,{\rm d}v{\rm d}\tau+(1+\frac{t}{2})^{-\frac{1}{2}}\int_\frac{t}{2}^t\int_{\mathbb{R}_v^3}e^{-\nu(v)(t-\tau)}\mu^{\delta}\,{\rm d}v{\rm d}\tau\\
       &\lesssim(1+\frac{t}{2})^{-1}\int_0^\frac{t}{2}(1+\tau)^{-\frac{1}{2}}\,{\rm d}\tau+(1+t)^{-\frac{1}{2}}\int_{\mathbb{R}_v^3}\int_\frac{t}{2}^t e^{-\nu(v)(t-\tau)}\nu^{-1}(v)\nu(v)\mu^\delta\,{\rm d}\tau{\rm d}v\\
       &\lesssim(1+t)^{-\frac{1}{2}}\,.
   \end{align*}
   The last step holds because $\int_\frac{t}{2}^t e^{-\nu(v)(t-\tau)}\nu^{-1}(v)\,{\rm d}\tau\le1$. This ends the proof of Lemma \ref{lemma2.5}.
\end{proof}
\section{$\|f\|_{\widetilde{L}_v^2(\dot{B}^{\frac12}_{2,\infty}\cap\dot{B}^{\frac32}_{2,\infty})}$ estimates}\label{Sect-Besov}
In this section, we estimate the Besov norms of $f$ using the inner product method. 
The analysis is divided into two parts: the low-frequency part and the high frequency part. 

We begin with the following Cauchy problem:
\begin{align}\label{Cauchy}
    \begin{cases}
    \partial_t f +v\cdot \nabla_x f+\mathcal{L} f=\Gamma(f,f)+S,\quad(t,x,v)\in\mathbb{R^+}\times\mathbb{R}_x^3\times\mathbb{R}_v^3,\\[2mm]
    f(0,x,v)=f_0(x,v)=\frac{F_0(x,v)}{\sqrt{\mu}},\qquad\quad(x,v)\in\mathbb{R}_x^3\times\mathbb{R}_v^3.
\end{cases}
\end{align}
\subsection{The macroscopic estimates} Let $-3<\gamma<0$ and $N\ge5$ be an integer. Define $w_{\ell,\beta}:=\langle v\rangle^{\ell+\gamma|\beta|}$, here $\beta$ is the order of the derivative of $v$. Assume that $f$ is a classical solution to the Cauchy problem \eqref{Cauchy}. 
It satisfies the {\it a priori} estimate
  \begin{align}
\sup_t\|f\|_X+\|S\|_\mathcal{Z}\le\epsilon\,,
\label{aps}
\end{align}
    where we define 
    \[\|f\|_X:=\|f\|_{\widetilde{L}_v^2(\dot{B}_{2,\infty}^\frac{1}{2}\cap\dot{B}^\frac{3}{2}_{2,\infty})}+\|f\|_{L_v^2(\dot{H}^1\cap\dot{H}^N)}+\sum_{|\alpha|+|\beta|\le N}\left\|w_{\ell,\beta}\partial_x^\alpha\partial_v^\beta \mathbf{P}^{\bot}f\right\|\,,\]
    and $\epsilon>0$ is a constant small enough.

Before carrying out the estimate, we first perform a macroscopic estimate for $f$ in frequency space, which serves as a preparation for the subsequent decomposition. To this end, we need to derive the macroscopic equation for $f(t,x,v)$. Define moment functions $A_{mj}(f)$ and $B_j(f),~1\leq m,j\leq3$, as in \cite{D-2008-JDE},  
\begin{equation*}
A_{mj}(f)=\int_{\mathbb{ R}^3}( v_m v_j-1)\mu^\frac12fd v,\quad B_j(f)=\frac{1}{10}\int_{\mathbb{ R}^3}(| v|^2-5) v_j\mu^\frac12fd v\,,
\end{equation*}
then, we can derive the macroscopic equation is
\begin{equation}\label{Macro-equation}
\begin{cases}
\partial_ta+\nabla_x\cdot b=\int_{\mathbb{R}^3}\mu^{\frac{1}{2}}S\,\mathrm{d}v,\\
\partial_tb+\nabla_x(a+2c)+\nabla_x A({\bf P}^\bot f)=\int_{\mathbb{R}^3}v\mu^{\frac{1}{2}}S\,\mathrm{d}v,\\
\partial_tc+\frac13\nabla_x\cdot b+\frac53\nabla_x\cdot B({\bf P}^\bot f)=\frac{1}{6}\int_{\mathbb{R}^3}(|v|^2-3)\mu^{\frac{1}{2}}S\,\mathrm{d}v,
\end{cases}
\end{equation}
and
\begin{equation}\label{moment-equation}
\begin{cases}
\partial_tA_{mj}({\bf P}^\bot f)+\partial_{x_m}b_j+\partial_{x_j}b_m
-\frac23\delta_{mj}\nabla_x\cdot B({\bf P}^\bot f)
=A_{mj}(r+G),\\
\partial_t B_{j}({\bf P}^\bot f)+\partial_{x_j}c=B_j(r+G)
\end{cases}
\end{equation}
with
\begin{equation}\label{r-G}
  r=- v\cdot\nabla_x{\bf P}^\bot f-\mathcal{L}{\bf P}^\bot f,\quad G={\Gamma}(f,f)+S,\notag
\end{equation}
where $r$ is a linear term related only to the microscopic component $\{{\bf I-P}\}f$ and $G$ is a quadratic nonlinear term. Then we have the following lemma.
\begin{lemma}\label{Macro estimate for Low}
There exists an interactive functional $\mathcal{E}_{int,\frac{1}{2}}(t,\xi)$ satisfying
\begin{equation*}
\mathcal{E}_{int,\frac12}(t,\xi)\lesssim \left|\hat{f}\right|_{L_v^2}^2+|\xi|^2\left|\hat{f}\right|_{L_v^2}
\end{equation*}
such that
\begin{equation}\label{Lemma4.2-1}
\begin{aligned}
&\hspace{3mm}\frac{d}{dt}\mathcal{E}_{int,\frac12}(t,\xi)+|\xi|^2\left|{\bf P}\hat{f}\right|_{L_v^2}^2\\
&\lesssim |\xi|^2\left|{\bf P}^\bot\hat{f}\,\langle v\rangle^{\frac{\gamma}{2}}\right|_{L_v^2}^2+\left|{\bf P}^\bot\hat{f}\,\langle v\rangle^{\frac{\gamma}{2}}\right|_{L_v^2}^2+\left(\int_{\mathbb{R}^3}\hat{S}\mu^\delta\,\mathrm{d}v\right)^2+\left|\hat{\Gamma}(f,f)\,\langle v\rangle^{-\frac{\gamma}{2}}\right|_{L_v^2}^2\,.\notag
\end{aligned}
\end{equation}
\end{lemma}
\begin{proof}
     For fixed $m$, we have
    \begin{equation}
\begin{aligned}\label{Ori-E}
&\hspace{3mm}-\partial_t\left[\sum_j\partial_jA_{jm}({\bf P}^\bot f)+\frac 12\partial_mA_{mm}({\bf P}^\bot f)\right]-\Delta_x b_m-\partial_m\partial_m b_m\\
&=\frac 12\sum_{j\neq m}\partial_mA_{jj}(r)-\sum_j\partial_j A_{jm}(r),
\end{aligned}
\end{equation}
Applying Fourier transform to \eqref{Ori-E}, we get
\begin{equation}
\begin{aligned}\label{O-E-F}
&\hspace{3mm}-\partial_t\left[\sum_ji\xi_jA_{jm}({\bf P}^\bot\hat{f})+\frac 12i\xi_mA_{mm}({\bf P}^\bot\hat{f})\right]+|\xi|^2 \hat{b}_m+\xi_m^2 \hat{b}_m\\
&=\frac 12\sum_{j\neq m}i\xi_mA_{jj}(\hat{r}+\hat{G})-\sum_j i\xi_j A_{jm}(\hat{r}+\hat{G}),\notag
\end{aligned}
\end{equation}
next, multiply both sides of the above equation by $\bar{\hat{b}}_m$ and then take the real part, we get
\begin{align*}
&\hspace{3mm} -\sum_m\left\{\partial_t\left[\sum_j i\xi_jA_{jm}({\bf P}^\bot\hat{f})+\frac 12 i\xi_mA_{mm}({\bf P}^\bot\hat{f})\right]\bar{\hat{b}}_m\right\}+|\xi|^2\sum_m  \left|\hat{b}_m\right|^2+\sum_m \xi_m^2\left|\hat{b}_m\right|^2\\
&=-\sum_m\partial_t\left[\sum_j i\xi_jA_{jm}({\bf P}^\bot\hat{f})\bar{\hat{b}}_m+\frac 12 i\xi_mA_{mm}({\bf P}^\bot\hat{f})\bar{\hat{b}}_m\right]+2|\xi|^2\left|\hat{b}\right|^2\\
&\hspace{13mm}+\sum_m\left\{\left[\sum_j i\xi_jA_{jm}({\bf P}^\bot\hat{f})+\frac 12 i\xi_mA_{mm}({\bf P}^\bot\hat{f})\right]\partial_t\bar{\hat{b}}_m\right\}\\
&=\sum_m\left\{\bar{\hat{b}}_m\left[\frac 12\sum_{j\neq m}i\xi_mA_{jj}(\hat{r}+\hat{G})-\sum_ji\xi_j A_{jm}(\hat{r}+\hat{G}) \right]\right\}\,,
\end{align*}
furthermore, we have
\begin{align}\label{b}
&\hspace{3mm} -\sum_m\partial_t\left[\sum_j i\xi_jA_{jm}({\bf P}^\bot\hat{f})\bar{\hat{b}}_m+\frac 12 i\xi_mA_{mm}({\bf P}^\bot\hat{f})\bar{\hat{b}}_m\right]+2|\xi|^2\left|\hat{b}\right|^2\notag\\
&=\underbrace{\sum_m\left\{\bar{\hat{b}}_m\left[\frac 12\sum_{j\neq m}i\xi_mA_{jj}(\hat{r}+\hat{G})-\sum_ji\xi_j A_{jm}(\hat{r}+\hat{G}) \right]\right\}}_{L_1}\notag\\
&\hspace{13mm}-\underbrace{\sum_m\left\{\left[\sum_j i\xi_jA_{jm}({\bf P}^\bot\hat{f})+\frac 12 i\xi_mA_{mm}({\bf P}^\bot\hat{f})\right]\partial_t\bar{\hat{b}}_m\right\}}_{L_2}\,.
\end{align}
For $L_1$, recalling the definition of $A$, one has from the Young inequality that 
\begin{align*}
    i\xi_jA_{jm}(\hat{r})\bar{\hat{b}}_m&=\int_{\mathbb{R}^3} i(v_jv_m-1)\mu^\frac{1}{2} (-iv\cdot\xi{\bf P}^\bot \hat{f}-\mathcal{L}{\bf P}^\bot \hat{f})\,\mathrm{d}v\,\xi_j\bar{\hat{b}}_m\\
    &\lesssim\left(\int_{\mathbb{R}^3} i(v_jv_m-1)\mu^\frac{1}{2} (-iv\cdot\xi{\bf P}^\bot \hat{f}-\mathcal{L}{\bf P}^\bot \hat{f})\,\mathrm{d}v\right)^2+\eta|\xi|^2\left|\hat{b}_m\right|^2\\
    &\lesssim\left|\mu^\delta\right|_{L_v^2}^2\left||\xi|{\bf P}^\bot \hat{f}\,\langle v\rangle^{\frac{\gamma}{2}}\right|_{L_v^2}^2+\left|\mu^\delta\right|_{L_v^2}^2\left|{\bf P}^\bot \hat{f}\,\langle v\rangle^{\frac{\gamma}{2}}\right|_{L_v^2}^2+\eta|\xi|^2\left|\hat{b}_m\right|^2\\
    &\lesssim|\xi|^2\left|{\bf P}^\bot \hat{f}\,\langle v\rangle^{\frac{\gamma}{2}}\right|_{L_v^2}^2+\left|{\bf P}^\bot \hat{f}\,\langle v\rangle^{\frac{\gamma}{2}}\right|_{L_v^2}^2+\eta|\xi|^2\left|\hat{b}_m\right|^2\,,
\end{align*}
and
\begin{align*}
    i\xi_jA_{jm}(\hat{\Gamma}(f,f))\bar{\hat{b}}_m&=\int_{\mathbb{R}^3}i(v_jv_m-1)\mu^\frac{1}{2}\hat{\Gamma}(f,f)\,\mathrm{d}v\xi_j\bar{\hat{b}}_m\\
    &\lesssim\left(\int_{\mathbb{R}^3}i(v_jv_m-1)\mu^\frac{1}{2}\hat{\Gamma}(f,f)\,\mathrm{d}v\right)^2+\eta|\xi|^2\left|\hat{b}_m\right|^2\\
    &\lesssim\left|\hat{\Gamma}(f,f)\,\langle v\rangle^{-\frac{\gamma}{2}}\right|_{L_v^2}^2+\eta|\xi|^2\left|\hat{b}_m\right|^2\,.
\end{align*}
Similarly, we can also obtain that
\begin{align*}
    i\xi_jA_{jm}(\hat{S})\bar{\hat{b}}_m&\lesssim \left|A_{jm}(\hat{S})\right|^2+\eta|\xi|^2\left|\hat{b}_m\right|^22\lesssim\left|\int_{\mathbb{R}^3}\hat{S}\mu^{\delta}\,\mathrm{d}v\right|^2+\eta|\xi|^2\left|\hat{b}_m\right|^2\,.
\end{align*}
For $L_2$, applying Fourier transform to $\eqref{Macro-equation}_2$ yields that 
\begin{align}\label{a0}
    \partial_t\hat{b}+i\xi(\hat{a}+2\hat{c})+i\xi A({\bf P}^\bot \hat{f})=\int_{\mathbb{R}^3}\hat{S}\mu^\delta\,\mathrm{d}v\,,
\end{align}
then we get
\[-\partial_t\bar{\hat{b}}=-i\xi(\bar{\hat{a}}+2\bar{\hat{c}})-i\xi A\left({\bf P}^\bot \bar{\hat{f}}\right)-\int_{\mathbb{R}^3} \bar{\hat{S}}\mu^\delta\,\mathrm{d}v\,.\]
Substitute it into $L_2$, we get
\begin{align*}
    -i\xi_jA_{jm}({\bf P}^\bot\hat{f})\partial_t\bar{\hat{b}}_m&\lesssim i\xi_jA_{jm}({\bf P}^\bot\hat{f})\left( i\xi(\bar{\hat{a}}+2\bar{\hat{c}})+i\xi A\left({\bf P}^\bot \bar{\hat{f}}\right)+\int_{\mathbb{R}^3} \bar{\hat{S}}\mu^\delta\,\mathrm{d}v\right)\\
    &\lesssim|\xi|^2\left|A_{jm}({\bf P}^\bot\hat{f})\right|^2+\eta|\xi|^2\left|\hat{a}+2\hat{c}\right|^2+\eta\left|\int_{\mathbb{R}^3} \bar{\hat{S}}\mu^\delta\,\mathrm{d}v\right|^2\\
    &\lesssim|\xi|^2\left|{\bf P}^\bot\hat{f}\,\langle v\rangle^{\frac{\gamma}{2}}\right|_{L_v^2}^2+\eta|\xi|^2\left(\left|\hat{a}\right|^2+\left|\hat{c}\right|^2\right)+\eta\left|\int_{\mathbb{R}^3} \hat{S}\mu^\delta\,\mathrm{d}v\right|^2\,.
\end{align*}
Define 
\[\mathcal{E}_{int}^b(t,\xi):=\sum_j i\xi_jA_{jm}({\bf P}^\bot\hat{f})\bar{\hat{b}}_m+\frac 12 i\xi_mA_{mm}({\bf P}^\bot\hat{f})\bar{\hat{b}}_m\,.\]
Plugging all the related estimates into \eqref{b}, we obtain
\begin{align*}
    \frac{\mathrm{d}}{\mathrm{d}t}\mathcal{E}_{int}^b+|\xi|^2\left|\hat{b}\right|^2&\lesssim|\xi|^2\left|{\bf P}^\bot\hat{f}\,\langle v\rangle^{\frac{\gamma}{2}}\right|_{L_v^2}^2+\left|{\bf P}^\bot\hat{f}\,\langle v\rangle^{\frac{\gamma}{2}}\right|_{L_v^2}^2\\
    &\quad+\left|\hat{\Gamma}(f,f)\,\langle v\rangle^{-\frac{\gamma}{2}}\right|_{L_v^2}^2+\eta|\xi|^2\left(\left|\hat{a}\right|^2+\left|\hat{c}\right|^2\right)+\left(\int_{\mathbb{R}^3}\hat{S}\mu^\delta\,\mathrm{d}v\right)^2\,.
\end{align*}
We next focus on the macroscopic term $c$, applying the Fourier transform to $\eqref{moment-equation}_2$ yields that
\[\partial_tB_j({\bf P}^\bot\hat{f})+i\xi_j\hat{c}=B_j(\hat{r}+\hat{G})\,.\]
Multiply both sides of the above equation by $-i\xi\bar{\hat{c}}$, we have
\begin{align}\label{c}
    -\partial_t\left[B_j({\bf P}^\bot\hat{f})i\xi\bar{\hat{c}}\right]+|\xi|^2\left|\hat{c}\right|^2=-i\xi B_j(\hat{r}+\hat{G})\bar{\hat{c}}-i\xi B_j({\bf P}^\bot\hat{f})\partial_t\bar{\hat{c}}:=L_3+L_4\,.
\end{align}
Similar to the estimate of $L_1$, we have 
\begin{align*}
    -i\xi B_j(\hat{r})\bar{\hat{c}}&\lesssim\left|B_j(\hat{r})\right|^2+\eta|\xi|^2\left|\hat{c}\right|^2\\
    &\lesssim\left(\int_{\mathbb{R}^3}(|v|^2-5)v_j\mu^{\frac{1}{2}}(-iv\cdot\xi{\bf P}^\bot\hat{f}-\mathcal{L}{\bf P}^\bot\hat{f})\,\mathrm{d}v\right)^2+\eta|\xi|^2\left|\hat{c}\right|^2\\
    &\lesssim|\xi|^2\left|{\bf P}^\bot\hat{f}\,\langle v\rangle^{\frac{\gamma}{2}}\right|_{L_v^2}^2+\left|{\bf P}^\bot\hat{f}\,\langle v\rangle^{\frac{\gamma}{2}}\right|_{L_v^2}^2+\eta|\xi|^2\left|\hat{c}\right|^2\,,
\end{align*}
and
\begin{align*}
    -i\xi B_j(\hat{S})\bar{\hat{c}}\lesssim\left|B_j(\hat{S})\right|^2+\eta|\xi|^2\left|\bar{\hat{c}}\right|^2\lesssim\left(\int_{\mathbb{R}^3}\mu^\delta\hat{S}\,\mathrm{d}v\right)^2+\eta|\xi|^2\left|\hat{c}\right|^2\,.
\end{align*}
As to $L_3$, applying the Fourier transform to $\eqref{Macro-equation}_3$, we have
\[\partial_t\hat{c}+\frac{1}{3}i\xi\cdot\hat{b}+\frac{5}{3}i\xi\cdot B({\bf P}^\bot\hat{f})=\frac{1}{6}\int_{\mathbb{R}^3}(|v|^2-3)\mu^{\frac{1}{2}}\hat{S}\,\mathrm{d}v\,,\]
then we get
\[-\partial_t\bar{\hat{c}}=-\frac{1}{3}i\xi\cdot\bar{\hat{b}}-\frac{5}{3}i\xi\cdot B\left({\bf P}^\bot\bar{\hat{f}}\right)-\frac{1}{6}\int_{\mathbb{R}^3}(|v|^2-3)\mu^{\frac{1}{2}}\bar{\hat{S}}\,\mathrm{d}v\,.\]
Inserting it into $L_4$, we further get
\begin{align*}
    L_4&\lesssim i\xi B_j({\bf P}^\bot\hat{f})\left(-\frac{1}{3}i\xi\cdot\bar{\hat{b}}-\frac{5}{3}i\xi\cdot B\left({\bf P}^\bot\bar{\hat{f}}\right)-\frac{1}{6}\int_{\mathbb{R}^3}(|v|^2-3)\mu^{\frac{1}{2}}\bar{\hat{S}}\,\mathrm{d}v\right)\\
    &\lesssim|\xi|^2\left|B_j({\bf P}^\bot\hat{f})\right|^2+\eta|\xi|^2\left|\hat{b}\right|^2+\eta\left(\int_{\mathbb{R}^3}\mu^\delta\hat{S}\,\mathrm{d}v\right)^2\\
    &\lesssim|\xi|^2\left|{\bf P}^\bot\hat{f}\,\langle v\rangle^{\frac{\gamma}{2}}\right|_{L_v^2}^2+\eta|\xi|^2\left|\hat{b}\right|^2+\eta\left(\int_{\mathbb{R}^3}\mu^\delta\hat{S}\,\mathrm{d}v\right)^2\,.
\end{align*}
Define 
\[\mathcal{E}_{int}^c:=B_j({\bf P}^\bot\hat{f})i\xi\bar{\hat{c}}\,.\]
Plugging all the related estimates into \eqref{c}, it arrives at
\begin{align*}
    \frac{\mathrm{d}}{\mathrm{d}t}\mathcal{E}_{int}^c+|\xi|^2\left|\hat{c}\right|^2\lesssim|\xi|^2\left|{\bf P}^\bot\hat{f}\,\langle v\rangle^{\frac{\gamma}{2}}\right|_{L_v^2}^2+\left|{\bf P}^\bot\hat{f}\,\langle v\rangle^{\frac{\gamma}{2}}\right|_{L_v^2}^2+\eta|\xi|^2\left|\hat{b}\right|^2+\left(\int_{\mathbb{R}^3}\mu^\delta\hat{S}\,\mathrm{d}v\right)^2\,.
\end{align*}
For the macroscopic term $a$, multiply both sides of \eqref{a0} by $-i\xi\bar{\hat{a}}$, one has
\begin{align*}
    -i\xi\bar{\hat{a}}\,\partial_t\hat{b}+|\xi|^2(\hat{a}+2\hat{c})\bar{\hat{a}}+|\xi|^2A({\bf P}^\bot\hat{f})\bar{\hat{a}}=-\int_{\mathbb{R}^3}\hat{S}\mu^\delta\,\mathrm{d}v\,i\xi\bar{\hat{a}}\,,
\end{align*}
    this simplifies to
    \begin{align}\label{a}
        -\partial_t(i\xi\hat{b}\bar{\hat{a}})+|\xi|^2|\hat{a}|^2=\underbrace{-2|\xi|^2\hat{c}\,\bar{\hat{a}}-|\xi|^2A({\bf P}^\bot\hat{f})\,\bar{\hat{a}}-\int_{\mathbb{R}^3}\hat{S}\mu^\delta\,\mathrm{d}v\,i\xi\bar{\hat{a}}}_{L_5}-i\xi\hat{b}\partial_t\bar{\hat{a}}\,.
    \end{align}
Due to
\[|\xi|^2\hat{c}\,\bar{\hat{a}}\lesssim|\xi|^2|\hat{c}|^2+\eta|\xi|^2\left|\hat{a}\right|^2\,,\]
and 
\begin{align*}
    |\xi|^2A({\bf P}^\bot \hat{f})\,\bar{\hat{a}}&\lesssim|\xi|^2\left|A({\bf P}^\bot\hat{f})\right|^2+\eta|\xi|^2\left|\hat{a}\right|^2\\
    &\lesssim|\xi|^2\left|{\bf P}^\bot\hat{f}\,\langle  v \rangle^{\frac{\gamma}{2}}\right|_{L_v^2}^2+\eta|\xi|^2\left|\hat{a}\right|^2\,,
\end{align*}
and 
\begin{align*}
\int_{\mathbb{R}^3}\hat{S}\mu^\delta\,\mathrm{d}v\,i\xi\bar{\hat{a}}\lesssim\left(\int_{\mathbb{R}^3}\hat{S}\mu^\delta\,\mathrm{d}v\right)^2+\eta|\xi|^2\left|\hat{a}\right|^2\,.
\end{align*}
Inserting the above estimates into $L_5$, we arrives at
\[L_5\lesssim|\xi|^2|\hat{c}|^2+\eta|\xi|^2|\hat{a}|^2+|\xi|^2\left|{\bf P}^\bot\hat{f}\,\langle v\rangle^{\frac{\gamma}{2}}\right|_{L_v^2}^2+\left(\int_{\mathbb{R}^3}\hat{S}\mu^\delta\,\mathrm{d}v\right)^2\,.\]
Applying the Fourier transform to $\eqref{Macro-equation}_1$, one has
\[\partial_t\hat{a}+i\xi\cdot\hat{b}=\int_{\mathbb{R}^3}\mu^{\frac{1}{2}}\hat{S}\,\mathrm{d}v\,,\]
then
\[-\partial_t\bar{\hat{a}}=-i\xi\cdot\bar{\hat{b}}-\int_{\mathbb{R}^3}\mu^{\frac{1}{2}}\bar{\hat{S}}\,\mathrm{d}v\,.\]
Consequently,
\begin{align*}
    -i\xi\hat{b}\,\partial_t\bar{\hat{a}}&=|\xi|^2|\hat{b}|^2-i\xi\hat{b}\int_{\mathbb{R}^3}\mu^{\frac{1}{2}}\bar{\hat{S}}\,\mathrm{d}v\\
    &\lesssim|\xi|^2|\hat{b}|^2+\left(\int_{\mathbb{R}^3}\hat{S}\mu^\delta\,\mathrm{d}v\right)^2+\eta|\xi|^2|\hat{b}|^2\\
    &\lesssim|\xi|^2|\hat{b}|^2+\left(\int_{\mathbb{R}^3}\hat{S}\mu^\delta\,\mathrm{d}v\right)^2\,.
\end{align*}
Define
\[\mathcal{E}_{int}^a:=i\xi\,\hat{b}\,\bar{\hat{a}}\,.\]
Plugging all the related estimates into \eqref{a}, we obtain
\begin{align}
    \frac{\mathrm{d}}{\mathrm{d}t}\mathcal{E}_{int}^a+|\xi|^2|\hat{a}|^2\lesssim|\xi|^2|\hat{c}|^2+|\xi|^2\left|{\bf P}^\bot\hat{f}\,\langle v\rangle^{\frac{\gamma}{2}}\right|_{L_v^2}^2+\left(\int_{\mathbb{R}^3}\hat{S}\mu^\delta\,\mathrm{d}v\right)^2+|\xi|^2|\hat{b}|^2\,.\notag
\end{align}
Set
\[\mathcal{E}_{int,\frac{1}{2}}:=\mathcal{E}_{int}^b+k_1\mathcal{E}_{int}^c+k_2\mathcal{E}_{int}^a\,,\,0<k_2<k_1<1\,.\]
Due to the fact that
\[\left|{\bf P}\hat{f}\right|_{L_v^2}^2\lesssim|(\hat{a},\hat{b},\hat{c})|^2\,,\]
we then obtain the desired result. This completes the proof.
\end{proof}
\subsection{Low frequency and high frequency equations}We introduce operators $\mathscr{P}_L$ and $\mathscr{P}_H$ which decompose a function into its low and high frequency parts. Here $\mathscr{P}_L$ is defined by
\[\mathscr{P}_Lf:=\mathcal{F}^{-1}[\widehat{\chi}_1\hat{f}]\]
where we take $r_1<r_\infty, 0\le\widehat{\chi}_1(\xi)\le1$ and $\widehat{\chi}_1(\xi)\in C^\infty(\mathbb{R}^3)$
\[\widehat{\chi}_1(\xi)=\begin{cases}
    1,\quad|\xi|\le r_1\\
    0,\quad|\xi|\ge r_\infty\,,
\end{cases}\]

Similarly, we define $\widehat{\chi}_\infty(\xi):=1-\widehat{\chi}_1(\xi)$ and the operator $\mathscr{P}_Hf:=\mathcal{F}^{-1}[\widehat{\chi}_\infty\hat{f}]$.
We set $f=f_L+f_H$, with $f_L=\mathscr{P}_Lf, f_H=\mathscr{P}_Hf$. Applying the operators $\mathscr{P}_L$ and $\mathscr{P}_H$ to \eqref{f-perturbation}, we obtain
\begin{align}\label{low}
    \partial_tf_L+v\cdot\nabla_xf_L+\mathcal{L}f_L=\mathscr{P}_L\Gamma(f,f)+\mathscr{P}_LS\,,
\end{align}
and
\begin{align}\label{high}
    \partial_tf_H+v\cdot\nabla_xf_H+\mathcal{L}f_H=\mathscr{P}_H\Gamma(f,f)+\mathscr{P}_HS\,.\notag
\end{align}
We now proceed to estimate the low frequency part. With the help of the macroscopic estimate, we obtain the following lemma:
\begin{lemma}\label{lemma3.2}
    It holds that
  \begin{equation}\label{Low E-1}
    \begin{aligned}
    \|f_L\|_{\widetilde{L}_v^2(\dot{B}_{2,\infty}^{\frac{1}{2}})}^2&\lesssim\|f_0\|_{\widetilde{L}_v^2(\dot{B}_{2,\infty}^{\frac{1}{2}})}^2+\eta\sup_t\|\mathbf{P}^{\bot}f\,\langle v\rangle^{-\frac{\gamma}{2}}\|_{L_v^2(L_x^2\cap\dot{H}_x^1)}^2\\
    &\qquad+
    \epsilon^2\sup_t\|f\|_{\widetilde{L}_v^2(\dot{B}_{2,\infty}^{\frac{1}{2}})}^2+\sup_t\|S\|_{\widetilde{L}_v^2(\dot{B}_{2,\infty}^{-\frac{3}{2}}\cap\dot{B}_{2,\infty}^{-\frac{1}{2}})}^2\,.\notag
\end{aligned}
\end{equation}
  
\end{lemma}
\begin{proof}

Applying the Fourier transform to \eqref{low}, one has
\[\partial_{t}\hat{f}_L+i\xi\cdot v\hat{f}_L+\mathcal{L}\hat{f}_L=\widehat{\chi}_1\hat{\Gamma}(f,f)+\widehat{\chi}_1\hat{S}\,.\]

Multiply the above equation by $\bar{\hat{f}}_L$ and take the real part, we get:
\begin{align*}
\frac{1}{2}\frac{{\rm d}}{{\rm d}t}|\hat{f}_L|^2+\bar{\hat{f}}_L\mathcal{L}\hat{f}_L=\widehat{\chi}_1\hat{\Gamma}(f,f)\bar{\hat{f}}_L+\widehat{\chi}_1\hat{S}\bar{\hat{f}}_L\,.
\end{align*}
Integrating both sides with respect to $v$ yields that
\begin{align}\label{U1}
\frac{{\rm d}}{{\rm d}t}|\hat{f}_L|_{L_v^2}^2+\eta|\mathbf{P}^{\bot}\hat{f}_L|_\nu^2\lesssim\left|\widehat{\chi}_1\hat{\Gamma}(f,f)\,\langle v\rangle^{-\frac{\gamma}{2}}\right|_{L_v^2}^2+\int_{\mathbb{R}_v^3}\widehat{\chi}_1\hat{S}\hat{f}\,\mathrm{d}v\,,
\end{align}
where we have used the fact that
\begin{align*}
    \int_{\mathbb{R}_v^3} f\mathcal{L}f\,\mathrm{d}v\ge\eta\left|\mathbf{P}^{\bot}f\right|_\nu^2\,,
\end{align*}
 and 
\begin{align*}
    \int_{\mathbb{R}_v^3} \widehat{\chi}_1\hat{\Gamma}(f,f)\hat{f}_L\,\mathrm{d}v&=\int_{\mathbb{R}_v^3} \widehat{\chi}_1\hat{\Gamma}(f,f)\mathbf{P}\hat{f}_L\,\mathrm{d}v+\int_{\mathbb{R}_v^3} \widehat{\chi}_1\hat{\Gamma}(f,f)\mathbf{P}^{\bot}\hat{f}_L\,\mathrm{d}v=\int_{\mathbb{R}_v^3} \widehat{\chi}_1\hat{\Gamma}(f,f)\mathbf{P}^{\bot}\hat{f}_L\,\mathrm{d}v\\
     &\lesssim\left|\widehat{\chi}_1\hat{\Gamma}(f,f)\,\langle v\rangle^{-\frac{\gamma}{2}}\right|_{L_v^2}^2+\eta'\left|\mathbf{P}^{\bot}\hat{f}_L\,\langle v\rangle^{\frac{\gamma}{2}}\right|_{L_v^2}^2\,.
\end{align*}
Recall the macroscopic estimate in Lemma \ref{Macro estimate for Low}
\begin{align}\label{mm}
    \frac{\mathrm{d}}{\mathrm{d}t}\mathcal{E}_{int}(t,\xi)+|\xi|^2\left|\mathbf{P}\hat{f}_L\right|_{L_v^2}^2\lesssim\left|\mathbf{P}^{\bot}\hat{f}_L\right|_\nu^2+|\xi|^2\left|\mathbf{P}^{\bot}\hat{f}_L\right|_\nu^2+\left(\int_{\mathbb{R}_v^3} \widehat{\chi}_1\hat{S}\mu^\delta\,\mathrm{d}v\right)^2+\left|\widehat{\chi}_1\hat{\Gamma}(f,f)\,\langle v\rangle^{-\frac{\gamma}{2}}\right|_{L_v^2}^2\,,
\end{align}
then adding $\lambda\eqref{U1}+\eqref{mm}\,,\lambda_1:=\lambda\eta>1$ yields that
\begin{align}\label{3.15}
    \frac{\mathrm{d}}{\mathrm{d}t}\left|\hat{f}_L\right|_{L_v^2}^2+\lambda_1\left|{\bf P}^\bot \hat{f}_L\right|_\nu^2+|\xi|^2\left|{\bf P}\hat{f}_L\right|_{L_v^2}^2\lesssim\left|\widehat{\chi}_1\hat{\Gamma}(f,f)\,\langle v\rangle^{-\frac{\gamma}{2}}\right|_{L_v^2}^2+\int_{\mathbb{R}_v^3}\widehat{\chi}_1\hat{S}\hat{f}\,\mathrm{d}v+\left(\int_{\mathbb{R}_v^3}\widehat{\chi}_1\hat{S}\mu^\delta\,\mathrm{d}v\right)^2\,.
\end{align}
     \noindent Furthermore, due to
\begin{align}\label{000}
    \left|\mathbf{P}^{\bot}\hat{f}_L\right|_{L_v^2}^2\lesssim\left|\mathbf{P}^{\bot}\hat{f}_L\,\langle v\rangle^{\frac{\gamma}{2}}\right|_{L_v^2}^2+\eta_2\left|\mathbf{P}^{\bot}\hat{f}_L\,\langle v\rangle^{-\frac{\gamma}{2}}\right|_{L_v^2}^2\,,
\end{align}
\eqref{3.15} becomes
\begin{align}
    \frac{\mathrm{d}}{\mathrm{d}t}\left|\hat{f}_L\right|_{L_v^2}^2+|\xi|^2\left|\hat{f}_L\right|_{L_v^2}^2\lesssim\eta|\xi|^2\left|\mathbf{P}^{\bot}\hat{f}\,\langle v\rangle^{-\frac{\gamma}{2}}\right|_{L_v^2}^2+\left|\hat{\Gamma}(f,f)\langle v\rangle^{-\frac{\gamma}{2}}\right|_{L_v^2}^2+\int_{\mathbb{R}_v^3} \hat{S}\hat{f}\,\mathrm{d}v+\left(\int_{\mathbb{R}_v^3}\hat{S}\mu^\delta\,\mathrm{d}v\right)^2\,,\notag
\end{align}
Solving the above inequality directly gives that
\begin{align}\label{l-f}
    \left|\hat{f}_L\right|_{L_v^2}^2&\lesssim e^{-|\xi|^2t}\left|\hat{f_0}\right|_{L_v^2}^2+\eta\int_0^t e^{-\lambda|\xi|^2(t-\tau)}|\xi|^2\left|\mathbf{P}^{\bot}\hat{f}\,\langle v\rangle^{-\frac{\gamma}{2}}\right|_{L_v^2}^2\,\mathrm{d}\tau\notag\\
    &\quad+\int_0^t e^{-|\xi|^2(t-\tau)}\left|\hat{\Gamma}(f,f)\langle v\rangle^{-\frac{\gamma}{2}}\right|_{L_v^2}^2\,\mathrm{d}\tau+\int_0^t \int_{\mathbb{R}_v^3} e^{-|\xi|^2(t-\tau)}\hat{S}\hat{f}\,\mathrm{d}v\mathrm{d}\tau\notag\\
    &\quad+\int_0^t e^{-|\xi|^2(t-\tau)}\left(\int_{\mathbb{R}_v^3}\hat{S}\mu^\delta\,\mathrm{d}v\right)^2\mathrm{d}\tau\,.
\end{align}
Multiplying both sides of the above inequality by $\varphi_j^2(\xi)2^j$, then integrating the result with respect to $\xi$ and finally taking the supremum over $j\in\mathbb{Z}$ yields that
\begin{align}\label{U2}
    \|f_L\|_{\widetilde{L}_v^2(\dot{B}_{2,\infty}^{\frac{1}{2}})}^2&\lesssim\|f_0\|_{\widetilde{L}_v^2(\dot{B}_{2,\infty}^{\frac{1}{2}})}^2+\underbrace{\eta\sup_{j\in\mathbb{Z}}\int_0^t \int_{\mathbb{R}_\xi^3}e^{-|\xi|^2(t-\tau)}|\xi|^2\left|\mathbf{P}^{\bot}\hat{f}\,\langle v\rangle^{-\frac{\gamma}{2}}\right|_{L_v^2}^2\varphi_j^22^j\,\mathrm{d}\xi\mathrm{d}\tau}_{Q_1}\notag\\
    &\quad\underbrace{+\sup_{j\in\mathbb{Z}}\int_0^t \int_{\mathbb{R}_\xi^3}e^{-|\xi|^2(t-\tau)}\left|\hat{\Gamma}(f,f)\langle v\rangle^{-\frac{\gamma}{2}}\right|_{L_v^2}^2\varphi_j^22^j\,\mathrm{d}\xi\mathrm{d}\tau}_{Q_2}\notag\\
    &\quad+\underbrace{\sup_{j\in\mathbb{Z}}\int_0^t \int_{\mathbb{R}_v^3} e^{-|\xi|^2(t-\tau)}\hat{S}\hat{f}\varphi_j^2\,2^j\,\mathrm{d}\xi\mathrm{d}v\mathrm{d}\tau}_{Q_3}\notag\\
    &\quad+\underbrace{\sup_{j\in\mathbb{Z}}\int_0^t\int_{\mathbb{R}_{\xi}^3} e^{-|\xi|^2(t-\tau)}\left(\int_{\mathbb{R}_v^3}\hat{S}\mu^\delta\,\mathrm{d}v\right)^2\varphi_j^2\,2^j\,\mathrm{d}\xi\mathrm{d}\tau}_{Q_4}\,.
\end{align}
     Since $L^2\cap\dot{H}^N\subset\dot{B}_{2,\infty}^\frac{1}{2}$, it's straightforward to change $Q_1$ into
     \begin{align*}
         Q_1&\lesssim\eta\sup_{j\in\mathbb{Z}}\int_0^t e^{-2^{2j}(t-\tau)}2^{2j}\mathrm{d}\tau\sup_t\sup_{j\in\mathbb{Z}}\int_{\mathbb{R}_\xi^3}\left|\varphi_j\,2^{\frac{j}{2}}\mathbf{P}^{\bot}\hat{f}\langle v\rangle^{-\frac{\gamma}{2}}\right|_{L_v^2}^2\,\mathrm{d}\xi\\
         &\lesssim\eta\sup_t\int_{\mathbb{R}_{v}^3}\sup_{j\in\mathbb{Z}}\int_{\mathbb{R}_\xi^3}\left|\varphi_j\,2^{\frac{j}{2}}\mathbf{P}^{\bot}\hat{f}\langle v\rangle^{-\frac{\gamma}{2}}\right|^2\,\mathrm{d}\xi\mathrm{d}v\\&\lesssim\eta\sup_t\left\|\mathbf{P}^{\bot}f\langle v\rangle^{-\frac{\gamma}{2}}\right\|_{L_v^2(\dot{B}_{2,\infty}^{\frac{1}{2}})}^2\lesssim\eta\sup_t\left\|\mathbf{P}^{\bot}f\langle v\rangle^{-\frac{\gamma}{2}}\right\|_{L_v^2(L_x^2\cap\dot{H}^N_x)}\,.
     \end{align*}
     For $Q_2$, one can derive from Lemma \ref{nonlinear} that
     \begin{align*}
         \left|\Gamma(f,g)\langle v\rangle^{-\frac{\gamma}{2}}\right|_{L_v^2}\lesssim|f|_{H_v^2}|g|_{L_v^2}\,,
     \end{align*}
 then by using \eqref{newinterpolation}, we have
 \begin{align*}
     Q_2&\lesssim\sup_{j\in\mathbb{Z}}\int_0^t \int_{\mathbb{R}_\xi^3}e^{-|\xi|^2(t-\tau)}|\xi|^2\left|\hat{\Gamma}(f,f)\langle v\rangle^{-\frac{\gamma}{2}}\right|_{L_v^2}^2\varphi_j^22^{-j}\,\mathrm{d}\xi\mathrm{d}\tau\\
     &\lesssim\sup_{j\in\mathbb{Z}}\int_0^t e^{-2^{2j}(t-\tau)}2^{2j}\,\mathrm{d}\tau\sup_t\sup_{j\in\mathbb{Z}}\int_{\mathbb{R}_\xi^3}\left|\left|f\right|_{H_v^2}\left|f\right|_{L_v^2}\right|^2\varphi_j^2\,2^{-j}\,\mathrm{d}\xi\lesssim\sup_t\left|\left|f\right|_{H_v^2}\left|f\right|_{L_v^2}\right|_{\dot{B}_{2,\infty}^{-\frac{1}{2}}}^2\\
     &\lesssim\sup_t\|f\|_{\widetilde{H}_v^2(\dot{B}_{2,\infty}^{\frac{1}{2}})}^2\|f\|_{\widetilde{L}_v^2(\dot{B}_{2,\infty}^{\frac{1}{2}})}^2\lesssim\epsilon^2\sup_t\|f\|_{\widetilde{L}_v^2(\dot{B}_{2,\infty}^{\frac{1}{2}})}^2\,.
 \end{align*}   
 The Cauchy inequality implies that
 \begin{align*}
     \int_{\mathbb{R}_v^3 }\hat{S}\hat{f}\,\mathrm{d}v\lesssim|\hat{S}|_{L_v^2}|\hat{f}|_{L_v^2}\,,
 \end{align*}
 then we obtain
 \begin{align*}
     Q_3&\lesssim\int_0^t \sup_{j\in\mathbb{Z}}\int_{\mathbb{R}_{\xi,v}^6} e^{-|\xi|^2(t-\tau)}|\xi|^2\hat{S}\hat{f}\varphi_j^2\,2^{-j}\,\mathrm{d}\xi\mathrm{d}v\mathrm{d}\tau\\
     &\lesssim\sup_t\sup_{j\in\mathbb{Z}}\int_{\mathbb{R}_\xi^3}2^{-\frac{3j}{2}}\varphi_j\left|\hat{S}\right|_{L_v^2}\left|\hat{f}\right|_{L_v^2}\varphi_j\,2^{\frac{j}{2}}\,\mathrm{d}\xi\lesssim\sup_t\|S\|_{\widetilde{L}_v^2(\dot{B}_{2,\infty}^{-\frac{3}{2}})}^2+\eta\sup_t\|f\|_{\widetilde{L}_v^2(\dot{B}_{2,\infty}^{\frac{1}{2}})}^2\,.
 \end{align*}
 As to $Q_4$, due to
 \begin{align*}
     \left(\int_{\mathbb{R}_v^3}\hat{S}\mu^\delta\,\mathrm{d}v\right)^2\lesssim\left|\mu^\delta\right|_{L_v^2}^2\left|\hat{S}\right|_{L_v^2}^2\lesssim\left|\hat{S}\right|_{L_v^2}^2\,,
 \end{align*}
 then we have
 \begin{align*}
     Q_4&\lesssim\int_0^t\sup_{j\in\mathbb{Z}}\int_{\mathbb{R}_{\xi^3}} e^{-|\xi|^2(t-\tau)}|\xi|^2|\xi|^{-2}\left|\hat{S}\right|_{L_v^2}^2\varphi_j^2\,2^j\,\mathrm{d}\xi\mathrm{d}\tau\lesssim\sup_t\sup_{j\in\mathbb{Z}}\int_{\mathbb{R}_{\xi}^3}2^{-j}\varphi_j^2\left|\hat{S}\right|_{L_v^2}^2\,\mathrm{d}\xi\lesssim\sup_t\|S\|_{\widetilde{L}_v^2(\dot{B}_{2,\infty}^{-\frac{1}{2}})}^2\,.
 \end{align*}
Plugging all the estimates into \eqref{U2}, it completes the proof of Lemma \ref{lemma3.2}.
\end{proof}
Next, we consider the high frequency part. The lemma is stated as follows.
\begin{lemma}\label{lemma3.3}
It holds that
\begin{align}
    \|f_H\|_{\widetilde{L}_v^2(\dot{B}_{2,\infty}^{\frac{1}{2}}\cap\dot{B}_{2,\infty}^{\frac{3}{2}})}^2&\lesssim\|f_0\|_{\widetilde{L}_v^2(\dot{B}_{2,\infty}^{\frac{1}{2}}\cap\dot{B}_{2,\infty}^{\frac{3}{2}})}^2+\eta\sup_t\|\mathbf{P}^{\bot}f\,\langle v\rangle^{-\frac{\gamma}{2}}\|_{L_v^2(L_x^2\cap\dot{H}_x^2)}^2\notag\\
    &\qquad+\epsilon^2\sup_t\|f\|_{L_v^2(\dot{H}^1\cap\dot{H}^2)}^2+
    \epsilon^2\sup_t\|f\|_{\widetilde{L}_v^2(\dot{B}_{2,\infty}^{\frac{3}{2}})}^2+\sup_t\|S\|_{\widetilde{L}_v^2(\dot{B}_{2,\infty}^{\frac{1}{2}})}^2\,.\notag
\end{align}

\end{lemma}
\begin{proof}

Similar to \eqref{U1}, we have
\begin{align*}
    &\hspace{3mm}\frac{\rm d}{{\rm d}t}\left(\left|\hat{f}_H\right|_{L_v^2}^2+\left||\xi|\hat{f}_H\right|_{L_v^2}^2\right)+\left(\left|\mathbf{P}^{\bot}\hat{f}_H\,\langle v\rangle^{\frac{\gamma}{2}}\right|_{L_v^2}^2+\left||\xi|\mathbf{P}^{\bot}\hat{f}_H\,\langle v\rangle^{\frac{\gamma}{2}}\right|_{L_v^2}^2\right)\\
    &\lesssim\left|\hat{\Gamma}(f,f)\,\langle v\rangle^{-\frac{\gamma}{2}}\right|_{L_v^2}^2+\left||\xi|\,\hat{\Gamma}(f,f)\,\langle v\rangle^{-\frac{\gamma}{2}}\right|_{L_v^2}^2+\int_{\mathbb{R}_v^3} \hat{S}\hat{f}\,\mathrm{d}v+\int_{\mathbb{R}_v^3} \hat{S}\hat{f}\,|\xi|^2\,\mathrm{d}v\,.
\end{align*}
Combining the estimate of macroscopic part:
\begin{align*}
    \frac{\rm d}{{\rm d}t}\mathcal{E}_{int}(t,\xi)+|\xi|^2\left|P\hat{f}_H\right|_{L_v^2}^2\lesssim\left|\mathbf{P}^{\bot}\hat{f}_H\,\langle v\rangle^{\frac{\gamma}{2}}\right|_{L_v^2}^2+|\xi|^2\left|\mathbf{P}^{\bot}\hat{f}_H\,\langle v\rangle^{\frac{\gamma}{2}}\right|_{L_v^2}^2+\left(\int_{\mathbb{R}_v^3} { \bf P }\hat{S}\,\mathrm{d}v\right)^2+\left|\hat{\Gamma}(f,f)\,\langle v\rangle^{-\frac{\gamma}{2}}\right|_{L_v^2}^2\,,
\end{align*}
we get from \eqref{000} that
\begin{align*}
    &\hspace{3mm}\frac{\rm d}{{\rm d}t}\left(\left|\hat{f}_H\right|_{L_v^2}^2+\left||\xi|\hat{f}_H\right|^2_{L_v^2}\right)+\left|\hat{f}_H\right|_{L_v^2}^2+\left||\xi|\hat{f}_H\right|^2_{L_v^2}\\
    &\lesssim\eta\left||\xi|\,\langle v\rangle^{-\frac{\gamma}{2}}\mathbf{P}^{\bot}\hat{f}\right|_{L_v^2}^2+\left|\hat{\Gamma}(f,f)|\xi|\,\langle v\rangle^{-\frac{\gamma}{2}}\right|_{L_v^2}^2+\int_{\mathbb{R}_v^3} \hat{S}\hat{f}\,|\xi|^2\,\mathrm{d}v+\left(\int_{\mathbb{R}_v^3} P\hat{S}\,\mathrm{d}v\right)^2\,.
\end{align*}
Solving the above inequality directly gives that
\begin{align}\label{h-f}
&\hspace{3mm}\left|\hat{f}_H\right|_{L_v^2}^2(t,\xi)+\left||\xi|\,\hat{f}_H\right|_{L_v^2}^2(t,\xi)\notag\\
&\lesssim\left(\left|\hat{f}(0,\xi,v)\right|_{L_v^2}^2+\left||\xi|\,\hat{f}(0,\xi,v)\right|_{L_v^2}^2\right)e^{-t}
    +\eta\int_0^t e^{-(t-\tau)}\left||\xi|\,\langle v\rangle^{-\frac{\gamma}{2}}\,\mathbf{P}^{\bot}\hat{f}\right|_{L_v^2}^2\,\mathrm{d}\tau\notag\\
    &\qquad+\int_0^t e^{-(t-\tau)}\left|\hat{\Gamma}(f,f)|\xi|\,\langle v\rangle^{-\frac{\gamma}{2}}\right|_{L_v^2}^2\,\mathrm{d}\tau
    +\int_0^t e^{-(t-\tau)}\left||\xi|\hat{S}\right|_{L_v^2}\left||\xi|\hat{f}\right|_{L_v^2}\,\mathrm{d}\tau\notag\\
    &\qquad+\int_0^t e^{-(t-\tau)}\left|\hat{S}\right|_{L_v^2}^2\,\mathrm{d}\tau\,,
\end{align}
multiplying the above inequality by $\varphi_j^2\,2^j$, then integrating the result over $\mathbb{R}_{\xi}^3$ and taking the supremum over $j\in\mathbb{Z}$, one has
\begin{align}\label{hh}
    \|f_H\|_{\widetilde{L}_v^2(\dot{B}_{2\,\infty}^{\frac{1}{2}}\cap\dot{B}_{2\,\infty}^{\frac{3}{2}})}^2&\lesssim\|f_0\|_{\widetilde{L}_v^2(\dot{B}_{2\,\infty}^{\frac{1}{2}}\cap\dot{B}_{2\,\infty}^{\frac{3}{2}})}^2
    +\underbrace{\eta\sup_{j\in\mathbb{Z}}\int_0^t\int_{\mathbb{R}_{\xi}^3} e^{-(t-\tau)}\left||\xi|\,\langle v\rangle^{-\frac{\gamma}{2}}\,\mathbf{P}^{\bot}\hat{f}\varphi_j\right|_{L_v^2}^2\,2^j\,\mathrm{d}\xi\mathrm{d}\tau}_{Q_5}\notag\\
    &\quad+\underbrace{\sup_{j\in\mathbb{Z}}\int_0^t\int_{\mathbb{R}_{\xi}^3}  e^{-(t-\tau)}\left|\hat{\Gamma}(f,f)|\xi|\,\langle v\rangle^{-\frac{\gamma}{2}}\varphi_j\right|_{L_v^2}^2\,2^j\,\mathrm{d}\xi\mathrm{d}\tau}_{Q_6}\notag\\
    &\quad+\underbrace{\sup_{j\in\mathbb{Z}}\int_0^t \int_{\mathbb{R}_{\xi}^3} e^{-(t-\tau)}\left|\varphi_j|\xi|\hat{S}\right|_{L_v^2}\left|\varphi_j|\xi|\hat{f}\right|_{L_v^2}\,2^j\,\mathrm{d}\xi\mathrm{d}\tau}_{Q_7}\notag\\
    &\quad+\underbrace{\sup_{j\in\mathbb{Z}}\int_0^t \int_{\mathbb{R}_{\xi}^3} e^{-(t-\tau)}\left|\hat{S}\right|_{L_v^2}^2\,\varphi_j^2\,2^j\,\mathrm{d}\xi\mathrm{d}\tau}_{Q_8}\,,
\end{align}
Since $e^{-(t-\tau)}$ is integrable with respect to $\tau$, we don't need to match a $|\xi|^2$ to ensure the integrability. It is straightforward that
\begin{align*}
    Q_5&\lesssim\eta\sup_t\sup_{j\in\mathbb{Z}}\int_{\mathbb{R}_\xi^3}\left|2^{\frac{3j}{2}}\,\varphi_j{\bf P}^\bot \hat{f}\,\langle v\rangle^{-\frac{\gamma}{2}}\right|_{L_v^2}^2\,\mathrm{d}\xi\int_0^t e^{-(t-\tau)}\,\mathrm{d}\tau\\
    &\lesssim\sup_t\left\|{\bf P}^\bot f\,\langle v\rangle^{-\frac{\gamma}{2}}\right\|_{L_v^2(\dot{B}_{2,\infty}^{\frac{3}{2}})}^2\lesssim\sup_t\left\|{\bf P}^\bot f\,\langle v\rangle^{-\frac{\gamma}{2}}\right\|_{L_v^2(L_x^2\cap\dot{H}_x^N)}^2\,.
\end{align*}
And for $Q_6$, \eqref{111} implies that
\begin{align*}
    Q_6&\lesssim\sup_t\left||f|_{H_v^2}|f|_{L_v^2}\right|_{\dot{B}_{2,\infty}^{\frac{3}{2}}}^2\\
    &\lesssim\sup_t\|f\|_{L_x^\infty H_v^2}^2\|f\|_{\dot{B}_{2,\infty}^{\frac{3}{2}}L_v^2}^2+\sup_t\|f\|_{\dot{B}_{2,\infty}^{\frac{3}{2}}H_v^2}^2\|f\|_{L_x^\infty L_v^2}^2\\
    &\lesssim\sup_t\|f\|_{H_v^2(\dot{H}^1\cap\dot{H}^2)}^2\|f\|_{\dot{B}_{2,\infty}^{\frac{3}{2}}L_v^2}^2+\sup_t\|f\|_{\dot{B}_{2,\infty}^{\frac{3}{2}}H_v^2}^2\|f\|_{L_v^2(\dot{H}^1\cap\dot{H}^2)}^2\\
    &\lesssim\epsilon^2\sup_t\|f\|_{\dot{B}_{2,\infty}^{\frac{3}{2}}L_v^2}^2+\epsilon^2\|f\|^2_{L_v^2(\dot{H}^1\cap\dot{H}^2)}\lesssim\epsilon^2\sup_t\|f\|_{\widetilde{L}_v^2(\dot{B}_{2,\infty}^\frac{3}{2})}+\epsilon^2\|f\|^2_{L_v^2(\dot{H}^1\cap\dot{H}^2)}\,.
\end{align*}
By applying the Cauchy and the Young inequality, we have
\begin{align*}
    Q_7&\lesssim\sup_t\sup_{j\in\mathbb{Z}}\int_{\mathbb{R}_\xi^3}\left|\,2^j\varphi_j\hat{S}\right|_{L_v^2}\left|2^{\frac{3}{2}j}\varphi_j\hat{f}\right|_{L_v^2}\,\mathrm{d}\xi\lesssim\sup_t\|S\|_{\widetilde{L}_v^2(\dot{B}_{2,\infty}^{\frac{1}{2}})}^2+\eta\sup_t\|f\|_{\widetilde{L}_v^2(\dot{B}_{2,\infty}^{\frac{3}{2}})}^2\,.
\end{align*}
As to $Q_8$, we have
\begin{align*}
    Q_8\lesssim\sup_t\|S\|_{\widetilde{L}_v^2(\dot{B}_{2,\infty}^{\frac{1}{2}})}^2\,.
\end{align*}
Plugging all the estimates into \eqref{hh}, we obtain the desired result.
\end{proof}
Since 
\[\|f\|_{\widetilde{L}_v^2(\dot{B}_{2,\infty}^\frac{1}{2}\cap\dot{B}_{2,\infty}^{\frac{3}{2}})}\lesssim\|f_L\|_{\widetilde{L}_v^2(\dot{B}_{2,\infty}^\frac{1}{2})}+\|f_H\|_{\widetilde{L}_v^2(\dot{B}_{2,\infty}^\frac{1}{2}\cap\dot{B}_{2,\infty}^\frac{3}{2})}\,.\]
 Then combining Lemma \ref{lemma3.2} and Lemma \ref{lemma3.3}, we finally obtain the following proposition.
\begin{proposition}\label{besov}
It holds that
    \begin{align}
        \sup_t\|f\|_{\widetilde{L}_v^2(\dot{B}_{2,\infty}^{\frac{1}{2}}\cap\dot{B}_{2,\infty}^{\frac{3}{2}})}^2&\lesssim\|f_0\|_{\widetilde{L}_v^2(\dot{B}_{2,\infty}^{\frac{1}{2}}\cap\dot{B}_{2,\infty}^{\frac{3}{2}})}^2+\eta\sup_t\|\mathbf{P}^{\bot}f\,\langle v\rangle^{-\frac{\gamma}{2}}\|_{L_v^2(L_x^2\cap\dot{H}^2_x)}^2\notag\\
    &\qquad+
    \sup_t\|S\|_{\widetilde{L}_v^2(\dot{B}_{2,\infty}^{-\frac{3}{2}}\cap\dot{B}_{2,\infty}^{\frac{1}{2}})}^2\,.\notag
    \end{align}

\end{proposition}

\section{$\dot{H}_{x}^NL^2_v$-inner-estimate without weight}\label{Sect-noweight}
This section is devoted to establishing the unweighted {\it a prior} bounds for $f$. We derive the estimates for $\|\partial_x^\alpha f\|^2$ with $|\alpha|=1,2,\dots,N$ and the $L^2$ estimates only for ${\bf P}^\bot f$. We first state a macroscopic estimate that will be used in the subsequent analysis.
\begin{lemma}[Higher order macro estimate] 
    Let $-3<\gamma<0$ and assume that $f(t,x,v)$ is the classical solution to the system \eqref{Cauchy}, then there exists an interactive function $\mathcal{G}(t)$ satisfying 
    \[\mathcal{G}(t)\lesssim\left\|f\right\|_{\dot{H}_x^1\cap\dot{H}_x^N(L_v^2)}^2\]
    such that
    \begin{equation}\label{M-E-4}
        \begin{aligned}
            \frac{\mathrm{d}}{\mathrm{d}t}\mathcal{G}(t)+\sum_{2\le|\alpha|\le N}\left\|\partial_x^{\alpha}\mathbf{P}f\right\|_{L_{x,v}^2}^2\lesssim \sum_{1\le|\alpha|\le N}\left\|\partial_x^{\alpha}\mathbf{P}^{\bot}f\right\|_{L_{x,v}^2}^2+\sum_{1\le|\alpha|\le N-1}\left\|\partial_x^\alpha S\right\|_{L_{x,v}^2}^2\,.
        \end{aligned}
    \end{equation}
\end{lemma}
\begin{proof}
    The proof is similar to Lemma \ref{Macro estimate for Low}, so we omit it.
\end{proof}
We now focus on the $L^2$ bounds for $\|\partial_x^\alpha f\|$ with $|\alpha|=1,2,\dots, N$, we obtain the following result.
\begin{lemma}\label{lemma4.2}
    Let $N\ge5$, assume $f(t,x.v)$ is the classical solution of the system \eqref{Cauchy}, then it holds that
    \begin{equation}\label{H1}
        \begin{aligned}
             \|f\|_{L_v^2(\dot{H}^1\cap\dot{H}^N)}^2&\lesssim\|f_0\|_{L_v^2(\dot{H}^1\cap\dot{H}^N)}^2+\eta\sup\limits_{t}\left\|\mathbf{P}^{\bot}f \,\langle v\rangle^{-\frac{\gamma}{2}}\right\|_{L_v^2(\dot{H}^1\cap\dot{H}^N)}^2\\
             &\qquad+\sup\limits_{t}\|S\|_{L_v^2(\dot{H}^1\cap\dot{H}^N)}^2+\sup\limits_{t}|(a,b,c)|_{\dot{B}_{2,\infty}^{\frac{1}{2}}}^2\,,\notag
        \end{aligned}
    \end{equation}
    where $\epsilon,\eta>0$ are small constants.
\end{lemma}
\begin{proof}
    Acting the operator $\partial_x^{\alpha}$ with $1\le|\alpha|\le N$ in both sides of $\eqref{f-perturbation}$, then multiply the result by $\partial_x^{\alpha}f$ in both sides and integrating over $\mathbb{R}_x^3\times\mathbb{R}_v^3$ yields
    \[\frac{1}{2}\frac{\mathrm{d}}{\mathrm{d}t}\left\|\partial_x^{\alpha}f\right\|_{L_{x,v}^2}^2+\int_{\mathbb{R}_{x,v}^6} \mathcal{L}(\partial_x^{\alpha}f)\partial_x^{\alpha}f \,\mathrm{d}v\mathrm{d}x=\underbrace{\int_{\mathbb{R}_{x,v}^6}  \partial_x^{\alpha}\Gamma(f,f)\partial_x^{\alpha}f\,\mathrm{d}v\mathrm{d}x}_{I_1}+\underbrace{\int_{\mathbb{R}_{x,v}^6} \partial_x^{\alpha}S\,\partial_x^{\alpha}f \,\mathrm{d}v\mathrm{d}x}_{I_2}\,.\]
Recall
\begin{align}\label{D2}
    \int_{\mathbb{R}_{v}^3} f\mathcal{L}f \,\mathrm{d}v\ge\delta_0\left|\mathbf{P}^{\bot}f\right|_{\nu}^2\,,
\end{align}
where $\delta_0>0$, and the interpolation
\[\left\|\partial_x^{\alpha}\mathbf{P}^{\bot}f \right\|^2\lesssim\left\|\partial_x^{\alpha}\mathbf{P}^{\bot}f \,\langle v\rangle^{\frac{\gamma}{2}}\right\|^2+\eta_1\left\|\partial_x^{\alpha}\mathbf{P}^{\bot}f \,\langle v\rangle^{-\frac{\gamma}{2}}\right\|^2\,,\]
we obtain
    \begin{equation}\label{B1}
        \frac{\mathrm{d}}{\mathrm{d}t}\left\|\partial_x^{\alpha}f\right\|^2+\left\|\partial_x^{\alpha}\mathbf{P}^{\bot}f \right\|^2\lesssim\eta_2\left\|\partial_x^{\alpha}\mathbf{P}^{\bot}f \,\langle v\rangle^{-\frac{\gamma}{2}}\right\|^2+I_3+I_4\,.
    \end{equation}
For $I_1$, Lemma \ref{lemma5.2} implies that
\begin{align}\label{zz1}
    I_1\lesssim\int_{\mathbb{R}_x^3} \left(\left|\mu^\delta \partial_x^{\alpha'}f\right|_{H_v^2}+|\partial_x^{\alpha'}f|_{L_v^2}\right)|\partial_x^{\alpha-\alpha'}f|_{\nu}|\partial_x^{\alpha}f|_{\nu}\,\mathrm{d}x\,,
\end{align}
or
\begin{align}\label{zz2}
    I_1\lesssim\int_{\mathbb{R}_x^3} \left(\left|\mu^\delta \partial_x^{\alpha-\alpha'}f\right|_{H_v^2}+|\partial_x^{\alpha-\alpha'}f|_{L_v^2}\right)|\partial_x^{\alpha'}f|_{\nu}|\partial_x^{\alpha}f|_{\nu}\,\mathrm{d}x\,.
\end{align}
Next, we distinguish several cases. If $|\alpha'|=0$, then we choose the form \eqref{zz1}, and apply the {\it a priori} assumption \eqref{aps} to obtain
\begin{align*}
    I_1\lesssim\|\mu^\delta f\|_{L_x^\infty H_v^2}\|\partial_x^{\alpha}f\|^2+\|f\|_{L_x^\infty L_v^2}\|\partial_x^{\alpha}f\|^2\lesssim\|f\|_{H_v^2(\dot{H}^1\cap\dot{H}^2)}\|\partial_x^{\alpha}f\|^2\lesssim\epsilon\|\partial_x^{\alpha}f\|^2\,.
\end{align*}
If $|\alpha'|=|\alpha|$, we instead use \eqref{zz2}, and the estimate is identical to that in the first case, that is
\[I_1\lesssim\epsilon\|\partial_x^{\alpha}f\|^2.\]
If $0<|\alpha'|<N$, we apply the $L^3-L^6-L^2$ estimate, one has
\begin{align*}
    I_1&\lesssim\|\mu^\delta\partial_x^{\alpha'}f\|_{L_x^3H_v^2}\|\partial_x^{\alpha-\alpha'}f\|_{L_x^6L_v^2}\|\partial_x^{\alpha}f\|\\
    &\lesssim\|f\|_{H_v^2(\dot{H}^{|\alpha'|+1}_x\cap\dot{H}_x^{|\alpha'|})}\|f\|_{L_v^2\dot{H}_x^{|\alpha|-|\alpha'|+1}}\|\partial_x^{\alpha}f\|^2\lesssim\epsilon^2\|f\|_{L_v^2(\dot{H}^1\cap\dot{H^N})}^2+\eta\|\partial_x^\alpha f\|^2
\end{align*}
As to $I_2$, it follows from Cauchy's inequality and Young;s inequality that
\begin{align*}
I_4\lesssim\left\|\partial_x^{\alpha}S\right\|\left\|\partial_x^{\alpha}f\right\|\lesssim \left\|\partial_x^{\alpha}S\right\|^2+\eta\left\|\partial_x^{\alpha}f\right\|^2\,.
\end{align*}
Inserting the estimates of $I_3$ and $I_4$ into \eqref{B1}, summing all the results for $1\le|\alpha|\le N$ and combining it with \eqref{M-E-4} linearly, we get
    \begin{align}\label{B2}
        &\hspace{3mm}\frac{\mathrm{d}}{\mathrm{d}t}\sum_{1\le|\alpha|\le N} \left\|\partial_x^{\alpha}f\right\|_{L_{x,v}^2}^2+\lambda\left(\sum_{1\le|\alpha|\le N} \left\|\partial_x^{\alpha}\mathbf{P}^{\bot}f \right\|^2+\sum_{2\le|\alpha|\le N} \left\|\partial_x^{\alpha}\mathbf{P}f \right\|^2\right)\notag\\
        &\lesssim \sum_{1\le|\alpha|\le N} \eta_3\left\|\partial_x^{\alpha}\mathbf{P}^{\bot}f \,\langle v\rangle^{-\frac{\gamma}{2}}\right\|^2+\sum_{1\le|\alpha|\le N} \|\partial_x^{\alpha}S\|^2+\epsilon^2\sum_{1\le|\alpha|\le N}\|\partial_x^\alpha f\|^2\,,
    \end{align}
Consequently, it follows from \eqref{B2} that
    \begin{align}\label{B3}
           &\hspace{3mm} \frac{\mathrm{d}}{\mathrm{d}t}\|f\|_{L_v^2(\dot{H}^1\cap\dot{H}^N)}^2+\lambda\|f\|_{L_v^2(\dot{H}^1\cap\dot{H}^N)}^2\notag\\
           &\lesssim\eta_3\left\|\mathbf{P}^{\bot}f \,\langle v\rangle^{-\frac{\gamma}{2}}\right\|_{L_v^2(\dot{H}^1\cap\dot{H}^N)}^2+\|S\|_{L_v^2(\dot{H}^1\cap\dot{H}^N)}^2+\lambda\left\|\nabla_x \mathbf{P}f\right\|^2\,,\notag
        \end{align}
from which, we then get
    \begin{align}
      \|f\|_{L_v^2(\dot{H}^1\cap\dot{H}^N)}^2
    &\lesssim\|f_0\|_{L_v^2(\dot{H}^1\cap\dot{H}^N)}^2+\eta_3\int_0^t e^{-\lambda(t-\tau)}\left\|\mathbf{P}^{\bot}f \,\langle v\rangle^{-\frac{\gamma}{2}}\right\|_{L_v^2(\dot{H}^1\cap\dot{H}^N)}^2 \,\mathrm{d}\tau\notag\\
       &\qquad+\int_0^t e^{-\lambda(t-\tau)} \|S\|_{L_v^2(\dot{H}^1\cap\dot{H}^N)}^2\,\mathrm{d}\tau+\lambda\int_0^t e^{-\lambda(t-\tau)} \|\nabla_x\mathbf{P}f\|^2\,\mathrm{d}\tau\notag\\
       &\lesssim\|f_0\|_{L_v^2(\dot{H}^1\cap\dot{H}^N)}^2+\eta_3\sup\limits_{t}\left\|\mathbf{P}^{\bot}f \,\langle v\rangle^{-\frac{\gamma}{2}}\right\|_{L_v^2(\dot{H}^1\cap\dot{H}^N)}^2\notag\\
       &\qquad+\sup\limits_{t}\|S\|_{L_v^2(\dot{H}^1\cap\dot{H}^N)}^2+\sup\limits_{t}|(a,b,c)|_{\dot{B}_{2,\infty}^{\frac{1}{2}}}^2\,.\notag
    \end{align}
Here the last inequality is due to Lemma \ref{interpolation-besov-other} and Young's inequality:
\[\|\nabla_x\mathbf{P}f\|^2\lesssim\|\mathbf{P}f\|_{L_v^2(\dot{B}_{2,\infty}^{\frac{1}{2}})}^2+\eta_4\|\nabla_x^2\mathbf{P}f\|^2\lesssim|(a,b,c)|_{\dot{B}_{2,\infty}^\frac{1}{2}}^2+\eta_4\|\nabla_x^2{\bf P}f\|^2\,,\]
where $\eta_4$ is small, this completes the proof of Lemma \ref{lemma4.2}.
\end{proof}
We next state the $L^2$ bound for ${\bf P}^\bot f$, which will be used later.
\begin{lemma}\label{lemma4.3}
 Let $N\ge5, -3<\gamma<0$, assume $f(x,t,v)$ is the classical solution of the system \eqref{Cauchy}, then it holds that
    \begin{align}
        \|\mathbf{P}^{\bot}f\|^2&\lesssim \|\mathbf{P}^{\bot}f_0\|^2+\eta\sup_t \|\mathbf{P}^{\bot}f\,\langle v\rangle^{-\frac{\gamma}{2}}\|^2+\eta'\sup\limits_t\left\|\mathbf{P}^{\bot}f\right\|^2\notag\\
        &\qquad+\sup\limits_t|(a,b,c)|_{\dot{B}_{2,\infty}^\frac{1}{2}\cap\dot{H}^2}^2+\sup\limits_t\|f\|_{L_v^2(\dot{H}^1)}\,,\notag
    \end{align}
    where $\eta>0$ is a small constant.
\end{lemma}
\begin{proof}
We consider the microscopic component of $\eqref{f-perturbation}$. Applying $\mathbf{P}^{\bot}$, we obtain
\begin{align}\label{D1}
    \partial_t\mathbf{P}^{\bot}f+v\cdot\nabla_x\mathbf{P}^{\bot}f+v\cdot\nabla_x\mathbf{P}f+{\mathcal{L}}\left(\mathbf{P}^{\bot}f\right)={\bf P}(v\cdot\nabla_xf)+\mathbf{P}^{\bot}\Gamma(f,f)+\mathbf{P}^{\bot}S\,,
\end{align}
where we have used the decomposition
\[\mathbf{P}^{\bot}(v\cdot\nabla_x f)=v\cdot\nabla_x f-{\bf P}(v\cdot\nabla_xf)=v\cdot\nabla_x \mathbf{P}^{\bot}f+v\cdot\nabla_x \mathbf{P}f-{\bf P}(v\cdot\nabla_xf)\,.\]
Multiplying \eqref{D1} by $\mathbf{P}^{\bot}f$, then integrating over $\mathbb{R}_x^3\times\mathbb{R}_v^3$ and we get from  \eqref{D2} that
\begin{equation}\label{D3}
\begin{aligned}
    &\hspace{3mm}\frac{\mathrm{d}}{\mathrm{d}t}\left\|\mathbf{P}^{\bot}f\right\|^2+\delta_0\left\|\mathbf{P}^{\bot}f\langle v\rangle^{\frac{\gamma}{2}}\right\|^2\\
    &\lesssim\int_{\mathbb{R}_{x,v}^6}\left(v\cdot\nabla_x\mathbf{P}f\right)\mathbf{P}^{\bot}f\,\mathrm{d}v\mathrm{d}x+\int_{\mathbb{R}_{x,v}^6}{\bf P}(v\cdot\nabla_xf)\mathbf{P}^{\bot}f\,\mathrm{d}v\mathrm{d}x\\
    &\quad+\int_{\mathbb{R}_{x,v}^6} \mathbf{P}^{\bot}\Gamma(f,f)\mathbf{P}^{\bot}f\,\mathrm{d}v\mathrm{d}x+\int_{\mathbb{R}_{x,v}^6}\mathbf{P}^{\bot}S\mathbf{P}^{\bot}f\,\mathrm{d}v\mathrm{d}x\,.\notag
    \end{aligned}
\end{equation}
This, together with Young's inequality 
\begin{align}\label{D4}
\left\|\mathbf{P}^{\bot}f\right\|^2\lesssim\left\|\mathbf{P}^{\bot}f\,\langle v\rangle^{\frac{\gamma}{2}}\right\|^2+\eta\left\|\mathbf{P}^{\bot}f\,\langle v\rangle^{-\frac{\gamma}{2}}\right\|^2\,.\notag
\end{align}
yields that
\begin{equation}
    \begin{aligned}
        \left\|\mathbf{P}^{\bot}f\right\|^2&\lesssim\left\|\mathbf{P}^{\bot}f_0\right\|^2+\eta\int_0^t e^{-\lambda(t-\tau)}\left\|\mathbf{P}^{\bot}f\,\langle v\rangle^{-\frac{\gamma}{2}}\right\|^2\,\mathrm{d}\tau\\
        &\quad\underbrace{-\int_0^t \int_{\mathbb{R}_{x,v}^6} e^{-\lambda(t-\tau)}\left(v\cdot\nabla_x\mathbf{P}f\right)\mathbf{P}^{\bot}f\,\mathrm{d}v\mathrm{d}x\,\mathrm{d}\tau}_{I_3}+\underbrace{\int_0^t\int_{\mathbb{R}_{x,v}^6} e^{-\lambda(t-\tau)}{\bf P}(v\cdot\nabla_xf)\mathbf{P}^{\bot}f\,\mathrm{d}v\mathrm{d}x\mathrm{d}\tau}_{I_4}\\
        &\quad\underbrace{+\int_0^t\int_{\mathbb{R}_{x,v}^6} e^{-\lambda(t-\tau)}\mathbf{P}^{\bot}\Gamma(f,f)\mathbf{P}^{\bot}f\,\mathrm{d}v\mathrm{d}x\mathrm{d}\tau}_{I_5}+\underbrace{\int_0^t\int_{\mathbb{R}_{x,v}^6} e^{-\lambda(t-\tau)}\mathbf{P}^{\bot}S\,\mathbf{P}^{\bot}f\,\mathrm{d}v\mathrm{d}x\mathrm{d}\tau}_{I_6}\,.\notag
    \end{aligned}
\end{equation}
We now estimate each term on the right-hand side of the above inequality. For $I_3$, it follows
\begin{align*}
    I_3&\lesssim\int_0^t \int_{\mathbb{R}_{x,v}^6} e^{-\lambda(t-\tau)}\nabla_x(a,b,c)\mu^{\delta}\mathbf{P}^{\bot}f\,\mathrm{d}v\mathrm{d}x\,\mathrm{d}\tau\\
    &\lesssim\sup\limits_t\|(a,b,c)\|_{\dot{H}^1}^2+\eta\sup\limits_t\left\|\mathbf{P}^{\bot}f\right\|^2\,,
\end{align*}
here, $\eta>0$ is sufficiently small.

For $I_6$, since
\[{\bf P}(v\cdot\nabla_xf)=\int_{\mathbb{R}_v^3} (v\cdot\nabla_xf)\,\phi_i\,\mathrm{d}v\,\mu^{\delta}\lesssim|\nabla_xf|_{L_v^2}\,\mu^\delta\,,\]
it's easy to obtain
\begin{align*}
    I_4\lesssim\int_0^t\int_{\mathbb{R}_{x,v}^6} e^{-\lambda(t-\tau)}|\nabla_xf|_{L_v^2}\,\mu^{\delta}\mathbf{P}^{\bot}f\,\mathrm{d}v\mathrm{d}x\mathrm{d}\tau\lesssim\eta\sup\limits_t\left\|\mathbf{P}^{\bot}f\right\|^2+\sup\limits_t\|\nabla_xf\|^2\,.
\end{align*}
As to $I_5$, Lemmas \ref{interpolation-besov-other} and \ref{lemma5.2} imply that
\begin{align*}
    &\hspace{3mm}\int_{\mathbb{R}_{x,v}^6}\Gamma(f,f)\mathbf{P}^{\bot}f\,\mathrm{d}v\mathrm{d}x\\
    &\lesssim \int_{\mathbb{R}_x^3}\Big( \left|\mu^\delta{\bf P}f\right|_{H_v^2}+|{\bf P}f|_{L_v^2}\Big)|{\bf P}f|_\nu|\mathbf{P}^{\bot}f|_{\nu}\,\mathrm{d}x+\int_{\mathbb{R}_x^3}\left( \left|\mu^\delta{\bf P}^\bot f\right|_{H_v^2}+\left|{\bf P}^\bot f\right|_{L_v^2}\right)\left|{\bf P}^\bot f\right|_\nu^2\,\mathrm{d}x\\
    &\quad+\int_{\mathbb{R}_x^3}\left( \left|\mu^\delta{\bf P} f\right|_{H_v^2}+\left|{\bf P} f\right|_{L_v^2}\right)\left|{\bf P}^\bot f\right|_\nu^2\,\mathrm{d}x+\int_{\mathbb{R}_x^3}\left( \left|\mu^\delta{\bf P}^\bot f\right|_{H_v^2}+\left|{\bf P}^\bot f\right|_{L_v^2}\right)\left|{\bf P} f\right|_\nu|{\bf P}^\bot f|_\nu\,\mathrm{d}x\\
    &\lesssim\Big\|{\bf P}f\Big\|_{L_x^4H_v^2}\Big\|{\bf P}f\Big\|_{L_x^4L_v^2}\Big\|\mathbf{P}^{\bot}f\Big\|+\Big\|{\bf P}^\bot f\Big\|_{L_x^\infty H_v^2}\Big\|{\bf P}^\bot f\Big\|^2\\
    &\quad+\Big\|{\bf P}f\Big\|_{L_x^\infty H_v^2}\Big\|{\bf P}^\bot f\Big\|^2+\Big\|{\bf P}f\Big\|_{L_x^\infty L_v^2}\Big\|{\bf P}^\bot f\Big\|_{L_x^2H_v^2}\Big\|{\bf P}^\bot f\Big\|\\
    &\lesssim\Big|(a,b,c)\Big|_{\dot{B}_{2,\infty}^\frac{1}{2}\cap\dot{H}^1}^2\Big\|{\bf P}^\bot f\Big\|+\Big\|{\bf P}^\bot f\Big\|_{H_v^2(\dot{H}_x^1\cap\dot{H}_x^2)}\Big\|{\bf P}^\bot f\Big\|^2\\
    &\quad+\Big|(a,b,c)\Big|_{\dot{H}^1\cap\dot{H}^2}\Big\|{\bf P}^\bot f\Big\|^2+\Big\|{\bf P}^\bot f\Big\|_{H_v^2L_x^2}\Big|(a,b,c)\Big|_{\dot{H}^1\cap\dot{H}^2}\Big\|{\bf P}^\bot f\Big\|\\
    &\lesssim\epsilon^2\Big|(a,b,c)\Big|_{\dot{B}_{2,\infty}^\frac{1}{2}\cap\dot{H}^1}^2+\eta\Big\|{\bf P}^\bot f\Big\|^2+\epsilon\Big\|{\bf P}^\bot f\Big\|^2+\epsilon^2\Big|(a,b,c)\Big|_{\dot{H}^1\cap\dot{H}^2}^2\\
    &\lesssim\epsilon^2\Big|(a,b,c)\Big|_{\dot{B}_{2,\infty}^\frac{1}{2}\cap\dot{H}^2}^2+\eta'\Big\|{\bf P}^\bot f\Big\|^2\,,
\end{align*}
where $\eta'=\max\{\eta,\epsilon\}$ is suitably small.
Then we have
\begin{align*}
    I_5\lesssim\eta'\sup_t\left\|\mathbf{P}^{\bot}f\right\|^2+\epsilon^2\sup_t|(a,b,c)|_{\dot{B}_{2,\infty}^\frac{1}{2}\cap\dot{H}^2}^2\,.
\end{align*}
It's straightforward to estimate $I_6$ by
\begin{align*}
    I_6\lesssim\eta\sup\limits_t\left\|\mathbf{P}^{\bot}f\right\|^2+\sup\limits_t\|S\|^2.
\end{align*}
Combining all the above estimates, we finally obtain the desired estimate.  The proof of Lemma \ref{lemma4.3} is finished.
    
\end{proof}

\section{$\dot{H}_{x,v}^N$-inner-estimate with weight}\label{Sect-weight}
In this section, we prove the Sobolev estimates used to close the a priori estimates with pure velocity weights. The first concerns microscopic estimates with the velocity weight $w_{\ell,\beta}=\langle v\rangle^{\ell+\gamma|\beta|}$, where $\beta$ denotes the order of differentiation in $v$.
\begin{lemma}\label{lemma5.3}
    Let $N\ge5, -3<\gamma<0$ and $1\le|\alpha|+|\beta|\le N, \ell\ge -N-\frac{\gamma}{2}$, assume that $f(t,x,v)$ is the solution of the system \eqref{Cauchy}, then it holds for some $\varepsilon>0$ that
    \begin{align}
        \sum_{1\le|\alpha|+|\beta|\le N}\sup_t\|w_{\ell,\beta}\partial_x^{\alpha}\partial_v^{\beta}{\bf P}^\bot f\|^2&\lesssim\sum_{1\le|\alpha|+|\beta|\le N}\|w_{\ell,\beta}\partial_x^{\alpha}\partial_v^{\beta}{\bf P}^\bot f_0\|^2+\sup\limits_t\|f\|_{L_v^2(\dot{H}^1\cap\dot{H}^N)}^2\notag\\
        &\qquad+\sup\limits_t\left\|\mathbf{P}^{\bot}f\right\|^2+\sup\limits_t|(a,b,c)|_{\dot{B}_{2,\infty}^\frac{1}{2}}^2+\eta\sup_t\|w_{\ell,0}{\bf P}^\bot f\|^2\notag\\
        &\qquad+\sum_{1\le|\alpha|+|\beta|\le N}\sup_t\|w_{\ell-(1+\frac{\varepsilon}{2})\gamma,\beta}\,\partial_x^\alpha\partial_v^\beta S\|^2\notag\\
        &\qquad+\eta'\sum_{\substack{|\alpha|+|\beta|\le N}}\sum_{j\in\mathbb{Z}^+}\sup_t\|w_{\ell,\beta}\partial_x^\alpha\partial_v^\beta{\bf P}^\bot f\|^2_{L^2_v(V_j)L_x^2(\mathbb{R}^3)}\,,\notag
    \end{align}
    where $\eta,\eta'>0$ are small constants. 
\end{lemma}
\begin{proof}
Note that $\eqref{f-perturbation}$ can be rewritten as
\begin{equation}\label{B4}
    \partial_t f+v\cdot\nabla_xf+\nu(v)f-\mathcal{K}f=\Gamma(f,f)+S\,.
\end{equation}
In what follows, we only consider three representative cases; the remaining cases can be obtained by interpolation.

\noindent \underline{{\it Case 1: $|\alpha|=N,\ |\beta|=0$.}} Applying $\partial_x^{\alpha}$ to \eqref{f-perturbation} and multiplying the resulting equation by $w_{\ell,0}^2\partial_x^{\alpha}f$, we get
 \begin{equation}
     \begin{aligned}
         &\hspace{3mm}\frac{1}{2}\partial_t\left(w_{\ell,0}\partial_x^{\alpha}f\right)^2+(v\cdot\nabla_x\partial_x^{\alpha}f)\,w_{\ell,0}^2\partial_x^{\alpha}f+\frac{1}{2}\nu(v)\left(\partial_x^{\alpha}fw_{\ell,0}\right)^2\\
         &=-\frac{1}{2}\nu(v)\left(\partial_x^{\alpha}fw_{\ell,0}\right)^2+\partial_x^{\alpha}\mathcal{K}f\,w_{\ell,0}^2\partial_x^{\alpha}f+\partial_x^{\alpha}\Gamma(f,f)w_{\ell,0}^2\partial_x^{\alpha}f+\partial_x^{\alpha}S\,w_{\ell,0}^2\partial_x^{\alpha}f\,.\notag
     \end{aligned}
 \end{equation}
Solving the above inequality directly yields that   
    \begin{align*}
        \left(w_{\ell,0}\partial_x^{\alpha}f\right)^2&\lesssim\left(w_{\ell,0}\partial_x^{\alpha}f_0\right)^2-\int_0^t e^{-\nu(v)(t-\tau)} (v\cdot\nabla_x\partial_x^{\alpha}f)\,w_{\ell,0}^2\partial_x^{\alpha}f\,\mathrm{d}\tau\\
        &\quad- \frac{1}{2}\int_0^t e^{-\nu(v)(t-\tau)}\nu(v)\left(\partial_x^{\alpha}fw_{\ell,0}\right)^2\,\mathrm{d}\tau+\int_0^t e^{-\nu(v)(t-\tau)} \partial_x^{\alpha}\mathcal{K}f\,w_{\ell,0}^2\partial_x^{\alpha}f\,\mathrm{d}\tau\\
        &\quad+\int_0^t e^{-\nu(v)(t-\tau)} \partial_x^{\alpha}\Gamma(f,f)w_{\ell,0}^2\partial_x^{\alpha}f\,\mathrm{d}\tau+\int_0^t e^{-\nu(v)(t-\tau)} \partial_x^{\alpha}S\,w_{\ell,0}^2\partial_x^{\alpha}f\,\mathrm{d}\tau.
    \end{align*}

Integrating the above inequality over $\mathbb{R}_x^3\times\mathbb{R}_v^3$ yields
\begin{align}\label{B5}  \left\|w_{\ell,0}\partial_x^{\alpha}f\right\|^2&\lesssim\left\|w_{\ell,0}\partial_x^{\alpha}f_0\right\|^2\underbrace{-\frac{1}{2}\int_0^t\int_{\mathbb{R}_{x,v}^6}e^{-\nu(v)(t-\tau)}\nu(v)\left(\partial_x^{\alpha}f w_{\ell,0}\right)^2 \, \mathrm{d}v\mathrm{d}x\mathrm{d}\tau}_{J_1}\notag\\
        &\quad+\underbrace{\int_0^t\int_{\mathbb{R}_{x,v}^6}  e^{-\nu(v)(t-\tau)} \partial_x^{\alpha}\mathcal{K}f\,w_{\ell,0}^2\partial_x^{\alpha}f\, \mathrm{d}v\mathrm{d}x\mathrm{d}\tau}_{J_2}\notag\\
        &\quad+\underbrace{\int_0^t\int_{\mathbb{R}_{x,v}^6} e^{-\nu(v)(t-\tau)} \partial_x^{\alpha}\Gamma(f,f)w_{\ell,0}^2\partial_x^{\alpha}f\, \mathrm{d}v\mathrm{d}x\mathrm{d}\tau}_{J_3}\notag\\
        &\quad+\underbrace{\int_0^t\int_{\mathbb{R}_{x,v}^6} e^{-\nu(v)(t-\tau)} \partial_x^{\alpha}S\,w_{\ell,0}^2\partial_x^{\alpha}f\, \mathrm{d}v\mathrm{d}x\mathrm{d}\tau}_{J_4}\,.
    \end{align}
    We now estimate $J_2$, $J_3$ and $J_4$ separately.
For $J_2$, by splitting the velocity space $\mathbb{R}_v^3$ as in section \ref{appendix}, one has 
\begin{align*}
    J_2&\lesssim\underbrace{\sum_{j\in\mathbb{Z}^+}\int_0^t\int_{V_j\times\mathbb{R}_x^3}e^{-\nu(v)(t-\tau)}\nu(v)\nu^{-1}(v)w^2_{\ell,0}\mathcal{K}\partial_x^\alpha f\partial_x^\alpha f\,{\rm d}v{\rm d}x{\rm d}\tau}_{J_{2,1}}\\
    &\quad+\underbrace{\int_0^t\int_{V_0\times\mathbb{R}_x^3}e^{-\nu(v)(t-\tau)}w^2_{\ell,0}\mathcal{K}\partial_x^\alpha f\partial_x^\alpha f\,{\rm d}v{\rm d}x{\rm d}\tau}_{J_{2,2}}\,.
\end{align*}
Next, it follows that
\begin{align*}
    J_{2,1}&\lesssim\sum_{j\in\mathbb{Z}^+}\int_0^t\int_{V_j\times\mathbb{R}_x^3}e^{-\tilde{c}_j(t-\tau)}c_j\nu^{-1}(v)w^2_{\ell,0}\mathcal{K}\partial_x^\alpha f\partial_x^\alpha f\,{\rm d}v{\rm d}x{\rm d}\tau\\
    &\lesssim\sum_{j\in\mathbb{Z}^+}\sup_t\int_{V_j\times\mathbb{R}_x^3}\nu^{-1}(v)w_{\ell,0}^2\mathcal{K}\partial_x^\alpha f\partial_x^\alpha f\,{\rm d}v{\rm d}x\sup_{j\in\mathbb{Z}^+}\int_0^t e^{-\tilde{c}_j(t-\tau)}c_j\,{\rm d}\tau\\
    &\lesssim\sum_{j\in\mathbb{Z}^+}\sup_t\int_{V_j\times\mathbb{R}_x^3}\nu^{-1}(v)w^2_{\ell,0}\mathcal{K}\partial_x^\alpha f\partial_x^\alpha f\,{\rm d}v{\rm d}x\\
    &\lesssim\eta\sum_{j\in\mathbb{Z}^+}\sup_t\|w_{\ell,0}\partial_x^\alpha f\|^2_{L^2(V_j)L_x^2}+\sup_t\|\partial_x^\alpha f\|^2+\eta'\sup_t\|w_{\ell,0}\partial_x^\alpha f\|^2\,,
\end{align*}
where $\eta$ and $\eta'$ are chosen sufficiently small. In the last step, we have used Lemma \ref{lemma5.1}.

\noindent Regarding $J_{2,2}$, it is straightforward to see that
\begin{align}\label{j22}
    J_{2,2}&\lesssim\int_0^t\int_{V_0\times\mathbb{R}_x^3}e^{-\lambda(t-\tau)}w^2_{\ell,0}\mathcal{K}\partial_x^\alpha f\partial_x^\alpha f\,{\rm d}v{\rm d}x{\rm d}\tau\nonumber\lesssim\sup_t\int_{\mathbb{R}_{x,v}^6}w^2_{\ell,0}\mathcal{K}\partial_x^\alpha f\partial_x^\alpha f\,{\rm d}v{\rm d}x\int_0^t e^{-\lambda(t-\tau)}{\rm d}\tau\nonumber\\
    &\lesssim\sup_t\int_{\mathbb{R}_{x,v}^6}w^2_{\ell,0}\mathcal{K}\partial_x^\alpha f\partial_x^\alpha f\,{\rm d}v{\rm d}x\,.
\end{align}
Lemma \ref{lemma9.3} implies that
\begin{equation}
    \begin{aligned}
     \left|\int_{\mathbb{R}_v^3} {w}_{\ell,0}^2\partial_x^{\alpha}[\mathcal{K}f]\partial_x^{\alpha}f\,\mathrm{d}v\right|&=\left|\int_{\mathbb{R}_v^3} {w}_{\ell,0}^2[\mathcal{K}\partial_x^{\alpha}f]\partial_x^{\alpha}f\,\mathrm{d}v\right|\lesssim\eta\left|{w}_{\ell,0}\partial_x^{\alpha}f\right|_{\nu}^2+C_{\eta}\left|\bar\chi(v)\partial_x^{\alpha}f{w}_{\ell,0}\right|_{L_v^2}^2\,,\notag
    \end{aligned}
\end{equation}
where $\eta$ is chosen sufficiently small. Substituting this into \eqref{j22}, we obtain
    \begin{align*}
        J_{2,2}&\lesssim\eta\sup_t\|w_{\ell,0}\partial_x^\alpha f\|^2+C_{\eta}\sup_t\int_{\mathbb{R}_{x,v}^6} w_{\ell,0}^2\,\bar\chi^2(v)\left|\partial_x^{\alpha}f\right|^2\, \mathrm{d}v\mathrm{d}x\\
        &\lesssim\eta\sup_t\|w_{\ell,0}\partial_x^\alpha f\|^2+C_\eta\|\partial_x^\alpha f\|^2\,.
    \end{align*}
    Consequently, we have
    \begin{align}
        J_{2}\lesssim\eta\sum_{j\in\mathbb{Z}^+}\sup_t\|w_{\ell,0}\partial_x^\alpha f\|^2_{L^2(V_j)L_x^2}+\sup_t\|\partial_x^\alpha f\|^2+
        \eta'\sup_t\|w_{\ell,0}\partial_x^\alpha f\|^2\,.\notag
    \end{align}
\noindent For $J_3$, , we also have
\begin{align*}
	J_3&\lesssim\underbrace{\sum_{j\in\mathbb{Z}^+}\int_0^t\int_{V_j\times\mathbb{R}_x^3} e^{-\nu(v)(t-\tau)}\nu(v)\nu^{-1}(v)w^2_{\ell,0}\partial_x^{\alpha}\Gamma(f,f)\partial_x^\alpha f\,{\rm d}v{\rm d}x{\rm d}\tau}_{J_{3,1}}\\
	&\quad+\underbrace{\int_0^t\int_{V_0\times\mathbb{R}_x^3} e^{-\nu(v)(t-\tau)}w^2_{\ell,0}\partial_x^{\alpha}\Gamma(f,f)\partial_x^{\alpha}f\,{\rm d}v{\rm d}x{\rm d}\tau}_{J_{3,2}}\,.
	\end{align*}
  Next,  Lemma \ref{lemma5.2} gives that
    \begin{align}
        &\hspace{3mm}\sup_t\int_{V_j\times\mathbb{R}_x^3}\nu^{-1}(v)w^2_{\ell,0}\Gamma(\partial_x^{\alpha_1}f,\partial_x^{\alpha_2}f)\partial_x^\alpha f\,{\rm d}v{\rm d}x\notag\\
        &\lesssim\sup_t\int_{\mathbb{R}_x^3}\left(|\mu^\delta \partial_x^{\alpha_1}f|_{H_v^2(\mathbb{R}^3)}+|w_{\ell,0}\partial_x^{\alpha_1}f|_{L_v^2(\mathbb{R}^3)}\right)\times|w_{\ell,0}\partial_x^{\alpha_2}f|_{L_v^2(\widetilde{V}_j)}|w_{\ell,0}\partial_x^{\alpha}f|_{L_v^2(V_j)}\,{\rm d}x\,,\notag
    \end{align}
    or
    \begin{align}\label{GG2}
        &\hspace{3mm}\sup_t\int_{V_j\times\mathbb{R}_x^3}\nu^{-1}(v)w^2_{\ell,0}\Gamma(\partial_x^{\alpha_1}f,\partial_x^{\alpha_2}f)\partial_x^\alpha f\,{\rm d}v{\rm d}x\notag\\
        &\lesssim\sup_t\int_{\mathbb{R}_x^3}\left(|\mu^\delta \partial_x^{\alpha_2}f|_{H^2_v(\mathbb{R}^3)}+|w_{\ell,0}\partial_x^{\alpha_2}f|_{L_v^2(\mathbb{R}^3)}\right)\times|w_{\ell,0}\partial_x^{\alpha_1}f|_{L_v^2(\widetilde{V}_j)}|w_{\ell,0}\partial_x^{\alpha}f|_{L^2_v(V_j)}\,{\rm d}x\,,
    \end{align}
    where $\alpha_1+\alpha_2=\alpha$. If $|\alpha_1|=0$, we have
    \begin{align}
        &\hspace{3mm}\sup_t\int_{V_j\times\mathbb{R}_x^3}\nu^{-1}(v)w^2_{\ell,0}\Gamma(f,\partial_x^{\alpha}f)\partial_x^\alpha f\,{\rm d}v{\rm d}x\notag\\
        &\lesssim\sup_t\Bigg\{\Big(\|\mu^\delta{\bf P} f\|_{L_x^\infty H^2_v}+\|w_{\ell,0}{\bf P}f\|_{L_x^\infty L_v^2}\Big)\times\|w_{\ell,0}\partial_x^\alpha f\|_{L_x^2(\mathbb{R}^3)L^2(\widetilde{V}_j)}\|w_{\ell,0}\partial_x^\alpha f\|_{L_x^2(\mathbb{R}^3)L^2(V_j)}\Bigg\}\nonumber\\
       &\quad+\sup_t\Bigg\{\Big(\|\mu^\delta{\bf P}^\bot f\|_{L_x^\infty H^2_v}+\|w_{\ell,0}{\bf P}^\bot f\|_{L_x^\infty L^2_v}\Big)\times\|w_{\ell,0}\partial_x^\alpha f\|_{L_x^2(\mathbb{R}^3)L^2(\widetilde{V}_j)}\|w_{\ell,0}\partial_x^\alpha f\|_{L_x^2(\mathbb{R}^3)L^2(V_j)}\Bigg\}\nonumber\\
       &\lesssim\sup_t\Bigg\{\Big(|(a,b,c)|_{\dot{H}_x^1\cap\dot{H}_x^2}+\|{\bf P}^\bot f\|_{ H_v^2(\dot{H}_x^1\cap\dot{H}_x^2)}+\|w_{\ell,0}{\bf P}^\bot f\|_{L_v^2(\dot{H}_x^1\cap\dot{H}_x^2)}\Big)\nonumber\\
       &\qquad\times\|w_{\ell,0}\partial_x^\alpha f\|_{L^2(\widetilde{V}_j)L_x^2(\mathbb{R}^3)}\|w_{\ell,0}\partial_x^\alpha f\|_{L^2(V_j)L_x^2(\mathbb{R}^3)}\Bigg\}\nonumber\\
       &\lesssim\epsilon^2\sup_t\|w_{\ell,0}\partial_x^\alpha f\|^2_{L^2(\widetilde{V}_j)L_x^2(\mathbb{R}^3)}+\eta\sup_t\|w_{\ell,0}\partial_x^\alpha f\|^2_{L^2({V}_j)L_x^2(\mathbb{R}^3)}\,.\notag
    \end{align}
    If $|\alpha_1|=|\alpha|$, we apply \eqref{GG2}, and the same estimate follows by the same argument.

    \noindent If $1\le|\alpha_1|\le|\alpha|-1$ and $|\alpha_1|+3\le N$, we apply $L^3-L^6-L^2$ estimate to obtain
    \begin{align}
&\hspace{3mm}\sup_t\int_{V_j\times\mathbb{R}_x^3}\nu^{-1}(v)w^2_{\ell,0}\Gamma(\partial_x^{\alpha_1}f,\partial_x^{\alpha_2}f)\partial_x^\alpha f\,{\rm d}v{\rm d}x\notag\\
        &\lesssim \sup_t \Bigg\{
\Big( \|\mu^\delta\partial_x^{\alpha_1}\mathbf{P} f\|_{L_x^3(\mathbb{R}^3)H_v^2(\mathbb{R}^3)} + \|w_{\ell,0}\partial_x^{\alpha_1}\mathbf{P} f\|_{L_x^3(\mathbb{R}^3)L^2_v(\mathbb{R}^3)} \Big) \nonumber \\
&\qquad \times \|w_{\ell,0}\partial_x^{\alpha_2} f\|_{L_x^6(\mathbb{R}^3)L^2(\widetilde{V}_j)} \|w_{\ell,0}\partial_x^{\alpha} f\|_{L_x^2(\mathbb{R}^3)L^2(V_j)} \Bigg\} \nonumber \\
&\quad + \sup_t \Bigg\{
\Big( \|\mu^\delta\partial_x^{\alpha_1}\mathbf{P}^\bot f\|_{L_x^3(\mathbb{R}^3)H_v^2(\mathbb{R}^3)} + \|w_{\ell,0}\partial_x^{\alpha_1}\mathbf{P}^\bot f\|_{L_x^3(\mathbb{R}^3)L^2_v(\mathbb{R}^3)} \Big) \nonumber \\
&\qquad \times \|w_{\ell,0}\partial_x^{\alpha_2} f\|_{L_x^6(\mathbb{R}^3)L^2(\widetilde{V}_j)} \|w_{\ell,0}\partial_x^{\alpha} f\|_{L_x^2(\mathbb{R}^3)L^2(V_j)} \Bigg\} \nonumber \\
&\lesssim \sup_t\bigg\{|(a,b,c)|^2_{\dot{H}_x^{|\alpha_1|}\cap\dot{H}_x^{|\alpha_1|+1}}\|w_{\ell,0}\partial_x^{\alpha_2+e_i}f\|_{L^2(\widetilde{V}_j)L_x^2(\mathbb{R}^3)}^2\bigg\}+\eta\sup_t\|w_{\ell,0}\partial_x^{\alpha} f\|_{L^2(V_j)L_x^2(\mathbb{R}^3)}^2\nonumber\\
&\quad+\sup_t\bigg\{\Big(\|w_{\ell,0}{\bf P}^\bot f\|^2_{L^2_v(\dot{H}_x^{|\alpha_1|}\cap\dot{H}_x^{|\alpha_1|+1})}+\|{\bf P}^\bot f\|^2_{H_v^2(\dot{H}_x^{|\alpha_1|}\cap\dot{H}^{|\alpha_1|+1})}\Big)\|w_{\ell,0}\partial_x^{\alpha_2+e_i}f\|_{L^2(\widetilde{V}_j)L_x^2}^2\bigg\}\nonumber\\
&\lesssim\epsilon^2\sup_t\|w_{\ell,0}\partial_x^{\alpha_2+e_i}f\|_{L^2(\widetilde{V}_j)L_x^2(\mathbb{R}^3)}^2+\eta\sup_t\|w_{\ell,0}\partial_x^{\alpha} f\|_{L^2(V_j)L_x^2(\mathbb{R}^3)}^2\,.\nonumber
\end{align}
If $1\le|\alpha_1|\le|\alpha|-1$ and $|\alpha_1|+3> N$, then $|\alpha_2|+3\le N$. We apply the formulation \eqref{GG2} again. The same argument then yields the same estimate.

\noindent Consequently, we have
\begin{align*}
    J_{3,1}&\lesssim\sum_{j\in\mathbb{Z}^+}\sup_t\int_{V_j\times\mathbb{R}_x^3}\nu^{-1}(v)w^2_{\ell,0}\partial_x^{\alpha}\Gamma(f,f)\partial_x^\alpha f\,{\rm d}v{\rm d}x\\
    &\lesssim\sum_{j\in\mathbb{Z}^+}\sum_{\alpha_1+\alpha_2=\alpha}\sup_t\int_{V_j\times\mathbb{R}_x^3}\nu^{-1}(v)w^2_{\ell,0}\Gamma(\partial_x^{\alpha_1}f,\partial_x^{\alpha_2}f)\partial_x^\alpha f\,{\rm d}v{\rm d}x\\
		&\lesssim\eta'\sum_{j\in\mathbb{Z}^+}\sum_{1\le|\alpha'|\le |\alpha|}\sup_t\|w_{\ell,0}\partial_x^{\alpha'} f\|^2_{L_v^2(V_j)L_x^2(\mathbb{R}^3)}\,,
\end{align*}
where $\eta'=\max\{\eta,\epsilon^2\}$.
    
    \noindent As for $J_{3,2}$, it holds from Lemma \ref{lemma9.4} that
\begin{align}
   &\hspace{3mm}\int_{\mathbb{R}_{x,v}^6}  \Gamma(\partial_x^{\alpha_1}f,\partial_x^{\alpha_2}f)\partial_x^{\alpha}f\, \mathrm{d}v\mathrm{d}x\notag\\
   &\lesssim\int_{\mathbb{R}_{x}^3}  \left(|\mu^\delta\partial_x^{\alpha_1}f|_{H_v^2}+|\partial_x^{\alpha_1}f|_{L_v^2}\right)|\partial_x^{\alpha_2}f|_{\nu}|\partial_x^{\alpha}f|_{\nu}\, \mathrm{d}x\lesssim\int_{\mathbb{R}_x^3}|\partial_x^{\alpha_1}f|_{H_v^2}|\partial_x^{\alpha_2}f|_{L_v^2}|\partial_x^{\alpha}f|_{L_v^2}\,{\rm d}x\,,\notag
\end{align}
or 
\begin{align}\label{jj2}
     &\hspace{3mm}\int_{\mathbb{R}_{x,v}^6}  \Gamma(\partial_x^{\alpha_1}f,\partial_x^{\alpha_2}f){w}_{\ell,0}^2\partial_x^{\alpha}f\, \mathrm{d}v\mathrm{d}x\notag\\
   &\lesssim\int_{\mathbb{R}_{x}^3}  \left(|\mu^\delta\partial_x^{\alpha_2}f|_{H_v^2}+|\partial_x^{\alpha_2}f|_{L_v^2}\right)|\partial_x^{\alpha_1}f|_{\nu}|\partial_x^{\alpha}f|_{\nu}\, \mathrm{d}x\lesssim\int_{\mathbb{R}_x^3}|\partial_x^{\alpha_2}f|_{H_v^2}|\partial_x^{\alpha_1}f|_{L_v^2}|\partial_x^{\alpha}f|_{L_v^2}\,{\rm d}x\,,
\end{align}
If $|\alpha_1|=0$, then we have
\begin{align*}
    &\hspace{3mm}\int_{\mathbb{R}_{x,v}^6}  \Gamma(f,\partial_x^{\alpha}f)\partial_x^{\alpha}f\, \mathrm{d}v\mathrm{d}x\lesssim\int_{\mathbb{R}_{x}^3}| f|_{H_v^2}|\partial_x^{\alpha}f|_{L_v^2}|\partial_x^{\alpha}f|_{L_v^2}\, \mathrm{d}x\\
   &\lesssim\|{\bf P}f\|_{L_x^\infty H_v^2}\|\partial_x^\alpha f\|^2+\|{\bf P}^\bot f\|_{L_x^\infty H_v^2}\|\partial_x^\alpha f\|^2\lesssim|(a,b,c)|_{\dot{H}^1\cap\dot{H}^2}\|\partial_x^{\alpha}f\|^2+\left\|{\bf P}^\bot f\right\|_{H_v^2(\dot{H}_x^1\cap\dot{H}^2)}\|\partial_x^{\alpha}f\|^2\\
   &\lesssim\epsilon\|\partial_x^{\alpha}f\|^2\,.
\end{align*}
\noindent If $|\alpha_1|= |\alpha|$, we apply \eqref{jj2} and obtain the same estimate by the same argument.

\noindent If $1\le|\alpha_1|\le|\alpha|-1$ and $|\alpha_1|+3\le N$, then it's straightforward to see that
\begin{align*}
    &\hspace{3mm}\int_{\mathbb{R}_{x,v}^6}  \Gamma(\partial_x^{\alpha_1}f,\partial_x^{\alpha_2}f)\partial_x^{\alpha}f\, \mathrm{d}v\mathrm{d}x\\
    &\lesssim\int_{\mathbb{R}_x^3}|\partial_x^{\alpha_1}{\bf P}f|_{H_v^2}|\partial_x^{\alpha_2}f|_{L_v^2}|\partial_x^\alpha f|_{L_v^2}\,{\rm d}x+\int_{\mathbb{R}_x^3}|\partial_x^{\alpha_1}{\bf P}^\bot f|_{H_v^2}|\partial_x^{\alpha_2}f|_{L_v^2}|\partial_x^\alpha f|_{L_v^2}\,{\rm d}x\\
    &\lesssim\|\partial_x^{\alpha_1}{\bf P}f\|_{L_x^3H_v^2}\|\partial_x^{\alpha_2}f\|_{L_x^6L_v^2}\|\partial_x^\alpha f\|+\|\partial_x^{\alpha_1}{\bf P}^\bot f\|_{L_x^3H_v^2}\|\partial_x^{\alpha_2}f\|_{L_x^6L_v^2}\|\partial_x^\alpha f\|\\
    &\lesssim\left(|(a,b,c)|_{\dot{H}_x^{1+|\alpha_1|}\cap\dot{H}_x^{|\alpha_1|}}+\left\|{\bf P}^\bot f\right\|_{H_v^2(\dot{H}^{1+|\alpha_1|}\cap\dot{H}^{|\alpha_1|}}\right)\|\partial_x^{\alpha_2+e_i}f\|\|\partial_x^\alpha f\|\\
    &\lesssim\epsilon^2\|\partial_x^{\alpha_2+e_i}f\|^2+\eta\|\partial_x^\alpha f\|^2\,.
\end{align*}
If $1\le|\alpha_1|\le|\alpha|-1$ and $|\alpha_1|+3> N$, then $1\le|\alpha_2|\le|\alpha|-1,\,|\alpha_2|+3\le N$. We apply \eqref{jj2}, and the same estimate follows by the same argument. Collecting all cases related to $J_{3,2}$, we obtain
\begin{align*}
    J_{3,2}&\lesssim\sum_{\alpha_1+\alpha_2=\alpha}\int_0^t\int_{V_0\times\mathbb{R}_x^3}e^{-(t-\tau)}\Gamma(\partial_x^{\alpha_1}f,\partial_x^{\alpha_2}f)\partial_x^\alpha f\,{\rm d}v{\rm d}x{\rm d}\tau\\
    &\lesssim\sum_{\alpha_1+\alpha_2=\alpha}\sup_t\int_{\mathbb{R}_{x,v}^6}\Gamma(\partial_x^{\alpha_1}f,\partial_x^{\alpha_2}f)\partial_x^\alpha f\,{\rm d}v{\rm d}x\int_0^t e^{-(t-\tau)}\,{\rm d}\tau\\
    &\lesssim\eta'\sum_{1\le|\alpha'|\le|\alpha|}\sup_t\|\partial_x^{\alpha'}f\|^2\,,
\end{align*}
 where $\eta'=\max\{\epsilon,\epsilon^2,\eta\}$ is a small constant.
\noindent For $J_4$, using Cauchy's inequality and Young's inequality, one can easily derive that
    \begin{align*}
        J_4&=\int_0^t\int_{\mathbb{R}_{x,v}^6}e^{-\nu(v)(t-\tau)} \langle v\rangle^{-\frac{\gamma}{2}}\partial_x^{\alpha}S\,w_{\ell,0}^2\partial_x^{\alpha}f \,\langle v\rangle^{\frac{\gamma}{2}}\, \mathrm{d}v\mathrm{d}x\mathrm{d}\tau\\
        &\lesssim\left(\int_0^t\int_{\mathbb{R}_{x,v}^6}e^{-\nu(v)(t-\tau)}\langle v\rangle^{-\gamma}\left|w_{\ell,0}\partial_x^{\alpha}S\right|^2 \, \mathrm{d}v\mathrm{d}x\mathrm{d}\tau\right)^{\frac{1}{2}}\left(\int_0^t\int_{\mathbb{R}_{x,v}^6}e^{-\nu(v)(t-\tau)}\langle v\rangle^{\gamma}\left|w_{\ell,0}\partial_x^{\alpha}f\right|^2  \, \mathrm{d}v\mathrm{d}x\mathrm{d}\tau\right)^{\frac{1}{2}}\\
        &\lesssim\underbrace{\eta\int_0^t\int_{\mathbb{R}_{x,v}^6}e^{-\nu(v)(t-\tau)}\langle v\rangle^{\gamma}\left|w_{\ell,0}\partial_x^{\alpha}f\right|^2  \, \mathrm{d}v\mathrm{d}x\mathrm{d}\tau}_{J_{4,1}}+\underbrace{C_{\eta}\int_0^t\int_{\mathbb{R}_{x,v}^6}e^{-\nu(v)(t-\tau)}\langle v\rangle^{-\gamma}\left|w_{\ell,0}\partial_x^{\alpha}S\right|^2 \, \mathrm{d}v\mathrm{d}x\mathrm{d}\tau}_{J_{4,2}}\,.
    \end{align*}
Similar to $J_{2,1}$, $J_{4,1}$ can be also absorbed by $J_1$ since $\eta$ is small enough.

\noindent With respect to $J_{4,2}$, taking $\varepsilon>0$, Lemma \ref{lem:max_estimate} implies that
\begin{equation*}
    \begin{aligned}
        J_{4,2}&\lesssim\int_0^t\int_{\mathbb{R}_{x,v}^6}e^{-\nu(v)(t-\tau)}\nu(v)^{1+\varepsilon}\langle v\rangle^{(-2-\varepsilon)\gamma}\left|w_{\ell,0}\partial_x^{\alpha}S\right|^2 \, \mathrm{d}v\mathrm{d}x\mathrm{d}\tau\\
        &\lesssim\int_0^t\int_{\mathbb{R}_{x,v}^6}(1+t-\tau)^{-1-\varepsilon}\langle v\rangle^{(-2-\varepsilon)\gamma}\left|w_{\ell,0}\partial_x^{\alpha}S\right|^2 \, \mathrm{d}v\mathrm{d}x\mathrm{d}\tau\\
        &\lesssim\sup_t\|w_{\ell -(1+\frac{\varepsilon}{2})\gamma,0}\,\partial_x^\alpha S\|^2\int_0^t(1+t-\tau)^{-1-\varepsilon}\,\mathrm{d}\tau\\
        &\lesssim\sup_t\|w_{\ell-(1+\frac{\varepsilon}{2})\gamma,0}\partial_x^\alpha S\|^2\,.
    \end{aligned}
\end{equation*}
Plugging all the above estimates into \eqref{B5} for $|\alpha|=1,2,\dots,N$ , we obtain the desired result.

\noindent \underline{ {\it Case 2: $|\alpha|=N-1, |\beta|=1$.} }
In this case, we consider the microscopic equation.
Applying the operator $\mathbf{P}^{\bot}$ to \eqref{B4} yields
\begin{equation}\label{B6}
    \begin{aligned}
&\hspace{3mm}\partial_t\mathbf{P}^{\bot}f+v\cdot\nabla_x\mathbf{P}^{\bot}f+v\cdot\nabla_x \mathbf{P}f+\frac{1}{2}\nu(v)\mathbf{P}^{\bot}f\\
        &=-\frac{1}{2}\nu(v)\mathbf{P}^{\bot}f+{\bf P}(v\cdot\nabla_x f)+\mathcal{K}\mathbf{P}^{\bot}f+\mathbf{P}^{\bot}\Gamma(f,f)+\mathbf{P}^{\bot}S\quad,
    \end{aligned}
\end{equation}
where we decompose
\[\mathbf{P}^{\bot}(v\cdot\nabla_x f)=v\cdot\nabla_x f-{\bf P}(v\cdot\nabla_xf)=v\cdot\nabla_x \mathbf{P}^{\bot}f+v\cdot\nabla_x \mathbf{P}f-{\bf P}(v\cdot\nabla_xf)\,.\]
Next, we apply the operator $\partial_x^{\alpha}\partial_v$ with $|\alpha|=N-1$ to both sides of \eqref{B6}. Multiplying the resulting identity by $w_{\ell,1}^2\partial_x^{\alpha}\partial_v\mathbf{P}^{\bot}f$, we obtain
\begin{equation*}
    \begin{aligned}
        &\hspace{3mm}\frac{1}{2}\partial_t\left(w_{\ell,1}\partial_x^{\alpha}\partial_v\mathbf{P}^{\bot}f\right)^2+\frac{1}{2}\nu(v)\left(w_{\ell,1}\partial_x^{\alpha}\partial_v\mathbf{P}^{\bot}f\right)^2\\
        &=-\frac{1}{2}\nu(v)\left(w_{\ell,1}\partial_x^{\alpha}\partial_v\mathbf{P}^{\bot}f\right)^2-\partial_v\left[\nu(v)\right]\partial_x^{\alpha}\mathbf{P}^{\bot}f\,w_{\ell,1}^2\partial_x^{\alpha}\partial_v\mathbf{P}^{\bot}f\\
        &\quad-\partial_x^{\alpha}\partial_v\left(v\cdot\nabla_x\mathbf{P}^{\bot}f\right)w_{\ell,1}^2\partial_x^{\alpha}\partial_v\mathbf{P}^{\bot}f-\partial_x^{\alpha}\partial_v(v\cdot\nabla_x \mathbf{P}f)w_{\ell,1}^2\partial_x^{\alpha}\partial_v\mathbf{P}^{\bot}f\\
        &\quad+\partial_x^{\alpha}\partial_v\left[{\bf P}(v\cdot\nabla_xf)\right]w_{\ell,1}^2\partial_x^{\alpha}\partial_v\mathbf{P}^{\bot}f+\partial_x^{\alpha}\partial_v(\mathcal{K}\mathbf{P}^{\bot}f)w_{\ell,1}^2\partial_x^{\alpha}\partial_v\mathbf{P}^{\bot}f\\
        &\quad+\partial_x^{\alpha}\partial_v[\mathbf{P}^{\bot}\Gamma(f,f)]w_{\ell,1}^2\partial_x^{\alpha}\partial_v\mathbf{P}^{\bot}f+\partial_x^{\alpha}\partial_v(\mathbf{P}^{\bot}S)w_{\ell,1}^2\partial_x^{\alpha}\partial_v\mathbf{P}^{\bot}f.
    \end{aligned}
\end{equation*}
Solving the above inequality and then integrating both sides of the result inequality over $\mathbb{R}_x^3\times\mathbb{R}_v^3$, we get

    \begin{align*}
        \left\|w_{\ell,1}\partial_x^{\alpha}\partial_v\mathbf{P}^{\bot}f\right\|^2\lesssim&\left\|w_{\ell,1}\partial_x^{\alpha}\partial_v\mathbf{P}^{\bot}f_0\right\|^2\underbrace{-\frac{1}{2}\int_0^t\int_{\mathbb{R}_{x,v}^6} e^{-\nu(v)(t-\tau)}\nu(v)\left(w_{\ell,1}\partial_x^{\alpha}\partial_v\mathbf{P}^{\bot}f\right)^2\,\mathrm{d}v\mathrm{d}x\mathrm{d}\tau}_{J_5}\\
        &\underbrace{-\int_0^t\int_{\mathbb{R}_{x,v}^6} e^{-\nu(v)(t-\tau)}\partial_v\left[\nu(v)\right]\partial_x^{\alpha}\mathbf{P}^{\bot}f\,w_{\ell,1}^2\partial_x^{\alpha}\partial_v\mathbf{P}^{\bot}f\,\mathrm{d}v\mathrm{d}x\mathrm{d}\tau}_{J_6}\\
        &\underbrace{-\int_0^t\int_{\mathbb{R}_{x,v}^6} e^{-\nu(v)(t-\tau)}\partial_x^{\alpha}\partial_v\left(v\cdot\nabla_x\mathbf{P}^{\bot}f\right)w_{\ell,1}^2\partial_x^{\alpha}\partial_v\mathbf{P}^{\bot}f\,\mathrm{d}v\mathrm{d}x\mathrm{d}\tau}_{J_7}\\
        &\underbrace{-\int_0^t\int_{\mathbb{R}_{x,v}^6} e^{-\nu(v)(t-\tau)}\partial_x^{\alpha}\partial_v(v\cdot\nabla_x \mathbf{P}f)w_{\ell,1}^2\partial_x^{\alpha}\partial_v\mathbf{P}^{\bot}f\,\mathrm{d}v\mathrm{d}x\mathrm{d}\tau}_{J_8}\\
        &\underbrace{+\int_0^t\int_{\mathbb{R}_{x,v}^6} e^{-\nu(v)(t-\tau)}\partial_x^{\alpha}\partial_v\left[{\bf P}(v\cdot\nabla_xf)\right]w_{\ell,1}^2\partial_x^{\alpha}\partial_v\mathbf{P}^{\bot}f\,\mathrm{d}v\mathrm{d}x\mathrm{d}\tau}_{J_9}\\
        &\underbrace{+\int_0^t\int_{\mathbb{R}_{x,v}^6} e^{-\nu(v)(t-\tau)}\partial_x^{\alpha}\partial_v(\mathcal{K}\mathbf{P}^{\bot}f)w_{\ell,1}^2\partial_x^{\alpha}\partial_v\mathbf{P}^{\bot}f\,\mathrm{d}v\mathrm{d}x\mathrm{d}\tau}_{J_{10}}\\
        &\underbrace{+\int_0^t\int_{\mathbb{R}_{x,v}^6} e^{-\nu(v)(t-\tau)}\partial_x^{\alpha}\partial_v[\mathbf{P}^{\bot}\Gamma(f,f)]w_{\ell,1}^2\partial_x^{\alpha}\partial_v\mathbf{P}^{\bot}f\,\mathrm{d}v\mathrm{d}x\mathrm{d}\tau}_{J_{11}}\\
        &\underbrace{+\int_0^t\int_{\mathbb{R}_{x,v}^6} e^{-\nu(v)(t-\tau)}\partial_x^{\alpha}\partial_v(\mathbf{P}^{\bot}S)w_{\ell,1}^2\partial_x^{\alpha}\partial_v\mathbf{P}^{\bot}f\,\mathrm{d}v\mathrm{d}x\mathrm{d}\tau}_{J_{12}}\,.
    \end{align*}
Set $\widetilde{w}_{\ell,1}:=e^{-\frac{\nu(v)}{2}(t-\tau)}w_{\ell,1}$. It can be deduced from Lemma 2 in \cite{Strain-Guo-2008-ARMA} that
\begin{align*}
    &\hspace{3mm}\int_{\mathbb{R}_v^3} \widetilde{w}_{\ell,1}^2\partial_v[\nu(v)]\partial_x^{\alpha}\mathbf{P}^{\bot}f\,\partial_v\partial_x^{\alpha}\mathbf{P}^{\bot}f\,\mathrm{d}v\\
    &\lesssim\eta\sum_{|\beta'|\le 1}|\widetilde{w}_{\ell,\beta'}\partial_x^\alpha\partial_v^{\beta'}{\bf P}^\bot f|_{L_\nu^2}^2+|\bar{\chi}(v)\widetilde{w}_{\ell,1}\partial_x^\alpha{\bf P}^\bot f|^2_{L_v^2}\,.
    \end{align*}
    Then we get
    \begin{align*}
        J_6&\lesssim\eta\sum_{|\beta'|\le1}\int_0^t\int_{\mathbb{R}_{x,v}^6}e^{-\nu(v)(t-\tau)}\nu(v)|w_{\ell,\beta'}\partial_x^\alpha\partial_v^{\beta'}{\bf P}^\bot f|^2\,{\rm d}v{\rm d}x{\rm d}\tau\\
        &\quad+\int_0^t\int_{\mathbb{R}_{x,v}^6}e^{-\nu(v)(t-\tau)}\bar{\chi}^2(v)w^2_{\ell,1}|\partial_x^\alpha{\bf P}^\bot f|^2{\rm d}v{\rm d}x{\rm d}\tau\\
        &\lesssim\eta\sum_{|\beta'|\le1}\int_0^t\int_{\mathbb{R}_{x,v}^6}e^{-\nu(v)(t-\tau)}\nu(v)|w_{\ell,\beta'}\partial_x^\alpha\partial_v^{\beta'}{\bf P}^\bot f|^2\,{\rm d}v{\rm d}x{\rm d}\tau+\sup_t\|\partial_x^\alpha{\bf P}^\bot f\|^2\int_0^t e^{-(t-\tau)}\,{\rm d}\tau\\
        &\lesssim\eta\sum_{|\beta'|\le1}\int_0^t\int_{\mathbb{R}_{x,v}^6}e^{-\nu(v)(t-\tau)}\nu(v)|w_{\ell,\beta'}\partial_x^\alpha\partial_v^{\beta'}{\bf P}^\bot f|^2\,{\rm d}v{\rm d}x{\rm d}\tau+\sup_t\|\partial_x^\alpha{\bf P}^\bot f\|^2\,,
    \end{align*}
    where the first term in the right-hand side of the last inequality can be absorbed by the dissipative terms. 
    
\noindent We next estimate the transport terms.
For $J_7$, it is easy to get that
\begin{align*}
    \hspace{3mm}\int_{\mathbb{R}_x^3} \partial_x^{\alpha}\partial_v\left(v\cdot\nabla_x\mathbf{P}^{\bot}f\right)\partial_x^{\alpha}\partial_v\mathbf{P}^{\bot}f\,\mathrm{d}x&=\sum_{i=1}^3\int_{\mathbb{R}_x^3}\partial_x^{\alpha+e_i}\mathbf{P}^{\bot}f\,\partial_x^{\alpha}\partial_v\mathbf{P}^{\bot}f\,\mathrm{d}x+\int_{\mathbb{R}_x^3}v\cdot\nabla_x\left(\partial_x^{\alpha}\partial_v\mathbf{P}^{\bot}f\right)\partial_x^{\alpha}\partial_v\mathbf{P}^{\bot}f\,\mathrm{d}x\\
    &=\sum_{i=1}^3\int_{\mathbb{R}_x^3}\partial_x^{\alpha+e_i}\mathbf{P}^{\bot}f\,\partial_x^{\alpha}\partial_v\mathbf{P}^{\bot}f\,\mathrm{d}x.
\end{align*}
Here we integrate by parts to obtain the last equality. Then, by Cauchy’s inequality and Young’s inequality, we deduce that
\begin{align*}
    J_7&\lesssim\sum_{i=1}^3\int_0^t\int_{\mathbb{R}_{x,v}^6} e^{-\nu(v)(t-\tau)}w_{\ell,1}^2\partial_x^{\alpha+e_i}\mathbf{P}^{\bot}f\,\partial_x^{\alpha}\partial_v\mathbf{P}^{\bot}f\,\mathrm{d}v\mathrm{d}x\mathrm{d}\tau\\
    &\lesssim\sum_{i=1}^3\int_0^t\int_{\mathbb{R}_{x,v}^6} e^{-\nu(v)(t-\tau)}\langle v\rangle^{\gamma}w_{\ell,0}\partial_x^{\alpha+e_i}\mathbf{P}^{\bot}f\,w_{\ell,1}\partial_x^{\alpha}\partial_v\mathbf{P}^{\bot}f\,\mathrm{d}v\mathrm{d}x\mathrm{d}\tau\\
    &\lesssim\underbrace{\eta\sum_{i=1}^3\int_0^t\int_{\mathbb{R}_{x,v}^6} e^{-\nu(v)(t-\tau)}\left|w_{\ell,1}\partial_x^{\alpha}\partial_v\mathbf{P}^{\bot}f\,\langle v\rangle^{\frac{\gamma}{2}}\right|^2\,\mathrm{d}v\mathrm{d}x\mathrm{d}\tau}_{J_{7,1}}\\
    &\quad+\underbrace{C_{\eta}\sum_{i=1}^3\int_0^t\int_{\mathbb{R}_{x,v}^6} e^{-\nu(v)(t-\tau)}\left|\langle v\rangle^{\frac{\gamma}{2}}w_{\ell,0}\partial_x^{\alpha+e_i}\mathbf{P}^{\bot}f\right|^2\,\mathrm{d}v\mathrm{d}x\mathrm{d}\tau}_{J_{7,2}}\,.
\end{align*}
Now $J_{7,1}$ can be absorbed into $J_5$ provided $\eta$ is chosen sufficiently small and $J_{7,2}$ will be absorbed into $J_1$ after enlarging the corresponding constant. 

\noindent For $J_8$, due to
\[v\cdot\nabla_x\mathbf{P}f\sim\nabla_x(a,b,c)\mu^{\delta},\]
then we have
\begin{align*}
    J_8&\lesssim\int_0^t\int_{\mathbb{R}_{x,v}^6} e^{-\nu(v)(t-\tau)}w_{\ell,1}^2\partial_x^{\alpha}\partial_v\left[\nabla_x(a,b,c)\mu^{\delta}\right]\partial_x^{\alpha}\partial_v\mathbf{P}^{\bot}f\,\mathrm{d}v\mathrm{d}x\mathrm{d}\tau\\
    &\lesssim\int_0^t\int_{\mathbb{R}_{x,v}^6} e^{-\nu(v)(t-\tau)}w_{\ell,1}^2\partial_x^{\alpha+1}(a,b,c)\mu^{\delta}\partial_x^{\alpha}\partial_v\mathbf{P}^{\bot}f\,\mathrm{d}v\mathrm{d}x\mathrm{d}\tau\\
    &\lesssim\underbrace{\eta\sup\limits_t\left\|\partial_x^{\alpha}\partial_v\mathbf{P}^{\bot}f\right\|^2}_{J_{8,1}}+\underbrace{C_{\eta}\sup\limits_t\left\|(a,b,c)\right\|_{\dot{H}^N}^2}_{J_{8,2}},
\end{align*}
where $J_{8,2}$ can be absorbed by magnifying $\eqref{M-E-4}$.

\noindent For $J_9$, recall the definition of ${\bf P}$, denote $\phi_i:=\mu^\delta, v\mu^\delta, (|v|^2-1)\mu^\delta$, we then have 
\[{\bf P}\left(v\cdot\nabla_xf\right)=\int_{\mathbb{R}_v^3} v\cdot\nabla_xf\,\phi_i\, \mathrm{d}v\,\mu^\delta\,,\]
then 
\[\partial_x^{\alpha}\partial_v\left[{\bf P}\left(v\cdot\nabla_xf\right)\right]=\int_{\mathbb{R}_v^3} v\cdot\nabla_x\partial_x^{\alpha}f\,\phi_i\, \mathrm{d}v\,\mu^\delta:=m(t,x)\mu^{\delta}\,.\]
Consequently,
\begin{align*}
    J_9&\lesssim\int_0^t\int_{\mathbb{R}_{x,v}^6} e^{-\nu(v)(t-\tau)}m(t,x)\mu^{\delta}w_{\ell,1}^2\partial_x^{\alpha}\partial_v\mathbf{P}^{\bot}f\,\mathrm{d}v\mathrm{d}x\mathrm{d}\tau\\
    &\lesssim\eta\sup\limits_t\left\|\partial_x^{\alpha}\partial_v\mathbf{P}^{\bot}f\right\|^2+C_{\eta}\sup\limits_t\|m(t,x)\|_{L_x^2}^2\\
    &\lesssim\eta\sup\limits_t\left\|\partial_x^{\alpha}\partial_v\mathbf{P}^{\bot}f\right\|^2+C_{\eta}\sup\limits_t\left\|\int_{\mathbb{R}_v^3} v\cdot\nabla_x\partial_x^{\alpha}f\,\phi_i\, \mathrm{d}v\right\|_{L_x^2}^2\\
    &\lesssim\eta\sup\limits_t\left\|\partial_x^{\alpha}\partial_v\mathbf{P}^{\bot}f\right\|^2+C_\eta\sup\limits_t\|f\|_{L_v^2(\dot{H}^{|\alpha|+1})}^2\,.
\end{align*}
As for $J_{10}$, its estimate follows by the same argument as for $J_2$, yielding
\begin{align*}
J_{10}&\lesssim\sum_{j\in\mathbb{Z}^+}\int_0^t\int_{V_j\times\mathbb{R}_x^3}e^{-\nu(v)(t-\tau)}w^2_{\ell,1}\partial_x^\alpha\partial_v(\mathcal{K}{\bf P}^\bot f)\partial_x^\alpha\partial_v{\bf P}^\bot f\,{\rm d}v{\rm d}x{\rm d}\tau\\
&\quad+\int_0^t\int_{V_0\times\mathbb{R}_x^3}e^{-\nu(v)(t-\tau)}w^2_{\ell,1}\partial_x^\alpha\partial_v(\mathcal{K}{\bf P}^\bot f)\partial_x^\alpha\partial_v{\bf P}^\bot f\,{\rm d}v{\rm d}x{\rm d}\tau\\
&\lesssim\sum_{j\in\mathbb{Z}^+}\sup_t\int_{V_j\times\mathbb{R}_x^3}\nu^{-1}(v)w^2_{\ell,1}\partial_x^\alpha\partial_v(\mathcal{K}{\bf P}^\bot f)\partial_x^\alpha\partial_v{\bf P}^\bot f\,{\rm d}v{\rm d}x\\
&\quad+\int_0^t\int_{V_0\times\mathbb{R}_x^3}e^{-(t-\tau)}w^2_{\ell,1}\partial_x^\alpha\partial_v(\mathcal{K}{\bf P}^\bot f)\partial_x^\alpha\partial_v{\bf P}^\bot f\,{\rm d}v{\rm d}x{\rm d}\tau\\
&\lesssim\eta\sum_{j\in\mathbb{Z}^+}\sup_t\sum_{|\beta'|\le1}\|w_{\ell,\beta'}\partial_x^\alpha\partial_v^{\beta'}{\bf P}^\bot f\|^2_{L^2(V_j)L_x^2(\mathbb{R}_x^3)}+\sup_t\|\partial_x^\alpha{\bf P}^\bot f\|^2\\
&\quad+\eta'\sum_{|\beta'|\le 1}\sup_t\|w_{\ell,\beta'}\partial_x^\alpha\partial_v^{\beta'}{\bf P}^\bot f\|^2+\sup_t\int_{\mathbb{R}_{x,v}^6}w^2_{\ell,1}\partial_x^\alpha\partial_v(\mathcal{K}{\bf P}^\bot f)\partial_x^\alpha\partial_v{\bf P}^\bot f\,{\rm d}v{\rm d}x\\
&\lesssim\eta\sum_{j\in\mathbb{Z}^+}\sum_{|\beta'|\le1}\sup_t\|w_{\ell,\beta'}\partial_x^\alpha\partial_v^{\beta'}{\bf P}^\bot f\|^2_{L^2(V_j)L_x^2(\mathbb{R}_x^3)}+\sup_t\|\partial_x^\alpha{\bf P}^\bot f\|^2\\
&\quad+\eta''\sum_{|\beta'|\le 1}\sup_t\|w_{\ell,\beta'}\partial_x^\alpha\partial_v^{\beta'}{\bf P}^\bot f\|^2\,.
\end{align*}
The term $J_{11}$ is treated in the similar way as $J_3$, we obtain
\begin{align*}
    J_{11}&\lesssim\underbrace{\sum_{j\in\mathbb{Z}^+}\sup_t\int_{V_j\times\mathbb{R}_x^3}\nu^{-1}(v)w^2_{\ell,1}\partial_x^\alpha\partial_v\Gamma(f,f)\partial_x^\alpha\partial_v{\bf P}^\bot f\,{\rm d}v{\rm d}x}_{J_{11,1}}\\
    &\quad+\underbrace{\sup_t\int_{\mathbb{R}_{x,v}^6}w^2_{\ell,1}\partial_x^\alpha\partial_v\Gamma(f,f)\partial_x^\alpha\partial_v{\bf P}^\bot f\,{\rm d}v{\rm d}x}_{J_{11,2}}\,.
\end{align*}
We can deduce from Lemma \ref{lemma5.2} that
\begin{align*}
	&\hspace{3mm}\sup_t\int_{V_j\times\mathbb{R}_x^3}\nu^{-1}(v)w_{\ell,1}^2\Gamma(\partial_x^{\alpha_1}\partial_v^{\beta_1}f,\partial_x^{\alpha_2}\partial_v^{\beta_2}f)\partial_x^\alpha\partial_v{\bf P}^\bot f\,{\rm d}v{\rm d}x\\
	&\lesssim\sup_t\int_{\mathbb{R}_x^3}\left(|\mu^\delta\partial_x^{\alpha_1}\partial_v^{\beta_1}f|_{H^2_v(\mathbb{R}^3)}+|w_{\ell,1}\partial_x^{\alpha_1}\partial_v^{\beta_1}f|_{L^2_v(\mathbb{R}^3)}\right)\times|w_{\ell,1}\partial_x^{\alpha_2}\partial_v^{\beta_2}f|_{L^2(\widetilde{V}_j)}|w_{\ell,1}\partial_x^\alpha\partial_v{\bf P}^\bot f|_{L^2(V_j)}\,{\rm d}x\,,
\end{align*}
or
\begin{align}\label{GG22}
	&\hspace{3mm}\sup_t\int_{V_j\times\mathbb{R}_x^3}\nu^{-1}(v)w_{\ell,1}^2\Gamma(\partial_x^{\alpha_1}\partial_v^{\beta_1}f,\partial_x^{\alpha_2}\partial_v^{\beta_2}f)\partial_x^\alpha\partial_v{\bf P}^\bot f\,{\rm d}v{\rm d}x\\
    &\lesssim\sup_t\int_{\mathbb{R}_x^3}\left(|\mu^\delta\partial_x^{\alpha_2}\partial_v^{\beta_2}f|_{H^2_v(\mathbb{R}^3)}+|w_{\ell,1}\partial_x^{\alpha_2}\partial_v^{\beta_2}f|_{L^2_v(\mathbb{R}^3)}\right)|w_{\ell,1}\partial_x^{\alpha_1}\partial_v^{\beta_1}f|_{L^2_v(\widetilde{V}_j)}|w_{\ell,1}\partial_x^\alpha\partial_v{\bf P}^\bot f|_{L^2_v(V_j)}\,{\rm d}x\,,\notag
\end{align}
where $\alpha_1+\alpha_2=\alpha,\beta_1+\beta_2=\beta$ with $|\alpha|=N-1,|\beta|=1$.

\noindent If $|\alpha_1|=0$, we have 
\begin{align}
&\hspace{3mm}\sup_t\int_{V_j\times\mathbb{R}_x^3}\nu^{-1}(v)w_{\ell,1}^2\Gamma(\partial_x^{\alpha_1}\partial_v^{\beta_1}f,\partial_x^{\alpha_2}\partial_v^{\beta_2}f)\partial_x^\alpha\partial_v{\bf P}^\bot f\,{\rm d}v{\rm d}x\notag\\&\lesssim\sup_t\Bigg\{\Big(\|\mu^\delta\partial_v^{\beta_1}{\bf P}f\|_{L_x^\infty(\mathbb{R}^3) H_v^2(\mathbb{R}^3)}+\|w_{\ell,1}\partial_v^{\beta_1}{\bf P}f\|_{L_x^\infty(\mathbb{R}^3)L_v^2(\mathbb{R}^3)}\Big)\nonumber\\
&\qquad\times\|w_{\ell,1}\partial_x^{\alpha}\partial_v^{\beta_2}{\bf P}f\|_{L_x^2(\mathbb{R}^3)L^2(\widetilde{V}_j)}\|w_{\ell,1}\partial_x^\alpha\partial_v{\bf P}^\bot f\|_{L_x^2(\mathbb{R}^3)L^2(V_j)}\Bigg\}\nonumber\\
	&\quad+\sup_t\Bigg\{\Big(\|\mu^\delta\partial_v^{\beta_1}{\bf P}^\bot f\|_{L_x^\infty(\mathbb{R}^3) H_v^2(\mathbb{R}^3)}+\|w_{\ell,1}\partial_v^{\beta_1}{\bf P}^\bot f\|_{L_x^\infty(\mathbb{R}^3)L_v^2(\mathbb{R}^3)}\Big)\nonumber\\
&\qquad\times\|w_{\ell,1}\partial_x^{\alpha}\partial_v^{\beta_2}{\bf P}^\bot f\|_{L_x^2(\mathbb{R}^3)L^2(\widetilde{V}_j)}\|w_{\ell,1}\partial_x^\alpha\partial_v{\bf P}^\bot f\|_{L_x^2(\mathbb{R}^3)L^2(V_j)}\Bigg\}\nonumber\\
&\quad+\sup_t\Bigg\{\Big(\|\mu^\delta\partial_v^{\beta_1}{\bf P}^\bot f\|_{L_x^\infty(\mathbb{R}^3) H_v^2(\mathbb{R}^3)}+\|w_{\ell,1}\partial_v^{\beta_1}{\bf P}^\bot f\|_{L_x^\infty(\mathbb{R}^3)L_v^2(\mathbb{R}^3)}\Big)\nonumber\\
&\qquad\times\|w_{\ell,1}\partial_x^{\alpha}\partial_v^{\beta_2}{\bf P}f\|_{L_x^2(\mathbb{R}^3)L^2(\widetilde{V}_j)}\|w_{\ell,1}\partial_x^\alpha\partial_v{\bf P}^\bot f\|_{L_x^2(\mathbb{R}^3)L^2(V_j)}\Bigg\}\nonumber\\
	&\quad+\sup_t\Bigg\{\Big(\|\mu^\delta\partial_v^{\beta_1}{\bf P} f\|_{L_x^\infty(\mathbb{R}^3) H_v^2(\mathbb{R}^3)}+\|w_{\ell,1}\partial_v^{\beta_1}{\bf P}f\|_{L_x^\infty(\mathbb{R}^3)L_v^2(\mathbb{R}^3)}\Big)\nonumber\\
&\qquad\times\|w_{\ell,1}\partial_x^{\alpha}\partial_v^{\beta_2}{\bf P}^\bot f\|_{L_x^2(\mathbb{R}^3)L^2(\widetilde{V}_j)}\|w_{\ell,1}\partial_x^\alpha\partial_v{\bf P}^\bot f\|_{L_x^2(\mathbb{R}^3)L^2(V_j)}\Bigg\}\nonumber\\
	&\lesssim\sup_t\Big\{|(a,b,c)|_{\dot{H}_x^1\cap\dot{H}_x^2}\|\partial_x^{\alpha}(a,b,c)|v|^{-1}\mu^\delta\|_{L^2(\widetilde{V}_j)L_x^2(\mathbb{R}^3)}\|w_{\ell,1}\partial_x^\alpha\partial_v{\bf P}^\bot f\|_{L^2(V_j)L_x^2(\mathbb{R}^3)}\Big\}\nonumber\\
	&\quad+\sup_t\Bigg\{\Big(\|w_{\ell,\beta_1}\partial_v^{\beta_1}{\bf P}^\bot f\|_{L_v^2(\dot{H}^1_x\cap\dot{H}^2_x)}+\|\partial_v^{\beta_1}{\bf P}^\bot f\|_{H_v^2(\dot{H}^1\cap\dot{H}^2)}\Big)\nonumber\\
    &\qquad\times\|w_{\ell,\beta_2}\partial_x^\alpha\partial_v^{\beta_2}{\bf P}^\bot f\|_{L^2(\widetilde{V}_j)L_x^2(\mathbb{R}^3)}\|w_{\ell,1}\partial_x^\alpha\partial_v{\bf P}^\bot f\|_{L^2(V_j)L_x^2(\mathbb{R}^3)}\Bigg\}\nonumber\\
    &\quad+\sup_t\Bigg\{\Big(\|w_{\ell,\beta_1}\partial_v^{\beta_1}{\bf P}^\bot f\|_{L_v^2(\dot{H}^1_x\cap\dot{H}^2_x)}+\|\partial_v^{\beta_1}{\bf P}^\bot f\|_{H_v^2(\dot{H}^1\cap\dot{H}^2)}\Big)\nonumber\\
    &\qquad\times\|\partial_x^\alpha(a,b,c)|v|^{-1}\mu^\delta\|_{L^2(\widetilde{V}_j)L_x^2(\mathbb{R}^3)}\|w_{\ell,1}\partial_x^\alpha\partial_v{\bf P}^\bot f\|_{L^2(V_j)L_x^2(\mathbb{R}^3)}\Bigg\}\nonumber\\
    &\quad+\sup_t\Big\{|(a,b,c)|_{\dot{H}_x^1\cap\dot{H}_x^2}\|w_{\ell,\beta_2}\partial_x^\alpha\partial_v^{\beta_2}{\bf P}^\bot f\|_{L^2(\widetilde{V}_j)L_x^2(\mathbb{R}^3)}\|w_{\ell,1}\partial_x^\alpha\partial_v{\bf P}^\bot f\|_{L^2(V_j)L_x^2(\mathbb{R}^3)}\Big\}\nonumber\\
	&\lesssim\epsilon^2\frac{1}{2^{2j}}\sup_t|(a,b,c)|^2_{\dot{H}_x^{|\alpha|}}+\eta\sup_t\|w_{\ell,1}\partial_x^\alpha\partial_v{\bf P}^\bot f\|^2_{L^2(V_j)L_x^2(\mathbb{R}^3)}+\epsilon^2\|w_{\ell,\beta_2}\partial_x^\alpha\partial_v^{\beta_2}{\bf P}^\bot  f\|^2_{L^2(\widetilde{V}_j)L_x^2(\mathbb{R}^3)}\nonumber\,.
\end{align}
If $|\alpha_2|=0$, we apply \eqref{GG22}, and a similar estimate follows by the same argument.

\noindent If $1\le|\alpha_1|\le N-2, |\alpha_1|+|\beta_1|+3\le N$, by applying the $L^3-L^6-L^2$ estimate. We deduce that
\begin{align}
	&\hspace{3mm}\sup_t\int_{V_j\times\mathbb{R}_x^3}\nu^{-1}(v)w_{\ell,1}^2\Gamma(\partial_x^{\alpha_1}\partial_v^{\beta_1}f,\partial_x^{\alpha_2}\partial_v^{\beta_2}f)\partial_x^\alpha\partial_v{\bf P}^\bot f\,{\rm d}v{\rm d}x\notag\\&\lesssim\sup_t\Bigg\{\Big(\|\mu^\delta\partial_x^{\alpha_1}\partial_v^{\beta_1}{\bf P}f\|_{L_x^3(\mathbb{R}^3) H_v^2(\mathbb{R}^3)}+\|w_{\ell,1}\partial_x^{\alpha_1}\partial_v^{\beta_1}{\bf P}f\|_{L_x^3(\mathbb{R}^3)L_v^2(\mathbb{R}^3)}\Big)\nonumber\\
&\qquad\times\|w_{\ell,1}\partial_x^{\alpha_2}\partial_v^{\beta_2}{\bf P}f\|_{L_x^6(\mathbb{R}^3)L^2(\widetilde{V}_j)}\|w_{\ell,1}\partial_x^\alpha\partial_v{\bf P}^\bot f\|_{L_x^2(\mathbb{R}^3)L^2(V_j)}\Bigg\}\nonumber\\
		&\quad+\sup_t\Bigg\{\Big(\|\mu^\delta\partial_x^{\alpha_1}\partial_v^{\beta_1}{\bf P}^\bot f\|_{L_x^3(\mathbb{R}^3) H_v^2(\mathbb{R}^3)}+\|w_{\ell,1}\partial_x^{\alpha_1}\partial_v^{\beta_1}{\bf P}^\bot f\|_{L_x^3(\mathbb{R}^3)L_v^2(\mathbb{R}^3)}\Big)\nonumber\\
	&\qquad\times\|w_{\ell,1}\partial_x^{\alpha_2}\partial_v^{\beta_2}{\bf P}^\bot f\|_{L_x^6(\mathbb{R}^3)L^2(\widetilde{V}_j)}\|w_{\ell,1}\partial_x^\alpha\partial_v{\bf P}^\bot f\|_{L_x^2(\mathbb{R}^3)L^2(V_j)}\Big\}\nonumber\\
    &\quad\sup_t\Bigg\{\Big(\|\mu^\delta\partial_x^{\alpha_1}\partial_v^{\beta_1}{\bf P}f\|_{L_x^3(\mathbb{R}^3) H_v^2(\mathbb{R}^3)}+\|w_{\ell,1}\partial_x^{\alpha_1}\partial_v^{\beta_1}{\bf P}f\|_{L_x^3(\mathbb{R}^3)L_v^2(\mathbb{R}^3)}\Big)\nonumber\\
&\qquad\times\|w_{\ell,1}\partial_x^{\alpha_2}\partial_v^{\beta_2}{\bf P}^\bot f\|_{L_x^6(\mathbb{R}^3)L^2(\widetilde{V}_j)}\|w_{\ell,1}\partial_x^\alpha\partial_v{\bf P}^\bot f\|_{L_x^2(\mathbb{R}^3)L^2(V_j)}\Bigg\}\nonumber\\
		&\quad+\sup_t\Bigg\{\Big(\|\mu^\delta\partial_x^{\alpha_1}\partial_v^{\beta_1}{\bf P}^\bot f\|_{L_x^3(\mathbb{R}^3) H_v^2(\mathbb{R}^3)}+\|w_{\ell,1}\partial_x^{\alpha_1}\partial_v^{\beta_1}{\bf P}^\bot f\|_{L_x^3(\mathbb{R}^3)L_v^2(\mathbb{R}^3)}\Big)\nonumber\\
	&\qquad\times\|w_{\ell,1}\partial_x^{\alpha_2}\partial_v^{\beta_2}{\bf P}f\|_{L_x^6(\mathbb{R}^3)L^2(\widetilde{V}_j)}\|w_{\ell,1}\partial_x^\alpha\partial_v{\bf P}^\bot f\|_{L_x^2(\mathbb{R}^3)L^2(V_j)}\Big\}\nonumber\\
		&\lesssim\sup_t\bigg\{|\partial_x^{\alpha_1}(a,b,c)|_{L_x^2\cap \dot{H}_x^1}\|\partial_x^{\alpha_2}(a,b,c)|v|^{-1}\mu^\delta\|_{L_v^2(\widetilde{V}_j)\dot{H}_x^1(\mathbb{R}^3)}\|w_{\ell,1}\partial_x^\alpha\partial_v{\bf P}^\bot f\|_{L^2(V_j)L_x^2(\mathbb{R}^3)}\Big\}\nonumber\\
        &\quad+\sup_t\Bigg\{\Big(\|w_{\ell,\beta_1}\partial_x^{\alpha_1}\partial_v^{\beta_1}{\bf P}^\bot f\|_{L_v^2{H}_x^1}+\|\partial_x^{\alpha_1}\partial_v^{\beta_1}{\bf P}^\bot f\|_{H_v^2{H}_x^1}\Big)\nonumber\\
        &\qquad\times\|w_{\ell,\beta_2}\partial_x^{\alpha_2}\partial_v^{\beta_2}{\bf P}^\bot f\|_{L^2(\widetilde{V}_j)\dot{H}_x^1}\|w_{\ell,1}\partial_x^{\alpha}\partial_v{\bf P}^\bot f\|_{L^2(V_j)L_x^2(\mathbb{R}^3)}\Bigg\}\nonumber\\
        &\quad+\sup_t\bigg\{|\partial_x^{\alpha_1}(a,b,c)|_{L_x^2\cap \dot{H}_x^1}\|w_{\ell,\beta_2}\partial_x^{\alpha_2}\partial_v^{\beta_2}{\bf P}^\bot f\|_{L^2(\widetilde{V}_j)\dot{H}_x^1}\|w_{\ell,1}\partial_x^\alpha\partial_v{\bf P}^\bot f\|_{L^2(V_j)L_x^2(\mathbb{R}^3)}\Big\}\nonumber\\
        &\quad+\sup_t\Bigg\{\Big(\|w_{\ell,\beta_1}\partial_x^{\alpha_1}\partial_v^{\beta_1}{\bf P}^\bot f\|_{L_v^2{H}_x^1}+\|\partial_x^{\alpha_1}\partial_v^{\beta_1}{\bf P}^\bot f\|_{H_v^2{H}_x^1}\Big)\nonumber\\
        &\qquad\times\|\partial_x^{\alpha_2}(a,b,c)|v|^{-1}\mu^\delta\|_{L^2(\widetilde{V}_j)\dot{H}_x^1}\|w_{\ell,1}\partial_x^{\alpha}\partial_v{\bf P}^\bot f\|_{L^2(V_j)L_x^2(\mathbb{R}^3)}\Bigg\}\nonumber\\
        &\lesssim\epsilon^2\frac{1}{2^{2j}}\sup_t|(a,b,c)|^2_{\dot{H}^{|\alpha_2|+1}}+\eta\sup_t\|w_{\ell,1}\partial_x^\alpha\partial_v{\bf P}^\bot f\|^2_{L^2(V_j)L_x^2}+\epsilon^2\sup_t\|w_{\ell,\beta_2}\partial_x^{\alpha_2+e_i}\partial_v^{\beta_2}{\bf P}^\bot f\|^2_{L^2(\widetilde{V}_j)L_x^2}\nonumber\,.
\end{align}
If $1\le|\alpha_1|\le N-2$ and $|\alpha_1|+|\beta_1|+3>N$, then $1\le|\alpha_2|\le N-2,\,|\alpha_2|+|\beta_2|+3\le N$. We use the formulation \eqref{GG22} and proceed as in the previous case.

\noindent Consequently,
\begin{align*}
    J_{11,1}&\lesssim \sum_{j\in\mathbb{Z}^+}\sum_{\substack{\alpha_1+\alpha_2=\alpha\\\beta_1+\beta_2=\beta}}\sup_t\int_{V_j\times\mathbb{R}_x^3}\nu^{-1}(v)w^2_{\ell,1}\Gamma(\partial_x^{\alpha_1}\partial_v^{\beta_1}f,\partial_x^{\alpha_2}\partial_v^{\beta_2}f)\partial_x^\alpha{\bf P}^\bot f\,{\rm d}v{\rm d}x\\
    &\lesssim\eta\sum_{j\in\mathbb{Z}^+}\sum_{\substack{1\le|\alpha'|\le|\alpha|\\|\beta'|\le|\beta|}}\sup_t\|w_{\ell,\beta'}\partial_x^{\alpha'}\partial_v^{\beta'}{\bf P}^\bot f\|^2_{L^2_v(V_j)L_x^2(\mathbb{R}^3)}+\sup_t|(a,b,c)|^2_{\dot{H}^1\cap\dot{H}^N}\,.
\end{align*}
As to $J_{11,2}$, applying Lemma \ref{lemma9.4} yields
\begin{align*}
    &\hspace{3mm}\int_{\mathbb{R}_{x,v}^6}w_{\ell,1}^2\Gamma(\partial_x^{\alpha_1}\partial_v^{\beta_1}f,\partial_x^{\alpha_2}\partial_v^{\beta_2}f)\partial_x^\alpha\partial_v{\bf P}^\bot f\,\mathrm{d}v\mathrm{d}x\\
    &\lesssim\int_{\mathbb{R}_x^3}\Big(|\mu^\delta\partial_x^{\alpha_1}\partial_v^{\beta_1}{\bf P}f|_{H_v^2}+|w_{\ell,1}\partial_x^{\alpha_1}\partial_v^{\beta_1}{\bf P}f|_{L_v^2}\Big)\times|w_{\ell,1}\partial_x^{\alpha_2}\partial_v^{\beta_2}{\bf P}f|_\nu|w_{\ell,1}\partial_x^\alpha\partial_v{\bf P}^\bot f|_\nu\,{\rm d}x\\
    &\quad+\int_{\mathbb{R}_x^3}\Big(|\mu^\delta\partial_x^{\alpha_1}\partial_v^{\beta_1}{\bf P}^\bot f|_{H_v^2}+|w_{\ell,1}\partial_x^{\alpha_1}\partial_v^{\beta_1}{\bf P}^\bot f|_{L_v^2}\Big)\times|w_{\ell,1}\partial_x^{\alpha_2}\partial_v^{\beta_2}{\bf P}^\bot f|_\nu|w_{\ell,1}\partial_x^\alpha\partial_v{\bf P}^\bot f|_\nu\,{\rm d}x\\
    &\quad+\int_{\mathbb{R}_x^3}\Big(|\mu^\delta\partial_x^{\alpha_1}\partial_v^{\beta_1}{\bf P}f|_{H_v^2}+|w_{\ell,1}\partial_x^{\alpha_1}\partial_v^{\beta_1}{\bf P}f|_{L_v^2}\Big)\times|w_{\ell,1}\partial_x^{\alpha_2}\partial_v^{\beta_2}{\bf P}^\bot f|_\nu|w_{\ell,1}\partial_x^\alpha\partial_v{\bf P}^\bot f|_\nu\,{\rm d}x\\
    &\quad+\int_{\mathbb{R}_x^3}\Big(|\mu^\delta\partial_x^{\alpha_1}\partial_v^{\beta_1}{\bf P}^\bot f|_{H_v^2}+|w_{\ell,1}\partial_x^{\alpha_1}\partial_v^{\beta_1}{\bf P}^\bot f|_{L_v^2}\Big)\times|w_{\ell,1}\partial_x^{\alpha_2}\partial_v^{\beta_2}{\bf P}f|_\nu|w_{\ell,1}\partial_x^\alpha\partial_v{\bf P}^\bot f|_\nu\,{\rm d}x\\
    &:=\widetilde{J}_{11,1}+\widetilde{J}_{11,2}+\widetilde{J}_{11,3}+\widetilde{J}_{11,4}\,,
\end{align*}
here $|\alpha_1|+|\alpha_2|=|\alpha|,|\beta_1|+|\beta_2|=|\beta|$ with $|\alpha|=N-1,\,|\beta|=1$.

\noindent If $|\alpha_1|=0$, we have
\begin{align*}
   \widetilde{J}_{11,1}&\lesssim\left(\left\|\mu^\delta \partial_v^{\beta_1}{\bf P}f\right\|_{L_x^\infty H_v^2}+\left\|{w}_{\ell,1}\partial_v^{\beta_1}{\bf P}f\right\|_{L_x^\infty L_v^2}\right)\times\left\|{w}_{\ell,1}\partial_x^\alpha\partial_v^{\beta_2}{\bf P}f\right\|_\nu\left\|{w}_{\ell,1}\partial_x^\alpha\partial_v{\bf P}^\bot f\right\|_\nu\\
    &\lesssim|(a,b,c)|_{\dot{H}_x^1\cap\dot{H}^2_x}|(a,b,c)|_{\dot{H}_x^{|\alpha|}}\left\|{w}_{\ell,1}\partial_x^\alpha\partial_v{\bf P}^\bot f\right\|\\
    &\lesssim|(a,b,c)|_{\dot{H}_x^{1}\cap\dot{H}^2_x}^2| (a,b,c)|_{\dot{H}_x^{|\alpha|}}^2+\eta\left\|{w}_{\ell,1}\partial_x^\alpha\partial_v{\bf P}^\bot f\right\|^2\\
    &\lesssim\epsilon^2| (a,b,c)|^2_{\dot{H}_x^{|\alpha|}}+\eta\left\|{w}_{\ell,1}\partial_x^\alpha\partial_v{\bf P}^\bot f\right\|^2\,,
\end{align*}
and 
\begin{align*}
    \widetilde{J}_{11,2}&\lesssim\left(\left\|\mu^\delta \partial_v^{\beta_1}{\bf P}^\bot f\right\|_{L_x^\infty H_v^2}+\left\|{w}_{\ell,1}\partial_v^{\beta_1}{\bf P}^\bot f\right\|_{L_x^\infty L_v^2}\right)\times\left\|{w}_{\ell,1}\partial_x^\alpha\partial_v^{\beta_2}{\bf P}^\bot f\right\|_\nu\left\|{w}_{\ell,1}\partial_x^\alpha\partial_v{\bf P}^\bot f\right\|_\nu\\
    &\lesssim\Big(\left\|\partial_v^{\beta_1}{\bf P}^\bot f\right\|_{H_v^2(\dot{H}^{1}\cap\dot{H}^2)}+\left\|w_{\ell,\beta_1}\partial_v^{\beta_1}{\bf P}^\bot f\right\|_{L_v^2(\dot{H}^{1}\cap\dot{H}^2)}\Big)\times\left\|{w}_{\ell,1}\partial_x^\alpha\partial_v^{\beta_2}{\bf P}^\bot f\right\|\left\|{w}_{\ell,1}\partial_x^\alpha\partial_v{\bf P}^\bot f\right\|\\
    &\lesssim\epsilon^2\left\|{w}_{\ell,\beta_2}\partial_x^\alpha\partial_v^{\beta_2}{\bf P}^\bot f\right\|^2+\eta\left\|{w}_{\ell,1}\partial_x^\alpha\partial_v{\bf P}^\bot f\right\|^2\,.
\end{align*}
As to $\widetilde{J}_{11,3}$ and $\widetilde{J}_{11,4}$, one can obtain the following result in the same way:
\begin{align*}
    \widetilde{J}_{11,3}&\lesssim\epsilon^2\|w_{\ell,\beta_2}\partial_x^\alpha\partial_v^{\beta_2}{\bf P}^\bot f\|^2+\eta\|w_{\ell,1}\partial_x^\alpha\partial_v{\bf P}^\bot f\|^2\,,
\end{align*}
and 
\begin{align*}
    \widetilde{J}_{11,4}&\lesssim\epsilon^2|(a,b,c)|^2_{\dot{H}^{|\alpha|}}+\eta\|w_{\ell,1}\partial_x^\alpha\partial_v{\bf P}^\bot f\|^2\,.
\end{align*}
If $|\alpha_2|=0$, we choose the formulation \eqref{rrr} and the same result can be obtained in the same way.

\noindent If $1\le|\alpha_1|\le N-2$ and $|\alpha_1|+|\beta_1|+3\le N$, we apply the $L^3-L^6-L^2$ estimate, one has
\begin{align*}
    \widetilde{J}_{11,1}&\lesssim\left(\left\|\mu^\delta\partial_x^{\alpha_1}\partial_v^{\beta_1}{\bf P}f\right\|_{L_x^3H_v^2}+\left\|{w}_{\ell,1}\partial_x^{\alpha_1}\partial_v^{\beta_1}{\bf P}f\right\|_{L_x^3L_v^2}\right)\times\left\|{w}_{\ell,1}\partial_x^{\alpha_2}\partial_v^{\beta_2}{\bf P}f\right\|_{L_x^6L_{\nu}^2}\left\|{w}_{\ell,1}\partial_x^\alpha\partial_v{\bf P}^\bot f\right\|_\nu\\
    &\lesssim\big|(a,b,c)\big|_{\dot{H}_x^{|\alpha_1|}\cap\dot{H}_x^{|\alpha_1|+1}}\big| (a,b,c)\big|_{\dot{H}_x^{|\alpha_2|+1}}\left\|{w}_{\ell,1}\partial_x^\alpha\partial_v{\bf P}^\bot f\right\|\\
    &\lesssim\big|(a,b,c)\big|_{\dot{H}_x^{|\alpha_1|}\cap\dot{H}_x^{|\alpha_1|+1}}^2\big| (a,b,c)\big|_{\dot{H}_x^{|\alpha_2|+1}}^2+\eta    \left\|{w}_{\ell,1}\partial_x^\alpha\partial_v{\bf P}^\bot f\right\|^2\\
    &\lesssim\epsilon^2\big|(a,b,c)\big|^2_{\dot{H}^{|\alpha_2|+1}}+\eta\left\|{w}_{\ell,1}\partial_x^\alpha\partial_v{\bf P}^\bot f\right\|^2\,,
\end{align*}
and
\begin{align*}
    \widetilde{J}_{11,2}&\lesssim\left(\left\|\mu^\delta\partial_x^{\alpha_1}\partial_v^{\beta_1}{\bf P}^\bot f\right\|_{L_x^3H_v^2}+\left\|{w}_{\ell,1}\partial_x^{\alpha_1}\partial_v^{\beta_1}{\bf P}^\bot f\right\|_{L_x^3L_v^2}\right)\times\left\|{w}_{\ell,1}\partial_x^{\alpha_2}\partial_v^{\beta_2}{\bf P}^\bot f\right\|_{L_x^6L_{\nu}^2}\left\|{w}_{\ell,1}\partial_x^\alpha\partial_v{\bf P}^\bot f\right\|_\nu\\
    &\lesssim\Big(\left\|\partial_x^{\alpha_1}\partial_v^{\beta_1}{\bf P}^\bot f\right\|_{H_v^2H_x^1}+\left\|w_{\ell,\beta_1}\partial_x^{\alpha_1}\partial_v^{\beta_1}{\bf P}^\bot f\right\|_{L_v^2H_x^1}\Big)\left\|w_{\ell,\beta_2}\partial_x^{\alpha_2+e_i}\partial_v^{\beta_2}{\bf P}^\bot f\right\|\left\|{w}_{\ell,1}\partial_x^\alpha\partial_v{\bf P}^\bot f\right\|\\
    &\lesssim\epsilon^2\left\|{w}_{\ell,\beta_2}\partial_x^{\alpha_2+e_i}\partial_v^{\beta_2}{\bf P}^\bot f\right\|^2+\eta\left\|{w}_{\ell,1}\partial_x^\alpha\partial_v{\bf P}^\bot f\right\|^2.
\end{align*}
For $\widetilde{J}_{11,3}$ and $\widetilde{J}_{11,4}$, we also obtain
\begin{align*}
    \widetilde{J}_{11,3}\lesssim\epsilon^2\left\|{w}_{\ell,\beta_2}\partial_x^{\alpha_2+e_i}\partial_v^{\beta_2}{\bf P}^\bot f\right\|^2+\eta\left\|{w}_{\ell,1}\partial_x^\alpha\partial_v{\bf P}^\bot f\right\|^2\,,
\end{align*}
and 
\begin{align*}
    \widetilde{J}_{11,4}\lesssim\epsilon^2\big|(a,b,c)\big|^2_{\dot{H}^{|\alpha_2|+1}}+\eta\left\|{w}_{\ell,1}\partial_x^\alpha\partial_v{\bf P}^\bot f\right\|^2\,.
\end{align*}
If $1\le|\alpha_1|$ and $|\alpha_1|+|\beta_1|+3>N$, we then choose the formulation \eqref{rrr} and use the same method.

\noindent Consequently, we have
\begin{align*}
    J_{11,2}&\lesssim\sup_t\|f\|_{L_v^2(\dot{H}^1\cap\dot{H}^N)}^2+\eta\sum_{\substack{1\le|\alpha'|\le|\alpha|\\|\beta'|\le|\beta|}}\sup_t\|w_{\ell,\beta'}\partial_x^{\alpha'}\partial_v^{\beta'}{\bf P}^\bot f\|^2\,.
\end{align*}
The estimate of $J_{12}$ is similar to $J_4$, that is,
\begin{align*}
    J_{12}\lesssim\underbrace{\eta\int_0^t\int_{\mathbb{R}_{x,v}^6}e^{-\nu(v)(t-\tau)}\langle v\rangle^{\gamma}\left|w_{\ell,1}\partial_x^{\alpha}\partial_v\mathbf{P}^{\bot}f\right|^2  \, \mathrm{d}v\mathrm{d}x\mathrm{d}\tau}_{J_{12,1}}+\sup_t\|w_{\ell-(1+\frac{\varepsilon}{2})\gamma,1}\partial_x^\alpha\partial_v S\|^2\,.
\end{align*}
And $J_{12,1}$ will be absorbed by $J_5$ provided that $\eta$ is small enough.

\noindent\underline{\it {Case 3: $|\alpha|=0, |\beta|=N$.}}
Applying the operator $\partial_v^{\beta}$ with $|\beta|=N$ to \eqref{B6}, and multiplying the resulting equation by $w_{\ell,N}^2\partial_v^\beta \mathbf{P}^{\bot}f$, we obtain
\begin{align*}
    &\hspace{3mm}\frac{1}{2}\partial_t\left(w_{\ell,N}\partial_v^\beta \mathbf{P}^{\bot}f\right)^2+\frac{1}{2}\nu(v)\left(w_{\ell,N}\partial_v^\beta \mathbf{P}^{\bot}f\right)^2\\
        &=-\frac{1}{2}\nu(v)\left(w_{\ell,N}\partial_v^\beta \mathbf{P}^{\bot}f\right)^2-\sum_{|\beta'|=1}^{N
        }\partial_v^{\beta'} \left[\nu(v)\right]\partial_v^{\beta-\beta'}\mathbf{P}^{\bot}f\,w_{\ell,N}^2\partial_v^\beta \mathbf{P}^{\bot}f\\
        &\quad-\partial_v^\beta \left(v\cdot\nabla_x\mathbf{P}^{\bot}f\right)w_{\ell,N}^2\partial_v^\beta \mathbf{P}^{\bot}f-\partial_v^\beta (v\cdot\nabla_x \mathbf{P}f)w_{\ell,N}^2\partial_v^\beta \mathbf{P}^{\bot}f\\
        &\quad+\partial_v^\beta \left[{\bf P}(v\cdot\nabla_xf)\right]w_{\ell,N}^2\partial_v^\beta \mathbf{P}^{\bot}f+\partial_v^\beta (\mathcal{K}\mathbf{P}^{\bot}f)w_{\ell,N}^2\partial_v^\beta \mathbf{P}^{\bot}f\\
        &\quad+\partial_v^\beta [\mathbf{P}^{\bot}\Gamma(f,f)]w_{\ell,N}^2\partial_v^\beta \mathbf{P}^{\bot}f+\partial_v^\beta (\mathbf{P}^{\bot}S)w_{\ell,N}^2\partial_v^\beta \mathbf{P}^{\bot}f.
\end{align*}
Solving the above equation directly and then integrating both sides of the result inequality over $\mathbb{R}_x^3\times\mathbb{R}_v^3$, one has

    \begin{align}
        \left\|w_{\ell,N}\partial_v^\beta \mathbf{P}^{\bot}f\right\|^2\lesssim&\left\|w_{\ell,N}\partial_v^\beta \mathbf{P}^{\bot}f_0\right\|^2\underbrace{-\frac{1}{2}\int_0^t\int_{\mathbb{R}_{x,v}^6} e^{-\nu(v)(t-\tau)}\nu(v)\left(w_{\ell,N}\partial_v^\beta \mathbf{P}^{\bot}f\right)^2\,\mathrm{d}v\mathrm{d}x\mathrm{d}\tau}_{J_{13}}\notag\\
        &\underbrace{-\sum_{|\beta'|=1}^{N
        }\int_0^t\int_{\mathbb{R}_{x,v}^6} e^{-\nu(v)(t-\tau)}\partial_v^{\beta'} \left[\nu(v)\right]\partial_v^{\beta-\beta'}\mathbf{P}^{\bot}f\,w_{\ell,N}^2\partial_v^\beta \mathbf{P}^{\bot}f\,\mathrm{d}v\mathrm{d}x\mathrm{d}\tau}_{J_{14}}\notag\\
        &\underbrace{-\int_0^t\int_{\mathbb{R}_{x,v}^6} e^{-\nu(v)(t-\tau)}\partial_v^\beta \left(v\cdot\nabla_x\mathbf{P}^{\bot}f\right)w_{\ell,N}^2\partial_v^\beta \mathbf{P}^{\bot}f\,\mathrm{d}v\mathrm{d}x\mathrm{d}\tau}_{J_{15}}\notag\\
        &\underbrace{-\int_0^t\int_{\mathbb{R}_{x,v}^6} e^{-\nu(v)(t-\tau)}\partial_v^\beta (v\cdot\nabla_x \mathbf{P}f)w_{\ell,N}^2\partial_v^\beta \mathbf{P}^{\bot}f\,\mathrm{d}v\mathrm{d}x\mathrm{d}\tau}_{J_{16}}\notag\\
        &\underbrace{+\int_0^t\int_{\mathbb{R}_{x,v}^6} e^{-\nu(v)(t-\tau)}\partial_v^\beta \left[{\bf P}(v\cdot\nabla_xf)\right]w_{\ell,N}^2\partial_v^\beta \mathbf{P}^{\bot}f\,\mathrm{d}v\mathrm{d}x\mathrm{d}\tau}_{J_{17}}\notag\\
        &\underbrace{+\int_0^t\int_{\mathbb{R}_{x,v}^6} e^{-\nu(v)(t-\tau)}\partial_v^\beta (\mathcal{K}\mathbf{P}^{\bot}f)w_{\ell,N}^2\partial_v^\beta \mathbf{P}^{\bot}f\,\mathrm{d}v\mathrm{d}x\mathrm{d}\tau}_{J_{18}}\notag\\
        &\underbrace{+\int_0^t\int_{\mathbb{R}_{x,v}^6} e^{-\nu(v)(t-\tau)}\partial_v^\beta [\mathbf{P}^{\bot}\Gamma(f,f)]w_{\ell,N}^2\partial_v^\beta \mathbf{P}^{\bot}f\,\mathrm{d}v\mathrm{d}x\mathrm{d}\tau}_{J_{19}}\notag\\
        &\underbrace{+\int_0^t\int_{\mathbb{R}_{x,v}^6} e^{-\nu(v)(t-\tau)}\partial_v^\beta (\mathbf{P}^{\bot}S)w_{\ell,N}^2\partial_v^\beta \mathbf{P}^{\bot}f\,\mathrm{d}v\mathrm{d}x\mathrm{d}\tau}_{J_{20}}\,.\notag
    \end{align}

Similar to $J_6$, if we denote $\widetilde{w}_{\ell,N}:=e^{-\frac{\nu(v)}{2}(t-\tau)}w_{\ell,N}$, it holds from Lemma 2 in \cite{Strain-Guo-2008-ARMA} that 
\begin{align*}
    \sum_{|\beta'|=1}^{|\beta|}\int_{\mathbb{R}_v^3}\widetilde{w}^2_{\ell,N}\partial_v^{\beta'}[\nu(v)]\partial_v^{\beta-\beta'}{\bf P}^\bot f\partial_v^\beta{\bf P}^\bot f\,{\rm d}v\lesssim\eta\sum_{|\beta'|\le|\beta|}|\widetilde{w}_{\ell,\beta'}\partial_v^{\beta'}{\bf P}^\bot f|^2_{L_\nu^2}+|\bar{\chi}(v)\widetilde{w}_{\ell,N}{\bf P}^\bot f|^2_{L_v^2}\,.
\end{align*}
We further obtain
\begin{align*}
    J_{14}&\lesssim\eta\sum_{|\beta'|\le|\beta|}\int_0^t\int_{\mathbb{R}_{x,v}^6}e^{-\nu(v)(t-\tau)}\nu(v)|w_{\ell,\beta'}\partial_v^{\beta'}{\bf P}^\bot f|^2\,{\rm d}v{\rm d}x{\rm d}\tau+\sup_t\|{\bf P}^\bot f\|^2\,.
\end{align*}
For $J_{15}$, it's straightforward to see that
\begin{align*}
     &\hspace{3mm}\int_{\mathbb{R}_x^3} \partial_v^{\beta}\left(v\cdot\nabla_x\mathbf{P}^{\bot}f\right)\partial_v^{\beta}\mathbf{P}^{\bot}f\,\mathrm{d}x\lesssim\int_{\mathbb{R}_x^3} v\cdot\left(\nabla_x\partial_v^{\beta}\mathbf{P}^{\bot}f\right)\partial_v^{\beta}\mathbf{P}^{\bot}f\,\mathrm{d}x+\int_{\mathbb{R}_x^3} \left(\partial_v^{\beta-e_i}\nabla_x\mathbf{P}^{\bot}f\right)\partial_v^{\beta}\mathbf{P}^{\bot}f\,\mathrm{d}x\\
     &=\int_{\mathbb{R}_x^3} \left(\partial_v^{\beta-e_i}\nabla_x\mathbf{P}^{\bot}f\right)\partial_v^{\beta}\mathbf{P}^{\bot}f\,\mathrm{d}x\,,
\end{align*}
then Cauchy's inequality and Young's inequality imply
\begin{align*}
    J_{15}&\lesssim\int_0^t\int_{\mathbb{R}_{x,v}^6} e^{-\nu(v)(t-\tau)}\,w_{\ell,\beta-e_i}\,\partial_v^{\beta-e_i}\nabla_x\mathbf{P}^{\bot}f\,\langle v\rangle^{\gamma}w_{\ell,N}\partial_v^{\beta}\mathbf{P}^{\bot}f\,\mathrm{d}v\mathrm{d}x\mathrm{d}\tau\\
    &\lesssim C_{\eta}\int_0^t\int_{\mathbb{R}_{x,v}^6} e^{-\nu(v)(t-\tau)}\,w_{\ell,\beta-e_i}^2\,\left|\partial_v^{\beta-e_i}\nabla_x\mathbf{P}^{\bot}f\right|^2\,\langle v\rangle^{\gamma}\,\mathrm{d}v\mathrm{d}x\mathrm{d}\tau\\
    &\quad+\eta\int_0^t\int_{\mathbb{R}_{x,v}^6} e^{-\nu(v)(t-\tau)}\,w_{\ell,N}^2\,\langle v\rangle^{\gamma}\left|\partial_v^{\beta}\mathbf{P}^{\bot}f\right|^2\,\mathrm{d}v\mathrm{d}x\mathrm{d}\tau\,.
\end{align*}
For $J_{16}$, 
\begin{align*}
    J_{16}&\lesssim\int_0^t\int_{\mathbb{R}_{x,v}^6} e^{-\nu(v)(t-\tau)}\nabla_x(a,b,c)\,\mu^{\delta}w_{\ell,N}^2\partial_v^{\beta}\mathbf{P}^{\bot}f\,\mathrm{d}v\mathrm{d}x\mathrm{d}\tau\\
    &\lesssim\sup\limits_t\left\|(a,b,c)\right\|_{\dot{H}^1}^2+\eta\sup\limits_t\left\|\partial_v^{\beta}\mathbf{P}^{\bot}f\right\|^2\,.
\end{align*}
For $J_{17}$, similar to $J_9$, set $\tilde{m}(t,x):=\int_{\mathbb{R}_v^3} v\cdot\nabla_xf\,\phi_i\, \mathrm{d}v$, then we have
\begin{align*}
    J_{17}&\lesssim\int_0^t\int_{\mathbb{R}_{x,v}^6} e^{-\nu(v)(t-\tau)}\tilde{m}(t,x)\mu^{\delta}w_{\ell,N}^2\partial_v^{\beta}\mathbf{P}^{\bot}f\,\mathrm{d}v\mathrm{d}x\mathrm{d}\tau\\
    &\lesssim\sup\limits_t\|\tilde{m}(t,x)\|_{L_x^2}^2+\eta\sup\limits_t\left\|w_{\ell,N}\partial_v^{\beta}\mathbf{P}^{\bot}f\right\|^2\\
    &\lesssim\sup\limits_t\|f\|_{L_v^2(\dot{H}^1)}^2+\eta\sup\limits_t\left\|w_{\ell,N}\partial_v^{\beta}\mathbf{P}^{\bot}f\right\|^2\,.
\end{align*}
As to $J_{18}$, we have
\begin{align*}
    J_{18}&\lesssim\sum_{j\in\mathbb{Z}^+}\sup_t\int_{V_j\times\mathbb{R}_x^3}\nu^{-1}(v)w^2_{\ell,N}\partial_v^\beta(\mathcal{K}{\bf P}^\bot f)\partial_v^\beta{\bf P}^\bot f\,{\rm d}v{\rm d}x\\
    &\quad+\int_0^t\int_{V_0\times\mathbb{R}_x^3}e^{-\nu(v)(t-\tau)}w^2_{\ell,N}\partial_v^\beta(\mathcal{K}{\bf P}^\bot f)\partial_v^\beta{\bf P}^\bot f\,{\rm d}v{\rm d}x{\rm d}\tau\\
    &\lesssim\eta\sum_{j\in\mathbb{Z}^+}\sum_{|\beta'|\le|\beta|}\sup_t\|w_{\ell,\beta'}\partial_v^{\beta'}{\bf P}^\bot f\|^2_{L_v^2(V_j)L_x^2(\mathbb{R}^3)}+\sup_t\|{\bf P}^\bot f\|^2\\
    &\quad+\eta'\sum_{|\beta'|\le|\beta|}\sup_t\|w_{\ell,\beta'}\partial_v^{\beta'}{\bf P}^\bot f\|^2+\sup_t\int_{\mathbb{R}_{x,v}^6}w^2_{\ell,N}\partial_v^\beta(\mathcal{K}{\bf P}^\bot f)\partial_v^\beta{\bf P}^\bot f\,{\rm d}v{\rm d}x\\
    &\lesssim\eta\sum_{j\in\mathbb{Z}^+}\sum_{|\beta'|\le|\beta|}\sup_t\|w_{\ell,\beta'}\partial_v^{\beta'}{\bf P}^\bot f\|^2_{L^2(V_j)L_x^2(\mathbb{R}^3)}+\sup_t\|{\bf P}^\bot f\|^2+\eta''\sum_{|\beta'|\le|\beta|}\sup_t\|w_{\ell,\beta'}\partial_v^{\beta'}{\bf P}^\bot f\|^2\,.
\end{align*}
And for $J_{19}$, we have 
\begin{align*}
    J_{19}&\lesssim\underbrace{\sum_{j\in\mathbb{Z}^+}\sup_t\int_{V_j\times\mathbb{R}_x^3}\nu^{-1}(v)w^2_{\ell,N}\partial_v^{\beta}\Gamma(f,f)\partial_v^\beta{\bf P}^\bot f\,{\rm d}v{\rm d}x}_{J_{19,1}}+\underbrace{\sup_t\int_{\mathbb{R}_{x,v}^6}w^2_{\ell,N}\partial_v^{\beta}\Gamma(f,f)\partial_v^\beta{\bf P}^\bot f\,{\rm d}v{\rm d}x}_{J_{19,2}}\,.
\end{align*}
We deduce from Lemma \ref{lemma5.2} that
\begin{align*}
    &\hspace{3mm}\sup_t\int_{V_j\times\mathbb{R}_x^3}\nu^{-1}(v)w_{\ell,N}^2\Gamma(\partial_v^{\beta_1}f,\partial_v^{\beta_2}f)\partial_v^\beta{\bf P}^\bot f\,{\rm d}v{\rm d}x\\
    &\lesssim\sup_t\int_{\mathbb{R}_x^3}\Big(|\mu^\delta\partial_v^{\beta_1}{\bf P}f|_{H_v^2(\mathbb{R}^3)}+|w_{\ell,N}\partial_v^{\beta_1}{\bf P}f|_{L^2_v(\mathbb{R}^3)}\Big)\times|w_{\ell,N}\partial_v^{\beta_2}{\bf P}f|_{L^2(\widetilde{V}_j)}|w_{\ell,N}\partial_v^\beta{\bf P}^\bot f|_{L^2(V_j)}\,{\rm d}x\\
    &\quad+\sup_t\int_{\mathbb{R}_x^3}\Big(|\mu^\delta\partial_v^{\beta_1}{\bf P}^\bot f|_{H_v^2(\mathbb{R}^3)}+|w_{\ell,N}\partial_v^{\beta_1}{\bf P}^\bot f|_{L^2_v(\mathbb{R}^3)}\Big)\times|w_{\ell,N}\partial_v^{\beta_2}{\bf P}^\bot f|_{L^2(\widetilde{V}_j)}|w_{\ell,N}\partial_v^\beta{\bf P}^\bot f|_{L^2(V_j)}\,{\rm d}x\\
    &\quad+\sup_t\int_{\mathbb{R}_x^3}\Big(|\mu^\delta\partial_v^{\beta_1}{\bf P}f|_{H_v^2(\mathbb{R}^3)}+|w_{\ell,N}\partial_v^{\beta_1}{\bf P}f|_{L^2_v(\mathbb{R}^3)}\Big)\times|w_{\ell,N}\partial_v^{\beta_2}{\bf P}^\bot f|_{L^2(\widetilde{V}_j)}|w_{\ell,N}\partial_v^\beta{\bf P}^\bot f|_{L^2(V_j)}\,{\rm d}x\\
    &\quad+\sup_t\int_{\mathbb{R}_x^3}\Big(|\mu^\delta\partial_v^{\beta_1}{\bf P}^\bot f|_{H_v^2(\mathbb{R}^3)}+|w_{\ell,N}\partial_v^{\beta_1}{\bf P}^\bot f|_{L^2_v(\mathbb{R}^3)}\Big)\times|w_{\ell,N}\partial_v^{\beta_2}{\bf P}f|_{L^2(\widetilde{V}_j)}|w_{\ell,N}\partial_v^\beta{\bf P}^\bot f|_{L^2(V_j)}\,{\rm d}x\\
    &:=\widetilde{J}_{19,1}+\widetilde{J}_{19,2}+\widetilde{J}_{19,3}+\widetilde{J}_{19,4}\,,
\end{align*}
or 
\begin{align}\label{GG222}
    &\hspace{3mm}\sup_t\int_{V_j\times\mathbb{R}_x^3}\nu^{-1}(v)w_{\ell,N}^2\Gamma(\partial_v^{\beta_1}f,\partial_v^{\beta_2}f)\partial_v^\beta{\bf P}^\bot f\,{\rm d}v{\rm d}x\\
    &\lesssim\sup_t\int_{\mathbb{R}_x^3}\Big(|\mu^\delta\partial_v^{\beta_2}{\bf P}f|_{H_v^2(\mathbb{R}^3)}+|w_{\ell,N}\partial_v^{\beta_2}{\bf P}f|_{L^2_v(\mathbb{R}^3)}\Big)\times|w_{\ell,N}\partial_v^{\beta_1}{\bf P}f|_{L^2(\widetilde{V}_j)}|w_{\ell,N}\partial_v^\beta{\bf P}^\bot f|_{L^2(V_j)}\,{\rm d}x\notag\\
    &\quad+\sup_t\int_{\mathbb{R}_x^3}\Big(|\mu^\delta\partial_v^{\beta_2}{\bf P}^\bot f|_{H_v^2(\mathbb{R}^3)}+|w_{\ell,N}\partial_v^{\beta_2}{\bf P}^\bot f|_{L^2_v(\mathbb{R}^3)}\Big)\times|w_{\ell,N}\partial_v^{\beta_1}{\bf P}^\bot f|_{L^2(\widetilde{V}_j)}|w_{\ell,N}\partial_v^\beta{\bf P}^\bot f|_{L^2(V_j)}\,{\rm d}x\notag\\
    &\quad+\sup_t\int_{\mathbb{R}_x^3}\Big(|\mu^\delta\partial_v^{\beta_2}{\bf P}f|_{H_v^2(\mathbb{R}^3)}+|w_{\ell,N}\partial_v^{\beta_2}{\bf P}f|_{L^2_v(\mathbb{R}^3)}\Big)\times|w_{\ell,N}\partial_v^{\beta_1}{\bf P}^\bot f|_{L^2(\widetilde{V}_j)}|w_{\ell,N}\partial_v^\beta{\bf P}^\bot f|_{L^2(V_j)}\,{\rm d}x\notag\\
    &\quad+\sup_t\int_{\mathbb{R}_x^3}\Big(|\mu^\delta\partial_v^{\beta_2}{\bf P}^\bot f|_{H_v^2(\mathbb{R}^3)}+|w_{\ell,N}\partial_v^{\beta_2}{\bf P}^\bot f|_{L^2_v(\mathbb{R}^3)}\Big)\times|w_{\ell,N}\partial_v^{\beta_1}{\bf P}f|_{L^2(\widetilde{V}_j)}|w_{\ell,N}\partial_v^\beta{\bf P}^\bot f|_{L^2(V_j)}\,{\rm d}x\,,\notag
\end{align}
where $|\beta_1|+|\beta_2|=|\beta|=N$.

\noindent If $|\beta_1|+3\le N$ and $|\beta_2|\le N-1$, we apply $L^4-L^4-L^2$ estimate to obtain
\begin{align*}
    \widetilde{J}_{19,1}&\lesssim\sup_t\Bigg\{\Big(\|\mu^\delta\partial_v^{\beta_1}{\bf P}f\|_{L_x^4(\mathbb{R}^3)H_v^2(\mathbb{R}^3)}+\|w_{\ell,N}\partial_v^{\beta_1}{\bf P}f\|_{L_x^4(\mathbb{R}^3)L_v^2(\mathbb{R}^3)}\Big)\nonumber\\
    &\qquad\times\|w_{\ell,N}\partial_v^{\beta_2}{\bf P}f\|_{L_x^4(\mathbb{R}^3)L^2(\widetilde{V}_j)}\|w_{\ell,N}\partial_v^\beta{\bf P}^\bot f\|_{L_x^2(\mathbb{R}^3)L^2(V_j)}\Bigg\}\\
    &\lesssim\sup_t\Big\{|(a,b,c)|_{\dot{B}^\frac{1}{2}_{2,\infty}\cap\dot{H}^1}\|(a,b,c)|v|^{-1}\mu^\delta\|_{L_v^2(\widetilde{V}_j)(\dot{B}^\frac{1}{2}_{2,\infty}\cap\dot{H}^1)}\|w_{\ell,N}\partial_v^\beta{\bf P}^\bot f\|_{L_v^2(V_j)L_x^2(\mathbb{R}^3)}\Big\}\\
    &\lesssim\sup_t\Big\{|(a,b,c)|_{\dot{B}^\frac{1}{2}_{2,\infty}\cap\dot{H}^1}^2\frac{1}{2^{2j}}|(a,b,c)|_{\dot{B}^\frac{1}{2}_{2,\infty}\cap\dot{H}^1}^2+\eta\|w_{\ell,N}\partial_v^\beta{\bf P}^\bot f\|_{L_v^2(V_j)L_x^2(\mathbb{R}^3)}^2\Big\}\\
    &\lesssim\frac{1}{2^{2j}}\epsilon^2|(a,b,c)|_{\dot{B}^\frac{1}{2}_{2,\infty}\cap\dot{H}^1}^2+\eta\|w_{\ell,N}\partial_v^\beta{\bf P}^\bot f\|_{L_v^2(V_j)L_x^2(\mathbb{R}^3)}^2\,.
\end{align*}
Notice that $L^2\cap\dot{H}^1\subset\dot{B}^\frac{1}{2}_{2,\infty}$, we then obtain
\begin{align*}
    \widetilde{J}_{19,2}&\lesssim\sup_t\Bigg\{\Big(\|\mu^\delta\partial_v^{\beta_1}{\bf P}^\bot f\|_{L_x^4(\mathbb{R}^3)H_v^2(\mathbb{R}^3)}+\|w_{\ell,N}\partial_v^{\beta_1}{\bf P}^\bot f\|_{L_x^4(\mathbb{R}^3)L_v^2(\mathbb{R}^3)}\Big)\nonumber\\
    &\qquad\times\|w_{\ell,N}\partial_v^{\beta_2}{\bf P}^\bot f\|_{L_x^4(\mathbb{R}^3)L^2(\widetilde{V}_j)}\|w_{\ell,N}\partial_v^\beta{\bf P}^\bot f\|_{L_x^2(\mathbb{R}^3)L^2(V_j)}\Bigg\}\\
    &\lesssim\sup_t\Bigg\{\Big(\|\partial_v^{\beta_1}{\bf P}^\bot f\|_{H_v^2(L_x^2\cap\dot{H}_x^1)}+\|w_{\ell,\beta_1}\partial_v^{\beta_1}{\bf P}^\bot f\|_{L_v^2(L_x^2\cap\dot{H}_x^1)}\Big)\\
    &\qquad\times\|w_{\ell,\beta_2}\partial_v^{\beta_2}{\bf P}^\bot f\|_{L_v^2(\widetilde{V}_j)(L_x^2\cap\dot{H}^1)}\|w_{\ell,N}\partial_v^\beta{\bf P}^\bot f\|_{L_v^2(V_j)L_x^2(\mathbb{R}^3)}\Bigg\}\\
    &\lesssim\epsilon^2\|w_{\ell,\beta_2}\partial_v^{\beta_2}{\bf P}^\bot f\|^2_{L_v^2(\widetilde{V}_j)(L_x^2\cap\dot{H}_x^1)}+\eta\|w_{\ell,N}\partial_v^\beta{\bf P}^\bot f\|_{L_v^2(V_j)L_x^2(\mathbb{R}^3)}^2\,.
\end{align*}
Similarly, one can deduce that
\begin{align*}
    \widetilde{J}_{19,3}&\lesssim\sup_t\Bigg\{\Big(\|\mu^\delta\partial_v^{\beta_1}{\bf P}f\|_{L_x^4(\mathbb{R}^3)H_v^2(\mathbb{R}^3)}+\|w_{\ell,N}\partial_v^{\beta_1}{\bf P}f\|_{L_x^4(\mathbb{R}^3)L_v^2(\mathbb{R}^3)}\Big)\nonumber\\
    &\qquad\times\|w_{\ell,N}\partial_v^{\beta_2}{\bf P}^\bot f\|_{L_x^4(\mathbb{R}^3)L^2(\widetilde{V}_j)}\|w_{\ell,N}\partial_v^\beta{\bf P}^\bot f\|_{L_x^2(\mathbb{R}^3)L^2(V_j)}\Bigg\}\\
    &\lesssim\sup_t\Big\{|(a,b,c)|_{\dot{B}^\frac{1}{2}_{2,\infty}\cap\dot{H}^1}\|w_{\ell,\beta_2}\partial_v^{\beta_2}{\bf P}^\bot f\|_{L_v^2(\widetilde{V}_j)(L_x^2\cap\dot{H}_x^1)}\|w_{\ell,N}\partial_v^\beta{\bf P}^\bot f\|_{L_v^2(V_j)L_x^2(\mathbb{R}^3)}\Big\}\\
    &\lesssim\epsilon^2\|w_{\ell,\beta_2}\partial_v^{\beta_2}{\bf P}^\bot f\|^2_{L_v^2(\widetilde{V}_j)(L_x^2\cap\dot{H}_x^1)}+\eta\|w_{\ell,N}\partial_v^\beta{\bf P}^\bot f\|_{L_v^2(V_j)L_x^2(\mathbb{R}^3)}^2\,,
\end{align*}
and 
\begin{align*}
    \widetilde{J}_{19,4}&\lesssim\sup_t\Bigg\{\Big(\|\mu^\delta\partial_v^{\beta_1}{\bf P}^\bot f\|_{L_x^4(\mathbb{R}^3)H_v^2(\mathbb{R}^3)}+\|w_{\ell,N}\partial_v^{\beta_1}{\bf P}^\bot f\|_{L_x^4(\mathbb{R}^3)L_v^2(\mathbb{R}^3)}\Big)\nonumber\\
    &\lesssim\sup_t\Bigg\{\Big(\|\partial_v^{\beta_1}{\bf P}^\bot f\|_{H_v^2(L_x^2\cap\dot{H}_x^1)}+\|w_{\ell,\beta_1}\partial_v^{\beta_1}{\bf P}^\bot f\|_{L_v^2(L_x^2\cap\dot{H}_x^1)}\Big)\\
    &\qquad\times\|(a,b,c)|v|^{-1}\mu^\delta\|_{L_v^2(\widetilde{V}_j)(\dot{B}^\frac{1}{2}_{2,\infty}\cap\dot{H}^1)}\|w_{\ell,N}\partial_v^\beta{\bf P}^\bot f\|_{L_v^2(V_j)L_x^2(\mathbb{R}^3)}\Bigg\}\\
    &\lesssim\frac{1}{2^{2j}}\epsilon^2|(a,b,c)|_{\dot{B}^\frac{1}{2}_{2,\infty}\cap\dot{H}^1}^2+\eta\|w_{\ell,N}\partial_v^\beta{\bf P}^\bot f\|_{L_v^2(V_j)L_x^2(\mathbb{R}^3)}^2\,.
\end{align*}
If $|\beta_1|+3\le N$ and $|\beta_2|=N$, that is $|\beta_1|=0,|\beta_2|=N$. It's clearly that $|\beta_2|+1> N$, we need to re-estimate  $\widetilde{J}_{19,2}$ and $\widetilde{J}_{19,3}$. By applying $L^\infty-L^2-L^2$ estimate, one can obtain
\begin{align*}
    \widetilde{J}_{19,2}&\lesssim\sup_t\Bigg\{\Big(\|\mu^\delta{\bf P}^\bot f\|_{L_x^\infty(\mathbb{R}^3)H_v^2(\mathbb{R}^3)}+\|w_{\ell,N}{\bf P}^\bot f\|_{L_x^\infty(\mathbb{R}^3)L_v^2(\mathbb{R}^3)}\Big)\nonumber\\
    &\qquad\times\|w_{\ell,N}\partial_v^{\beta}{\bf P}^\bot f\|_{L_x^2(\mathbb{R}^3)L^2(\widetilde{V}_j)}\|w_{\ell,N}\partial_v^\beta{\bf P}^\bot f\|_{L_x^2(\mathbb{R}^3)L^2(V_j)}\Bigg\}\\
    &\lesssim\sup_t\Bigg\{\Big(\|{\bf P}^\bot f\|_{H_v^2(\dot{H}^1_x\cap\dot{H}_x^2)}+\|w_{\ell,0}{\bf P}^\bot f\|_{L_v^2(\dot{H}_x^1\cap\dot{H}_x^2)}\Big)\\
    &\qquad\times\|w_{\ell,N}\partial_v^{\beta}{\bf P}^\bot f\|_{L_v^2(\widetilde{V}_j)L_x^2(\mathbb{R}^3)}\|w_{\ell,N}\partial_v^\beta{\bf P}^\bot f\|_{L_v^2(V_j)L_x^2(\mathbb{R}^3)}\Bigg\}\\
    &\lesssim\epsilon^2\|w_{\ell,N}\partial_v^\beta{\bf P}^\bot f\|^2_{L_v^2(\widetilde{V}_j)L_x^2(\mathbb{R}^3)}+\eta\|w_{\ell,N}\partial_v^\beta{\bf P}^\bot f\|^2_{L_v^2(V_j)L_x^2(\mathbb{R}^3)}\,,
\end{align*}
and 
\begin{align*}
    \widetilde{J}_{19,3}&\lesssim\sup_t\Bigg\{\Big(\|\mu^\delta{\bf P}f\|_{L_x^\infty(\mathbb{R}^3)H_v^2(\mathbb{R}^3)}+\|w_{\ell,N}{\bf P}f\|_{L_x^\infty(\mathbb{R}^3)L_v^2(\mathbb{R}^3)}\Big)\nonumber\\
    &\qquad\times\|w_{\ell,N}\partial_v^{\beta}{\bf P}^\bot f\|_{L_x^2(\mathbb{R}^3)L^2(\widetilde{V}_j)}\|w_{\ell,N}\partial_v^\beta{\bf P}^\bot f\|_{L_x^2(\mathbb{R}^3)L^2(V_j)}\Bigg\}\\
    &\lesssim\sup_t\Big\{|(a,b,c)|_{\dot{H}^1\cap\dot{H}^2}\|w_{\ell,N}\partial_v^{\beta}{\bf P}^\bot f\|_{L_x^2(\mathbb{R}^3)L^2(\widetilde{V}_j)}\|w_{\ell,N}\partial_v^\beta{\bf P}^\bot f\|_{L_x^2(\mathbb{R}^3)L^2(V_j)}\Big\}\\
    &\lesssim\epsilon^2\|w_{\ell,N}\partial_v^\beta{\bf P}^\bot f\|^2_{L_v^2(\widetilde{V}_j)L_x^2(\mathbb{R}^3)}+\eta\|w_{\ell,N}\partial_v^\beta{\bf P}^\bot f\|^2_{L_v^2(V_j)L_x^2(\mathbb{R}^3)}\,.
\end{align*}
If $|\beta_1|+3>N$, then $|\beta_2|+3\le N$. Assume that we further have $|\beta_1|\le N-1$. Then by choosing the form \eqref{GG222}, we will obtain similar result:
\begin{align*}
    &\hspace{3mm}\sup_t\int_{V_j\times\mathbb{R}_x^3}\nu^{-1}(v)w_{\ell,N}^2\Gamma(\partial_v^{\beta_1}f,\partial_v^{\beta_2}f)\partial_v^\beta{\bf P}^\bot f\,{\rm d}v{\rm d}x\\
    &\lesssim\frac{1}{2^{2j}}\epsilon^2|(a,b,c)|_{\dot{B}^\frac{1}{2}_{2,\infty}\cap\dot{H}^1}^2+\eta\|w_{\ell,N}\partial_v^\beta{\bf P}^\bot f\|_{L_v^2(V_j)L_x^2(\mathbb{R}^3)}^2+\epsilon^2\|w_{\ell,\beta_1}\partial_v^{\beta_1}{\bf P}^\bot f\|^2_{L_v^2(\widetilde{V}_j)(L_x^2\cap\dot{H}_x^1)}\,.
\end{align*}
If $|\beta_1|=N, |\beta_2|=0$, we also have
\begin{align*}
    &\hspace{3mm}\sup_t\int_{V_j\times\mathbb{R}_x^3}\nu^{-1}(v)w_{\ell,N}^2\Gamma(\partial_v^{\beta_1}f,\partial_v^{\beta_2}f)\partial_v^\beta{\bf P}^\bot f\,{\rm d}v{\rm d}x\\
    &\lesssim\frac{1}{2^{2j}}\epsilon^2|(a,b,c)|_{\dot{B}^\frac{1}{2}_{2,\infty}\cap\dot{H}^1}^2+\eta\|w_{\ell,N}\partial_v^\beta{\bf P}^\bot f\|_{L_v^2(V_j)L_x^2(\mathbb{R}^3)}^2+\epsilon^2\|w_{\ell,N}\partial_v^\beta{\bf P}^\bot f\|_{L_v^2(\widetilde{V}_j)L_x^2(\mathbb{R}^3)}^2\,.
\end{align*}
Consequently,
\begin{align*}
    J_{19,1}&\lesssim\sum_{j\in\mathbb{Z}^+}\sum_{\beta_1+\beta_2=\beta}\sup_t\int_{V_j\times\mathbb{R}_x^3}\nu^{-1}(v)w^2_{\ell,N}\Gamma(\partial_v^{\beta_1}f,\partial_v^{\beta_2}f)\partial_v^\beta{\bf P}^\bot f\,{\rm d}v{\rm d}x\\
    &\lesssim\eta\sum_{\substack{|\alpha'|+|\beta'|\le N\\|\alpha'|\le 1}}\sum_{j\in\mathbb{Z}^+}\sup_t\|w_{\ell,\beta'}\partial_x^{\alpha'}\partial_v^{\beta'}{\bf P}^\bot f\|^2_{L^2_v(V_j)L_x^2(\mathbb{R}^3)}+\sup_t|(a,b,c)|^2_{\dot{B}^\frac{1}{2}_{2,\infty}\cap\dot{H}^1}\,.
\end{align*}
For $J_{19,2}$, by using Lemma \ref{lemma9.4}, we have
\begin{align*}
    &\hspace{3mm}\int_{\mathbb{R}_{x,v}^6}w^2_{\ell,N}\Gamma(\partial_v^{\beta_1}f,\partial_v^{\beta_2}f)\partial_v^\beta {\bf P}^\bot f\,{\rm d}v\\
    &\lesssim\int_{\mathbb{R}_x^3}\Big(|\mu^\delta\partial_v^{\beta_1}{\bf P}f|+|w_{\ell,N}\partial_v^{\beta_1}{\bf P}f|_{L_v^2}\Big)\times|w_{\ell,N}\partial_v^{\beta_2}{\bf P}f|_{L_v^2}|w_{\ell,N}\partial_v^\beta{\bf P}^\bot f|_{L_v^2}\,{\rm d}x\\
    &\quad+\int_{\mathbb{R}_x^3}\Big(|\mu^\delta\partial_v^{\beta_1}{\bf P}^\bot f|+|w_{\ell,N}\partial_v^{\beta_1}{\bf P}^\bot f|_{L_v^2}\Big)\times|w_{\ell,N}\partial_v^{\beta_2}{\bf P}^\bot f|_{L_v^2}|w_{\ell,N}\partial_v^\beta{\bf P}^\bot f|_{L_v^2}\,{\rm d}x\\
    &\quad+\int_{\mathbb{R}_x^3}\Big(|\mu^\delta\partial_v^{\beta_1}{\bf P}f|+|w_{\ell,N}\partial_v^{\beta_1}{\bf P}f|_{L_v^2}\Big)\times|w_{\ell,N}\partial_v^{\beta_2}{\bf P}^\bot f|_{L_v^2}|w_{\ell,N}\partial_v^\beta{\bf P}^\bot f|_{L_v^2}\,{\rm d}x\\
    &\quad+\int_{\mathbb{R}_x^3}\Big(|\mu^\delta\partial_v^{\beta_1}{\bf P}^\bot f|+|w_{\ell,N}\partial_v^{\beta_1}{\bf P}^\bot f|_{L_v^2}\Big)\times|w_{\ell,N}\partial_v^{\beta_2}{\bf P}f|_{L_v^2}|w_{\ell,N}\partial_v^\beta{\bf P}^\bot f|_{L_v^2}\,{\rm d}x\,.
\end{align*}
The estimate proceeds in a similar way of $J_{19,1}$. We directly give the result:
\begin{align*}
    &\hspace{3mm}\int_{\mathbb{R}_{x,v}^6}w^2_{\ell,N}\Gamma(\partial_v^{\beta_1}f,\partial_v^{\beta_2}f)\partial_v^\beta {\bf P}^\bot f\,{\rm d}v\\
    &\lesssim\epsilon^2|(a,b,c)|^2_{\dot{B}^\frac{1}{2}_{2,\infty}\cap\dot{H}^1}+\eta\|w_{\ell,N}\partial_v^\beta{\bf P}^\bot f\|^2+\epsilon^2\|w_{\ell,\beta_2}\partial_x\partial_v^{\beta_2}{\bf P}^\bot f\|^2\,,
\end{align*}
or 
\begin{align*}
    &\hspace{3mm}\int_{\mathbb{R}_{x,v}^6}w^2_{\ell,N}\Gamma(\partial_v^{\beta_1}f,\partial_v^{\beta_2}f)\partial_v^\beta {\bf P}^\bot f\,{\rm d}v\\
    &\lesssim\epsilon^2|(a,b,c)|^2_{\dot{B}^\frac{1}{2}_{2,\infty}\cap\dot{H}^1}+\eta\|w_{\ell,N}\partial_v^\beta{\bf P}^\bot f\|^2+\epsilon^2\|w_{\ell,\beta_1}\partial_x\partial_v^{\beta_1}{\bf P}^\bot f\|^2\,.
\end{align*}
 Consequently, we have
\begin{align*}
    J_{19,2}&\lesssim\sum_{\beta_1+\beta_2=\beta}\sup_t\int_{\mathbb{R}_{x,v}^6}w^2_{\ell,N}\Gamma(\partial_v^{\beta_1}f,\partial_v^{\beta_2}f)\partial_v^\beta{\bf P}^\bot f\,{\rm d}v{\rm d}x\\
    &\lesssim\sup_t\big|(a,b,c)\big|^2_{\dot{B}^\frac{1}{2}_{2,\infty}\cap\dot{H}^1}+\eta'\sum_{\substack{|\alpha'|+|\beta'|\le N\\|\alpha'|\le 1}}\sup_t\|w_{\ell,\beta'}\partial_x^{\alpha'}\partial_v^{\beta'}{\bf P}^\bot f\|^2\,.
\end{align*}
\noindent The estimate of $J_{20}$ is obtained by the same way as for
$J_{12}$, and we have
\begin{align*}
    J_{20}\lesssim\sup_t\|w_{\ell-(1+\frac{\varepsilon}{2})\gamma,N}\,\partial_v^\beta S\|^2+\eta\int_0^t\int_{\mathbb{R}_{x,v}^6} e^{-\nu(v)(t-\tau)}w_{\ell,N}^2\left|\partial_v^{\beta}\mathbf{P}^{\bot}f\right|^2\,\nu(v)\,\mathrm{d}v\mathrm{d}x\mathrm{d}\tau\,.
\end{align*}
    Collecting all the relevant estimates, taking the supremum in $t$, and then forming a suitable linear combination of the estimates in all cases, we obtain the desired result. This ends the proof of Lemma  \ref{lemma5.3}.
\end{proof}
It remains now to bound $\left\|{\bf P}^\bot f\,\langle v\rangle^{-\frac{\gamma}{2}}\right\|^2$. For result in this direction, we have
\begin{lemma}\label{lemma5.4}
    Let $N\ge5, -3<\gamma<0$, then for $|\alpha|+|\beta|=0$, it holds that
    \begin{align}
        \sup_t\left\|w_{\ell,0}\mathbf{P}^{\bot}f\right\|^2&\lesssim\left\|w_{\ell,0}\mathbf{P}^{\bot}f_0\right\|^2+\sup\limits_t|(a,b,c)|_{\dot{B}_{2,\infty}^\frac{1}{2}\cap\dot{H}^2}^2+\sup\limits_t\|f\|_{L_v^2(\dot{H}^1)}^2+\sup_t\|{\bf P}^\bot f\|^2\notag\\
        &\qquad+\sup_t\|w_{\ell-(1+\frac{\varepsilon}{2})\gamma,0}\,S\|^2+\eta\sum_{j\in\mathbb{Z}^+}\sup_t\|w_{\ell,0}{\bf P}^\bot f\|^2_{L_v^2(V_j)L_x^2(\mathbb{R}^3)}\,,\notag
    \end{align}
    where $\eta>0$ is a small constant.
\end{lemma}
\begin{proof}
Recall the microscopic equation \eqref{B6}, 
\begin{align*}
&\hspace{3mm}\partial_t\mathbf{P}^{\bot}f+v\cdot\nabla_x\mathbf{P}^{\bot}f+v\cdot\nabla_x \mathbf{P}f+\frac{1}{2}\nu(v)\mathbf{P}^{\bot}f\\
        &=-\frac{1}{2}\nu(v)\mathbf{P}^{\bot}f+{\bf P}(v\cdot\nabla_xf)+\mathcal{K}\mathbf{P}^{\bot}f+\mathbf{P}^{\bot}\Gamma(f,f)+\mathbf{P}^{\bot}S\,,
    \end{align*}
by multiplying $w_{\ell,0}^2\mathbf{P}^{\bot}f$, we get
\begin{align*}
    &\hspace{3mm}\frac{1}{2}\partial_t\left(w_{\ell,0}\mathbf{P}^{\bot}f\right)^2+\frac{1}{2}\nu(v)\left(w_{\ell,0}\mathbf{P}^{\bot}f\right)^2\\
        &=-\frac{1}{2}\nu(v)\left(w_{\ell,0}\mathbf{P}^{\bot}f\right)^2
        -\left(v\cdot\nabla_x\mathbf{P}^{\bot}f\right)w_{\ell,0}^2\mathbf{P}^{\bot}f-(v\cdot\nabla_x \mathbf{P}f)w_{\ell,0}^2\mathbf{P}^{\bot}f\\
        &\quad+{\bf P}(v\cdot\nabla_xf)w_{\ell,0}^2\mathbf{P}^{\bot}f+(\mathcal{K}\mathbf{P}^{\bot}f)w_{\ell,0}^2\mathbf{P}^{\bot}f
        +[\mathbf{P}^{\bot}\Gamma(f,f)]w_{\ell,0}^2\mathbf{P}^{\bot}f+(\mathbf{P}^{\bot}S)w_{\ell,0}^2\mathbf{P}^{\bot}f.
\end{align*}
Solving the inequality directly and then integrating the resulting equation over $\mathbb{R}_x^3\times\mathbb{R}_v^3$, we obtain
    \begin{align}\label{5.13}
        \left\|w_{\ell,0}\mathbf{P}^{\bot}f\right\|^2&\lesssim\left\|w_{\ell,0}\mathbf{P}^{\bot}f_0\right\|^2-\underbrace{\int_0^t\int_{\mathbb{R}_{x,v}^6} e^{-\nu(v)(t-\tau)}\nu(v)\left|w_{\ell,0}\mathbf{P}^{\bot}f\right|^2\,\mathrm{d}v\mathrm{d}x\mathrm{d}\tau}_{J_{21}}\notag\\
        &\quad\underbrace{-\int_0^t\int_{\mathbb{R}_{x,v}^6} e^{-\nu(v)(t-\tau)}(v\cdot\nabla_x \mathbf{P}f)w_{\ell,0}^2\mathbf{P}^{\bot}f\,\mathrm{d}v\mathrm{d}x\mathrm{d}\tau}_{J_{22}}\notag\\
        &\quad\underbrace{+\int_0^t\int_{\mathbb{R}_{x,v}^6} e^{-\nu(v)(t-\tau)}{\bf P}(v\cdot\nabla_xf)w_{\ell,0}^2\mathbf{P}^{\bot}f\,\mathrm{d}v\mathrm{d}x\mathrm{d}\tau}_{J_{23}}\notag\\
        &\quad\underbrace{+\int_0^t\int_{\mathbb{R}_{x,v}^6} e^{-\nu(v)(t-\tau)}(\mathcal{K}\mathbf{P}^{\bot}f)w_{\ell,0}^2\mathbf{P}^{\bot}f\,\mathrm{d}v\mathrm{d}x\mathrm{d}\tau}_{J_{24}}\notag\\
        &\quad\underbrace{+\int_0^t\int_{\mathbb{R}_{x,v}^6} e^{-\nu(v)(t-\tau)}[\mathbf{P}^{\bot}\Gamma(f,f)]w_{\ell,0}^2\mathbf{P}^{\bot}f\,\mathrm{d}v\mathrm{d}x\mathrm{d}\tau}_{J_{25}}\notag\\
        &\quad\underbrace{+\int_0^t\int_{\mathbb{R}_{x,v}^6} e^{-\nu(v)(t-\tau)}(\mathbf{P}^{\bot}S)w_{\ell,0}^2\mathbf{P}^{\bot}f\,\mathrm{d}v\mathrm{d}x\mathrm{d}\tau}_{J_{26}}\,.
    \end{align}
All terms on the right-hand side of \eqref{5.13} can be estimated as in Lemma \ref{lemma5.3}.

It is straightforward to see that
\begin{align*}
    J_{22}\lesssim\sup\limits_t\|(a,b,c)\|_{\dot{H}^1}^2+\eta\sup\limits_t\left\|w_{\ell,0}\mathbf{P}^{\bot}f\right\|^2\,,
\end{align*}
and
\begin{align*}
    J_{23}&\lesssim\eta\sup\limits_t\left\|w_{\ell,0}\mathbf{P}^{\bot}f\right\|^2+\sup\limits_t\left\|\int_{\mathbb{R}_v^3}(v\cdot\nabla_xf)\phi_i\,\mathrm{d}v\right\|_{L_x^2}^2\lesssim\eta\sup\limits_t\left\|w_{\ell,0}\mathbf{P}^{\bot}f\right\|^2+\sup\limits_t\|f\|_{L_v^2(\dot{H}^1)}\,.
\end{align*}
For $J_{24}$, lemmas \ref{lemma5.1} and \ref{lemma9.3}  imply that
\begin{align*}
    J_{24}&\lesssim\sum_{j\in\mathbb{Z}^+}\sup_t\int_{V_j\times\mathbb{R}_x^3}\nu^{-1}(v)w^2_{\ell,0}\mathcal{K}{\bf P}^\bot f{\bf P}^\bot f\,{\rm d}v{\rm d}x+\sup_t\int_{V_0\times\mathbb{R}_x^3}e^{-\nu(v)(t-\tau)}w^2_{\ell,0}\mathcal{K}{\bf P}^\bot f{\bf P}^\bot f\,{\rm d}v{\rm d}x\\
    &\lesssim\eta\sum_{j\in\mathbb{Z}^+}\sup_t\|w_{\ell,0}{\bf P}^\bot f\|^2_{L^2(V_j)L_x^2(\mathbb{R}^3)}+\sup_t\|{\bf P}^\bot f\|^2+\sup_t\int_{\mathbb{R}_{x,v}^6}w^2_{\ell,0}\mathcal{K}{\bf P}^\bot f{\bf P}^\bot f\,{\rm d}v{\rm d}x\\
    &\lesssim\eta\sum_{j\in\mathbb{Z}^+}\sup_t\|w_{\ell,0}{\bf P}^\bot f\|^2_{L^2(V_j)L_x^2(\mathbb{R}^3)}+\sup_t\|{\bf P}^\bot f\|^2+\eta'\sup_t\|w_{\ell,0}{\bf P}^\bot f\|^2\,.
\end{align*}
We also get
\begin{align*}
    J_{25}&\lesssim\sum_{j\in\mathbb{Z}^+}\sup_t\int_{V_j\times\mathbb{R}_x^3}\nu^{-1}(v)\Gamma(f,f)w^2_{\ell,0}{\bf P}^\bot f\,{\rm d}v{\rm d}x+\sup_t\int_{\mathbb{R}_{x,v}^6}\Gamma(f,f){\bf P}^\bot f\,{\rm d}v{\rm d}x\\
    &:=J_{25,1}+J_{25,2}\,.
\end{align*}
Lemma \ref{lemma5.2} implies that
    \begin{align}
        &\hspace{3mm}\sup_t\int_{V_j\times\mathbb{R}_{x}^3}\nu^{-1}(v)w^2_{\ell,0}\Gamma(f,f){\bf P}^\bot f\nonumber\\
        &\lesssim\sup_t\int_{\mathbb{R}_x^3}\Big(|\mu^{\delta}{\bf P}f|_{H_v^2(\mathbb{R}^3)}+|w_{\ell,0}{\bf P}f|_{L_v^2(\mathbb{R}^3)}\Big)\times|w_{\ell,0}{\bf P}f|_{L_v^2(\widetilde{V}_j)}|w_{\ell,0}{\bf P}^\bot f|_{L_v^2(V_j)}\nonumber\\
        &\quad+\sup_t\int_{\mathbb{R}_x^3}\Big(|\mu^{\delta}{\bf P}^\bot f|_{H_v^2(\mathbb{R}^3)}+|w_{\ell,0}{\bf P}^\bot f|_{L_v^2(\mathbb{R}^3)}\Big)\times|w_{\ell,0}{\bf P}^\bot f|_{L_v^2(\widetilde{V}_j)}|w_{\ell,0}{\bf P}^\bot f|_{L_v^2(V_j)}\nonumber\\
        &\quad+\sup_t\int_{\mathbb{R}_x^3}\Big(|\mu^{\delta}{\bf P}f|_{H_v^2(\mathbb{R}^3)}+|w_{\ell,0}{\bf P}f|_{L_v^2(\mathbb{R}^3)}\Big)\times|w_{\ell,0}{\bf P}^\bot f|_{L_v^2(\widetilde{V}_j)}|w_{\ell,0}{\bf P}^\bot f|_{L_v^2(V_j)}\nonumber\\
        &\quad+\sup_t\int_{\mathbb{R}_x^3}\Big(|\mu^{\delta}{\bf P}^\bot f|_{H_v^2(\mathbb{R}^3)}+|w_{\ell,0}{\bf P}^\bot f|_{L_v^2(\mathbb{R}^3)}\Big)\times|w_{\ell,0}{\bf P}f|_{L_v^2(\widetilde{V}_j)}|w_{\ell,0}{\bf P}^\bot f|_{L_v^2(V_j)}\nonumber\\
        &\lesssim\sup_t\Bigg\{\Big(\|\mu^\delta{\bf P}f\|_{L_x^4(\mathbb{R}^3)H_v^2(\mathbb{R}^3)}+\|w_{\ell,0}{\bf P}f\|_{L_x^4L_v^2(\mathbb{R}^3)}\Big)\times\|w_{\ell,0}{\bf P}f\|_{L_x^4L_v^2(\widetilde{V}_j)}\|w_{\ell,0}{\bf P}^\bot f\|_{L_x^2L_v^2(V_j)}\Bigg\}\nonumber\\
        &\quad+\sup_t\Bigg\{\Big(\|\mu^\delta{\bf P}^\bot f\|_{L_x^\infty H_v^2(\mathbb{R}^3)}+\|w_{\ell,0}{\bf P}^\bot f\|_{L_x^\infty L_v^2(\mathbb{R}^3)}\Big)\times\|w_{\ell,0}{\bf P}^\bot f\|_{L_x^2L_v^2(\widetilde{V}_j)}\|w_{\ell,0}{\bf P}^\bot f\|_{L_x^2L_v^2(V_j)}\Bigg\}\nonumber\\
        &\quad+\sup_t\Bigg\{\Big(\|\mu^\delta{\bf P}f\|_{L_x^\infty(\mathbb{R}^3)H_v^2(\mathbb{R}^3)}+\|w_{\ell,0}{\bf P}f\|_{L_x^\infty L_v^2(\mathbb{R}^3)}\Big)\times\|w_{\ell,0}{\bf P}^\bot f\|_{L_x^2L_v^2(\widetilde{V}_j)}\|w_{\ell,0}{\bf P}^\bot f\|_{L_x^2L_v^2(V_j)}\Bigg\}\nonumber\\
        &\quad+\sup_t\Bigg\{\Big(\|\mu^\delta{\bf P}^\bot f\|_{L_x^2H_v^2(\mathbb{R}^3)}+\|w_{\ell,0}{\bf P}^\bot f\|_{L_x^2L_v^2(\mathbb{R}^3)}\Big)\times\|w_{\ell,0}{\bf P}f\|_{L_x^\infty L_v^2(\widetilde{V}_j)}\|w_{\ell,0}{\bf P}^\bot f\|_{L_x^2L_v^2(V_j)}\Bigg\}\nonumber\\
        &\lesssim\sup_t\Big\{|(a,b,c)|_{\dot{B}^\frac{1}{2}_{2,\infty}\cap\dot{H}^1}\frac{1}{2^{j}}|(a,b,c)|_{\dot{B}^\frac{1}{2}_{2,\infty}\cap\dot{H}^1}\|w_{\ell,0}{\bf P}^\bot f\|_{L^2(V_j)L_x^2(\mathbb{R}^3)}\Big\}\nonumber\\
        &\quad+\sup_t\Big\{\|w_{\ell,0}{\bf P}^\bot f\|_{H_v^2(\dot{H}^1\cap\dot{H}^2)}\|w_{\ell,0}{\bf P}^\bot f\|_{L^2(\widetilde{V}_j)L_x^2(\mathbb{R}^3)}\|w_{\ell,0}{\bf P}^\bot f\|_{L^2(V_j)L_x^2(\mathbb{R}^3)}\Big\}\nonumber\\
        &\quad+\sup_t\Big\{|(a,b,c)|_{\dot{H}^1\cap\dot{H}^2}\|w_{\ell,0}{\bf P}^\bot f\|_{L_v^2(\widetilde{V}_j)L_x^2}\|w_{\ell,0}{\bf P}^\bot f\|_{L_v^2({V}_j)L_x^2}\Big\}\nonumber\\
        &\quad+\sup_t\Big\{\epsilon\frac{1}{2^{j}}|(a,b,c)|_{\dot{H}^1\cap\dot{H}^2}\|w_{\ell,0}{\bf P}^\bot f\|_{L_v^2({V}_j)L_x^2}\Big\}\nonumber\\
        &\lesssim \epsilon^2\frac{1}{2^{2j}}\sup_t|(a,b,c)|^2_{\dot{B}_{2,\infty}^\frac{1}{2}\cap\dot{H}^2}+\eta\sup_t\|w_{\ell,0}{\bf P}^\bot f\|_{L^2(V_j)L_x^2(\mathbb{R}^3)}^2+\epsilon^2\sup_t\|w_{\ell,0}{\bf P}^\bot f\|_{L^2(\widetilde{V}_j)L_x^2(\mathbb{R}^3)}\,.\nonumber
    \end{align}
    Plugging it into the estimate of $J_{25,1}$ gives
    \begin{align*}
        J_{25,1}&\lesssim\sup_t|(a,b,c)|^2_{\dot{B}^\frac{1}{2}_{2,\infty}\cap\dot{H}^2}+\eta'\sup_t\|w_{\ell,0}{\bf P}^\bot f\|^2_{L_v^2(V_j)L_x^2(\mathbb{R}^3)}\,.
    \end{align*}
As to $J_{25,2}$, Lemma \ref{lemma9.4} gives that
\begin{align*}
   &\hspace{3mm}\int_{\mathbb{R}_{x,v}^6}  \Gamma(f,f)\mathbf{P}^{\bot}f\, \mathrm{d}v\mathrm{d}x\\
   &\lesssim\int_{\mathbb{R}_x^3}|{\bf P}f|_{H_v^2}|{\bf P}f|_{L_v^2}|{\bf P}^\bot f|_{L_v^2}\,{\rm d}x+\int_{\mathbb{R}_x^3}|{\bf P}^\bot f|_{H_v^2}|{\bf P}^\bot f|_{L_v^2}|{\bf P}^\bot f|_{L_v^2}\,{\rm d}x\\
   &\quad+\int_{\mathbb{R}_x^3}|{\bf P} f|_{H_v^2}|{\bf P}^\bot f|_{L_v^2}|{\bf P}^\bot f|_{L_v^2}\,{\rm d}x+\int_{\mathbb{R}_x^3}|{\bf P}^\bot f|_{H_v^2}|{\bf P}f|_{L_v^2}|{\bf P}^\bot f|_{L_v^2}\,{\rm d}x\\
   &\lesssim\|{\bf P}f\|_{L_x^4H_v^2}\|{\bf P}f\|_{L_x^4L_v^2}\|{\bf P}^\bot f\|_{L_x^2L_v^2}+\|{\bf P}^\bot f\|_{L_x^\infty H_v^2}\|{\bf P}^\bot f\|^2\\
   &\quad+\|{\bf P}f\|_{L_x^\infty H_v^2}\|{\bf P}^\bot f\|^2+\|{\bf P}^\bot f\|_{L_x^2H_v^2}\|{\bf P}f\|_{L_x^\infty L_v^2}\|{\bf P}^\bot f\|\\
   &\lesssim|(a,b,c)|^2_{\dot{B}^\frac{1}{2}_{2,\infty}\cap\dot{H}^1}\|{\bf P}^\bot f\|+\|{\bf P}^\bot f\|_{H_v^2(\dot{H}_x^1\cap\dot{H}_x^2)}\|{\bf P}^\bot f\|^2\\
   &\quad+|(a,b,c)|_{\dot{H}^1\cap\dot{H}^2}\|{\bf P}^\bot f\|^2+\|{\bf P}^\bot f\|_{H_v^2L_x^2}|(a,b,c)|_{\dot{H}^1\cap\dot{H}^2}\|{\bf P}^\bot f\|\\
   &\lesssim\epsilon^2|(a,b,c)|_{\dot{B}^\frac{1}{2}_{2,\infty}\cap\dot{H}^2}^2+\eta\|{\bf P}^\bot f\|^2\,.
\end{align*}
Then we have
\begin{align*}
    J_{25,2}
    &\lesssim\epsilon^2\sup_t|(a,b,c)|_{\dot{B}_{2,\infty}^\frac{1}{2}\cap{\dot{H}^2}}^2+\eta\sup_t\left\|\mathbf{P}^{\bot}f\right\|^2\,.
\end{align*}
It follows from Cauchy's inequality that
    \begin{align*}
        J_{26}\lesssim\sup_t\left\|w_{\ell-(1+\frac{\varepsilon}{2})\gamma,0}{\bf P}^\bot S\right\|^2+\eta\int_0^t\int_{\mathbb{R}_{x,v}^6} e^{-\nu(v)(t-\tau)}\nu(v)\left|w_{\ell,0}\mathbf{P}^{\bot}f\right|^2\,\mathrm{d}v\mathrm{d}x\mathrm{d}\tau\,.
    \end{align*}
    Combining all the relating estimates yields the desired result. This completes the proof of Lemma \ref{lemma5.4}.
\end{proof}

\section{Estimate for $\sum_{j\in\mathbb{Z}^+}\sup\limits_t\|w_{\ell,\beta}\partial_x^\alpha\partial_v^\beta{\bf P}^\bot f\|^2_{L^2(V_j)L_x^2(\mathbb{R}^3)}$}\label{Section-V-j}
It can be seen in section \ref{Sect-weight} that the {\it a prior} estimate is not yet closed. To this end, we have to bound the term $\sum_{j\in\mathbb{Z}^+}\sup\limits_t\|w_{\ell,\beta}\partial_x^\alpha\partial_v^\beta{\bf P}^\bot f\|^2_{L^2_v(V_j)L_x^2(\mathbb{R}^3)}$ for $|\alpha|+|\beta|\le N$ in this section. We begin with the following lemmas.
\begin{lemma}\label{lemma8.3}
    Let $N\ge 5$ be an integer and $w_{\ell,\beta}:=\langle v\rangle^{\ell+\gamma|\beta|}$ with $-3<\gamma<0,\ell>-N-\frac{\gamma}{2}$. Assume that $f$ is the smooth solution of the Cauchy problem \eqref{Cauchy}, then it holds that
    \begin{align*}
    &\hspace{3mm}\sum_{j\in\mathbb{Z}^+}\sum_{1\le|\alpha|+|\beta|\le N}\sup_t\|w_{\ell,\beta}\partial_x^\alpha\partial_v^\beta{\bf P}^\bot f\|^2_{L^2(V_j)L_x^2(\mathbb{R}^3)}\\
    &\lesssim\sum_{1\le|\alpha|+|\beta|\le N}\|w_{\ell,\beta}\partial_x^\alpha\partial_v^\beta{\bf P}^\bot f_0\|^2+\sup_t\|f\|^2_{L_v^2(\dot{H}^1\cap\dot{H}^N)}+\sup_t|(a,b,c)|^2_{\dot{B}_{2,\infty}^\frac{1}{2}}\\
    &\quad+\sup_t\|{\bf P}^\bot f\|^2+\sum_{j\in\mathbb{Z}^+}\sum_{1\le|\alpha|+|\beta|\le N}\sup_t\|\nu^{-1}(v)w_{\ell,\beta}\partial_x^\alpha\partial_v^\beta S\|^2_{L_v^2(V_j)L_x^2(\mathbb{R}^3)}\\
    &\quad+\epsilon^2\sum_{j\in\mathbb{Z}^+}\sup_t\|w_{\ell,0}{\bf P}^\bot f\|^2_{L_v^2(V_j)L_x^2(\mathbb{R}^3)}+\eta\sum_{|\alpha|+|\beta|\le N}\sup_t\|w_{\ell,\beta}\partial_x^\alpha\partial_v^\beta{\bf P}^\bot f\|^2\,.
\end{align*}
\end{lemma}
\begin{proof}
First, we rewrite
    \begin{align}\label{pp1}
        \partial_t f+v\cdot\nabla_x f+\nu(v)f-\mathcal{K}f=\Gamma(f,f)+S\,.
    \end{align}
    The computation is then divided into following three cases.
    
    {\bf Case I:}\quad If $|\beta|=0$, 
    applying the operator $\partial_x^\alpha$ with $|\alpha|=1,2,\dots,N$ to the above equation and then multiplying the resulting equation by $w_{\ell,0}^2\partial_x^\alpha f$ yields that
\begin{align}
    &\hspace{3mm}\frac{1}{2}\partial_t(w_{\ell,0}\partial_x^\alpha f)^2+\nu(v)(w_{\ell,0}\partial_x^\alpha f)^2\notag\\
    &=-(v\cdot\nabla_x\partial_x^\alpha f)w^2_{\ell,0}\partial_x^\alpha f+w^2_{\ell,0}\partial_x^\alpha\mathcal{K}f\partial_x^\alpha f+w^2_{\ell,0}\partial_x^\alpha\Gamma(f,f)\partial_x^\alpha f+w^2_{\ell,0}\partial_x^\alpha S\partial_x^\alpha f\,.\notag
\end{align}
Solving the above equation directly and then integrating the result over $V_j\times\mathbb{R}_x^3$, one has
\begin{align}\label{pp3}
    \|w_{\ell,0}\partial_x^\alpha f\|^2_{L^2(V_j)L^2_x(\mathbb{R}^n)}
    &\lesssim    \|e^{-\nu(v)t}w_{\ell,0}\partial_x^\alpha f_0\|^2_{L^2(V_j)L^2_x(\mathbb{R}^n)}\notag\\
    &\quad+\underbrace{\int_0^t\int_{V_j\times\mathbb{R}_x^3}e^{-2\nu(v)(t-\tau)}w^2_{\ell,0}\partial_x^\alpha\mathcal{K}f\partial_x^\alpha f\,\mathrm{d}v\mathrm{d}x{\rm d}\tau}_{G_1}\notag\\
    &\quad+\underbrace{\int_0^t\int_{V_j\times\mathbb{R}_x^3}e^{-2\nu(v)(t-\tau)}w^2_{\ell,0}\partial_x^\alpha\Gamma(f,f)\partial_x^\alpha f\,\mathrm{d}v\mathrm{d}x{\rm d}\tau}_{G_2}\notag\\
    &\quad+\underbrace{\int_0^t\int_{V_j\times\mathbb{R}_x^3}e^{-2\nu(v)(t-\tau)}w^2_{\ell,0}\partial_x^\alpha S\partial_x^\alpha f\,\mathrm{d}v\mathrm{d}x{\rm d}\tau}_{G_3}\,.
\end{align}
Since $e^{-\nu(v)t}\le1$, then we have for the first term on the right-hand side that
\[    \|e^{-\nu(v)t}w_{\ell,0}\partial_x^\alpha f_0\|^2_{L^2(V_j)L^2_x(\mathbb{R}^n)}\lesssim    \|w_{\ell,0}\partial_x^\alpha f_0\|^2_{L^2(V_j)L^2_x(\mathbb{R}^n)}\,.\]
For the remaining terms, using the fact that $\tilde{c}_j<\nu(v)\le c_j$ in the domain $V_j$, where $\tilde{c}_j=(1+2^{j+1})^\gamma$ and $c_j=(1+2^{j})^\gamma$. We then obtain
    \begin{align}
        e^{-2\nu(v)(t-\tau)}\nu(v)\lesssim e^{-2\tilde{c}_j(t-\tau)}c_j\,.\notag
    \end{align}
    We further have 
    \begin{align*}
        G_1&\lesssim\int_0^te^{-2\tilde{c}_j(t-\tau)}c_j\int_{V_j\times\mathbb{R}_x^3} \nu^{-1}(v)w^2_{\ell,0}\partial_x^\alpha\mathcal{K}f\partial_x^\alpha f\,{\rm d}v{\rm d}x{\rm d}\tau\\
        &\lesssim\sup_t\int_{V_j\times\mathbb{R}_x^3}\nu^{-1}(v)w^2_{\ell,0}\partial_x^\alpha\mathcal{K}f\partial_x^\alpha f\,{\rm d}v{\rm d}x\sup_{j\in\mathbb{Z}}\int_0^te^{-2\tilde{c}_j(t-\tau)}c_j\,{\rm d}\tau\\
        &\lesssim\sup_{j\in\mathbb{Z}}\frac{(1+2^j)^\gamma}{2(1+2^{j+1})^\gamma}\sup_t\int_{V_j\times\mathbb{R}_x^3}\nu^{-1}(v)w^2_{\ell,0}\partial_x^\alpha\mathcal{K}f\partial_x^\alpha f\,{\rm d}v{\rm d}x\\
        &\lesssim\sup_t\int_{V_j\times\mathbb{R}_x^3}\nu^{-1}(v)w^2_{\ell,0}\partial_x^\alpha\mathcal{K}f\partial_x^\alpha f\,{\rm d}v{\rm d}x\,.
    \end{align*}
    Similarly, for $G_2$, we deduce from the estimate of $J_{3,1}$ that
    \begin{align*}
        G_2&\lesssim\sup_t\int_{V_j\times\mathbb{R}^3_x}\nu^{-1}(v)w^2_{\ell,0}\partial_x^\alpha\Gamma(f,f)\partial_x^\alpha f\,{\rm d}v{\rm d}x\\
        &\lesssim\sum_{\alpha_1+\alpha_2=\alpha}\sup_t\int_{V_j\times\mathbb{R}^3_x}\nu^{-1}(v)w^2_{\ell,0}\Gamma(\partial_x^{\alpha_1}f,\partial_x^{\alpha_2}f)\partial_x^\alpha f\,{\rm d}v{\rm d}x\\
        &\lesssim\epsilon^2\sum_{1\le|\alpha'|\le|\alpha|}\sup_t\|w_{\ell,0}\partial_x^{\alpha'}f\|^2_{L^2_v(\widetilde{V}_j)L_x^2(\mathbb{R}^3)}+\eta\sup_t\|w_{\ell,0}\partial_x^{\alpha}f\|^2_{L^2_v({V}_j)L_x^2(\mathbb{R}^3)}\,.
    \end{align*}

\noindent Cauchy's and Young's inequalities imply that
    \begin{align*}
        G_3&\lesssim\sup_t\int_{V_j\times\mathbb{R}_x^3}\nu^{-1}(v)w^2_{\ell,0}\partial_x^\alpha Sf\partial_x^\alpha f\,{\rm d}v{\rm d}x\\
        &\lesssim\sup_t\|\nu^{-1}(v)w_{\ell,0}\partial_x^\alpha S\|_{L^2(V_j)L_x^2(\mathbb{R}^3)}^2+\eta\sup_t\|w_{\ell,0}^2\partial_x^\alpha f\|_{L(V_j)L_x^2(\mathbb{R}^3)}^2\,.
    \end{align*}
    Plugging the above estimates into \eqref{pp3}, taking supremum over $t$ and then combining all the estimates for $|\alpha|=1,\dots,N$ in a proper way. Finally summing the result over $j\in\mathbb{Z}^+$, we have from Lemma \ref{lemma5.1} that   
    \begin{align}
        &\hspace{3mm}\sum_{j\in\mathbb{Z}^+}\sum_{1\le|\alpha|\le N}\sup_t    \|w_{\ell,0}\partial_x^\alpha f\|^2_{L^2(V_j)L^2_x(\mathbb{R}^n)}\notag\\
        &\lesssim\sum_{1\le|\alpha|\le N}\|w_{\ell,0}\partial_x^\alpha f_0\|^2+\sup_{t}\| f\|^2_{L_v^2(\dot{H}^1\cap\dot{H}^N)}+\eta\sum_{1\le|\alpha|\le N}\sup_t\|w_{\ell,0}\partial_x^\alpha f\|^2\notag\\
        &\quad+\sum_{j\in\mathbb{Z}^+}\sum_{1\le|\alpha|\le N}\sup_t\|\nu^{-1}(v)w_{\ell,0}\partial_x^\alpha S\|_{L^2(V_j)L_x^2(\mathbb{R}^3)}^2.\notag
    \end{align}
    {\bf Case II:} If $|\alpha|\neq0$ and $|\beta|\neq0$, we just consider the case $|\alpha|=N-1, |\beta|=1$. Recall the microscopic form of \eqref{pp1}:
    \begin{align}\label{pp4}
        &\hspace{3mm}\partial_t {\bf P}^\bot f+v\cdot\nabla_x{\bf P}^\bot f+v\cdot\nabla_x{\bf P}f+\nu(v){\bf P}^\bot f={\bf P}(v\cdot\nabla_xf)+\mathcal{K}{\bf P}^\bot f+\Gamma(f,f)+{\bf P}^\bot S\,.
    \end{align}
    Applying the operator $\partial_x^\alpha\partial_v$ to \eqref{pp4} and then multiplying the result by $w_{\ell,1}^2\partial_x^\alpha\partial_v{\bf P}^\bot f$ gives the following result:
    \begin{align*}
                &\hspace{3mm}\frac{1}{2}\partial_t\left(w_{\ell,1}\partial_x^{\alpha}\partial_v\mathbf{P}^{\bot}f\right)^2+\nu(v)\left(w_{\ell,1}\partial_x^{\alpha}\partial_v\mathbf{P}^{\bot}f\right)^2\\
        &=-\partial_v\left[\nu(v)\right]\partial_x^{\alpha}\mathbf{P}^{\bot}f\,w_{\ell,1}^2\partial_x^{\alpha}\partial_v\mathbf{P}^{\bot}f\\
        &\quad-\partial_x^{\alpha}\partial_v\left(v\cdot\nabla_x\mathbf{P}^{\bot}f\right)w_{\ell,1}^2\partial_x^{\alpha}\partial_v\mathbf{P}^{\bot}f-\partial_x^{\alpha}\partial_v(v\cdot\nabla_x \mathbf{P}f)w_{\ell,1}^2\partial_x^{\alpha}\partial_v\mathbf{P}^{\bot}f\\
        &\quad+\partial_x^{\alpha}\partial_v\left[{\bf P}(v\cdot\nabla_xf)\right]w_{\ell,1}^2\partial_x^{\alpha}\partial_v\mathbf{P}^{\bot}f+\partial_x^{\alpha}\partial_v(\mathcal{K}\mathbf{P}^{\bot}f)w_{\ell,1}^2\partial_x^{\alpha}\partial_v\mathbf{P}^{\bot}f\\
        &\quad+\partial_x^{\alpha}\partial_v[\Gamma(f,f)]w_{\ell,1}^2\partial_x^{\alpha}\partial_v\mathbf{P}^{\bot}f+\partial_x^{\alpha}\partial_v(\mathbf{P}^{\bot}S)w_{\ell,1}^2\partial_x^{\alpha}\partial_v\mathbf{P}^{\bot}f.
    \end{align*}
    Solving the above inequality directly and then integrating both sides of the result inequality over $V_j\times\mathbb{R}_x^3$ gives that
    \begin{align}\label{pp5}
        \left\|w_{\ell,1}\partial_x^{\alpha}\partial_v\mathbf{P}^{\bot}f\right\|_{L^2(V_j)L_x^2(\mathbb{R}^3)}^2\lesssim&\left\|e^{-\nu(v)t}w_{\ell,1}\partial_x^{\alpha}\partial_v\mathbf{P}^{\bot}f_0\right\|_{L^2(V_j)L_x^2(\mathbb{R}^3)}^2\nonumber\\
        &\underbrace{-\int_0^t\int_{V_j\times\mathbb{R}_x^3} e^{-\nu(v)(t-\tau)}\partial_v\left[\nu(v)\right]\partial_x^{\alpha}\mathbf{P}^{\bot}f\,w_{\ell,1}^2\partial_x^{\alpha}\partial_v\mathbf{P}^{\bot}f\,\mathrm{d}v\mathrm{d}x\mathrm{d}\tau}_{G_4}\nonumber\\
        &\underbrace{-\int_0^t\int_{V_j\times\mathbb{R}_x^3} e^{-\nu(v)(t-\tau)}\partial_x^{\alpha}\partial_v\left(v\cdot\nabla_x\mathbf{P}^{\bot}f\right)w_{\ell,1}^2\partial_x^{\alpha}\partial_v\mathbf{P}^{\bot}f\,\mathrm{d}v\mathrm{d}x\mathrm{d}\tau}_{G_5}\nonumber\\
        &\underbrace{-\int_0^t\int_{V_j\times\mathbb{R}_x^3} e^{-\nu(v)(t-\tau)}\partial_x^{\alpha}\partial_v(v\cdot\nabla_x \mathbf{P}f)w_{\ell,1}^2\partial_x^{\alpha}\partial_v\mathbf{P}^{\bot}f\,\mathrm{d}v\mathrm{d}x\mathrm{d}\tau}_{G_6}\nonumber\\
        &\underbrace{+\int_0^t\int_{V_j\times\mathbb{R}_x^3} e^{-\nu(v)(t-\tau)}\partial_x^{\alpha}\partial_v\left[{\bf P}(v\cdot\nabla_xf)\right]w_{\ell,1}^2\partial_x^{\alpha}\partial_v\mathbf{P}^{\bot}f\,\mathrm{d}v\mathrm{d}x\mathrm{d}\tau}_{G_7}\nonumber\\
        &\underbrace{+\int_0^t\int_{V_j\times\mathbb{R}_x^3} e^{-\nu(v)(t-\tau)}\partial_x^{\alpha}\partial_v(\mathcal{K}\mathbf{P}^{\bot}f)w_{\ell,1}^2\partial_x^{\alpha}\partial_v\mathbf{P}^{\bot}f\,\mathrm{d}v\mathrm{d}x\mathrm{d}\tau}_{G_{8}}\nonumber\\
        &\underbrace{+\int_0^t\int_{V_j\times\mathbb{R}_x^3} e^{-\nu(v)(t-\tau)}\partial_x^{\alpha}\partial_v[\Gamma(f,f)]w_{\ell,1}^2\partial_x^{\alpha}\partial_v\mathbf{P}^{\bot}f\,\mathrm{d}v\mathrm{d}x\mathrm{d}\tau}_{G_{9}}\nonumber\\
        &\underbrace{+\int_0^t\int_{V_j\times\mathbb{R}_x^3} e^{-\nu(v)(t-\tau)}\partial_x^{\alpha}\partial_v(\mathbf{P}^{\bot}S)w_{\ell,1}^2\partial_x^{\alpha}\partial_v\mathbf{P}^{\bot}f\,\mathrm{d}v\mathrm{d}x\mathrm{d}\tau}_{G_{10}}\,.
    \end{align}
    It's straightforward to see that the first term on the right-hand side of the above inequality can be bounded by
    \[\left\|e^{-\nu(v)t}w_{\ell,1}\partial_x^{\alpha}\partial_v\mathbf{P}^{\bot}f_0\right\|_{L^2(V_j)L_x^2(\mathbb{R}^3)}^2\lesssim\left\|w_{\ell,1}\partial_x^{\alpha}\partial_v\mathbf{P}^{\bot}f_0\right\|_{L^2(V_j)L_x^2(\mathbb{R}^3)}^2.\]
    For $G_4$, we have 
    \begin{align*}
        G_4&\lesssim\sup_t\int_{V_j\times\mathbb{R}_x^3}\nu^{-1}(v)w^2_{\ell,1}\partial_v[\nu(v)]\partial_x^\alpha{\bf P}^\bot f\partial_v\partial_x^\alpha{\bf P}^\bot f\,{\rm d}v{\rm d}x\,.
    \end{align*}
Similar to $J_7$, by integrating in parts and using the fact that $\nu(v)\sim\langle v\rangle^\gamma$, one has
\begin{align*}
    G_5&\lesssim\sum_{i=1}^3\int_0^t\int_{V_j\times\mathbb{R}_x^3}e^{-\nu(v)(t-\tau)}\nu(v)w_{\ell,0}\partial_x^{\alpha+e_i}{\bf P}^\bot f\,w_{\ell,1}\partial_x^\alpha\partial_v{\bf P}^\bot f\,{\rm d}v{\rm d}x{\rm d}\tau\\
    &\lesssim\sum_{i=1}^3\sup_t\int_{V_j\times\mathbb{R}^3_x}w_{\ell,0}\partial_x^{\alpha+e_i}{\bf P}^\bot f\,w_{\ell,1}\partial_x^\alpha\partial_v{\bf P}^\bot f\,{\rm d}v{\rm d}x\\
    &\lesssim\sum_{i=1}^3\sup_t\|w_{\ell,0}\partial_x^{\alpha+e_i}{\bf P}^\bot f\|^2_{L(V_j)L_x^2(\mathbb{R}^3)}+\eta\sup_t\|w_{\ell,1}\partial_x^\alpha\partial_v{\bf P}^\bot f\|_{L(V_j)L_x^2(\mathbb{R}^3)}^2\,.
\end{align*}
As to $G_6$, it can be derived from the estimate of $J_8$ that
\begin{align*}
    G_6&\lesssim\sup_t\int_{V_j\times\mathbb{R}_x^3}\nu^{-1}(v)\partial_x^\alpha\partial_v(v\cdot\nabla_x{\bf P}f)w_{\ell,1}^2\partial_x^\alpha\partial_v{\bf P}^\bot f\,{\rm d}v{\rm d}x\\
    &\lesssim\sup_t\int_{V_j\times\mathbb{R}_x^3}\partial_x^{\alpha+e_i}(a,b,c)|v|^{-1}\mu^\delta w_{\ell.1}\partial_x^\alpha\partial_v{\bf P}^\bot f{\rm d}v{\rm d}x\\
    &\lesssim\sup_t\|\partial_x^{\alpha+e_i}(a,b,c)|v|^{-1}\mu^\delta\|_{L^2_v(V_j)L_x^2(\mathbb{R}^3)}^2+\eta\sup_t\|w_{\ell,1}\partial_x^\alpha\partial_v{\bf P}^\bot f\|^2_{L^2(V_j)L_x^2(\mathbb{R}^3)}\\
    &\lesssim\|(a,b,c)\|^2_{\dot{H}_x^{|\alpha|+1}}\frac{1}{2^j}+\eta\sup_t\|w_{\ell,1}\partial_x^\alpha\partial_v{\bf P}f\|^2_{L^2(V_j)L_x^2(\mathbb{R}^3)}\,.
\end{align*}
The last inequality holds because in the domain $V_j$, we have $|v|^{-1}\lesssim\frac{1}{2^j}$, and the extra velocity part can be absorbed by $\mu^\delta$.

\noindent For $G_7$, we have from the estimate of $J_9$ that
\begin{align*}
	G_7&\lesssim\sup_t\int_{V_j\times\mathbb{R}^3_x}\nu^{-1}(v)\partial_x^\alpha\partial_v[{\bf P}(v\cdot\nabla_x f)]w^2_{\ell.1}\partial_x^\alpha\partial_v{\bf P}^\bot f\,{\rm d}v{\rm d}x\\
	&\lesssim\sup_t\int_{V_j\times\mathbb{R}^3_x}m(t,x)\mu^\delta w^2_{\ell,1}\partial_x^\alpha\partial_v{\bf P}^\bot f\,{\rm d}v{\rm d}x\\
	&\lesssim\sup_t\|m(t,x)|v|^{-1}\mu^\delta\|_{L_x^2(\mathbb{R}^3)L^2(V_j)}^2+\eta\|w_{\ell,1}\partial_x^\alpha\partial_v{\bf P}^\bot f\|_{L_x^2(\mathbb{R}^3)L^2(V_j)}^2\\
	&\lesssim\sup_t|m(t,x)\frac{1}{2^j}|_{L_x^2(\mathbb{R}^3)}^2+\eta\|w_{\ell,1}\partial_x^\alpha\partial_v{\bf P}^\bot f\|_{L_x^2(\mathbb{R}^3)L^2(V_j)}^2\\
	&\lesssim\frac{1}{2^j}\sup_t\|f\|_{L_v^2(\mathbb{R}^3)\dot{H}_x^{|\alpha|+1}}^2+\eta\|w_{\ell,1}\partial_x^\alpha\partial_v{\bf P}^\bot f\|_{L^2(V_j)L_x^2(\mathbb{R}^3)}^2\,.
\end{align*}
The estimate for $G_8$ can be derived in a completely same way of $G_1$, one has
\begin{align*}
	G_8&\lesssim\sup_t\int_{V_j\times\mathbb{R}_x^3}\nu^{-1}(v)\partial_x^\alpha\partial_v(\mathcal{K}{\bf P}^\bot f)w_{\ell,1}^2\partial_x^\alpha\partial_v{\bf P}^\bot f\,{\rm d}v{\rm d}x\,.
\end{align*}
It holds from the estimate of $J_{11,1}$ that
\begin{align*}
    G_9&\lesssim\sup_t\int_{V_j\times\mathbb{R}_x^3}\nu^{-1}(v)w^2_{\ell,1}\partial_x^\alpha\partial_v\Gamma(f,f)\partial_x^\alpha\partial_v{\bf P}^\bot f\,{\rm d} v{\rm d}x\\
    &\lesssim\sum_{\substack{\alpha_1+\alpha_2=\alpha\\\beta_1+\beta_2=\beta}}\sup_t\int_{V_j\times\mathbb{R}_x^3}\nu^{-1}(v)w^2_{\ell,1}\Gamma(\partial_x^{\alpha_1}\partial_v^{\beta_1}f,\partial_x^{\alpha_2}\partial_v^{\beta_2}f)\partial_x^\alpha\partial_v{\bf P}^\bot f\,{\rm d}v{\rm d}x\\
    &\lesssim\frac{1}{2^{2j}}\epsilon^2\sup_t|(a,b,c)|^2_{\dot{H}^1\cap\dot{H}^N}+\eta\sup_t\|w_{\ell,1}\partial_x^\alpha\partial_v{\bf P}^\bot f\|^2_{L^2_v(V_j)L_x^2(\mathbb{R}^3)}\\
    &\quad+\epsilon^2\sum_{\substack{1\le|\alpha'|\le|\alpha|\\|\beta'|\le|\beta|}}\sup_t\|w_{\ell,\beta'}\partial_x^{\alpha'}\partial_v^{\beta'}{\bf P}^\bot f\|^2_{L^2_v(\widetilde{V}_j)L_x^2(\mathbb{R}^3)}\,.
\end{align*}
\noindent For $G_{10}$, we have 
\begin{align*}
    G_{10}&\lesssim\sup_t\int_{V_j\times\mathbb{R}_x^3}\nu^{-1}(v)w_{\ell,1}^2\partial_x^{\alpha}\partial_v({\bf P}^\bot S)\partial_x^\alpha\partial_v{\bf P}^\bot f\,{\rm d}v{\rm d}x\\
    &\lesssim\sup_t\|\nu^{-1}(v)w_{\ell,1}\partial_x^\alpha\partial_v({\bf P}^\bot S)\|^2_{L^2(V_j)L_x^2(\mathbb{R}^3)}+\eta\sup_t\|w_{\ell,1}\partial_x^\alpha\partial_v{\bf P}^\bot f\|^2_{L^2(V_j)L_x^2(\mathbb{R}^3)}\,.
\end{align*}
Inserting all the relating estimates into \eqref{pp5}, taking the supremum over $t$ and then combining all the estimates for $|\alpha|\ge1, |\alpha|+|\beta|\le N$. Finally summing the result over $j$,  we have from Lemma \ref{lemma5.1} that
\begin{align*}
    &\hspace{3mm}\sum_{j\in\mathbb{Z^+}}\sum_{\substack{|\alpha|+|\beta|\le N\\|\alpha|\ge 1}}\sup_t\|w_{\ell,\beta}\partial_x^\alpha\partial_v^\beta{\bf P}^\bot f\|^2_{L^2(V_j)L_x^2(\mathbb{R}^3)}\\
    &\lesssim\sum_{\substack{|\alpha|+|\beta|\le N\\|\alpha|\ge 1}}\|w_{\ell,\beta}\partial_x^\alpha\partial_v^\beta{\bf P}^\bot f_0\|^2+\sup_t\|f\|^2_{L_v^2(\dot{H}^1\cap\dot{H}^{N})}+\eta\sum_{\substack{|\alpha|+|\beta|\le N\\|\alpha|\ge 1}}\sup_t\|w_{\ell,\beta}\partial_x^\alpha\partial_v^{\beta}{\bf P}^\bot f\|^2\\
    &\quad+\sum_{j\in\mathbb{Z}^+}\sum_{\substack{|\alpha|+|\beta|\le N\\|\alpha|\ge 1}}\sup_t\|\nu^{-1}(v)w_{\ell,\beta}\partial_x^\alpha\partial_v^\beta{\bf P}^\bot S\|^2_{L^2(V_j)L_x^2(\mathbb{R}^3)}\,.
\end{align*}
Here we have used the fact that $\sum_{j\in\mathbb{Z}^+}\frac{1}{2^{2j}}$ is convergent.

\noindent{\bf Case III:} If $|\alpha|=0$, we just estimate the case $|\beta|=N$, the remaining cases $|\beta|<N, |\alpha|=0$ can be derived in the same way. Similar to {\bf Case II}, one can deduce that
\begin{align}\label{pp6}
    \left\|w_{\ell,N}\partial_v^\beta \mathbf{P}^{\bot}f\right\|^2_{L^2(V_j)L_x^2(\mathbb{R}^3)}\lesssim&\left\|e^{-\nu(v)t}w_{\ell,N}\partial_v^\beta \mathbf{P}^{\bot}f_0\right\|^2_{L^2(V_j)L_x^2(\mathbb{R}^3)}\notag\\
        &\underbrace{-\sum_{|\beta'|=1}^{N
        }\int_0^t\int_{V_j\times\mathbb{R}_x^3} e^{-\nu(v)(t-\tau)}\partial_v^{\beta'} \left[\nu(v)\right]\partial_v^{\beta-\beta'}\mathbf{P}^{\bot}f\,w_{\ell,N}^2\partial_v^\beta \mathbf{P}^{\bot}f\,\mathrm{d}v\mathrm{d}x\mathrm{d}\tau}_{G_{11}}\notag\\
        &\underbrace{-\int_0^t\int_{V_j\times\mathbb{R}_x^3} e^{-\nu(v)(t-\tau)}\partial_v^\beta \left(v\cdot\nabla_x\mathbf{P}^{\bot}f\right)w_{\ell,N}^2\partial_v^\beta \mathbf{P}^{\bot}f\,\mathrm{d}v\mathrm{d}x\mathrm{d}\tau}_{G_{12}}\notag\\
        &\underbrace{-\int_0^t\int_{V_j\times\mathbb{R}_x^3} e^{-\nu(v)(t-\tau)}\partial_v^\beta (v\cdot\nabla_x \mathbf{P}f)w_{\ell,N}^2\partial_v^\beta \mathbf{P}^{\bot}f\,\mathrm{d}v\mathrm{d}x\mathrm{d}\tau}_{G_{13}}\notag\\
        &\underbrace{+\int_0^t\int_{V_j\times\mathbb{R}_x^3} e^{-\nu(v)(t-\tau)}\partial_v^\beta \left[{\bf P}(v\cdot\nabla_xf)\right]w_{\ell,N}^2\partial_v^\beta \mathbf{P}^{\bot}f\,\mathrm{d}v\mathrm{d}x\mathrm{d}\tau}_{G_{14}}\notag\\
        &\underbrace{+\int_0^t\int_{V_j\times\mathbb{R}_x^3} e^{-\nu(v)(t-\tau)}\partial_v^\beta (\mathcal{K}\mathbf{P}^{\bot}f)w_{\ell,N}^2\partial_v^\beta \mathbf{P}^{\bot}f\,\mathrm{d}v\mathrm{d}x\mathrm{d}\tau}_{G_{15}}\notag\\
        &\underbrace{+\int_0^t\int_{V_j\times\mathbb{R}_x^3} e^{-\nu(v)(t-\tau)}\partial_v^\beta [\mathbf{P}^{\bot}\Gamma(f,f)]w_{\ell,N}^2\partial_v^\beta \mathbf{P}^{\bot}f\,\mathrm{d}v\mathrm{d}x\mathrm{d}\tau}_{G_{16}}\notag\\
        &\underbrace{+\int_0^t\int_{V_j\times\mathbb{R}_x^3} e^{-\nu(v)(t-\tau)}\partial_v^\beta (\mathbf{P}^{\bot}S)w_{\ell,N}^2\partial_v^\beta \mathbf{P}^{\bot}f\,\mathrm{d}v\mathrm{d}x\mathrm{d}\tau}_{G_{17}}\,.
\end{align}
We have, for the first term of the right-hand side of the above inequality, that
\[\|e^{-\nu(v)t}w_{\ell,N}\partial_v^\beta{\bf P}^\bot f_0\|^2_{L^2(V_j)L_x^2(\mathbb{R}^3)}\lesssim\|w_{\ell,N}\partial_v^\beta{\bf P}^\bot f_0\|^2_{L^2(V_j)L_x^2(\mathbb{R}^3)}\,.\]
For $G_{11}$, we have
\begin{align*}
    G_{11}&\lesssim\sum_{|\beta'|=1}^{|\beta|}\sup_t\int_{V_j\times\mathbb{R}_x^3}\nu^{-1}(v)\partial_v^{\beta'}[\nu(v)]\partial_v^{\beta-\beta'}{\bf P}^\bot f\partial_v^\beta{\bf P}^\bot f\,{\rm d}v{\rm d}x\,.
\end{align*}
For $G_{12}$, we deduce that
\begin{align*}
    G_{12}&\lesssim\sup_t\int_{V_j\times\mathbb{R}_x^3}\nu^{-1}(v)w^2_{\ell,N}\partial_v^\beta\left(v\cdot\nabla_x{\bf P}^\bot f\right)\partial_v^\beta{\bf P}^\bot f\,{\rm d}v{\rm d}x\\
    &\lesssim\sup_t\int_{V_j\times\mathbb{R}_x^3}\nu^{-1}(v)w^2_{\ell,N}\,v\cdot\nabla_x(\partial_v^\beta{\bf P}^\bot f)\partial_v^\beta{\bf P}^\bot f\,{\rm d}v{\rm d}x\\
    &\quad+\sup_t\int_{V_j\times\mathbb{R}_x^3}\nu^{-1}(v)w_{\ell,N}^2\nabla_x\partial_v^{\beta-e_i}{\bf P}^\bot f\partial_v^\beta{\bf P}^\bot f\,{\rm d}v{\rm d}x\\
    &\lesssim\sup_t\int_{V_j\times\mathbb{R}_x^3}w_{\ell,N}\partial_v^\beta{\bf P}^\bot f\,w_{\ell,\beta-e_i}\nabla_x\partial_v^{\beta-e_i}{\bf P}^\bot f,{\rm d}v{\rm d}x\\
    &\lesssim\sup_t\sum_{|\alpha_1|=1}\|w_{\ell,\beta-e_i}\partial_x^{\alpha_1}\partial_v^{\beta-e_i}{\bf P}^\bot f\|^2_{L^2(V_j)L_x^2(\mathbb{R}^3)}+\eta\sup_t\|w_{\ell,N}\partial_v^\beta{\bf P}^\bot f\|_{L^2(V_j)L_x^2(\mathbb{R}^3)}^2\,.
\end{align*}
For $G_{13}$, it follows that
\begin{align*}
    G_{13}&\lesssim\sup_t\int_{V_j\times\mathbb{R}_x^3}\nu^{-1}(v)w^2_{\ell,N}\partial_v^\beta(v\cdot\nabla_x{\bf P}f)\partial_v^\beta{\bf P}^\bot f\,{\rm d}v{\rm d}x\\
    &\lesssim\sup_t\int_{V_j\times\mathbb{R}_x^3}\nabla_x(a,b,c)\mu^\delta w_{\ell,N}\partial_v^\beta{\bf P}^\bot f\,{\rm d}v{\rm d}x\\
    &\lesssim\sup_t\|\nabla_x(a,b,c)|v|^{-1}\mu^\delta\|_{L^2(V_j)L_x^2(\mathbb{R}^3)}^2+\eta\|w_{\ell,N}\partial_v^\beta{\bf P}^\bot f\|^2_{L^2(V_j)L_x^2(\mathbb{R}^3)}\\
    &\lesssim\frac{1}{2^{2j}}|(a,b,c)|^2_{\dot{H}^1}+\eta\|w_{\ell,N}\partial_v^\beta{\bf P}^\bot f\|^2_{L^2(V_j)L_x^2(\mathbb{R}^3)}\,.
\end{align*}
It can be derived from the estimate of $J_{17}$ that
\begin{align*}
    G_{14}&\lesssim\sup_t\int_{V_j\times\mathbb{R}^3_x}\nu^{-1}(v)w_{\ell,N}^2\partial_v^\beta[{\bf P}(v\cdot\nabla_x f)]\partial_v^\beta{\bf P}^\bot f\,{\rm d}v{\rm d}x\\
    &\lesssim\sup_t\int_{V_j\times\mathbb{R}_x^3}\tilde{m}(t,x)\mu^\delta w_{\ell,N}^2\partial_v^\beta{\bf P}^\bot f\,{\rm d}v{\rm d}x\\
    &\lesssim\sup_t\|\tilde{m}(t,x)|v|^{-1}\mu^\delta\|_{L^2(V_j)L_x^2(\mathbb{R}^3)}^2+\eta\sup_t\|w_{\ell,N}\partial_v^\beta{\bf P}^\bot f\|^2_{L^2(V_j)L_x^2(\mathbb{R}^3)}\\
    &\lesssim\frac{1}{2^{2j}}\sup_t\|f\|_{L_v^2\dot{H}_x^1}^2+\eta\sup_t\|w_{\ell,N}\partial_v^\beta{\bf P}^\bot f\|^2_{L^2(V_j)L_x^2(\mathbb{R}^3)}\,.
\end{align*}
It's obvious that we have
\begin{align*}
    G_{15}\lesssim\sup_t\int_{V_j\times\mathbb{R}_x^3}\nu^{-1}(v)w^2_{\ell,N}\partial_v^\beta(\mathcal{K}{\bf P}^\bot f)\partial_v^\beta{\bf P}^\bot f\,{\rm d}v{\rm d}x\,.
\end{align*}
It can be derived from the estimate of $J_{19,1}$
that
\begin{align*}
    G_{16}&\lesssim\sup_t\int_{V_j\times\mathbb{R}_x^3}\nu^{-1}(v)w^2_{\ell,N}\partial_v^\beta\Gamma(f,f)\partial_v^\beta{\bf P}^\bot f\,{\rm d}v{\rm d}x\\
    &\lesssim\sum_{\beta_1+\beta_2=\beta}\sup_t\int_{V_j\times\mathbb{R}_x^3}\nu^{-1}(v)w^2_{\ell.N}\Gamma(\partial_v^{\beta_1}f,\partial_v^{\beta_2}f)\partial_v^\beta{\bf P}^\bot f\,{\rm d}v{\rm d}x\\
    &\lesssim\frac{1}{2^{2j}}\epsilon^2\sup_t|(a,b,c)|^2_{\dot{B}^\frac{1}{2}_{2,\infty}\cap\dot{H}^1}+\eta\sup_t\|w_{\ell,N}\partial_v^\beta{\bf P}^\bot f\|^2_{L_v^2(V_j)L_x^2(\mathbb{R}^3)}\\
    &\quad+\epsilon^2\sum_{\substack{1\le|\alpha'|+|\beta'|\le N\\|\alpha'|\le 1}}\sup_t\|w_{\ell,\beta'}\partial_x^{\alpha'}\partial_v^{\beta'}{\bf P}^\bot f\|^2_{L^2_v(\widetilde{V}_j)L_x^2(\mathbb{R}^3)}\,.
\end{align*}
\noindent We have for $G_{17}$ that
\begin{align*}
    G_{17}&\lesssim\sup_t\int_{V_j\times\mathbb{R}_x^3}\nu^{-1}(v)w_{\ell,N}^2\partial_v^\beta{\bf P}^\bot S\,\partial_v^\beta{\bf P}^\bot f\,{\rm d}v{\rm d}x\\
    &\lesssim\sup_t\|\nu^{-1}(v)w_{\ell,N}\partial_v^\beta{\bf P}^\bot S\|^2_{L^2(V_j)L_x^2}+\eta\sup_t\|w_{\ell,N}\partial_v^\beta{\bf P}^\bot f\|^2_{L^2(V_j)L_x^2}\,.
\end{align*}
Inserting all the relevant estimates into \eqref{pp6}, taking the supremum over $t$ and summing over $j\in\mathbb{Z^+}$, we combine the estimates from all cases to obtain, by Lemma \ref{lemma5.1}, that
\begin{align*}
    &\hspace{3mm}\sum_{j\in\mathbb{Z}^+}\sum_{1\le|\alpha|+|\beta|\le N}\sup_t\|w_{\ell,\beta}\partial_x^\alpha\partial_v^\beta{\bf P}^\bot f\|^2_{L^2(V_j)L_x^2(\mathbb{R}^3)}\\
    &\lesssim\sum_{1\le|\alpha|+|\beta|\le N}\|w_{\ell,\beta}\partial_x^\alpha\partial_v^\beta{\bf P}^\bot f_0\|^2+\sup_t\|f\|^2_{L_v^2(\dot{H}^1\cap\dot{H}^N)}+\sup_t|(a,b,c)|^2_{\dot{B}_{2,\infty}^\frac{1}{2}}\\
    &\quad+\sup_t\|{\bf P}^\bot f\|^2+\sum_{j\in\mathbb{Z}^+}\sum_{1\le|\alpha|+|\beta|\le N}\sup_t\|\nu^{-1}(v)w_{\ell,\beta}\partial_x^\alpha\partial_v^\beta S\|^2_{L_v^2(V_j)L_x^2(\mathbb{R}^3)}\\
    &\quad+\epsilon^2\sum_{j\in\mathbb{Z}^+}\sup_t\|w_{\ell,0}{\bf P}^\bot f\|^2_{L_v^2(V_j)L_x^2(\mathbb{R}^3)}+\eta\sum_{|\alpha|+|\beta|\le N}\sup_t\|w_{\ell,\beta}\partial_x^\alpha\partial_v^\beta{\bf P}^\bot f\|^2\,.
\end{align*}
This completes the proof of Lemma \ref{lemma8.3}.
\end{proof}
\begin{lemma}\label{lemma8.4}
    Under the assumption of Lemma \ref{lemma8.3}, we have the following result:
    \begin{align*}
        \hspace{3mm}\sum_{j\in\mathbb{Z}^+}\sup_t\|w_{\ell,0}{\bf P}^\bot f\|^2_{L^2(V_j)L_x^2(\mathbb{R}^3)}
        \lesssim&\|w_{\ell,0}{\bf P}^\bot f_0\|^2+\sup_t\|f\|^2_{L_v^2(\dot{H}^1)}+\sup_t\|{\bf P}^\bot f\|^2+\eta\sup_t\|w_{\ell,0}{\bf P}^\bot f\|^2\\
        &\quad+\sup_t|(a,b,c)|^2_{\dot{B}^\frac{1}{2}_{2,\infty}\cap\dot{H}^2}+\sum_{j\in\mathbb{Z}^+}\sup_t\|\nu^{-1}(v)w_{\ell,0}{\bf P}^\bot S\|^2_{L^2(V_j)L_x^2(\mathbb{R}^3)}\,.
    \end{align*}
\end{lemma}
\begin{proof}
    Similar to \eqref{5.13}, one gets
    \begin{align}\label{pp8}
        \left\|w_{\ell,0}\mathbf{P}^{\bot}f\right\|^2_{L^2(V_j)L_x^2(\mathbb{R}^3)}&\lesssim\left\|w_{\ell,0}\mathbf{P}^{\bot}f_0\right\|^2_{L^2(V_j)L_x^2(\mathbb{R}^3)}
        \underbrace{-\int_0^t\int_{V_j\times\mathbb{R}_{x}^3} e^{-\nu(v)(t-\tau)}(v\cdot\nabla_x \mathbf{P}f)w_{\ell,0}^2\mathbf{P}^{\bot}f\,\mathrm{d}v\mathrm{d}x\mathrm{d}\tau}_{G_{18}}\notag\\
        &\quad\underbrace{+\int_0^t\int_{V_j\times\mathbb{R}_{x}^3} e^{-\nu(v)(t-\tau)}{\bf P}(v\cdot\nabla_xf)w_{\ell,0}^2\mathbf{P}^{\bot}f\,\mathrm{d}v\mathrm{d}x\mathrm{d}\tau}_{G_{19}}\notag\\
        &\quad\underbrace{+\int_0^t\int_{V_j\times\mathbb{R}_{x}^3} e^{-\nu(v)(t-\tau)}(\mathcal{K}\mathbf{P}^{\bot}f)w_{\ell,0}^2\mathbf{P}^{\bot}f\,\mathrm{d}v\mathrm{d}x\mathrm{d}\tau}_{G_{20}}\notag\\
        &\quad\underbrace{+\int_0^t\int_{V_j\times\mathbb{R}_{x}^3} e^{-\nu(v)(t-\tau)}[\mathbf{P}^{\bot}\Gamma(f,f)]w_{\ell,0}^2\mathbf{P}^{\bot}f\,\mathrm{d}v\mathrm{d}x\mathrm{d}\tau}_{G_{21}}\notag\\
        &\quad\underbrace{+\int_0^t\int_{V_j\times\mathbb{R}_{x}^3} e^{-\nu(v)(t-\tau)}(\mathbf{P}^{\bot}S)w_{\ell,0}^2\mathbf{P}^{\bot}f\,\mathrm{d}v\mathrm{d}x\mathrm{d}\tau}_{G_{22}}\,.
    \end{align}
  Now we estimate the above terms separately. For $G_{18}$, it follows
    \begin{align*}
        G_{18}&\lesssim\sup_t\int_{V_j\times\mathbb{R}_{x}^3}\nu^{-1}(v)w^2_{\ell,0}(v\cdot\nabla_x{\bf P}f){\bf P}^\bot f\,{\rm d}v{\rm d}x\\
        &\lesssim\sup_t\int_{V_j\times\mathbb{R}_{x}^3}\partial_x(a,b,c)|v|^{-1}\mu^\delta w_{\ell,0}{\bf P}^\bot f\,{\rm d}v{\rm d}x\\
        &\lesssim\sup_t\|\partial_x(a,b,c)\frac{1}{2^j}\mu^\delta\|^2+\eta\sup_t\|w_{\ell,0}{\bf P}^\bot f\|_{L^2(V_j)L_x^2(\mathbb{R}^3)}^2\\
        &\lesssim\frac{1}{2^{2j}}|(a,b,c)|^2_{\dot{H}^1}+\eta\sup_t\|w_{\ell,0}{\bf P}^\bot f\|_{L^2(V_j)L_x^2(\mathbb{R}^3)}^2\,.
    \end{align*}
    For $G_{19}$, we set $m_2(t,x):=\int_{\mathbb{R}_v^3}v\cdot\nabla_xf\phi_i\,{\rm d}v$, then we obtain
    \begin{align*}
        G_{19}&\lesssim\sup_t\int_{V_j\times\mathbb{R}_{x}^3}\nu^{-1}(v){\bf P}(v\cdot\nabla_xf)w^2_{\ell,0}{\bf P}^\bot f\,{\rm d}v{\rm d}x\\
        &\lesssim\sup_t\int_{V_j\times\mathbb{R}_{x}^3}
        m_2(t,x)|v|^{-1}\mu^\delta w_{\ell,0}{\bf P}^\bot f\,{\rm d}v{\rm d}x\\
        &\lesssim\sup_t\|m_2(t,x)\frac{1}{2^{2j}}\mu^\delta\|^2+\eta\sup_t\|w_{\ell,0}{\bf P}^\bot f\|^2_{L^2(V_j)L_x^2(\mathbb{R}_x^3)}\\
        &\lesssim\frac{1}{2^{2j}}\sup_t\|f\|^2_{L_v^2(\dot{H}^1)}+\eta\sup_t\|w_{\ell,0}{\bf P}^\bot f\|^2_{L^2(V_j)L_x^2(\mathbb{R}_x^3)}\,.
    \end{align*}
    It's obviously to see that
    \begin{align*}
        G_{20}&\lesssim\sup_t\int_{V_j\times\mathbb{R}_{x}^3}\nu^{-1}(v)w^2_{\ell,0}\mathcal{K}{\bf P}^\bot f{\bf P}^\bot f\,{\rm d}v{\rm d}x\,.
    \end{align*}
    For $G_{21}$, we deduce from the estimate of $J_{25,1}$ that
    \begin{align*}
        G_{21}&\lesssim\sup_t\int_{V_j\times\mathbb{R}_x^3}\nu^{-1}(v)w^2_{\ell,0}\Gamma(f,f){\bf P}^\bot f\,{\rm d}v{\rm d}x\\
        &\lesssim\frac{1}{2^{2j}}\epsilon^2\sup_t|(a,b,c)|^2_{\dot{B}^\frac{1}{2}_{2,\infty}\cap\dot{H}^2}+\eta\|w_{\ell,0}{\bf P}^\bot f\|^2_{L_v^2(V_j)L_x^2(\mathbb{R}^3)}
        +\epsilon^2 \|w_{\ell,0}{\bf P}^\bot f\|^2_{L_v^2(\widetilde{V}_j)L_x^2(\mathbb{R}^3)}\,.
    \end{align*}
    \noindent One can derive from Cauchy's and Young's inequality that
    \begin{align*}
        G_{22}&\lesssim\sup_t\int_{V_j\times\mathbb{R}_x^3}\nu^{-1}(v){\bf P}^\bot S\,w_{\ell,0}^2{\bf P}^\bot f\,{\rm d}v{\rm d}x\\
        &\lesssim\sup_t\|\nu^{-1}(v)w_{\ell,0}{\bf P}^\bot S\|^2_{L^2(V_j)L_x^2(\mathbb{R}^3)}+\eta\sup_t\|w_{\ell,0}{\bf P}^\bot f\|^2_{L^2(V_j)L_x^2(\mathbb{R}^3)}\,.
    \end{align*}
    Plugging all the estimates into \eqref{pp8}, then taking the supremum over $t$ and summing over $j\in\mathbb{Z}^+$ yields the desired result.  The proof of Lemma \ref{lemma8.4} is finished.
    \end{proof}
We are now ready to close the {\it a prior} estimate and obtain the global existence of the solution of the system \eqref{Cauchy}. We have the following proposition.  
\begin{proposition}\label{existence}
   Let $-3<\gamma<0, N\ge5$ and $\ell\ge N-\frac{\gamma}{2}$, assume that $f_0$ and $S$ satisfy
   \begin{align}\label{initial}
       \|f_0\|_X+\|S\|_{\mathcal{Z}}\le\mathcal{M}_0,
   \end{align}
   where $\mathcal{M}_0>0$ is suitably small. Then there exists a unique global solution of the system \eqref{Cauchy} such that
    \begin{align}\label{solution}
        \sup_t\|f\|_X\lesssim\|f_0\|_X+\|S\|_{\mathcal{Z}}\lesssim\mathcal{M}_0\,\notag
    \end{align}
    
\end{proposition}
\begin{proof}
    The desired result follows from a suitable combination of Proposition \ref{besov}, Lemmas \ref{lemma4.2}–\ref{lemma4.3}, Lemmas \ref{lemma5.3}–\ref{lemma5.4}, and Lemmas \ref{lemma8.3}–\ref{lemma8.4}.
\end{proof}

\section{Time decay rates for the difference of two solutions}\label{Sect-time-devay}
With the global existence established in Proposition \ref{existence}, this section is devoted to proving the asymptotic stability of the global solution of the Cauchy problem \eqref{Cauchy}.

Let $f_1$ and $f_2$ be the pair of solutions obtained in Proposition \ref{existence}. Then their initial data $f_1(0,x,v)$ and $f_2(0,x,v)$ satisfy \eqref{initial}, and we have
\[\sup_t\|(f_1,f_2)\|_X\lesssim\mathcal{M}_0\,.\]
Define $g:=f_1-f_2$, we derive that $g$ satisfy
\begin{align}\label{X1}
    \partial_t g+v\cdot\nabla_x g+\mathcal{L}g=\Gamma(g,f_1)+\Gamma(f_2,g)\,,
\end{align}
with the initial data
\[g(0,x,v)=f_1(0,x,v)-f_2(0,x,v)\,.\]
It's clear that $g_0$ satisfies the conditions in Proposition \ref{existence} and consequently
\[\sup_t\|g\|_{X}\lesssim\mathcal{M}_0\,.\]
Now we proceed to obtain the time decay of $g$. We begin with the following lemmas.
\begin{lemma}\label{prop6.2}
    Let $N\ge5, -3<\gamma<0$, then for $t>0$, it holds that
    \begin{align*}
       \sup_t(1+t)^{\frac{1}{2}}\|g\|_{\widetilde{L}_v^2(\dot{B}_{2,\infty}^{1}\cap\dot{B}_{2,\infty}^{2})}^2&\lesssim\|g_0\|_{\widetilde{L}_v^2(\dot{B}_{2,\infty}^{\frac{1}{2}}\cap\dot{B}_{2,\infty}^{2})}^2+\eta\sup_t(1+t)^{\frac{1}{2}}\left\|\mathbf{P}^{\bot}g\,\langle v\rangle^{-\frac{\gamma}{2}}\right\|_{L_v^2(\dot{H}_x^1\cap\dot{H}_x^2)}^2\\
       &\qquad+\mathcal{M}_0^2\|{\bf P}^\bot g\|_{L_v^2(\dot{H}_x^1\cap\dot{H}_x^2)}^2
        \,,
    \end{align*}
    where we choose both $\eta$ and $\mathcal{M}_0$ to be suitably small.
\end{lemma}
\begin{proof}
Firstly, we consider the low frequency part.
Similar to \eqref{l-f}, we obtain for $g$ that
    \begin{align*}
        \left|\hat{g}_L\right|_{L_v^2}^2&\lesssim e^{-|\xi|^2t}\left|\hat{g}_0\right|_{L_v^2}^2+\eta\int_0^t e^{-|\xi|^2(t-\tau)}|\xi|^2\left|{\bf P}^\bot f\,\langle v\rangle^{-\frac{\gamma}{2}}\right|_{L_v^2}^2 \,\mathrm{d}\tau\\
        &\qquad+\int_0^t e^{-|\xi|^2(t-\tau)}\left|\Gamma(g,f_1)\,\langle v\rangle^{-\frac{\gamma}{2}}\right|_{L_v^2}^2 \,\mathrm{d}\tau+\int_0^t e^{-|\xi|^2(t-\tau)}\left|\Gamma(g,f_1)\,\langle v\rangle^{-\frac{\gamma}{2}}\right|_{L_v^2}^2 \,\mathrm{d}\tau\,.
    \end{align*}
     Multiplying both sides of the above inequality by $|\xi|^2\varphi_j^2$, then integrating the result with equation respect to $\xi$ and taking the supremum over $j\in\mathbb{Z}$, one obtains
\begin{align}
    \|g_L\|_{\widetilde{L}_v^2(\dot{B}_{2,\infty}^1)}^2&\lesssim \underbrace{\sup_{j\in\mathbb{Z}}\int_{\mathbb{R}_{\xi}^3} e^{- t|\xi|^2}|\xi|^2\left|\varphi_j\hat{g}_0\right|_{L_v^2}^2\,\mathrm{d}\xi}_{O_1}+\underbrace{\eta\sup_{j\in\mathbb{Z}}\int_0^t \int_{\mathbb{R}_\xi^3}e^{-|\xi|^2(t-\tau)}|\xi|^2\left||\xi|\mathbf{P}^{\bot}\hat{g}\,\langle v\rangle^{-\frac{\gamma}{2}}\varphi_j\right|_{L_v^2}^2\,\mathrm{d}\xi\mathrm{d}\tau}_{O_2}\notag\\
    &\quad+\underbrace{\sup_{j\in\mathbb{Z}}\int_0^t \int_{\mathbb{R}_\xi^3}e^{-|\xi|^2(t-\tau)}|\xi|^2\left|\hat{\Gamma}(g,f_1)\langle v\rangle^{-\frac{\gamma}{2}}\varphi_j\right|_{L_v^2}^2\,\mathrm{d}\xi\mathrm{d}\tau}_{O_3}\notag\\
    &\quad+\underbrace{\sup_{j\in\mathbb{Z}}\int_0^t \int_{\mathbb{R}_\xi^3}e^{-|\xi|^2(t-\tau)}|\xi|^2\left|\hat{\Gamma}(f_2,g)\langle v\rangle^{-\frac{\gamma}{2}}\varphi_j\right|_{L_v^2}^2\,\mathrm{d}\xi\mathrm{d}\tau}_{O_4}\,.\notag
\end{align}
For $O_1$, since $|\xi|\sim2^j$ and it holds from Lemma \ref{lem:max_estimate} that
\[e^{-2^{2j}t}2^{mj}\lesssim(1+t)^{-\frac{m}{2}}\,,\] one has
\begin{align*}
    O_1\lesssim \sup_{j\in\mathbb{Z}}\int_{\mathbb{R}_\xi^3} e^{-2^{2j}t}2^j\left|\varphi_j2^{\frac{j}{2}}\hat{g}_0\right|_{L_v^2}^2\,\mathrm{d}\xi\lesssim(1+t)^{-\frac{1}{2}}\|g_0\|_{\widetilde{L}_v^2(\dot{B}_{2,\infty}^\frac{1}{2})}^2\,.
\end{align*}
It's easy to obtain
\begin{align*}
    O_{2}&\lesssim\eta\sup_t(1+t)^{\frac{1}{2}}\left\|\mathbf{P}^{\bot}g\langle v\rangle^{-\frac{\gamma}{2}}\right\|_{\widetilde{L}_v^2(\dot{B}_{2,\infty}^1)}^2\sup_{j\in\mathbb{Z}}\int_0^t e^{-2^{2j} (t-\tau)}2^{2j}(1+\tau)^{-\frac{1}{2}}\,\mathrm{d}\tau\\
    &\lesssim\eta(1+t)^{-\frac{1}{2}}\sup_t(1+t)^{\frac{1}{2}}\left\|\mathbf{P}^{\bot}g\langle v\rangle^{-\frac{\gamma}{2}}\right\|_{L_v^2(\dot{B}_{2,\infty}^1)}^2\\
    &\lesssim\eta(1+t)^{-\frac{1}{2}}\sup_t(1+t)^{\frac{1}{2}}\left\|\mathbf{P}^{\bot}g\langle v\rangle^{-\frac{\gamma}{2}}\right\|_{L_v^2(\dot{H}_x^1)}^2\,,
    \end{align*}
    here we have used the fact that
\begin{align}\label{6,3}
    &\hspace{3mm}\int_0^te^{-2^{2j}(t-\tau)}2^{2j}(1+\tau)^{-\frac{1}{2}}\,\mathrm{d}\tau\notag\\
    &\lesssim\int_0^{\frac{t}{2}}e^{-2^{2j}(t-\tau)}2^{2j}(1+\tau)^{-\frac{1}{2}}\,\mathrm{d}\tau+\int_{\frac{t}{2}}^te^{-2^{2j}(t-\tau)}2^{2j}(1+\tau)^{-\frac{1}{2}}\,\mathrm{d}\tau\notag\\
    &\lesssim\int_0^{\frac{t}{2}}(1+t-\tau)^{-1}(1+\tau)^{-\frac{1}{2}}\,\mathrm{d}\tau+\int_{\frac{t}{2}}^te^{-2^{2j}(t-\tau)}2^{2j}(1+\tau)^{-\frac{1}{2}}\,\mathrm{d}\tau\notag\\
    &\lesssim(1+\frac{t}{2})^{-1}\int_0^{\frac{t}{2}}(1+\tau)^{-\frac{1}{2}}\,\mathrm{d}\tau+(1+\frac{t}{2})^{-\frac{1}{2}}\int_{\frac{t}{2}}^te^{-2^{2j}(t-\tau)}2^{2j}\,\mathrm{d}\tau\lesssim(1+t)^{-\frac{1}{2}}\,.
\end{align}
We next focus on $O_3$, and  the estimate of $O_4$ follows in the same way. By Lemma \ref{lemma5.2} and Proposition \ref{prop2.1} with $s_1=\frac{1}{2}, s_2=1, q_1=q_2=\infty$, we can deduce that
\begin{align*}
    O_3&\lesssim\sup_{j\in\mathbb{Z}}\int_0^te^{-2^{2j}(t-\tau)}2^{2j}\varphi_j^2\left\|\Gamma(g,f_1)\langle v\rangle^{-\frac{\gamma}{2}}\right\|_{{B_{2,\infty}^0(L_v^2)}}^2\,\mathrm{d}\tau\\
    &\lesssim\sup_{j\in\mathbb{Z}}\int_0^te^{-2^{2j}(t-\tau)}2^{2j}\left||f_1|_{H_v^2}|g|_{L_v^2}\right|_{\dot{B}_{2,\infty}^0}^2\,\mathrm{d}\tau\\
    &\lesssim\sup_{j\in\mathbb{Z}}\int_0^te^{-2^{2j}(t-\tau)}2^{2j}(1+\tau)^{-\frac{1}{2}}(1+\tau)^{\frac{1}{2}}\|f_1\|_{\dot{B}_{2,\infty}^{\frac{1}{2}}H_v^2}^2\|g\|_{\dot{B}_{2,\infty}^1L_v^2}^2\,\mathrm{d}\tau\\
    &\lesssim\sup_t\Big\{\|{\bf P}f_1\|_{\dot{B}_{2,\infty}^{\frac{1}{2}}H_v^2}^2+\|{\bf P}^\bot f_1\|_{\dot{B}^\frac{1}{2}_{2,\infty}H_v^2}^2\Big\}\sup_t\left[(1+t)^{\frac{1}{2}}\|g\|_{\dot{B}_{2,\infty}^1L_v^2}^2\right]\sup_{j\in\mathbb{Z}}\int_0^te^{-2^{2j}(t-\tau)}2^{2j}(1+\tau)^{-\frac{1}{2}}\,\mathrm{d}\tau\\
    &\lesssim\sup_t\Big\{|(a,b,c)|^2_{\dot{B}^\frac{1}{2}_{2,\infty}}+\|{\bf P}^\bot f\|^2_{H_v^2(L_x^2\cap\dot{H}_x^1)}\Big\}\sup_t(1+t)^{\frac{1}{2}}\|g\|_{\widetilde{L}_v^2(\dot{B}_{2,\infty}^1)}^2\\
    &\lesssim\mathcal{M}_0^2(1+t)^{-\frac{1}{2}}\sup_t(1+t)^{\frac{1}{2}}\|g\|_{\widetilde{L}_v^2(\dot{B}_{2,\infty}^1)}^2\,.
\end{align*}
Thus, the estimate of $O_4$ follows immediately:
\begin{align*}
    O_4\lesssim(1+t)^{-\frac{1}{2}}\sup_t\|f_2\|_{\dot{B}_{2,\infty}^{\frac{1}{2}}H_v^2}^2\sup_t(1+t)^{\frac{1}{2}}\|g\|_{\dot{B}_{2,\infty}^1L_v^2}^2\lesssim\mathcal{M}_0^2(1+t)^{-\frac{1}{2}}\sup_t(1+t)^{\frac{1}{2}}\|g\|_{\widetilde{L}_v^2(\dot{B}_{2,\infty}^1)}^2\,.
\end{align*}
Secondly, for the high frequency part, similar to \eqref{h-f}, we obtain for $g$ that
\begin{align*}
    \left|\hat{g}_H\right|_{L_v^2}^2+\left||\xi|\hat{g}_H\right|_{L_v^2}^2&\lesssim e^{-t}\left(\left|\hat{g_0}\right|_{L_v^2}^2+\left||\xi|\hat{g_0}\right|_{L_v^2}^2\right)+\eta\int_0^t e^{-(t-\tau)}\left||\xi|{\bf P}^\bot \hat{g}\,\langle v\rangle^{-\frac{\gamma}{2}}\right|_{L_v^2}^2\,\mathrm{d}\tau\\
    &\lesssim\int_0^t e^{-(t-\tau)}\left||\xi|\Gamma(g,f_1)\,\langle v\rangle^{-\frac{\gamma}{2}}\right|_{L_v^2}^2\,\mathrm{d}\tau+\int_0^t e^{-(t-\tau)}\left||\xi|\Gamma(f_2,g)\,\langle v\rangle^{-\frac{\gamma}{2}}\right|_{L_v^2}^2\,\mathrm{d}\tau\,.
\end{align*}
Multiplying both sides of the above inequality by $\varphi_j^2|\xi|^2$, integrating with respect to $|\xi|$, and taking supremum over $j\in\mathbb{Z}$, one has
\begin{align*}
    \|g_H\|_{\widetilde{L}_v^2(\dot{B}_{2,\infty}^1\cap\dot{B}_{2,\infty}^2)}^2&\lesssim \underbrace{e^{-t}\|g_0\|_{\widetilde{L}_v^2(\dot{B}_{2,\infty}^1\cap\dot{B}_{2,\infty}^2)}^2}_{O_5}+\underbrace{\eta\sup_{j\in\mathbb{Z}}\int_0^t\int_{\mathbb{R}_\xi^3} e^{-(t-\tau)}\left||\xi|^2\mathbf{P}^{\bot}\hat{g}\,\langle v\rangle^{-\frac{\gamma}{2}}\right|_{L_v^2}^2\varphi_j^2\,\mathrm{d}\xi\mathrm{d}\tau}_{O_6}\notag\\
    &\quad\underbrace{+\sup_{j\in\mathbb{Z}}\int_0^t\int_{\mathbb{R}_\xi^3} e^{-(t-\tau)}\left||\xi|^2\hat{\Gamma}(g,f_1)\langle v\rangle^{-\frac{\gamma}{2}}\right|_{L_v^2}^2\varphi_j^2\,\mathrm{d}\xi\mathrm{d}\tau}_{O_{7}}\notag\\
    &\quad\underbrace{+\sup_{j\in\mathbb{Z}}\int_0^t\int_{\mathbb{R}_\xi^3} e^{-(t-\tau)}\left||\xi|^2\hat{\Gamma}(f_2,g)\langle v\rangle^{-\frac{\gamma}{2}}\right|_{L_v^2}^2\varphi_j^2\,\mathrm{d}\xi\mathrm{d}\tau}_{O_{8}}\,.
\end{align*}
For $O_{5}$, from $e^{-t}\lesssim(1+t)^{-\frac{1}{2}}$, one has
\[O_{5}\lesssim(1+t)^{-\frac{1}{2}}\|g_0\|_{\widetilde{L}_v^2(\dot{B}_{2,\infty}^1\cap\dot{B}_{2,\infty}^2)}^2\,\]
and 
\begin{align*}
    O_{6}&\lesssim\eta\sup_t(1+t)^{\frac{1}{2}}\left\|\mathbf{P}^{\bot}g\langle v\rangle^{-\frac{\gamma}{2}}\right\|_{\widetilde{L}_v^2(\dot{B}_{2,\infty}^2)}^2\sup_{j\in\mathbb{Z}}\int_0^t e^{-(t-\tau)}(1+\tau)^{-\frac{1}{2}}\,\mathrm{d}\tau\\
    &\lesssim(1+t)^{-\frac{1}{2}}\eta\sup_t(1+t)^{\frac{1}{2}}\left\|\mathbf{P}^{\bot}g\langle v\rangle^{-\frac{\gamma}{2}}\right\|_{L_v^2(\dot{B}_{2,\infty}^2)}^2\\
    &\lesssim(1+t)^{-\frac{1}{2}}\eta\sup_t(1+t)^{\frac{1}{2}}\left\|\mathbf{P}^{\bot}g\langle v\rangle^{-\frac{\gamma}{2}}\right\|_{L_v^2(\dot{H}_x^2)}^2\,,
\end{align*}
    where the following fact was used
    \begin{align}\label{S1}
        \int_0^t e^{-(t-\tau)}(1+\tau)^{-\frac{1}{2}}\,\mathrm{d}\tau&=\int_0^{\frac{t}{2}} e^{-(t-\tau)}(1+\tau)^{-\frac{1}{2}}\,\mathrm{d}\tau+\int_{\frac{t}{2}}^t e^{-(t-\tau)}(1+\tau)^{-\frac{1}{2}}\,\mathrm{d}\tau\notag\\
        &\lesssim e^{-\frac{t}{2}}\int_0^{\frac{t}{2}} (1+\tau)^{-\frac{1}{2}}\,\mathrm{d}\tau+(1+\frac{t}{2})^{-\frac{1}{2}}\int_{\frac{t}{2}}^t e^{-(t-\tau)}\,\mathrm{d}\tau\lesssim(1+t)^{-\frac{1}{2}}\,.
    \end{align}
Next, Lemma \ref{interpolation-besov-other} implies
\begin{align*}
    &\hspace{3mm}\left\|\Gamma(g,f_1)\langle v\rangle^{-\frac{\gamma}{2}}\right\|_{\dot{B}_{2,\infty}^2(L_v^2)}^2
\lesssim\left||f_1|_{H_v^2}|g|_{L_v^2}\right|_{\dot{B}_{2,\infty}^2}^2\lesssim\|f_1\|_{L_x^{\infty}H_v^2}^2\|g\|_{\dot{B}_{2,\infty}^2L_v^2}^2+\|f_1\|_{\dot{B}_{2,\infty}^{2}H_v^2}^2\|g\|_{L_x^\infty L_v^2}\\
&\lesssim\|f_1\|_{H_v^2(\dot{H}_x^1\cap\dot{H}^2_x)}^2\|g\|_{\dot{B}_{2,\infty}^2L_v^2}^2+\|f_1\|_{\dot{B}_{2,\infty}^{2}H_v^2}^2\left(\|{\bf P}g\|_{L_v^2(\dot{B}^1_{2,\infty}\cap\dot{H}^2)}^2+\|{\bf P}^\bot g\|_{L_v^2(\dot{H}^1_x\cap\dot{H}_x^2)}^2\right)\\
&\lesssim\mathcal{M}_0^2\|g\|_{\widetilde{L}_v^2(\dot{B}_{2,\infty}^2)}^2+\mathcal{M}_0^2|(a^g,b^g,c^g)|_{\dot{B}_{2,\infty}^1}^2+\mathcal{M}_0^2\|{\bf P}^\bot g\|^2_{L_v^2(\dot{H}_x^1\cap\dot{H}_x^2)}.
\end{align*}
Therefore, it follows
\begin{align*}
    O_7
    &\lesssim\mathcal{M}_0^2(1+t)^{-\frac{1}{2}}\sup_t(1+t)^{\frac{1}{2}}\|g\|_{\widetilde{L}_v^2(\dot{B}_{2,\infty}^2)}^2+\mathcal{M}_0^2(1+t)^{-\frac{1}{2}}\sup_t(1+t)^{\frac{1}{2}}|(a^g,b^g,c^g)|_{\dot{B}_{2,\infty}^1}^2\\
    &\qquad+\mathcal{M}_0^2(1+t)^{-\frac{1}{2}}\sup_t(1+t)^{\frac{1}{2}}\|{\bf P}^\bot g\|^2_{L_v^2(\dot{H}_x^1\cap\dot{H}_x^2)}\,.
\end{align*}
The estimate of $O_8$ shares the same bound as above.

Finally, choosing $\mathcal{M}_0$ suitably small and combing all the above estimates, we get the desired result. This completes the proof of Lemma \ref{prop6.2}.

\end{proof}

\begin{lemma}\label{lemma6.2}
    Under the same assumption of Lemma \ref{prop6.2}, we have 
    \begin{align}\label{t2}
        \sup_t(1+t)^{\frac{1}{2}}\|g\|_{L_v^2(\dot{H}^2\cap\dot{H}^{N-2})}^2&\lesssim\|g_0\|_{L_v^2(\dot{H}^2\cap\dot{H}^{N-2})}^2+\eta\sup_t(1+t)^{\frac{1}{2}}\left\|\mathbf{P}^{\bot}g\langle v\rangle^{-\frac{\gamma}{2}}\right\|_{L_v^2(\dot{H}^2\cap\dot{H}^{N-2})}^2\notag\\
        &\qquad+\sup_t(1+t)^{\frac{1}{2}}|(a^g,b^g,c^g)|_{\dot{B}_{2,\infty}^{1}}^2\,.\notag
    \end{align}
\end{lemma}
\begin{proof}
Similar to Lemma \ref{lemma4.2}, we obtain the following estimate
\begin{align}
    \|g\|_{L_v^2(\dot{H}^2\cap\dot{H}^{N-2})}^2&\lesssim \underbrace{e^{-\lambda t} \|g_0\|_{L_v^2(\dot{H}^2\cap\dot{H}^{N-2})}^2}_{O_{9}}+\underbrace{\eta\int_0^t e^{-\lambda(t-\tau)}\left\|\mathbf{P}^{\bot}g\langle v\rangle^{-\frac{\gamma}{2}}\right\|_{L_v^2(\dot{H}^2\cap\dot{H}^{N-2})}^2\,\mathrm{d}\tau}_{O_{10}}\notag\\
    &\qquad+\underbrace{\eta'\int_0^t e^{-\lambda(t-\tau)}\|\mathbf{P}g\|_{L_v^2(\dot{H}^3)}^2\,\mathrm{d}\tau}_{O_{11}}+\underbrace{\int_0^t e^{-\lambda(t-\tau)}\|\mathbf{P}g\|_{L_v^2(\dot{B}_{2,\infty}^{1})}^2\,\mathrm{d}\tau}_{O_{12}}\,,\notag
\end{align}
where we have used Lemma \ref{interpolation-besov-other} and Young's inequality,
\[\|\nabla^2\mathbf{P}f\|^2\lesssim\|{\bf P}f\|_{\dot{H}^2}\lesssim\|{\bf P}f\|_{\dot{B}_{2,\infty}^1}^\theta\|{\bf P}f\|_{\dot{H}^3}^{1-\theta}\lesssim\| \mathbf{P}f\|_{\dot{B}_{2,\infty}^{1}}^2+\eta\|\mathbf{P}f\|_{\dot{H}^3}^2\,.\]
It's easy to see that
\[O_{9}\lesssim(1+t)^{-\frac{1}{2}}\|g_0\|_{L_v^2(\dot{H}^2\cap\dot{H}^{N-2})}^2\]
and 
\[O_{10}\lesssim(1+t)^{-\frac{1}{2}}\eta\sup_t(1+t)^{\frac{1}{2}}\left\|\mathbf{P}^{\bot}g\langle v\rangle^{-\frac{\gamma}{2}}\right\|_{L_v^2(\dot{H}^2\cap\dot{H}^{N-2})}^2\,.\]
Due to \eqref{S1}, we have
\begin{align*}
    O_{11}\lesssim(1+t)^{-\frac{1}{2}}\eta\sup_t(1+t)^{\frac{1}{2}}\|{\bf P}g\|_{L_v^2(\dot{H}^3)}^2\,.
\end{align*}
As to $O_{12}$, one has
\[O_{12}\lesssim(1+t)^{-\frac{1}{2}}\sup_t\left[(1+t)^{\frac{1}{2}}|(a^g,b^g,c^g)|_{\dot{B}_{2,\infty}^{1}}^2\right]\,.\]
Collecting all the above estimates gives the desired estimate. The proof of Lemma \ref{lemma6.2} is complete.

\end{proof}
To absorb the difficult term
$\sup\limits_t(1+t)^{\frac{1}{2}}\left\|{\bf P}^\bot g\,\langle v\rangle^{-\frac{\gamma}{2}}\right\|_{L_v^2(\dot{H}_x^1\cap\dot{H}_x^2)}$, we need to establish the time decay of ${\bf P}^\bot g$ with velocity weight. Before doing so, we first derive an estimate on ${\bf P}^\bot g$ in $\dot{H}^1$:
\begin{lemma}\label{ppg-de}
    Under the assumption of Lemma \ref{prop6.2}, we have
    \begin{align*}
        \sup_t(1+t)^{\frac{1}{2}}\left\|{\bf P}^\bot g\right\|_{L_v^2(\dot{H}^1)}^2&\lesssim\left\|{\bf P}^\bot g_0\right\|_{L_v^2(\dot{H}^1)}^2+\sup_t(1+t)^\frac{1}{2}|(a^g,b^g,c^g)|_{\dot{B}_{2,\infty}^1}^2\\
        &\qquad+\eta\sup_t(1+t)^\frac{1}{2}\left\|{\bf P}^\bot g\,\langle v\rangle^{-\frac{\gamma}{2}}\right\|_{L_v^2(\dot{H}^1)}^2+\sup_t(1+t)^\frac{1}{2}\|g\|_{L_v^2(\dot{H}^2)}^2\,.
    \end{align*}
\end{lemma}
    \begin{proof}
        Applying ${\bf P}^{\bot}$ to \eqref{X1} yields that
        \begin{align}
            \partial_t{\bf P}^{\bot}g+\mathcal{L}{\bf P}^{\bot}g=-v\cdot\nabla_x{\bf P}^{\bot}g-v\cdot\nabla_x{\bf P}g+{\bf P}(v\cdot\nabla_x g)+{\bf P}^{\bot}\Gamma(g,f_1)+{\bf P}^{\bot}\Gamma(f_2,g)\,.\notag
        \end{align}
        Acting $\partial_x$ to the above equation, multiplying the resulting equation by $\partial_x {\bf P}^{\bot}g$ and then integrating over $\mathbb{R}_x^3\times\mathbb{R}_v^3$, we have from \eqref{D2} that
        \begin{align}\label{OO2}
            &\hspace{3mm}\frac{\mathrm{d}}{\mathrm{d}t}\left\|\partial_x{\bf P}^{\bot}g\right\|^2+\delta_0\left\|\partial_x{\bf P}^{\bot}g\right\|_{\nu}^2\notag\\
            &=-\int_{\mathbb{R}_{x,v}^6}\partial_x\left( v\cdot\nabla_x {\bf P}g\right)\partial_x{\bf P}^{\bot}g\,\mathrm{d}v\mathrm{d}x+\int_{\mathbb{R}_{x,v}^6} \partial_x\left[{\bf P}(v\cdot\nabla_x g)\right]\partial_x{\bf P}^{\bot}g\,\mathrm{d}v\mathrm{d}x\notag\\
            &\quad+\int_{\mathbb{R}_{x,v}^6} \partial_x\Gamma(g,f_1)\partial_x{\bf P}^{\bot}g\,\mathrm{d}v\mathrm{d}x+\int_{\mathbb{R}_{x,v}^6} \partial_x\Gamma(f_2,g)\partial_x{\bf P}^{\bot}g\,\mathrm{d}v\mathrm{d}x\,.
        \end{align}
       Note that
        \begin{align*}
           \hspace{3mm} \int_{\mathbb{R}_{x,v}^6} \partial_x\Gamma(g,f_1)\partial_x{\bf P}^{\bot}g\,\mathrm{d}v\mathrm{d}x=&\int_{\mathbb{R}_{x,v}^6} \partial_x\Gamma(g,f_1)\langle v\rangle^{-\frac{\gamma}{2}}\langle v\rangle^{\frac{\gamma}{2}}\partial_x{\bf P}^{\bot}g\,\mathrm{d}v\mathrm{d}x\\
    \lesssim&\left\|\partial_x\Gamma(g,f_1)\langle v\rangle^{-\frac{\gamma}{2}}\right\|^2+\eta\left\|\partial_x{\bf P}^\bot g\right\|_{\nu}^2\,,
        \end{align*}
        and
        \begin{align*}
            \int_{\mathbb{R}_v^3}\left|\partial_x{\bf P}^{\bot}g\right|^2\,\mathrm{d}v&=\int_{\mathbb{R}_v^3}\partial_x{\bf P}^{\bot}g\,\langle v\rangle^{\frac{\gamma}{2}}\partial_x{\bf P}^{\bot}g\,\langle v\rangle^{-\frac{\gamma}{2}}\,\mathrm{d}v\lesssim\left|\partial_x{\bf P}^{\bot}g\,\langle v\rangle^{\frac{\gamma}{2}}\right|_{L_v^2}^2+\eta\left|\partial_x{\bf P}^{\bot}g\,\langle v\rangle^{-\frac{\gamma}{2}}\right|_{L_v^2}^2\,.
        \end{align*}

        \noindent Then if we choose $0<\eta\ll\delta_0$, \eqref{OO2} can gives
        \begin{align}
            \hspace{3mm}\frac{\mathrm{d}}{\mathrm{d}t}\left\|\partial_x{\bf P}^{\bot}g\right\|^2+\lambda\left\|\partial_x{\bf P}^{\bot}g\right\|^2\lesssim&\int_{\mathbb{R}_{x,v}^6}\partial_x\left( v\cdot\nabla_x {\bf P}g\right)\partial_x{\bf P}^{\bot}g\,\mathrm{d}v\mathrm{d}x+\int_{\mathbb{R}_{x,v}^6} \partial_x\left[{\bf P}(v\cdot\nabla_x g)\right]\partial_x{\bf P}^{\bot}g\,\mathrm{d}v\mathrm{d}x\notag\\
            &\quad+\left\|\partial_x\Gamma(g,f_1)\langle v\rangle^{-\frac{\gamma}{2}}\right\|^2+\left\|\partial_x\Gamma(f_2,g)\langle v\rangle^{-\frac{\gamma}{2}}\right\|^2+\eta\left\|\partial_x{\bf P}^{\bot}g\,\langle v\rangle^{-\frac{\gamma}{2}}\right\|^2\,,\notag
        \end{align}
where $\lambda>0$ is a constant. Solving the above inequality directly, one has
        \begin{align*}
            \left\|\partial_x{\bf P}^{\bot}g\right\|^2&\lesssim e^{-\lambda t}\left\|\partial_x{\bf P}^{\bot}g_0\right\|^2+\int_0^t\int_{\mathbb{R}_{x,v}^6} e^{-\lambda(t-\tau)}\partial_x\left( v\cdot\nabla_x {\bf P}g\right)\partial_x{\bf P}^{\bot}g\,\mathrm{d}v\mathrm{d}x \mathrm{d}\tau\notag\\
            &\quad+\int_0^t \int_{\mathbb{R}_{x,v}^6} e^{-\lambda(t-\tau)}\partial_x\left[{\bf P}(v\cdot\nabla_x g)\right]\partial_x{\bf P}^{\bot}g\,\mathrm{d}v\mathrm{d}x\mathrm{d}\tau\notag\\
           &\quad+\int_0^t e^{-\lambda(t-\tau)}\left\|\partial_x\Gamma(g,f_1)\langle v\rangle^{-\frac{\gamma}{2}}\right\|^2\,\mathrm{d}\tau
           +\int_0^t  e^{-\lambda(t-\tau)}\left\|\partial_x\Gamma(f_2,g)\langle v\rangle^{-\frac{\gamma}{2}}\right\|^2\,\mathrm{d}\tau\notag\\
           &\quad+\eta\int_0^t e^{-\lambda(t-\tau)}\left\|\partial_x{\bf P}^{\bot}g\,\langle v\rangle^{-\frac{\gamma}{2}}\right\|^2\mathrm{d}\tau\notag\\
           &:=O_{13}+O_{14}+O_{15}+O_{16}+O_{17}+O_{18}\,.
        \end{align*}
For $O_{13}$, from
$e^{-\lambda t}\lesssim(1+t)^{-\frac{1}{2}}$, we have 
 \[O_{13}\lesssim(1+t)^{-\frac{1}{2}}\left\|\partial_x{\bf P}^{\bot }g_0\right\|^2\,.\]
Due to \eqref{S1}, we get
\begin{align*}
    O_{14}&\lesssim\int_0^t\int_{\mathbb{R}_{x,v}^6} e^{-\lambda(t-\tau)}\mu^\delta \nabla_x\partial_x(a,b,c)\partial_x{\bf P}^{\bot}g\,\mathrm{d}v\mathrm{d}x \mathrm{d}\tau\\
    &\lesssim\int_0^t e^{-\lambda(t-\tau)}\|\mu^{\delta}(a,b,c)\|_{\dot{H}^2(L_v^2)}^2\,\mathrm{d}\tau+\eta\int_0^t e^{-\lambda(t-\tau)}\left\|\partial_x{\bf P}^\bot g\right\|^2\,\mathrm{d}\tau\\
   & \lesssim\int_0^t e^{-\lambda(t-\tau)}(1+\tau)^{-\frac{1}{2}}(1+\tau)^{\frac{1}{2}}\|\mu^{\delta}(a,b,c)\|_{\dot{H}^2(L_v^2)}^2\,\mathrm{d}\tau\\
   &\qquad+\eta\int_0^t e^{-\lambda(t-\tau)}(1+\tau)^{-\frac{1}{2}}(1+\tau)^{\frac{1}{2}}\left\|\partial_x{\bf P}^\bot g\right\|^2\,\mathrm{d}\tau\\
   &\lesssim(1+t)^{-\frac{1}{2}}\sup_t(1+t)^{\frac{1}{2}}\|(a,b,c)\|_{\dot{H}^2}^2+\eta(1+t)^{-\frac{1}{2}}\sup_t(1+t)^{\frac{1}{2}}\left\|\partial_x{\bf P}^\bot g\right\|^2\,.
\end{align*}
For $O_{15}$, we have 
\begin{align}\label{OO3}
    &\hspace{3mm}\partial_x{\bf P}(v\cdot\nabla_x g)=\partial_x\int_{\mathbb{R}_v^3} v\cdot\nabla_x g\,\phi_i\,\mathrm{d}v\,\mu^\delta=\int_{\mathbb{R}_v^3} v\cdot\nabla_x\partial_x g\,\phi_i\,\mathrm{d}v\,\mu^\delta\notag\\
    &\lesssim|\nabla_x\partial_x g|_{L_v^2}|\mu^{\delta}|_{L_v^2}\,\mu^\delta\lesssim|\nabla_x\partial_x g|_{L_v^2}\,\mu^\delta\,,
\end{align}
        plugging this into $O_{15}$, we further get from \eqref{S1} that
        \begin{align*}
            O_{15}&\lesssim\int_0^t \int_{\mathbb{R}_v^3}\int_{\mathbb{R}_x^3}e^{-\lambda(t-\tau)}\mu^{\delta}|\nabla_x\partial_x g|_{L_v^2}\partial_x{\bf P}^\bot g\,\mathrm{d}x\mathrm{d}v\mathrm{d}\tau\\
            &\lesssim\int_0^t\int_{\mathbb{R}_v^3} e^{-\lambda(t-\tau)}\mu^\delta\|\nabla_x\partial_x g\|^2\,\mathrm{d}v\mathrm{d}\tau+\eta\int_0^t\int_{\mathbb{R}_v^3} e^{-\lambda(t-\tau)}\left|\partial_x{\bf P}^\bot g\right|_{L_x^2}^2\,\mathrm{d}v\mathrm{d}\tau\\
            &\lesssim(1+t)^{-\frac{1}{2}}\sup_t(1+t)^{\frac{1}{2}}\|g\|_{L_v^2(\dot{H}^2)}^2+\eta(1+t)^{-\frac{1}{2}}\sup_t(1+t)^{\frac{1}{2}}\left\|\partial_x{\bf P}^\bot g\right\|^2\,.
        \end{align*}
        For $O_{16}$,  it can be derived from Lemma \ref{nonlinear} that
        \begin{align*}
            \left|\Gamma(g,f_1)\langle v\rangle^{-\frac{\gamma}{2}}\right|_{L_v^2}^2\lesssim|f_1|_{{H}_v^2}^2|g|_{L_v^2}^2\,,
        \end{align*}
        then we can obtain
        \begin{align*}
            &\hspace{3mm}\left\|\partial_x\Gamma(g,f_1)\langle v\rangle^{-\frac{\gamma}{2}}\right\|^2=\left\|\Gamma(\partial_xg,f_1)\langle v\rangle^{-\frac{\gamma}{2}}+\Gamma(g,\partial_xf_1)\langle v\rangle^{-\frac{\gamma}{2}}\right\|^2\\
            &\lesssim\int_{\mathbb{R}_x^3}|{\bf P}f_1|_{{H}_v^2}^2|\partial_x{\bf P}g|_{L_v^2}^2\,\mathrm{d}x+\int_{\mathbb{R}_x^3}|{\bf P}^\bot f_1|_{{H}_v^2}^2|\partial_x{\bf P}^\bot g|_{L_v^2}^2\,\mathrm{d}x\\
            &\qquad+\int_{\mathbb{R}_x^3}|\partial_x{\bf P}f_1|_{{H}_v^2}^2|{\bf P}g|_{L_v^2}^2\,\mathrm{d}x+\int_{\mathbb{R}_x^3}|\partial_x{\bf P}^\bot f_1|_{{H}_v^2}^2|{\bf P}^\bot g|_{L_v^2}^2\,\mathrm{d}x\\
            &\lesssim\|{\bf P}f_1\|_{L^4_xH_v^2}^2\|\partial_x{\bf P}g\|_{L^4_xL_v^2}^2+\|{\bf P}^\bot f_1\|_{L^4_xH_v^2}^2\|\partial_x{\bf P}^\bot g\|_{L^4_xL_v^2}^2\\
            &\qquad+\|{\bf P}g\|_{L_x^\infty L_v^2}^2\|{\bf    P}f_1\|_{\dot{H}_x^1H_v^2}^2+\|{\bf P}^\bot g\|_{L_x^\infty L_v^2}^2\|{\bf    P}^\bot f_1\|_{\dot{H}_x^1H_v^2}^2\\
            &\lesssim\|{\bf P}f_1\|_{H_v^2(\dot{B}_{2,\infty}^\frac{1}{2}\cap\dot{H}^1)}^2\|{\bf P}g\|_{L_v^2(\dot{B}_{2,\infty}^1\cap\dot{H}^2)}^2+\|{\bf P}^\bot f_1\|_{H_v^2(L_x^2\cap\dot{H}^1_x)}^2\|{\bf P}^\bot g\|_{L_v^2(\dot{H}_x^1\cap\dot{H}_x^2)}^2\\
            &\qquad+\|{\bf P}g\|_{L_v^2(\dot{B}_{2,\infty}^1\cap\dot{H}_x^2)}^2\|{\bf P}f_1\|_{H_v^2\dot{H}_x^1}^2+\|{\bf P}^\bot g\|_{L_v^2(\dot{H}^1\cap\dot{H}^2)}^2\|{\bf P}f_1\|_{H_v^2\dot{H}_x^1}^2\\
            &\lesssim\Big(|(a^{f_1},b^{f_1},c^{f_1})|_{\dot{B}_{2,\infty}^\frac{1}{2}\cap\dot{H}^1}^2+|(a^{f_1},b^{f_1},c^{f_1})|_{\dot{H}_x^1}^2\Big)\times|(a^g,b^g,c^g)|_{\dot{B}_{2,\infty}^1\cap\dot{H}^2}^2\\
            &\qquad+\Big(\|{\bf P}^\bot f_1\|_{H_v^2(L_x^2\cap\dot{H}^1_x)}^2+\|{\bf P}^\bot f_1\|_{H_v^2\dot{H}_x^1}^2\Big)\times\|{\bf P}^\bot g\|_{L_v^2(\dot{H}_x^1\cap\dot{H}_x^2)}^2\\
            &\lesssim\|f_1\|_X^2\left(\|{\bf P}^\bot g\|^2_{L_v^2(\dot{H}^1\cap\dot{H}^2)}+|(a^g,b^g,c^g)|_{\dot{B}_{2,\infty}^1\cap\dot{H}^2}^2\right)\,,
        \end{align*}
        here we have used the interpolation in Lemma \ref{interpolation-besov-other} and the fact that $\dot{H}^s\hookrightarrow\dot{B}_{2,\infty}^s$. Inserting the above estimate into $O_{16}$, we further obtain
        \begin{align*}
            O_{16}&\lesssim\sup_t\|f_1\|_X^2\int_0^t e^{-\lambda(t-\tau)}\left(\|{\bf P}^\bot g\|^2_{L_v^2(\dot{H}^1\cap\dot{H}^2)}+|(a^g,b^g,c^g)|_{\dot{B}_{2,\infty}^1\cap\dot{H}^2}^2\right)\,\mathrm{d}\tau\\
            &\lesssim\mathcal{M}_0^2\sup_t(1+t)^{\frac{1}{2}}\|{\bf P}^\bot g\|^2_{L_v^2(\dot{H}^1\cap\dot{H}^2)}+\mathcal{M}_0^2\sup_t(1+t)^\frac{1}{2}|(a^g,b^g,c^g)|_{\dot{B}_{2,\infty}^1\cap\dot{H}^2}^2\,.
        \end{align*}
        The estimate of $O_{17}$ can be obtained by the same way.
        
        \noindent As for $O_{18}$, it's clear that
        \begin{align*}
            O_{18}&\lesssim\eta\int_0^t e^{-\lambda(t-\tau)}(1+\tau)^{-\frac{1}{2}}(1+\tau)^{\frac{1}{2}}\left\|\partial_x{\bf P}^\bot g\,\langle v\rangle^{-\frac{\gamma}{2}}\right\|^2\\
            &\lesssim\eta(1+t)^{-\frac1 2}\sup_t(1+t)^{\frac{1}{2}}\left\|\partial_x{\bf P}^\bot g\,\langle v\rangle^{-\frac{\gamma}{2}}\right\|^2\,.
        \end{align*}
       Combining all the relevant estimates, we obtain the desired result. This completes the proof of Lemma \ref{ppg-de}.
    \end{proof}
   Motivated by the ideas in Section \ref{Section-V-j}, we establish the following lemma to derive the time decay of $\partial_x^\alpha g$ with velocity weight.
    \begin{lemma}\label{lemma8.5}
    Let $N\ge5$ be an integer and $-3<\gamma<0$. Assume that $g$ satisfying \eqref{X1}. Then it holds  that
        \begin{align*}
        &\hspace{3mm}\sum_{j\in\mathbb{Z}^+}\sum_{1\le|\alpha|\le N-2}\sup_t\Big\{(1+t)^\frac{1}{2}\|w_{\ell,0}\partial_x^\alpha{\bf P}^\bot g\|^2_{L^2(V_j)L_x^2(\mathbb{R}^3)}\Big\}\\
        &\lesssim\sum_{1\le|\alpha|\le N-2}\|w_{\ell-\frac{\gamma}{4},0}\partial_x^\alpha{\bf P}^\bot g_0\|^2+\sup_t\Big\{(1+t)^\frac{1}{2}|(a^g,b^g,c^g)|^2_{\dot{B}_{2,\infty}^1}\Big\}\\
        &\quad+\sup_t\Big\{(1+t)^\frac{1}{2}\|g\|^2_{L_v^2(\dot{H}^2\cap\dot{H}^{N-2})}\Big\}+\sup_t\Big\{(1+t)^\frac{1}{2}\|{\bf P}^\bot g\|^2_{L_v^2(\dot{H}^1)}\Big\}\\
        &\quad+\eta\sup_t\Big\{(1+t)^{\frac{1}{2}}\|w_{\ell,0}{\bf P}^\bot g\|^2_{L_v^2(\dot{H}^1\cap\dot{H}^{N-2})}\Big\}\,.
    \end{align*}
    \end{lemma}
\begin{proof}
The proof is divided into two cases.

    {\bf Case I:} $1\le |\alpha|\le N-3$.  In this case, we consider the microscopic equation.
    Applying $\mathbf{P}^{\bot}$ to \eqref{X1}, one has
    \begin{align}\label{OOOO1}
        &\hspace{3mm}\partial_t \mathbf{P}^{\bot}g+v\cdot\nabla_x\mathbf{P}^{\bot}g+v\cdot\nabla_x\mathbf{P}g+\nu(v){\bf P}^{\bot}g\notag\\
        &={\bf P}(v\cdot\nabla_x g)+\mathcal{K}{\bf P}^{\bot}g+{\bf P}^{\bot}\Gamma(g,f_1)+{\bf P}^{\bot}\Gamma(f_2,g)\,,
    \end{align}
    where we decompose
    \begin{align*}
        {\bf P}^{\bot}(v\cdot\nabla_x g)=v\cdot\nabla_x g-{\bf P}(v\cdot \nabla_x g)=v\cdot \nabla_x{\bf P}^{\bot}g+v\cdot\nabla_x{\bf P}g-{\bf P}(v\cdot\nabla_x g) \,,  
    \end{align*}
    and 
    \begin{align*}
\mathcal{L}{\bf P}^{\bot}g=\nu(v){\bf P}^{\bot}g-\mathcal{K}{\bf P}^{\bot}g\,.
    \end{align*}
    Acting the operator $\partial_x^{\alpha}$ on \eqref{OOOO1} with $1\le|\alpha|\le N-3$ and then multiplying the result by $w_{\ell,0}^2\partial_x^{\alpha}{\bf P}^{\bot}g$, it yields that
    \begin{align}
        &\hspace{3mm}\frac{1}{2}\frac{\mathrm{d}}{\mathrm{d}t}\left(w_{\ell,0}\partial_x^{\alpha}{\bf P}^{\bot}g\right)^2+\nu(v)\left(w_{\ell,0}\partial_x^{\alpha}{\bf P}^{\bot}g\right)^2\notag\\
        &=-\partial_x^{\alpha}(v\cdot\nabla_x\mathbf{P}^{\bot}g)w_{\ell,0}^2\partial_x^{\alpha}{\bf P}^{\bot}g-\partial_x^{\alpha}(v\cdot\nabla_x\mathbf{P}g)w_{\ell,0}^2\partial_x^{\alpha}{\bf P}^{\bot}g\notag\\
        &\quad+\partial_x^{\alpha}({\bf P}(v\cdot\nabla_x g))w_{\ell,0}^2\partial_x^{\alpha}{\bf P}^{\bot}g+\partial_x^{\alpha}(\mathcal{K}{\bf P}^{\bot}g)w_{\ell,0}^2\partial_x^{\alpha}{\bf P}^{\bot}g+\partial_x^{\alpha}({\bf P}^{\bot}\Gamma(g,f_1))w_{\ell,0}^2\partial_x^{\alpha}{\bf P}^{\bot}g\notag\\
        &\quad+\partial_x^{\alpha}({\bf P}^{\bot}\Gamma(f_2,g))w_{\ell,0}^2\partial_x^{\alpha}{\bf P}^{\bot}g\,.\notag
    \end{align}
 Solving the above equation directly and then integrating the result over $V_j\times\mathbb{R}_x^3$, we get
    \begin{align}\label{pp9}
        \left\|w_{\ell,0}\partial_x^{\alpha}{\bf P}^{\bot}g\right\|^2_{L^2(V_j)L_x^2(\mathbb{R}^3)}&\lesssim \int_{V_j\times\mathbb{R}_x^3} e^{-\nu(v)t}\left|w_{\ell,0}\partial_x^{\alpha}{\bf P}^{\bot}g_0\right|^2\,\mathrm{d}v\mathrm{d}x\notag\\
        &\quad\underbrace{-\int_0^t\int_{V_j\times\mathbb{R}_x^3} e^{-\nu(v)(t-\tau)} \partial_x^{\alpha}(v\cdot\nabla_x\mathbf{P}g)w_{\ell,0}^2\partial_x^{\alpha}{\bf P}^{\bot}g\,\mathrm{d}v\mathrm{d}x\mathrm{d}\tau}_{O_{19}}\notag\\
        &\quad+\underbrace{\int_0^t\int_{V_j\times\mathbb{R}_x^3} e^{-\nu(v)(t-\tau)} \partial_x^{\alpha}({\bf P}(v\cdot\nabla_x g))w_{\ell,0}^2\partial_x^{\alpha}{\bf P}^{\bot}g\,\mathrm{d}v\mathrm{d}x\mathrm{d}\tau}_{O_{20}}\notag\\
        &\quad+\underbrace{\int_0^t\int_{V_j\times\mathbb{R}_x^3} e^{-\nu(v)(t-\tau)}\partial_x^{\alpha}(\mathcal{K}{\bf P}^{\bot}g) w_{\ell,0}^2\partial_x^{\alpha}{\bf P}^{\bot}g\,\mathrm{d}v\mathrm{d}x\mathrm{d}\tau}_{O_{21}}\notag\\
        &\quad+\underbrace{\int_0^t\int_{V_j\times\mathbb{R}_x^3} e^{-\nu(v)(t-\tau)} \partial_x^{\alpha}({\bf P}^{\bot}\Gamma(g,f_1))w_{\ell,0}^2\partial_x^{\alpha}{\bf P}^{\bot}g\,\mathrm{d}v\mathrm{d}x\mathrm{d}\tau}_{O_{22}}\notag\\
        &\quad+\underbrace{\int_0^t\int_{V_j\times\mathbb{R}_x^3} e^{-\nu(v)(t-\tau)} \partial_x^{\alpha}({\bf P}^{\bot}\Gamma(f_2,g))w_{\ell,0}^2\partial_x^{\alpha}{\bf P}^{\bot}g\,\mathrm{d}v\mathrm{d}x\mathrm{d}\tau}_{O_{23}}\,.
    \end{align}
    Lemma \ref{lem:max_estimate} implies that
    \begin{align}
        &\hspace{3mm}\int_{V_j\times\mathbb{R}_x^3} e^{-\nu(v)t}\left|w_{\ell,0}\partial_x^{\alpha}{\bf P}^{\bot}g_0\right|^2\,\mathrm{d}v\mathrm{d}x\nonumber\\
        &\lesssim\int_{V_j\times\mathbb{R}_x^3} e^{-\nu(v)t}\nu^\frac{1}{2}(v)\left|w_{\ell-\frac{\gamma}{4},0}\partial_x^{\alpha}{\bf P}^{\bot}g_0\right|^2\,\mathrm{d}v\mathrm{d}x\lesssim(1+t)^{-\frac{1}{2}}\|w_{\ell-\frac{\gamma}{4},0}\partial_x^{\alpha}{\bf P}^{\bot}g_0\|^2_{L^2(V_j)L_x^2(\mathbb{R}^3)}\nonumber\,.
    \end{align}
    By using \eqref{6,3}, we have
    \begin{align*}
        O_{19}&\lesssim\int_0^t\int_{V_j\times\mathbb{R}_x^3}e^{-\tilde{c}_j(t-\tau)}c_j(1+\tau)^{-\frac{1}{2}}(1+\tau)^\frac{1}{2}\nu^{-1}(v)\partial_x^\alpha(v\cdot\nabla_x{\bf P}g)w^2_{\ell,0}\partial_x^\alpha{\bf P}^\bot g\,{\rm d}v{\rm d}x{\rm d}\tau\\
        &\lesssim\sup_t\Big\{(1+t)^\frac{1}{2}\int_{V_j\times\mathbb{R}_x^3}\nu^{-1}\partial_x^\alpha(v\cdot\nabla_x{\bf P}g)w^2_{\ell,0}\partial_x^\alpha{\bf P}^\bot g\,{\rm d}v{\rm d}x\Big\}\int_0^te^{-\tilde{c}_j(t-\tau)}c_j(1+\tau)^{-\frac{1}{2}}\,{\rm d}\tau\\
        &\lesssim(1+t)^{-\frac{1}{2}}\sup_t\Big\{(1+t)^\frac{1}{2}\int_{V_j\times\mathbb{R}_x^3}\partial_x^\alpha\nabla_x(a^g,b^g,c^g)|v|^{-1}\mu^\delta w_{\ell,0}\partial_x^\alpha{\bf P}^\bot g\,{\rm d}v{\rm d}x\Big\}\\
        &\lesssim(1+t)^{-\frac{1}{2}}\sup_t\Bigg\{(1+t)^{\frac{1}{2}}\Big(\|\frac{1}{2^j}\partial_x^\alpha\nabla_x(a^g,b^g,c^g)\mu^\delta\|^2+\eta\|w_{\ell,0}\partial_x^\alpha{\bf P}^\bot g\|^2_{L^2(V_j)L_x^2(\mathbb{R}^3)}\Big)\Bigg\}\\
        &\lesssim\frac{1}{2^{2j}}(1+t)^{-\frac{1}{2}}\sup_t\Big\{(1+t)^{\frac{1}{2}}|(a^g,b^g,c^g)|^2_{\dot{H}^{|\alpha|+1}}\Big\}\\
        &\qquad+\eta(1+t)^{-\frac{1}{2}}\sup_t\Big\{(1+t)^{\frac{1}{2}}\|w_{\ell,0}\partial_x^\alpha{\bf P}^\bot g\|^2_{L^2(V_j)L_x^2(\mathbb{R}^3)}\Big\}\,.
    \end{align*}
Next, we denote $m_3(t,x):=\int_{\mathbb{R}_v^3}v\cdot\nabla_x\partial_x^\alpha g\,\phi_i\,{\rm d}v$, for $O_{20}$, we have
    \begin{align*}
        O_{20}&\lesssim(1+t)^{-\frac{1}{2}}\sup_t\Big\{(1+t)^{\frac{1}{2}}\int_{L^2(V_j)L_x^2(\mathbb{R}^3)}\nu^{-1}(v)\partial_x^\alpha({\bf P}(v\cdot\nabla_x g))w^2_{\ell,0}\partial_x^\alpha{\bf P}^\bot g\,{\rm d}v{\rm d}x\Big\}\\
        &\lesssim(1+t)^{-\frac{1}{2}}\sup_t\Big\{(1+t)^{\frac{1}{2}}\int_{L^2(V_j)L_x^2(\mathbb{R}^3)}m_3(t,x)|v|^{-1}\mu^\delta w_{\ell,0}\partial_x^\alpha{\bf P}^\bot g\,{\rm d}v{\rm d}x\Big\}\\
        &\lesssim(1+t)^{-\frac{1}{2}}\sup_t\Big\{(1+t)^{\frac{1}{2}}\frac{1}{2^{2j}}\|m_3(t,x)\mu^\delta\|^2\Big\}\\
        &\qquad+\eta(1+t)^{-\frac{1}{2}}\sup_t\Big\{(1+t)^{\frac{1}{2}}\|w_{\ell,0}\partial_x^\alpha{\bf P}^\bot g\|^2_{L^2(V_j)L_x^2}\Big\}\\
        &\lesssim\frac{1}{2^{2j}}(1+t)^{-\frac{1}{2}}\sup_t\Big\{(1+t)^{\frac{1}{2}}\|g\|^2_{L_v^2(\dot{H}^{|\alpha|+1})}\Big\}\\
        &\qquad+\eta(1+t)^{-\frac{1}{2}}\sup_t\Big\{(1+t)^{\frac{1}{2}}\|w_{\ell,0}\partial_x^\alpha{\bf P}^\bot g\|^2_{L^2(V_j)L_x^2}\Big\}\,.
    \end{align*}
    As for $O_{21}$, we can obtain
    \begin{align*}
        O_{21}&\lesssim(1+t)^{-\frac{1}{2}}\sup_t\Big\{(1+t)^{\frac{1}{2}}\int_{V_j\times\mathbb{R}_x^3}\nu^{-1}(v)w^2_{\ell,0}\partial_x^\alpha(\mathcal{K}{\bf P}^\bot g)\partial_x^\alpha{\bf P}^\bot g\,{\rm d}v{\rm d}x\Big\}\,.
    \end{align*}
Applying Lemma \ref{lemma5.2}, one has
    \begin{align*}
        &\hspace{3mm}\int_{V_j\times\mathbb{R}_x^3}\nu^{-1}(v)w^2_{\ell,0}\Gamma(\partial_x^{\alpha_1}f_1,\partial_x^{\alpha_2}g)\partial_x^\alpha{\bf P}^\bot g\,{\rm d}v{\rm d}x\\
        &\lesssim\int_{\mathbb{R}_x^3}\Big(|\mu^\delta\partial_x^{\alpha_1}f_1|_{H_v^2(\mathbb{R}^3)}+|w_{\ell,0}\partial_x^{\alpha_1}f_1|_{L_v^2(\mathbb{R}^3)}\Big)\times|w_{\ell,0}\partial_x^{\alpha_2}{\bf P}g|_{L_v^2(\widetilde{V}_j)}|w_{\ell,0}\partial_x^\alpha{\bf P}^\bot g|_{L_v^2(V_j)}\,{\rm d}x\\
        &\quad+\int_{\mathbb{R}_x^3}\Big(|\mu^\delta\partial_x^{\alpha_1}f_1|_{H_v^2(\mathbb{R}^3)}+|w_{\ell,0}\partial_x^{\alpha_1}f_1|_{L_v^2(\mathbb{R}^3)}\Big)\times|w_{\ell,0}\partial_x^{\alpha_2}{\bf P}^\bot g|_{L_v^2(\widetilde{V}_j)}|w_{\ell,0}\partial_x^\alpha{\bf P}^\bot g|_{L_v^2(V_j)}\,{\rm d}x\,.
    \end{align*}
    If $|\alpha_2|=0$, then $1\le|\alpha_1|\le N-3$. We apply $L^2-L^\infty-L^2$ estimate to obtain
    \begin{align*}
        &\hspace{3mm}\int_{V_j\times\mathbb{R}_x^3}\nu^{-1}(v)w^2_{\ell,0}\Gamma(\partial_x^{\alpha}f_1,g)\partial_x^\alpha{\bf P}^\bot g\,{\rm d}v{\rm d}x\\
        &\lesssim\Big(\|\mu^\delta\partial_x^{\alpha}f_1\|_{L_x^2H_v^2}+\|w_{\ell,0}\partial_x^{\alpha} f_1\|\Big)\times\|w_{\ell,0}{\bf P}g\|_{L_x^\infty L^2(\widetilde{V}_j)}\|w_{\ell,0}\partial_x^\alpha {\bf P}^\bot g\|_{L_x^2L_v^2(V_j)}\\
        &\quad+\Big(\|\mu^\delta\partial_x^{\alpha} f_1\|_{L_x^2H_v^2}+\|w_{\ell,0}\partial_x^{\alpha} f_1\|\Big)\times\|w_{\ell,0}{\bf P}^\bot g\|_{L_x^\infty L^2(\widetilde{V}_j)}\|w_{\ell,0}\partial_x^\alpha {\bf P}^\bot g\|_{L_x^2L_v^2(V_j)}\\
        &\lesssim\mathcal{M}_0\|(a^g,b^g,c^g)|v|^{-1}\mu^\delta\|_{L^2(\widetilde{V}_j)(\dot{B}_{2,\infty}^1\cap\dot{H}_x^2)}\|w_{\ell,0}\partial_x^\alpha{\bf P}^\bot g\|_{L_v^2(V_j)L_x^2}\\
        &\quad+\mathcal{M}_0\|w_{\ell,0}{\bf P}^\bot g\|_{L^2(\widetilde{V}_j)(\dot{H}_x^1\cap\dot{H}_x^2)}\|w_{\ell,0}\partial_x^\alpha{\bf P}^\bot g\|_{L_v^2(V_j)L_x^2}\\
        &\lesssim\frac{1}{2^{2j}}\mathcal{M}_0^2|(a^g,b^g,c^g)|^2_{\dot{B}_{2,\infty}^1\cap\dot{H}^2}+\eta\|w_{\ell,0}\partial_x^\alpha{\bf P}^\bot g\|_{L_v^2(V_j)L_x^2}^2+\mathcal{M}_0^2\sum_{1\le|\alpha'|\le 2}\|w_{\ell,0}\partial_x^{\alpha'}{\bf P}^\bot g\|^2_{L_v^2(\widetilde{V}_j)L_x^2}\,.
    \end{align*}
    If $|\alpha_2|=1$, then $0\le|\alpha_1|\le N-4$. We apply $L^4-L^4-L^2$ estimate to obtain
    \begin{align*}
        &\hspace{3mm}\int_{V_j\times\mathbb{R}_x^3}\nu^{-1}(v)w^2_{\ell,0}\Gamma(\partial_x^{\alpha_1}f_1,\partial_x^{\alpha_2}g)\partial_x^\alpha{\bf P}^\bot g\,{\rm d}v{\rm d}x\\
        &\lesssim\Big(\|\mu^\delta\partial_x^{\alpha_1}f_1\|_{L_x^4H_v^2}+\|w_{\ell,0}\partial_x^{\alpha_1}f_1\|_{L_x^4L_v^2}\Big)\times\|w_{\ell,0}\partial_x^{\alpha_2}{\bf P}g\|_{L_x^4(\mathbb{R}^3)L_v^2(\widetilde{V}_j)}\|w_{\ell ,0}\partial_x^\alpha{\bf P}^\bot g\|_{L_x^2(\mathbb{R}^3)L_v^2(V_j)}\\
        &\quad+\Big(\|\mu^\delta\partial_x^{\alpha_1}f_1\|_{L_x^4H_v^2}+\|w_{\ell,0}\partial_x^{\alpha_1}f_1\|_{L_x^4L_v^2}\Big)\times\|w_{\ell,0}\partial_x^{\alpha_2}{\bf P}^\bot g\|_{L_x^4(\mathbb{R}^3)L_v^2(\widetilde{V}_j)}\|w_{\ell ,0}\partial_x^\alpha{\bf P}^\bot g\|_{L_x^2(\mathbb{R}^3)L_v^2(V_j)}\\
        &\lesssim|\partial_x^{\alpha_1}(a^{f_1},b^{f_1},c^{f_1})|_{\dot{B}^\frac{1}{2}_{2,\infty}\cap\dot{H}^1}\|\partial_x^{\alpha_2}(a^g,b^g,c^g)|v|^{-1}\mu^\delta\|_{L_v^2(\widetilde{V}_j)(\dot{B}^\frac{1}{2}_{2,\infty}\cap\dot{H}^1)}\|w_{\ell ,0}\partial_x^\alpha{\bf P}^\bot g\|_{L_v^2(V_j)L_x^2(\mathbb{R}^3)}\\
        &\quad+\Big(\|\partial_x^{\alpha_1}{\bf P}^\bot f_1\|_{H_v^2(L_x^2\cap\dot{H}^1_x)}+\|w_{\ell,0}\partial_x^{\alpha_1}{\bf P}^\bot f_1\|_{L_v^2(L_x^2\cap\dot{H}^1_x)}\Big)\\[2mm]
        &\qquad\times\|\partial_x^{\alpha_2}(a^g,b^g,c^g)|v|^{-1}\mu^\delta\|_{L_v^2(\widetilde{V}_j)(\dot{B}^\frac{1}{2}_{2,\infty}\cap\dot{H}^1)}\|w_{\ell ,0}\partial_x^\alpha{\bf P}^\bot g\|_{L_v^2(V_j)L_x^2(\mathbb{R}^3)}\\
        &\quad+|\partial_x^{\alpha_1}(a^{f_1},b^{f_1},c^{f_1})|_{\dot{B}^\frac{1}{2}_{2,\infty}\cap\dot{H}^1}\|w_{\ell,0}\partial_x^{\alpha_2}{\bf P}^\bot g\|_{L_v^2(\widetilde{V}_j)(L_x^2\cap\dot{H}^1)}\|w_{\ell ,0}\partial_x^\alpha{\bf P}^\bot g\|_{L_v^2(V_j)L_x^2(\mathbb{R}^3)}\\
        &\quad+\Big(\|\partial_x^{\alpha_1}{\bf P}^\bot f_1\|_{H_v^2(L_x^2\cap\dot{H}^1_x)}+\|w_{\ell,0}\partial_x^{\alpha_1}{\bf P}^\bot f_1\|_{L_v^2(L_x^2\cap\dot{H}^1_x)}\Big)\\
        &\qquad\times\|w_{\ell,0}\partial_x^{\alpha_2}{\bf P}^\bot g\|_{L_v^2(\widetilde{V}_j)(L_x^2\cap\dot{H}^1)}\|w_{\ell ,0}\partial_x^\alpha{\bf P}^\bot g\|_{L_v^2(V_j)L_x^2(\mathbb{R}^3)}\\
        &\lesssim\frac{1}{2^j}\mathcal{M}_0|(a^g,b^g,c^g)|_{\dot{B}_{2,\infty}^1\cap\dot{H}^2}\|w_{\ell ,0}\partial_x^\alpha{\bf P}^\bot g\|_{L_v^2(V_j)L_x^2(\mathbb{R}^3)}\\
        &\quad+\mathcal{M}_0\|w_{\ell,0}\partial_x{\bf P}^\bot g\|_{L_v^2(V_j)(L_x^2\cap\dot{H}_x^1)}\|w_{\ell ,0}\partial_x^\alpha{\bf P}^\bot g\|_{L_v^2(V_j)L_x^2(\mathbb{R}^3)}\\
        &\lesssim\frac{1}{2^{2j}}\mathcal{M}_0^2|(a^g,b^g,c^g)|^2_{\dot{B}^1_{2,\infty}\cap\dot{H}^2}+\eta\|w_{\ell ,0}\partial_x^\alpha{\bf P}^\bot g\|_{L_v^2(V_j)L_x^2(\mathbb{R}^3)}^2+\mathcal{M}_0^2\sum_{1\le|\alpha'|\le 2}\|w_{\ell,0}\partial_x^{\alpha'}{\bf P}^\bot g\|_{L_v^2(\widetilde{V}_j)L_x^2}\,.
    \end{align*}
    If $|\alpha_2|\ge2$, we have
    \begin{align*}
        &\hspace{3mm}\int_{V_j\times\mathbb{R}_x^3}\nu^{-1}(v)w^2_{\ell,0}\Gamma(\partial_x^{\alpha_1}f_1,\partial_x^{\alpha_2}g)\partial_x^\alpha{\bf P}^\bot g\,{\rm d}v{\rm d}x\\
        &\lesssim\Big(\|\mu^\delta\partial_x^{\alpha_1}f_1\|_{L_x^\infty H_v^2}+\|w_{\ell,0}\partial_x^{\alpha_1} f_1\|_{L_x^\infty L_v^2}\Big)\times\|w_{\ell,0}\partial_x^{\alpha_2}{\bf P}g\|_{L_x^2 L^2(\widetilde{V}_j)}\|w_{\ell,0}\partial_x^\alpha {\bf P}^\bot g\|_{L_x^2L_v^2(V_j)}\\
        &\quad+\Big(\|\mu^\delta\partial_x^{\alpha_1} f_1\|_{L_x^\infty H_v^2}+\|w_{\ell,0}\partial_x^{\alpha_1} f_1\|_{L_x^\infty L_v^2}\Big)\times\|w_{\ell,0}\partial_x^{\alpha_2}{\bf P}^\bot g\|_{L_x^2L^2(\widetilde{V}_j)}\|w_{\ell,0}\partial_x^\alpha {\bf P}^\bot g\|_{L_x^2L_v^2(V_j)}\\
        &\lesssim\|w_{\ell,0}\partial_x^{\alpha_1}f_1\|_{H_v^2(\dot{H}^1\cap\dot{H}^2)}\|\partial_x^{\alpha_2}(a^g,b^g,c^g)|v|^{-1}\mu^\delta |_{L^2(\widetilde{V}_j)L_x^2}\|w_{\ell,0}\partial_x^\alpha{\bf P}^\bot g\|_{L^2(V_j)L_x^2}\\
        &\quad+\|w_{\ell,0}\partial_x^{\alpha_1}f_1\|_{H_v^2(\dot{H}^1\cap\dot{H}^2)}\|w_{\ell,0}\partial_x^{\alpha_2}{\bf P}^\bot g\|_{L^2(V_j)L_x^2}\|w_{\ell,0}\partial_x^\alpha{\bf P}^\bot g\|_{L^2(V_j)L_x^2}\\
        &\lesssim\mathcal{M}_0^2\frac{1}{2^{2j}}|(a^g,b^g,c^g)|^2_{\dot{H}^{|\alpha_2|}}+\eta\|w_{\ell,0}\partial_x^\alpha{\bf P}^\bot g\|_{L^2(V_j)L_x^2}+\mathcal{M}_0^2\|w_{\ell,0}\partial_x^{\alpha_2}{\bf P}^\bot g\|_{L^2(V_j)L_x^2}\,.
    \end{align*}
    Thus, we have
    \begin{align*}
        O_{22}&\lesssim(1+t)^{-\frac{1}{2}}\sup_t\Big\{(1+t)^{\frac{1}{2}}\sum_{\alpha_1+\alpha_2=\alpha}\int_{V_j\times\mathbb{R}_x^3}\nu^{-1}(v)w^2_{\ell,0}\Gamma(\partial_x^{\alpha_1}f_1,\partial_x^{\alpha_2}g)\partial_x^\alpha{\bf P}^\bot g\Big\}\\
        &\lesssim\eta'(1+t)^{-\frac{1}{2}}\sum_{1\le|\alpha'|\le|\alpha|}\sup_t\Big\{(1+t)^{\frac{1}{2}}\|w_{\ell,0}\partial_x^{\alpha'}{\bf P}^\bot g\|^2_{L^2(\widetilde{V}_j)L_x^2(\mathbb{R}^3)}\Big\}\\
        &\quad+\frac{1}{2^{2j}}\mathcal{M}_0^2(1+t)^{-\frac{1}{2}}\sup_t\Big\{(1+t)^{\frac{1}{2}}|(a^g,b^g,c^g)|^2_{\dot{B}_{2,\infty}^1\cap\dot{H}^{N-2}}\Big\}\,.
    \end{align*}
    The estimate of $O_{23}$ is completely the same as $O_{22}$, so we omit it.

    \noindent Plugging all the estimates into \eqref{pp9}, then multiplying the result by $(1+t)^\frac{1}{2}$, taking supremum over $t$ and summing over $j\in\mathbb{Z}^+$ yields that
    \begin{align*}
        &\hspace{3mm}\sum_{j\in\mathbb{Z}^+}\sum_{1\le|\alpha|\le N-3}\sup_t\Big\{(1+t)^\frac{1}{2}\|w_{\ell,0}\partial_x^\alpha{\bf P}^\bot g\|^2_{L^2(V_j)L_x^2(\mathbb{R}^3)}\Big\}\\
        &\lesssim\sum_{1\le|\alpha|\le N-3}\|w_{\ell-\frac{\gamma}{4},0}\partial_x^\alpha{\bf P}^\bot g_0\|^2+\sup_t\Big\{(1+t)^\frac{1}{2}|(a^g,b^g,c^g)|^2_{\dot{B}_{2,\infty}^1}\Big\}\\
        &\quad+\sup_t\Big\{(1+t)^\frac{1}{2}\|g\|^2_{L_v^2(\dot{H}^2\cap\dot{H}^{N-2})}\Big\}+\sup_t\Big\{(1+t)^\frac{1}{2}\|{\bf P}^\bot g\|^2_{L_v^2(\dot{H}^1)}\Big\}\,.
    \end{align*}
    {\bf Case II:}  We consider  the original equation \eqref{X1}. Similar to ${\bf Case\,\,I}$, one gets
    \begin{align}\label{pp10}
        \|w_{\ell,0}\partial_x^\alpha g\|_{L_v^2(V_j)L_x^2(\mathbb{R}^3)}^2&\lesssim \int_{V_j\times\mathbb{R}_x^3}e^{-\nu(v)t}|w_{\ell,0}\partial_x^\alpha g_0|^2\,\mathrm{d}v\mathrm{d}x\notag\\
    &\qquad+\underbrace{\int_0^t\int_{V_j\times\mathbb{R}_x^3} e^{-\nu(v)(t-\tau)}w_{\ell,0}^2\partial_x^\alpha \mathcal{K}g\partial_x^\alpha  g\,\mathrm{d}v\mathrm{d}x\mathrm{d}\tau}_{O_{24}}\notag\\
    &\qquad+\underbrace{\int_0^t\int_{V_j\times\mathbb{R}_x^3} e^{-\nu(v)(t-\tau)}w_{\ell,0}^2\partial_x^\alpha \Gamma(g,f_1)\partial_x^\alpha  g\,\mathrm{d}v\mathrm{d}x\mathrm{d}\tau}_{O_{25}}\notag\\
    &\qquad+\underbrace{\int_0^t\int_{V_j\times\mathbb{R}_x^3} e^{-\nu(v)(t-\tau)}w_{\ell,0}^2\partial_x^\alpha \Gamma(f_2,g)\partial_x^\alpha  g\,\mathrm{d}v\mathrm{d}x\mathrm{d}\tau}_{O_{26}}\,.
    \end{align}
    It follows from Lemma \ref{lem:max_estimate} that
    \begin{align*}
        \hspace{3mm}\int_{V_j\times\mathbb{R}_x^3}e^{-\nu(v)t}|w_{\ell,0}\partial_x^\alpha g_0|^2\,\mathrm{d}v\mathrm{d}x\lesssim&\int_{V_j\times\mathbb{R}_x^3}e^{-\nu(v)t}\nu^{\frac{1}{2}}(v)|w_{\ell-\frac{\gamma}{4},0}\partial_x^\alpha g_0|^2\,\mathrm{d}v\mathrm{d}x\\
        \lesssim&(1+t)^{-\frac{1}{2}}\|w_{\ell-\frac{\gamma}{4},0}\partial_x^\alpha g_0\|^2_{L_v^2(V_j)L_x^2(\mathbb{R}^3)}\,.
    \end{align*}
    For $O_{24}$, it follows that
    \begin{align*}
        O_{24}&\lesssim(1+t)^{-\frac{1}{2}}\sup_t\Big\{(1+t)^\frac{1}{2}\int_{V_j\times\mathbb{R}_x^3}\nu^{-1}(v)w_{\ell,0}^2\partial_x^\alpha\mathcal{K}g\partial_x^\alpha g\,{\rm d}v{\rm d}x\Big\}\,.
    \end{align*}
    Using the same method in the estimate of $O_{22}$, we have
    \begin{align*}
      O_{25}&\lesssim(1+t)^{-\frac{1}{2}}\sup_t\Big\{(1+t)^\frac{1}{2}\int_{V_j\times\mathbb{R}_x^3}\nu^{-1}(v)w_{\ell,0}^2\Gamma(\partial_x^{\alpha_1}f_1,\partial_x^{\alpha_2}g)\partial_x^\alpha g\Big\}\\
      &\lesssim\eta'(1+t)^{-\frac{1}{2}}\sum_{1\le|\alpha'|\le|\alpha|}\sup_t\Big\{(1+t)^{\frac{1}{2}}\|w_{\ell,0}\partial_x^{\alpha'}{\bf P}^\bot g\|^2_{L^2(V_j)L_x^2(\mathbb{R}^3)}\Big\}\\
        &\quad+\frac{1}{2^{2j}}\mathcal{M}_0^2(1+t)^{-\frac{1}{2}}\sup_t\Big\{(1+t)^{\frac{1}{2}}|(a^g,b^g,c^g)|^2_{\dot{B}_{2,\infty}^1\cap\dot{H}^{|\alpha|}}\Big\}\\
        &\quad+\eta(1+t)^{-\frac{1}{2}}\sup_t\Big\{(1+t)^{\frac{1}{2}}\|w_{\ell,0}\partial_x^{\alpha} g\|^2_{L^2(V_j)L_x^2(\mathbb{R}^3)}\Big\}\,.
    \end{align*}
    The estimate of $O_{26}$ is similar, we omit it.

    \noindent Plugging all the estimates into \eqref{pp10}, multiplying the resulting inequality by $(1+t)^\frac{1}{2}$, and taking the supremum over $t$. We then combine all cases for $1\le|\alpha|\le N-2$ and sum over $j\in\mathbb{Z}^+$ to obtain the desired result. This ends the proof of Lemma \ref{lemma8.5}.
\end{proof}
Now we can further prove the following lemma:
\begin{lemma}\label{lemma6.5}
    Under the assumption of Lemma \ref{prop6.2}, we have
    \begin{align*}
        &\hspace{3mm}\sup_t(1+t)^\frac{1}{2}\left\|w_{\ell,0}{\bf P}^\bot g\right\|_{L_v^2(\dot{H}^1\cap\dot{H}^{N-2})}^2\\
        &\lesssim\left\|w_{\ell-\frac{\gamma}{4},0}{\bf P}^\bot g_0\right\|_{L_v^2(\dot{H}^1\cap\dot{H}^{N-2})}^2+\sup_t\Big\{(1+t)^\frac{1}{2}\left\|g\right\|_{L_v^2(\dot{H}^2\cap\dot{H}^{N-2})}^2\Big\}\\
        &\quad+\sup_t\Big\{(1+t)^\frac{1}{2}\left\|{\bf P}^\bot g\right\|_{L_v^2(\dot{H}^1)}^2\Big\}+\sup_t\Big\{(1+t)^\frac{1}{2}|(a^g,b^g,c^g)|^2_{\dot{B}_{2,\infty}^1}\Big\}\,.
    \end{align*}
\end{lemma}
\begin{proof}
As in the proof of Lemma \ref{lemma8.5}, we divide the argument into two steps.

{\bf Step I:} $1\le|\alpha|\le N-3$. We consider the microscopic equation.
    Applying $\mathbf{P}^{\bot}$ to \eqref{X1}, one has
    \begin{align}\label{OO1}
        &\hspace{3mm}\partial_t \mathbf{P}^{\bot}g+v\cdot\nabla_x\mathbf{P}^{\bot}g+v\cdot\nabla_x\mathbf{P}g+\frac{1}{2}\nu(v){\bf P}^{\bot}g\notag\\
        &=-\frac{1}{2}\nu(v){\bf P}^\bot g+{\bf P}(v\cdot\nabla_x g)+\mathcal{K}{\bf P}^{\bot}g+{\bf P}^{\bot}\Gamma(g,f_1)+{\bf P}^{\bot}\Gamma(f_2,g)\,,
    \end{align}
    where we decompose
    \begin{align*}
        {\bf P}^{\bot}(v\cdot\nabla_x g)=v\cdot\nabla_x g-{\bf P}(v\cdot \nabla_x g)=v\cdot \nabla_x{\bf P}^{\bot}g+v\cdot\nabla_x{\bf P}g-{\bf P}(v\cdot\nabla_x g) \,,  
    \end{align*}
    and 
    \begin{align*}
\mathcal{L}{\bf P}^{\bot}g=\nu(v){\bf P}^{\bot}g-\mathcal{K}{\bf P}^{\bot}g\,.
    \end{align*}
    Acting the operator $\partial_x^{\alpha}$ on \eqref{OO1} with $1\le|\alpha|\le N-3$ and then multiplying the result by $w_{\ell,0}^2\partial_x^{\alpha}{\bf P}^{\bot}g$, it yields that
    \begin{align}
        &\hspace{3mm}\frac{1}{2}\frac{\mathrm{d}}{\mathrm{d}t}\left(w_{\ell,0}\partial_x^{\alpha}{\bf P}^{\bot}g\right)^2+\frac{1}{2}\nu(v)\left(w_{\ell,0}\partial_x^{\alpha}{\bf P}^{\bot}g\right)^2\notag\\
        &=-\frac{1}{2}\nu(v)\left(w_{\ell,0}\partial_x^{\alpha}{\bf P}^{\bot}g\right)^2-\partial_x^{\alpha}(v\cdot\nabla_x\mathbf{P}^{\bot}g)w_{\ell,0}^2\partial_x^{\alpha}{\bf P}^{\bot}g-\partial_x^{\alpha}(v\cdot\nabla_x\mathbf{P}g)w_{\ell,0}^2\partial_x^{\alpha}{\bf P}^{\bot}g\notag\\
        &\quad+\partial_x^{\alpha}({\bf P}(v\cdot\nabla_x g))w_{\ell,0}^2\partial_x^{\alpha}{\bf P}^{\bot}g+\partial_x^{\alpha}(\mathcal{K}{\bf P}^{\bot}g)w_{\ell,0}^2\partial_x^{\alpha}{\bf P}^{\bot}g+\partial_x^{\alpha}({\bf P}^{\bot}\Gamma(g,f_1))w_{\ell,0}^2\partial_x^{\alpha}{\bf P}^{\bot}g\notag\\
        &\quad+\partial_x^{\alpha}({\bf P}^{\bot}\Gamma(f_2,g))w_{\ell,0}^2\partial_x^{\alpha}{\bf P}^{\bot}g\,.\notag
    \end{align}
 Solving the above equation directly and then integrating the result over $\mathbb{R}_x^3\times\mathbb{R}_v^3$, we get
    \begin{align}\label{wp}
        \left\|w_{\ell,0}\partial_x^{\alpha}{\bf P}^{\bot}g\right\|^2&\lesssim \int_{\mathbb{R}_{x,v}^6} e^{-\nu(v)t}\left|w_{\ell,0}\partial_x^{\alpha}{\bf P}^{\bot}g_0\right|^2\,\mathrm{d}v\mathrm{d}x-\frac{1}{2}\int_0^t\int_{\mathbb{R}_{x,v}^6} e^{-\nu(v)(t-\tau)} \nu(v)\left|w_{\ell,0}\partial_x^{\alpha}{\bf P}^{\bot}g\right|^2\,\mathrm{d}v\mathrm{d}x\mathrm{d}\tau\notag\\
        &\quad-\int_0^t\int_{\mathbb{R}_{x,v}^6} e^{-\nu(v)(t-\tau)} \partial_x^{\alpha}(v\cdot\nabla_x\mathbf{P}g)w_{\ell,0}^2\partial_x^{\alpha}{\bf P}^{\bot}g\,\mathrm{d}v\mathrm{d}x\mathrm{d}\tau\notag\\
        &\quad+\int_0^t\int_{\mathbb{R}_{x,v}^6} e^{-\nu(v)(t-\tau)} \partial_x^{\alpha}({\bf P}(v\cdot\nabla_x g))w_{\ell,0}^2\partial_x^{\alpha}{\bf P}^{\bot}g\,\mathrm{d}v\mathrm{d}x\mathrm{d}\tau\notag\\
        &\quad+\int_0^t\int_{\mathbb{R}_{x,v}^6} e^{-\nu(v)(t-\tau)}\partial_x^{\alpha}(\mathcal{K}{\bf P}^{\bot}g) w_{\ell,0}^2\partial_x^{\alpha}{\bf P}^{\bot}g\,\mathrm{d}v\mathrm{d}x\mathrm{d}\tau\notag\\
        &\quad+\int_0^t\int_{\mathbb{R}_{x,v}^6} e^{-\nu(v)(t-\tau)} \partial_x^{\alpha}({\bf P}^{\bot}\Gamma(g,f_1))w_{\ell,0}^2\partial_x^{\alpha}{\bf P}^{\bot}g\,\mathrm{d}v\mathrm{d}x\mathrm{d}\tau\notag\\
        &\quad+\int_0^t\int_{\mathbb{R}_{x,v}^6} e^{-\nu(v)(t-\tau)} \partial_x^{\alpha}({\bf P}^{\bot}\Gamma(f_2,g))w_{\ell,0}^2\partial_x^{\alpha}{\bf P}^{\bot}g\,\mathrm{d}v\mathrm{d}x\mathrm{d}\tau\notag\\
        &:=O_{27}+O_{28}+O_{29}+O_{30}+O_{31}+O_{32}+O_{33}\,.
    \end{align}
    For $O_{27}$, taking $x^\alpha=\nu(v)^{\frac{1}{2}}$, then it holds from Lemma \ref{lem:max_estimate} that
    \begin{align*}
        O_{27}&\lesssim\int_{\mathbb{R}_{x,v}^6} \nu(v)^{\frac{1}{2}}e^{-\nu(v)t}\left|w_{\ell-\frac{\gamma}{4},0}\partial_x^{\alpha}{\bf P}^{\bot}g_0\right|^2\,\mathrm{d}v\mathrm{d}x\\
        &\lesssim(1+t)^{-\frac{1}{2}}\int_{\mathbb{R}_{x,v}^6}\left|w_{\ell-\frac{\gamma}{4},0}\partial_x^{\alpha}{\bf P}^{\bot}g_0\right|^2\,\mathrm{d}v\mathrm{d}x\\
        &\lesssim(1+t)^{-\frac{1}{2}}\left\|w_{\ell-\frac{\gamma}{4},0}\partial_x^{\alpha}{\bf P}^{\bot}g_0\right\|^2\,.
    \end{align*}
For $O_{29}$, we get
from Lemma \ref{lemma2.5} that
\begin{align*}
    O_{29}&\lesssim\int_0^t\int_{\mathbb{R}_{x,v}^6} e^{-\nu(v)(t-\tau)} \partial_x^{\alpha+e_i}(a,b,c)\mu^\delta w_{\ell,0}^2\partial_x^{\alpha}{\bf P}^{\bot}g\,\mathrm{d}v\mathrm{d}x\mathrm{d}\tau\\
    &\lesssim\int_0^t\int_{\mathbb{R}_{x,v}^6} e^{-\nu(v)(t-\tau)} \left|\partial_x^{\alpha+e_i}(a,b,c)\right|^2\mu^\delta w_{\ell,0}^2\langle v\rangle^{-\gamma} \,\mathrm{d}v\mathrm{d}x\mathrm{d}\tau\\
    &\qquad+\eta\int_0^t\int_{\mathbb{R}_{x,v}^6} e^{-\nu(v)(t-\tau)}\langle v\rangle^{\gamma} \left|w_{\ell,0}\partial_x^{\alpha}{\bf P}^{\bot}g\right|^2\,\mathrm{d}v\mathrm{d}x\mathrm{d}\tau\\
    &\lesssim\int_0^t\int_{\mathbb{R}_{x,v}^6} e^{-\nu(v)(t-\tau)} (1+\tau)^{-\frac{1}{2}}\mu^{\delta'}(1+\tau)^\frac{1}{2} \left|\partial_x^{\alpha+e_i}(a,b,c)\right|^2\,\mathrm{d}v\mathrm{d}x\mathrm{d}\tau+\eta O_{20}\\
    &\lesssim(1+t)^{-\frac{1}{2}}\sup_t(1+t)^{\frac{1}{2}}\|(a,b,c)\|_{\dot{H}^{|\alpha|+1}}^2+\eta O_{20}\,,
\end{align*}
here we insert $w_{\ell,0}^2$ and $\langle v\rangle^{-\gamma} $ into $\mu^{\delta'}$.
For $O_{30}$, similar to \eqref{OO3}, one has
\[\partial_x^\alpha{\bf P}(v\cdot\nabla_x g)\lesssim|\nabla_x\partial_x^\alpha g|_{L_v^2}\,\mu^\delta\,.\]
Then we obtain
\begin{align*}
    O_{30}&\lesssim\int_0^t\int_{\mathbb{R}_{x,v}^6} e^{-\nu(v)(t-\tau)} |\nabla_x\partial_x^\alpha g|_{L_v^2}\,\mu^\delta w_{\ell,0}^2\partial_x^{\alpha}{\bf P}^{\bot}g\,\mathrm{d}v\mathrm{d}x\mathrm{d}\tau\\
    &\lesssim\int_0^t\int_{\mathbb{R}_{x,v}^6} e^{-\nu(v)(t-\tau)} |\nabla_x\partial_x^\alpha g|_{L_v^2}^2\,\mu^\delta w_{\ell,0}^2\langle v\rangle^{-\gamma} \,\mathrm{d}v\mathrm{d}x\mathrm{d}\tau\\
    &\qquad+\int_0^t\int_{\mathbb{R}_{x,v}^6} e^{-\nu(v)(t-\tau)} \langle v\rangle^{\gamma} w_{\ell,0}^2\left|\partial_x^{\alpha}{\bf P}^{\bot}g\right|^2\,\mathrm{d}v\mathrm{d}x\mathrm{d}\tau\\
    &\lesssim\int_0^t\int_{\mathbb{R}_{x,v}^6} e^{-\nu(v)(t-\tau)}(1+\tau)^{-\frac{1}{2}} \mu^{\delta'}(1+\tau)^{\frac{1}{2}}|\nabla_x\partial_x^\alpha g|_{L_v^2}^2\,\,\mathrm{d}v\mathrm{d}x\mathrm{d}\tau+\eta O_{20}\\
    &\lesssim(1+t)^{-\frac{1}{2}}\sup_t(1+t)^{\frac{1}{2}}\|g\|_{L_v^2(\dot{H}^{|\alpha|+1})}^2+\eta O_{20}\,.
    \end{align*}
    For $O_{31}$, we have
    \begin{align*}
        O_{31}&\lesssim\sum_{j\in\mathbb{Z}^+}\int_0^t\int_{V_j\times\mathbb{R}_x^3}e^{-\nu(v)(t-\tau)}\nu(v)\nu^{-1}(v)w^2_{\ell,0}\partial_x^\alpha\mathcal{K}{\bf P}^\bot g\partial_x^\alpha{\bf P}^\bot g\,{\rm d}v{\rm d}x{\rm d}\tau\\
        &\quad+\int_0^t\int_{V_0\times\mathbb{R}_x^3}e^{-\nu(v)(t-\tau)}w^2_{\ell,0}\partial_x^\alpha\mathcal{K}{\bf P}^\bot g\partial_x^\alpha{\bf P}^\bot g\,{\rm d}v{\rm d}x{\rm d}\tau\\
        &:=O_{31,1}+O_{31,2}\,.
    \end{align*}
    It follows from Lemma \ref{lemma5.1} and Lemma \ref{lemma8.5} that
    \begin{align*}
        O_{31,1}&\lesssim\sum_{j\in\mathbb{Z}^+}\int_0^t\int_{V_j\times\mathbb{R}_x^3}e^{-\tilde{c}_j(t-\tau)}c_j(1+\tau)^{-\frac{1}{2}}(1+\tau)^{\frac{1}{2}}\nu^{-1}(v)w^2_{\ell,0}\partial_x^\alpha\mathcal{K}{\bf P}^\bot g\partial_x^\alpha{\bf P}^\bot g\,{\rm d}v{\rm d}x{\rm d}\tau\\
        &\lesssim\sum_{j\in\mathbb{Z}^+}\sup_t\Big\{(1+t)^{\frac{1}{2}}\int_{V_j\times\mathbb{R}_x^3}\nu^{-1}(v)w^2_{\ell,0}\partial_x^\alpha\mathcal{K}{\bf P}^\bot g\partial_x^\alpha{\bf P}^\bot g\,{\rm d}v{\rm d}x\Big\}\sup_j\int_0^te^{-\tilde{c}_j(t-\tau)}c_j(1+\tau)^{-\frac{1}{2}}\,{\rm d}\tau\\
        &\lesssim(1+t)^{-\frac{1}{2}}\sum_{j\in\mathbb{Z}^+}\sup_t\Big\{(1+t)^{\frac{1}{2}}\int_{V_j\times\mathbb{R}_x^3}\nu^{-1}(v)w^2_{\ell,0}\partial_x^\alpha\mathcal{K}{\bf P}^\bot g\partial_x^\alpha{\bf P}^\bot g\,{\rm d}v{\rm d}x\Big\}\\
        &\lesssim\eta(1+t)^{-\frac{1}{2}}\sum_{j\in\mathbb{Z}^+}\sup_t\Big\{(1+t)^{\frac{1}{2}}\|w_{\ell,0}\partial_x^\alpha{\bf P}^\bot g\|^2_{L^2_v(V_j)L_x^2(\mathbb{R}^3)}\Big\}\\
        &\quad+(1+t)^{-\frac{1}{2}}\sup_t\Big\{(1+t)^{\frac{1}{2}}\|\partial_x^\alpha {\bf P}^\bot g\|^2\Big\}+\eta'(1+t)^{-\frac{1}{2}}\sup_t\Big\{(1+t)^\frac{1}{2}\|w_{\ell,0}\partial_x^\alpha{\bf P}^\bot g\|^2\Big\}\,.
    \end{align*}
As for $O_{31,2}$, since $|v|$ is bounded in the domain $V_0$, it follows from Lemma \ref{lemma9.3} and \eqref{S1} that
\begin{align*}
O_{31,2}&\lesssim\int_0^t\int_{V_0\times\mathbb{R}_x^3}e^{-(t-\tau)}(1+\tau)^{-\frac{1}{2}}(1+\tau)^\frac{1}{2}w^2_{\ell,0}\partial_x^\alpha\mathcal{K}{\bf P}^\bot g\partial_x^\alpha{\bf P}^\bot g\,{\rm d}v{\rm d}x{\rm d}\tau\\
&\lesssim(1+t)^{-\frac{1}{2}}\sup_t\Big\{(1+t)^{\frac{1}{2}}\int_{\mathbb{R}_{x,v}^6}w^2_{\ell,0}\partial_x^\alpha\mathcal{K}{\bf P}^\bot g\partial_x^\alpha{\bf P}^\bot g\,{\rm d}v{\rm d}x\Big\}\\
&\lesssim\eta(1+t)^{-\frac{1}{2}}\sup_t\Big\{(1+t)^{\frac{1}{2}}\|w_{\ell,0}\partial_x^\alpha{\bf P}^\bot g\|^2\Big\}+(1+t)^{-\frac{1}{2}}\sup_t\Big\{(1+t)^{\frac{1}{2}}\|\bar{\chi}(v)\partial_x^\alpha{\bf P}^\bot g\|^2\Big\}\\
&\lesssim\eta(1+t)^{-\frac{1}{2}}\sup_t\Big\{(1+t)^{\frac{1}{2}}\|w_{\ell,0}\partial_x^\alpha{\bf P}^\bot g\|^2\Big\}+(1+t)^{-\frac{1}{2}}\sup_t\Big\{(1+t)^{\frac{1}{2}}\|\partial_x^\alpha{\bf P}^\bot g\|^2\Big\}\,.
\end{align*}
We also obtain for $O_{32}$ that
\begin{align*}
    O_{32}&\lesssim(1+t)^{-\frac{1}{2}}\sum_{j\in\mathbb{Z}^+}\sum_{\alpha_1+\alpha_2=\alpha}\sup_t\Big\{(1+t)^{\frac{1}{2}}\int_{V_j\times\mathbb{R}_x^3}\nu^{-1}(v)w^2_{\ell,0}\Gamma(\partial_x^{\alpha_1}f_1,\partial_x^{\alpha_2}g)\partial_x^\alpha{\bf P}^\bot g\,{\rm d}v{\rm d}x\Big\}\\
    &\quad+(1+t)^{-\frac{1}{2}}\sum_{\alpha_1+\alpha_2=\alpha}\sup_t\Big\{(1+t)^{\frac{1}{2}}\int_{\mathbb{R}_{x,v}^6}\Gamma(\partial_x^{\alpha_1}f_1,\partial_x^{\alpha_2}g)\partial_x^\alpha{\bf P}^\bot g\,{\rm d}v{\rm d}x\Big\}\\
    &:=O_{32,1}+O_{32,2}\,.
\end{align*}
We further get from Lemma \ref{lemma8.5} that
\begin{align*}
    O_{32,1}&\lesssim(1+t)^{-\frac{1}{2}}\sup_t\Big\{(1+t)^{\frac{1}{2}}|a^g,b^g,c^g|^2_{\dot{B}^1_{2,\infty}\cap\dot{H}^{N-2}}\Big\}\\
    &\quad+\eta(1+t)^{-\frac{1}{2}}\sum_{j\in\mathbb{Z}^+}\sum_{1\le|\alpha'|\le N-2}\sup_t\Big\{(1+t)^{\frac{1}{2}}\|w_{\ell,0}\partial_x^{\alpha'}{\bf P}^\bot g\|^2_{L^2_v(V_j)L_x^2(\mathbb{R}^3)}\Big\}\,.
\end{align*}
It follows from Lemma \ref{lemma9.4} that, for $|\alpha_2|=0$, we have
    \begin{align*}
        &\hspace{3mm}\int_{\mathbb{R}_{x,v}^6}\Gamma(\partial_x^{\alpha}f_1,g)\partial_x^\alpha{\bf P}^\bot g\,{\rm d}v{\rm d}x\\
        &\lesssim\|\partial_x^{\alpha}f_1\|_{L_x^2H_v^2}\|{\bf P}g\|_{L_x^\infty L_v^2}\|\partial_x^\alpha{\bf P}^\bot g\|_{L_x^2L_v^2}+\|\partial_x^{\alpha}f_1\|_{L_x^2H_v^2}\|{\bf P}^\bot g\|_{L_x^\infty L_v^2}\|\partial_x^\alpha{\bf P}^\bot g\|_{L_x^2L_v^2}\\
        &\lesssim\mathcal{M}_0^2|(a^g,b^g,c^g)|^2_{\dot{B}^1_{2,\infty}\cap\dot{H}^2}+\eta\|\partial_x^\alpha{\bf P}^\bot g\|^2+\mathcal{M}_0^2\|{\bf P}^\bot g\|^2_{L_v^2(\dot{H}^1\cap\dot{H}^2)}\,.
    \end{align*}
    If $|\alpha_2|=1$, then $0\le|\alpha_1|\le N-4$. We apply $L^4-L^4-L^2$ estimate to obtain
    \begin{align*}
        &\hspace{3mm}\int_{\mathbb{R}_{x,v}^6}\Gamma(\partial_x^{\alpha_1}f_1,\partial_x^{\alpha_2}g)\partial_x^\alpha{\bf P}^\bot g\,{\rm d}v{\rm d}x\\
        &\lesssim\|\partial_x^{\alpha_1}f_1\|_{L_x^4H_v^2}\|\partial_x^{\alpha_2}{\bf P}g\|_{L_x^4L_v^2}\|\partial_x^\alpha{\bf P}^\bot g\|_{L_x^2L_v^2}+\|\partial_x^{\alpha_1}f_1\|_{L_x^4H_v^2}\|\partial_x^{\alpha_2}{\bf P}^\bot g\|_{L_x^4L_v^2}\|\partial_x^\alpha{\bf P}^\bot g\|_{L_x^2L_v^2}\\[2mm]
        &\lesssim|\partial_x^{\alpha_1}(a^{f_1},b^{f_1},c^{f_1})|_{\dot{B}^\frac{1}{2}_{2,\infty}\cap\dot{H}^1}|\partial_x^{\alpha_2}(a^g,b^g,c^g)|_{\dot{B}^\frac{1}{2}_{2,\infty}\cap\dot{H}^1}\|\partial_x^\alpha{\bf P}^\bot g\|\\
        &\quad+\|\partial_x^{\alpha_1}{\bf P}^\bot f_1\|_{H_v^2(L_x^2\cap\dot{H}^1_x)}|\partial_x^{\alpha_2}(a^g,b^g,c^g)|_{\dot{B}^\frac{1}{2}_{2,\infty}\cap\dot{H}^1}\|\partial_x^\alpha{\bf P}^\bot g\|\\
        &\quad+|\partial_x^{\alpha_1}(a^{f_1},b^{f_1},c^{f_1})|_{\dot{B}^\frac{1}{2}_{2,\infty}\cap\dot{H}^1}\|\partial_x^{\alpha_2}{\bf P}^\bot g\|_{L_v^2(L_x^2\cap\dot{H}^1)}\|\partial_x^\alpha{\bf P}^\bot g\|\\
        &\quad+\|\partial_x^{\alpha_1}{\bf P}^\bot f_1\|_{H_v^2(L_x^2\cap\dot{H}^1_x)}\|\partial_x^{\alpha_2}{\bf P}^\bot g\|_{L_v^2(L_x^2\cap\dot{H}^1)}\|\partial_x^\alpha{\bf P}^\bot g\|\\
        &\lesssim\mathcal{M}_0|(a^g,b^g,c^g)|_{\dot{B}_{2,\infty}^1\cap\dot{H}^2}\|\partial_x^\alpha{\bf P}^\bot g\|+\mathcal{M}_0\|\partial_x{\bf P}^\bot g\|_{L_v^2(L_x^2\cap\dot{H}_x^1)}\|\partial_x^\alpha{\bf P}^\bot g\|\\
        &\lesssim\mathcal{M}_0^2|(a^g,b^g,c^g)|^2_{\dot{B}^1_{2,\infty}\cap\dot{H}^2}+\eta\|\partial_x^\alpha{\bf P}^\bot g\|^2+\mathcal{M}_0^2\sum_{1\le|\alpha'|\le 2}\|\partial_x^{\alpha'}{\bf P}^\bot g\|^2\,.
    \end{align*}
    For $|\alpha_2|\ge 2$, we obtain
    \begin{align*}
        &\hspace{3mm}\int_{\mathbb{R}_{x,v}^6}\Gamma(\partial_x^{\alpha_1}f_1,\partial_x^{\alpha_2}g)\partial_x^\alpha{\bf P}^\bot g\,{\rm d}v{\rm d}x\\
        &\lesssim\|\partial_x^{\alpha_1}f_1\|_{L_x^\infty H_v^2}\|\partial_x^{\alpha_2}{\bf P}g\|_{L_x^2L_v^2}\|\partial_x^\alpha{\bf P}^\bot f\|_{L_x^2L_v^2}+\|\partial_x^{\alpha_1}f_1\|_{L_x^\infty H_v^2}\|\partial_x^{\alpha_2}{\bf P}^\bot g\|_{L_x^2L_v^2}\|\partial_x^\alpha{\bf P}^\bot f\|_{L_x^2L_v^2}\\
        &\lesssim\|\partial_x^{\alpha_1}f_1\|_{H_v^2(\dot{H}^1\cap\dot{H}^2)}|(a^g,b^g,c^g)|_{\dot{H}^{|\alpha_2|}}\|\partial_x^\alpha{\bf P}^\bot g\|+\|\partial_x^{\alpha_1}f_1\|_{H_v^2(\dot{H}^1\cap\dot{H}^2)}\|\partial_x^{\alpha_2}{\bf P}^\bot g\|\|\partial_x^\alpha{\bf P}^\bot g\|\\
        &\lesssim\mathcal{M}_0^2|(a^g,b^g,c^g)|_{\dot{H}^{|\alpha_2|}}^2+\eta\|\partial_x^\alpha{\bf P}^\bot g\|^2+\mathcal{M}_0^2\|\partial_x^{\alpha_2}{\bf P}^\bot g\|^2\,.
    \end{align*}
    Consequently, we obtain
    \begin{align*}
        O_{32,2}&\lesssim(1+t)^{-\frac{1}{2}}\sup_t\Big\{(1+t)^{\frac{1}{2}}|(a^g,b^g,c^g)|^2_{\dot{B}^1_{2,\infty}\cap\dot{H}^{N-2}}\Big\}\\
        &\quad+\eta(1+t)^{-\frac{1}{2}}\sum_{1\le|\alpha'|\le|\alpha|}\sup_t\Big\{(1+t)^{\frac{1}{2}}\|\partial_x^{\alpha'}{\bf P}^\bot g\|^2\Big\}\,.
    \end{align*}
The estimate of $O_{33}$ can be derived by a completely the same way as $O_{32}$, so we omit it.

\noindent Substituting all the estimates into \eqref{wp} and combining them with Lemma \ref{lemma8.5}, we obtain the desired result for the case $1\le|\alpha|\le N-3$.

{\bf Step II:}
\noindent $|\alpha|=N-2$. We consider  the original equation \eqref{X1}. Similar to \eqref{B5}, we have the following estimate for $g$
\begin{align*}\label{6.19}
    \|w_{\ell,0}\partial_x^\alpha g\|^2&\lesssim \int_{\mathbb{R}_{x,v}^6}e^{-\nu(v)t}|w_{\ell,0}\partial_x^\alpha g_0|^2\,\mathrm{d}v\mathrm{d}x-\frac{1}{2}\int_0^t\int_{\mathbb{R}_{x,v}^6} e^{-\nu(v)(t-\tau)}\nu(v)|w_{\ell,0}\partial_x^\alpha g|^2\,\mathrm{d}v\mathrm{d}x\mathrm{d}\tau\notag\\
    &\qquad+\int_0^t\int_{\mathbb{R}_{x,v}^6} e^{-\nu(v)(t-\tau)}w_{\ell,0}^2\partial_x^\alpha \mathcal{K}g\partial_x^\alpha  g\,\mathrm{d}v\mathrm{d}x\mathrm{d}\tau\notag\\
    &\qquad+\int_0^t\int_{\mathbb{R}_{x,v}^6} e^{-\nu(v)(t-\tau)}w_{\ell,0}^2\partial_x^\alpha \Gamma(g,f_1)\partial_x^\alpha  g\,\mathrm{d}v\mathrm{d}x\mathrm{d}\tau\notag\\
    &\qquad+\int_0^t\int_{\mathbb{R}_{x,v}^6} e^{-\nu(v)(t-\tau)}w_{\ell,0}^2\partial_x^\alpha \Gamma(f_2,g)\partial_x^\alpha  g\,\mathrm{d}v\mathrm{d}x\mathrm{d}\tau\notag\\
    &:=O_{34}+O_{35}+O_{36}+O_{37}+O_{38}\,.
\end{align*}
By using Lemma \ref{lem:max_estimate} and taking $x=\langle v\rangle^\frac{\gamma}{2}$ one has
\begin{align*}
    O_{34}&\lesssim \int_{\mathbb{R}_{x,v}^6}e^{-\langle v\rangle^{\gamma}t}\langle v\rangle^{\frac{\gamma}{2}}|w_{\ell-\frac{\gamma}{4},0}\partial_x^\alpha g_0|^2\,\mathrm{d}v\mathrm{d}x\lesssim(1+t)^{-\frac{1}{2}}\|w_{\ell-\frac{\gamma}{4},0}\partial_x^\alpha g_0\|^2\,.\notag
\end{align*}
For $O_{36}$, similar to the estimate of $O_{31}$, we have
\begin{align*}
    O_{36}&\lesssim(1+t)^{-\frac{1}{2}}\sum_{j\in\mathbb{Z}^+}\sup_t\Big\{(1+t)^{\frac{1}{2}}\int_{V_j\times\mathbb{R}_x^3}\nu^{-1}(v)w^2_{\ell,0}\partial_x^\alpha\mathcal{K}g\partial_x^\alpha g\,{\rm d}v{\rm d}x\Big\}\\
    &\quad+(1+t)^{-\frac{1}{2}}\sup_t\Big\{(1+t)^{\frac{1}{2}}\int_{\mathbb{R}_{x,v}^6}w^2_{\ell,0}\partial_x^\alpha\mathcal{K}g\partial_x^\alpha g\,{\rm d}v{\rm d}x\Big\}\\
    &\lesssim\eta(1+t)^{-\frac{1}{2}}\sum_{j\in\mathbb{Z}^+}\sup_t\Big\{(1+t)^{\frac{1}{2}}\|w_{\ell,0}\partial_x^\alpha g\|^2_{L_v^2(V_j)L_x^2(\mathbb{R}^3)}\Big\}\\
    &\quad+(1+t)^{-\frac{1}{2}}\sup_t\Big\{(1+t)^{\frac{1}{2}}\|\partial_x^\alpha g\|^2\Big\}+\eta'(1+t)^{-\frac{1}{2}}\sup_t\Big\{(1+t)^{\frac{1}{2}}\|w_{\ell,0}\partial_x^\alpha g\|^2\Big\}\,.
\end{align*}
The estimates of $O_{37}$ and $O_{38}$ can be obtained in the same way as $O_{32}$, so we omit them.

    \noindent Collecting all the estimates and applying Lemma \ref{lemma8.5}, we obtain the desired result. This completes the proof of Lemma \ref{lemma6.5}.
\end{proof}

Noting that $\dot{B}^1_{2,\infty}\cap\dot{B}_{2,\infty}^2\subset \dot{B}^1_{2,\infty}   $ and $\dot{H}^2\hookrightarrow\dot{B}^2_{2,\infty}$, and combining all the lemmas above, we obtain the following proposition.
\begin{proposition}\label{time decay}
    Let $N\ge5$, $-3<\gamma<0$ and $\ell\ge-N-\frac{\gamma}{2}$. Assume that $f_1(t,x,v), f_2(t,x,v)$ are the solutions of \eqref{f-perturbation} with initial data $f_1^0, f_2^0$ satisfying \eqref{initial} respectively, let $g(t,x,v):=f_1-f_2$ with initial data $g_0:=f_1^0-f_2^0$, suppose that
    \begin{align}
        \|g_0\|_{\mathcal{Y}_0}:=\|g_0\|_{\widetilde{L}_v^2(\dot{B}_{2,\infty}^{\frac{1}{2}})}+\|g_0\|_{L_v^2(\dot{H}^2\cap\dot{H}^{N-2})}+\|w_{\ell-\frac{\gamma}{4},0}\,\mathbf{P}^{\bot}g_0\|_{L_v^2(\dot{H}^1\cap\dot{H}^{N-2})}\le\delta_1\,,\notag
    \end{align}
    where $\delta_1$ is suitable small,
    then we have the time decay of $g(t,x,v)$:
    \begin{align}
\sup_t\|g\|_{\mathcal{Y}}&:=\sup_t\|g\|_{\widetilde{L}_v^2(\dot{B}^1_{2,\infty})}+\sup_t\|g\|_{L_v^2(\dot{H}^2\cap\dot{H}^{N-2})}+\sup_t\|w_{\ell,0}\mathbf{P}^{\bot} g\|_{L_v^2(\dot{H}^1\cap\dot{H}^{N-2})}\nonumber\\
&\lesssim(1+t)^{-\frac{1}{4}}\|g_0\|_{\mathcal{Y}_0}\,,\notag
    \end{align}
    here $w_{\ell,\beta}=\langle v\rangle^{\ell+\gamma|\beta|}$, and $\beta$ denotes the order of differentiation in $v$.
\end{proposition}

\section{The existence and stability of the periodic solution}\label{Sect-time-periodic}
This section is devoted to proving the existence and stability of periodic solutions to the system \eqref{f-perturbation}.

\begin{proof}[The proof of Theorem \ref{theorem-1}]
Our strategy is to use the time decay established in Proposition \ref{time decay} to construct a Cauchy sequence. Taking the limit of this sequence as the initial data, and applying Proposition \ref{existence}, we obtain a global solution $f^{per}$
 to \eqref{f-perturbation}. Finally, we show that $f^{per}$
 is indeed the desired periodic solution.

    We divide the proof into two parts. 
    
    Firstly, we begin with the trivial initial data $0$ which clearly satisfies \eqref{initial}. Define $f_0^*=f^*(0,x,v)=0$, then according to Proposition \ref{existence}, there exists a global solution $f^*(t,x,v)$ of the system \eqref{f-perturbation} satisfying
    \[\sup_t\|f^*(t,x,v)\|_X\lesssim\delta_0\,.\]
     We deduce that $\left\{f^*(nT)\right\}_{n\in\mathbb{Z}}$ is the Cauchy sequence in $\mathcal{Y}$.
    
    In fact, let $n>m$, if we set $f_1(t)=f^*(t), f_2(t)=f^*(t+(n-m)T)$, which obviously solve the system \eqref{f-perturbation}, then we have 
    \[\|f_1(0,x,v)\|_X=\|f^*_0\|_X=0\lesssim\delta_1\,,\]
    and 
    \[\|f_2(0,x,v)\|_X=\|f^*((n-m)T)\|_X\lesssim\|f^*_0\|_X\le\delta_1.\]
    Furthermore, we have
    \[\|f_2(0,x,v)-f_1(0,x,v)\|_{\mathcal{Y}_0}\lesssim\|f_2(0,x,v)\|_X\lesssim\delta_1.\]
    Applying Proposition \ref{time decay} yields that
    \begin{align*}
       &\hspace{3mm}\|f^*(n T)-f^*(mT)\|_{\mathcal{Y}}\lesssim\|f_2(mT)-f_1(mT)\|_{\mathcal{Y}}\lesssim(1+mT)^{-\frac{1}{2}}\delta_1.
    \end{align*}
    Let $m\to\infty$, one has
    \[\lim_{n,m\to\infty}\|f^*(n T)-f^*(mT)\|_{\mathcal{Y}}=0\,,\]
    So $\left\{f^*(nT)\right\}_{n\in\mathbb{Z}}$ is the Cauchy sequence in $\mathcal{Y}$. 

    Furthermore, we have \[\sup_t\|f^*(t,x,v)\|_{\mathcal{Y}_0}\lesssim \sup_t\|f^*(t,x,v)\|_X\lesssim\delta_1\,,\]
    here we just need to take the weight as $w_{\ell-\frac{\gamma}{2},\beta}$ in $X$ norms since Proposition \ref{existence} holds for all $\ell$ big enough.
    
    Then there exists a limit of the sequence $\{f^*(nT)\}$ in $\mathcal{Y}$, we denote the limit by $f^*_{\infty}$. The Fatou lemma implies that
    \[\|f^*_{\infty}\|_{\mathcal{Y}_0}\lesssim\delta_1.\]
    Secondly, we construct the periodic solution of \eqref{f-perturbation}.

    According to Proposition \ref{existence}, the system \eqref{f-perturbation} admits a global solution $f^{per}$ with the initial data $f_\infty^*$. 
    We next prove that $f^{per}$ is exactly the periodic solution with the period $T$.
    
    If we set $f_3(t):=f^{per}(t), f_4(t):=f^*(t+(n-1)T)$, then we have
    \begin{align*}
        &\hspace{3mm}\|f^{per}(T)-f^*(nT)\|_{\mathcal{Y}}=\|f_3(T)-f_4(T)\|_{\mathcal{Y}}\\
        &\lesssim\|f_3(0)-f_4(0)\|_{\mathcal{Y}}\lesssim\|f^*_{\infty}-f_4((n-1)T)\|_{\mathcal{Y}}\to 0\qquad as\quad n\to 0\,,
    \end{align*}
then 
\begin{align*}
    \|f^{per}(T)-f^*_{\infty}\|_{\mathcal{Y}}\lesssim\|f^{per}(T)-f^*(nT)\|_{\mathcal{Y}}+\|f^*(nT)-f^*_{\infty}\|_{\mathcal{Y}}\to 0,\qquad as \quad n\to\infty\,,
\end{align*}
this implies $f^{per}(T)=f^{per}(0)$, we deduce that $f^{per}(t+T)=f^{per}(t)$ for any $t$. To see this, define $\tilde{f}(t)=f^{per}(t+T)$, then $\tilde{f}(0)=f^{per}(T)=f^{per}(0)$, so that $\tilde{f}$ and share the same initial data. By uniqueness of solutions to \eqref{f-perturbation}, it follows that $\tilde{f}(t)=f^{per}(t+T)=f^{per}(t)$. That is, $f^{per}$ is exactly the desired periodic solution of system \eqref{f-perturbation}. 
This completes the proof of Theorem  \ref{theorem-1}. 
\end{proof}

We now turn to complete the proof of Theorem \ref{theorem-2}.

\begin{proof}[The proof of Theorem \ref{theorem-2}]
The stability follows immediately from Proposition \ref{time decay}. To this end, let $\tilde{f}(t,x,v)$ be the solution to \eqref{f-perturbation} with an initial data $\tilde{f}_0$ satisfying 
    \[\|\tilde{f}_0\|_X\lesssim\delta_1.\]
    Consider the perturbation $\tilde{f}-f^{per}$, we first estimate $\|\tilde{f}_0-f^*_{\infty}\|_{\mathcal{Y}_0}$. Indeed, since
$$
\|\tilde{f}_0\|_{\mathcal{Y}_0}\lesssim\|\tilde{f}_0\|_X\lesssim\delta_1,\ \|f^*_{\infty}\|_{\mathcal{Y}_0}\lesssim\delta_1\,,$$
  it follows that
     \[\|\tilde{f}_0-f^*_{\infty}\|_{\mathcal{Y}_0}\lesssim\delta_1\,.\]
   Then, by Proposition \ref{time decay}, we directly obtain
     \[\sup_t\|\tilde{f}-f^{per}\|_{\mathcal{Y}}\lesssim(1+t)^{-\frac{1}{4}}\delta_1\,.\]
     This completes the proof of Theorem \ref{theorem-2}.
\end{proof}

\section{Appendix}\label{appendix}
In this appendix, we collect several significant estimates involving collision frequency $\nu(v)$, the operators $\mathcal{K}$ and $\Gamma$. The first concerns inner product estimates between $\nu(v)$ and $\mathcal{K}$ on velocity annuli defined by \eqref{vel-dec}.

\begin{lemma}\label{lemma5.1}
    Let $N\ge 5$ be an integer and $w_{\ell,\beta}:=\langle v\rangle^{\ell+\gamma|\beta|}$ with $-3<\gamma<0,\ell>N-\frac{\gamma}{2},|\beta|>0$. Then for all $\eta>0$, $\exists C(\eta)>0$ such that
    \begin{align}
        &\hspace{3mm}\sum_{j\in\mathbb{Z}^+}\sup_t\int_{V_j(v)\times\mathbb{R}_x^3} \nu^{-1}(v){w}_{\ell,\beta}^2\partial_v^{\beta'}[\nu(v)]\partial_v^{\beta-\beta'}f\partial_v^\beta f\,\mathrm{d}v{\rm d}x\notag\\
&\lesssim\eta\sum_{j\in\mathbb{Z}^+}\sup_t\sum_{|\beta'|\le|\beta|}\|w_{\ell,\beta'}\partial_v^{\beta'}f\|^2_{L^2(V_j)L_x^2(\mathbb{R}^3)}+\sup_t\|f\|^2\,.
\label{nu-vj}
    \end{align}
    Moreover, if $|\beta|\ge0$, then we have 
    \begin{align}
    &\hspace{3mm}\sum_{j\in \mathbb{Z}^+}\sup_t\int_{V_j\times\mathbb{R}_x^3} \nu^{-1}(v)w_{\ell,\beta}^2\partial_v^\beta[\mathcal{K} f_1]  f_2\,\mathrm{d}v{\rm d}x\notag\\
    &\lesssim\eta_1\sum_{j\in \mathbb{Z}^+}\sup_{t}\sum_{|\beta'|\le|\beta|}\|w_{\ell,\beta}\partial_v^{\beta'} f_1\|_{L^2(\widetilde{V}_j)L_x^2(\mathbb{R}^3)}^2+\sup_{t}\|f_1\|^2\notag\\
    &\quad+\eta_2\sum_{j\in \mathbb{Z}}\sup_{t}\|{w}_{\ell,\beta}f_2\|_{L^2(V_j)L^2_x(\mathbb{R}^3)}^2+\eta_3\sum_{|\beta'|\le|\beta|}\sup_t\|w_{\ell,\beta}\partial_v^{\beta'}f_1\|^2\,,\label{k-vj}
    \end{align}
    where $\widetilde{V}_j=V_{j-1}\cup V_j\cup V_{j+1}$ for $j\geq 1$ and $\widetilde{V}_0=V_0\cup V_{1}$, $\eta_i>0(i=1,2,3)$ are small constants.
\end{lemma}
\begin{proof}
We first prove \eqref{nu-vj}. The argument follows closely that of Lemma 2 in \cite{Strain-Guo-2008-ARMA}. The only difference is that the integration is restricted to the velocity annulus $V_j$ instead of $\mathbb{R}_v^3$. To ensure the convergence of $\sum_{j\in\mathbb{Z}^+}\sup\limits_t\|w_{\ell,\beta}\bar{\chi}(v)f\|^2_{L_v^2(V_j)L_x^2(\mathbb{R}^3)}$, we proceed as follows:
    \begin{align*}
        &\hspace{3mm}\int_{V_j\times\mathbb{R}_x^3} |\bar{\chi}(v)w_{\ell,\beta}f|^2\,{\rm d}v{\rm d}x\lesssim\int_{V_j\times\mathbb{R}_x^3} \Big|\bar{\chi}(v)w_{\ell,\beta}|v| |v|^{-1}f\Big|^2\,{\rm d}v{\rm d}x\lesssim\frac{1}{2^{2j}}\int_{\mathbb{R}_{x,v}^6}|f|^2\,{\rm d}v{\rm d}x\,.
    \end{align*}
    Then 
    \[\sum_{j\in\mathbb{Z}^+}\sup\limits_t\|w_{\ell,\beta}\bar{\chi}(v)f\|^2_{L_v^2(V_j)L_x^2(\mathbb{R}^3)}\lesssim\sum_{j\in\mathbb{Z}^+}\frac{1}{2^{2j}}\sup_t\|f\|^2\lesssim\sup_t\|f\|^2\,.\]
    In the above estimate, we have made full use of the property of the space $V_j$: we attach a $|v|^{-1}\lesssim\frac{1}{2^j}$ as the convergence factor, and the extra $|v|$ can be bounded by a constant due to the cut-off function $\bar{\chi}(v)$, then we extend the integral domain to the whole space $\mathbb{R}_v^3\times \mathbb{R}_x^3$ in order to apply the known results in section \ref{Sect-noweight}. 

We next consider \eqref{k-vj}. We integrate over $V_j$ and follow the procedure in \cite{Strain-Guo-2008-ARMA}. We only highlight the differences from the original argument.

    \noindent Firstly, for ${\bf K}_1$, the estimate of ${\bf K}_1^{1-\chi}$ is completely the same. We have
    \[\Big|{\bf K}_1^{1-\chi}\Big|\lesssim\frac{\eta}{2}|w_{\ell,\beta}f_2|_{L^2_v(V_j)}\sum_{|\beta'|\le|\beta|}|w_{\ell,\beta}\partial_v^{\beta'}f_1|_{L^2_v(V_j)}\,.\]
    As to ${\bf K}_1^\chi$, due to the existence of $\bar{\chi}(v)$ and $\mu(v)$, we can also attach a convergence factor $|v|^{-1}\lesssim\frac{1}{2^j}$ and the extra $|v|$ can be bounded or absorbed by $\mu(v)$ respectively. Then we have
    \[\Big|{\bf K}_1^\chi\Big|\le \frac{1}{2^j}|f_1|_{L_v^2(\mathbb{R}^3)}|w_{\ell,\beta}f_2|_{L_v^2(V_j)}\lesssim\frac{1}{2^{2j}}|f_1|_{L_v^2(\mathbb{R}^3)}^2+\eta|w_{\ell,\beta}f_2|_{L_v^2(V_j)}^2\,.\]
    \noindent Secondly, for ${\bf K}_2$, the estimate of ${\bf K}_2^{1-\chi}$ is the same, so we omit it.

    \noindent For ${\bf K}_2^{\Upsilon}$, in the domain $V_j$, if $|v+\xi|>2|v|\ge 2^{j+1}$, then we attach a convergence factor $|\xi|^{-1}\lesssim\frac{1}{2^j}$, one has
    \begin{align*}
        \Big|\partial_v^\beta[{\bf K}_2^{\Upsilon}f_1]\Big|&\lesssim\frac{C}{m'}|w_{\ell,\beta}f_2|_{L_v^2(V_j)}\sum_{|\beta'|\le|\beta|}\int_{V_j\times\mathbb{R}_\xi^3}w^2_{\ell,\beta}(v+\xi)|\partial_v^{\beta-\beta'}f_1(v+\xi)|^2|\xi|^{-1}e^{-|\xi|^2}\,{\rm d}\xi{\rm d}v\\
        &\lesssim\frac{1}{2^j}\frac{C}{m'}|w_{\ell,\beta}f_2|_{L_v^2(V_j)}\sum_{|\beta'|\le|\beta|}\int_{\mathbb{R}_v^3\times\mathbb{R}_\xi^3}w^2_{\ell,\beta}(v+\xi)|\partial_v^{\beta-\beta'}f_1(v+\xi)|^2e^{-|\xi|^2}\,{\rm d}\xi{\rm d}v\\
        &\lesssim\frac{1}{2^j}\frac{C}{m'}|w_{\ell,\beta}f_2|_{L_v^2(V_j)}\sum_{|\beta'|\le|\beta|}|w_{\ell,\beta}\partial_v^{\beta-\beta'}f_1|_{L_v^2(\mathbb{R}^3)}\\
        &\lesssim\frac{C}{m'}|w_{\ell,\beta}f_2|_{L_v^2(V_j)}^2+\eta\frac{1}{2^{2j}}\sum_{|\beta'|\le|\beta|}|w_{\ell,\beta}\partial_v^{\beta-\beta'}f_1|_{L_v^2(\mathbb{R}^3)}^2\,.
    \end{align*}
    If $|v+\xi|\le2|v|$, then $|v|-|\xi|\le|v+\xi|< 2^{j+2}$ in the domain $V_j$. If $|\xi|$ is bounded so that $|v|-|\xi|\ge 2^{j-1}$, then we have
    \begin{align*}
        \Big|\partial_v^\beta[{\bf K}_2^{\Upsilon}f_1]\Big|&\lesssim\frac{C}{m'}|w_{\ell,\beta}f_2|_{L_v^2(V_j)}\sum_{|\beta'|\le|\beta|}\int_{V_j\times\mathbb{R}_\xi^3}w^2_{\ell,\beta}(v+\xi)|\partial_v^{\beta-\beta'}f_1(v+\xi)|^2e^{-|\xi|^2}\,{\rm d}\xi{\rm d}v\\
        &\lesssim\frac{C}{m'}|w_{\ell,\beta}f_2|_{L_v^2(V_j)}\sum_{|\beta'|\le|\beta|}\int_{\widetilde{V}_j\times\mathbb{R}_\xi^3}w^2_{\ell,\beta}(v)|\partial_v^{\beta-\beta'}f_1(v)|^2e^{-|\xi|^2}\,{\rm d}\xi{\rm d}v\\
        &\lesssim\frac{C}{m'}|w_{\ell,\beta}f_2|_{L_v^2(V_j)}\sum_{|\beta'|\le|\beta|}|w_{\ell,\beta}\partial_v^{\beta-\beta'}f_1|_{L^2_v(\widetilde{V}_j)}\,.
    \end{align*}
    If $|\xi|$ is unbounded, the  estimate can be derived in the same way as the case $|v+\xi|\ge2|v|$, we omit it.
    For ${\bf K}_2^{1-\Upsilon}$, we just need to match the convergence factor $|v|^{-1}\lesssim\frac{1}{2^j}$ since $|v|\le m'$ is bounded. Collecting all the above estimates, and applying Young’s inequality, we obtain \eqref{k-vj}. This completes the proof of Lemma \ref{lemma5.1}.
\end{proof}
The following lemma provides trilinear estimates on annulus $V_j$.
\begin{lemma}\label{lemma5.2}
    Let $N\ge5$, ${w}_{\ell,\beta}:=\langle v\rangle^{\ell+\gamma|\beta|}$ with $-3<\gamma<0,\ell>N-\frac{\gamma}{2},|\beta|\ge0$.  Assume that $\alpha_1+\alpha_2=\alpha,\beta_1+\beta_2=\beta,|\alpha|+|\beta|\le N$, then for $j\in\mathbb{Z}^+$, we have
    \begin{align}
        &\hspace{3mm}\sup_t\int_{V_j\times\mathbb{R}_x^3}\nu^{-1}(v){w}_{\ell,\beta}^2\Gamma(\partial_x^{\alpha_1}\partial_v^{\beta_1}f,\partial_x^{\alpha_2}\partial_v^{\beta_2}g)\partial_x^{\alpha}\partial_v^{\beta}h\,\mathrm{d}v{\rm d}x\notag\\
        &\lesssim\sup_t\int_{\mathbb{R}_x^3}\sum_{m\le2}\left\{|\nabla_v^m(\mu^\delta\partial_x^{\alpha_1}\partial_v^{\beta_1}f)|_{L^2(\mathbb{R}^3_v)}+|{w}_{\ell,\beta}\partial_x^{\alpha_1}\partial_v^{\beta_1}f|_{L^2(\mathbb{R}^3_v)}\right\}\notag\\
        &\qquad\qquad\times|{w}_{\ell,\beta}\partial_x^{\alpha_2}\partial_v^{\beta_2}g|_{L^2(\widetilde{V}_j)}|{w}_{\ell,\beta}\partial_x^{\alpha}\partial_v^{\beta}h|_{L^2(V_j)}\,{\rm d}x\,,\notag
    \end{align}
    or
    \begin{align}\label{form2}
        &\hspace{3mm}\sup_t\int_{V_j\times\mathbb{R}_x^3}\nu^{-1}(v){w}_{\ell,\beta}^2\Gamma(\partial_x^{\alpha_1}\partial_v^{\beta_1}f,\partial_x^{\alpha_2}\partial_v^{\beta_2}g)\partial_x^{\alpha}\partial_v^{\beta}h\,\mathrm{d}v{\rm d}x\notag\\
        &\lesssim\sup_t\int_{\mathbb{R}_x^3}\sum_{m\le2}\left\{|\nabla_v^m(\mu^\delta\partial_x^{\alpha_2}\partial_v^{\beta_2}f)|_{L^2(\mathbb{R}^3_v)}+|{w}_{\ell,\beta}\partial_x^{\alpha_2}\partial_v^{\beta_2}f|_{L^2(\mathbb{R}^3_v)}\right\}\notag\\
        &\qquad\qquad\times|{w}_{\ell,\beta}\partial_x^{\alpha_1}\partial_v^{\beta_1}g|_{L^2(\widetilde{V}_j)}|{w}_{\ell,\beta}\partial_x^{\alpha}\partial_v^{\beta}h|_{L^2(V_j)}\,{\rm d}x\,,\notag
    \end{align}
    where $\widetilde{V}_j=V_{j-1}\cup V_j\cup V_{j+1}$ for $j\geq 1$ and $\widetilde{V}_0=V_0\cup V_{1}$.
\end{lemma}
\begin{proof}
    Repeating the arguments used in Lemma 2 of \cite{Strain-Guo-2008-ARMA}, while replacing the integration domain in $v$ by $V_j$, we obtain the desired results. We briefly point out the differences.

    The treatment of the loss term is unchanged. It should be noted that the integration variable in the gain term is $u$, and the integration is still taken over $\mathbb{R}_u^3$. 

As for the gain term estimate, since a change of integration variable is involved, special care is required regarding the integration domain. Here we only consider case (2c) in \cite[pp.311]{Strain-Guo-2008-ARMA} as an illustrative example. In this case, we have $\{|u|\le\frac{|v|}{2},\,|v-u|\ge\frac{|v|}{2},\,|v|\ge 1\}$. Due to the fact that
\[|v'|^2+|u'|^2=|v|^2+|u|^2,\quad |v'-u'|=|v-u|.\]
Then one has
\[|v'|^2+|u'|^2\le\frac{5}{4}|v|^2,\quad|v'|\ge|v'-u'|-|u'|\ge\frac{|v|}{2}-|u'|.\]
If $|u'|\le\frac{|v|}{4}$, then it follows that $|v'|\ge\frac{|v|}{4}$. Consequently, for $v\in V_j$, we have $v'\in\widetilde{V}_j$. Hence, we obtain
\begin{align*}
    &\hspace{3mm}\int_{\mathbb{S}^2\times\mathbb{R}_u^3\times V_j} \nu^{-1}(v)\,\langle v\rangle^\gamma e^{-\frac{|u|^2}{4}}w^2_{\ell,\beta}(v)|\partial_x^{\alpha_1}\partial_v^{\beta_1}f(u')|^2|\partial_x^{\alpha_2}\partial_v^{\beta_2}g(v')|^2\,{\rm d}\omega{\rm d}u{\rm d}v\\
    &\lesssim\int_{\mathbb{S}^2\times\mathbb{R}_u^3\times {V}_j}w^2_{\ell,\beta}(u')|\partial_x^{\alpha_1}\partial_v^{\beta_1}f(u')|^2w^2_{\ell,\beta}(v')|\partial_x^{\alpha_2}\partial_v^{\beta_2}g(v')|^2\,{\rm d}\omega{\rm d}u{\rm d}v\\
    &=\int_{\mathbb{R}_u^3\times \widetilde{V}_j}w^2_{\ell,\beta}(u)|\partial_x^{\alpha_1}\partial_v^{\beta_1}f(u)|^2w^2_{\ell,\beta}(v)|\partial_x^{\alpha_2}\partial_v^{\beta_2}g(v)|^2\,{\rm d}u{\rm d}v\\
    &\lesssim|w_{\ell,\beta}\partial_x^{\alpha_1}\partial_v^{\beta_1}f|^2_{L_v^2(\mathbb{R}^3)}|w_{\ell,\beta}\partial_x^{\alpha_2}\partial_v^{\beta_2}|^2_{L_v^2(\widetilde{V}_j)}\,.
\end{align*}
If $|u'|>\frac{|v|}{4}$, then within the domain $V_j$, we have $u'\in\widetilde{V}_j$. It follows
\begin{align*}
    &\hspace{3mm}\int_{\mathbb{S}^2\times\mathbb{R}_u^3\times V_j} \nu^{-1}(v)\,\langle v\rangle^\gamma e^{-\frac{|u|^2}{4}}w^2_{\ell,\beta}(v)|\partial_x^{\alpha_1}\partial_v^{\beta_1}f(u')|^2|\partial_x^{\alpha_2}\partial_v^{\beta_2}g(v')|^2\,{\rm d}\omega{\rm d}u{\rm d}v\\
    &\lesssim\int_{\mathbb{S}^2\times\mathbb{R}_u^3\times {V}_j}w^2_{\ell,\beta}(u')|\partial_x^{\alpha_1}\partial_v^{\beta_1}f(u')|^2w^2_{\ell,\beta}(v')|\partial_x^{\alpha_2}\partial_v^{\beta_2}g(v')|^2\,{\rm d}\omega{\rm d}u{\rm d}v\\
    &=\int_{\widetilde{V}_j\times \mathbb{R}_v^3}w^2_{\ell,\beta}(u)|\partial_x^{\alpha_1}\partial_v^{\beta_1}f(u)|^2w^2_{\ell,\beta}(v)|\partial_x^{\alpha_2}\partial_v^{\beta_2}g(v)|^2\,{\rm d}u{\rm d}v\\
    &\lesssim|w_{\ell,\beta}\partial_x^{\alpha_2}\partial_v^{\beta_2}g|^2_{L_v^2(\mathbb{R}^3)}|w_{\ell,\beta}\partial_x^{\alpha_1}\partial_v^{\beta_1}f|^2_{L_v^2(\widetilde{V}_j)}\,.
\end{align*}
This completes the proof of Lemma \ref{lemma5.2}.
\end{proof}
 The following lemmas are used to treat the relevant integrations in $V_0$.
\begin{lemma}[\cite{Strain-Guo-2008-ARMA}]\label{lemma9.3}
     Let $N\ge 5$ be an integer and $w_{\ell,\beta}:=\langle v\rangle^{\ell+\gamma|\beta|}$ with $-3<\gamma<0,\ell>N-\frac{\gamma}{2},|\beta|\ge0$. Then for all $\eta>0$, $\exists C(\eta)>0$ such that
     \begin{align*}
         \int_{\mathbb{R}_v^3}w^2_{\ell,\beta}\partial_v^{\beta}[\mathcal{K}f_1]f_2\,{\rm d}v&\lesssim\Big\{\eta\sum_{|\beta'|\le|\beta|}|w_{\ell,\beta'}\partial_v^{\beta'}f_1|_{L_\nu^2}+C_{\eta}|w_{\ell,\beta}\bar{\chi}_{C_\eta}(v)f_1|_{L_v^2}\Big\}|w_{\ell,\beta  }f_2|_{L_\nu^2}\,,
     \end{align*}
     where $0\le\bar{\chi}_m(v)\le 1$ is a smooth cut-off function such that for $m\ge0$,
     \[\bar{\chi}_m(v)\equiv1,\, for\, |v|\le m;\qquad \bar{\chi}_m(v)\equiv0,\, for \,|v|\ge 2m.\]
\end{lemma}
\begin{lemma}[\cite{Guo-ARMA-2003,Strain-Guo-2008-ARMA,Duan-Lei-Yang-Zhao-CMP-2017}]\label{lemma9.4}
     Let $N\ge5$, ${w}_{\ell,\beta}:=\langle v\rangle^{\ell+\gamma|\beta|}$ with $-3<\gamma<0,\ell>N-\frac{\gamma}{2},|\beta|\ge0$.  Assume that $\alpha_1+\alpha_2=\alpha,\beta_1+\beta_2=\beta,|\alpha|+|\beta|\le N$, then it holds that
     \begin{align*}
         &\hspace{3mm}\int_{\mathbb{R}_v^3}w^2_{\ell,\beta}\Gamma(\partial_x^{\alpha_1}\partial_v^{\beta_1}f,\partial_x^{\alpha_2}\partial_v^{\beta_2}g)\partial_x^\alpha\partial_v^\beta h\,{\rm d}v\\
         &\lesssim\sum_{m\le2}\Big\{\Big|\nabla_v\Big(\mu^\delta\partial_x^{\alpha_1}\partial_v^{\beta_1}f\Big)\Big|_{L_v^2}+\Big|w_{\ell,\beta}\partial_x^{\alpha_1}\partial_v^{\beta_1}f\Big|_{L_v^2}\Big\}\times\Big|w_{\ell,\beta}\partial_x^{\alpha_2}\partial_v^{\beta_2}g\Big|_{L_\nu^2}\Big|w_{\ell,\beta}\partial_x^{\alpha}\partial_v^{\beta}h\Big|_{L_\nu^2}\,,
     \end{align*}
     or
     \begin{align}\label{rrr}
         &\hspace{3mm}\int_{\mathbb{R}_v^3}w^2_{\ell,\beta}\Gamma(\partial_x^{\alpha_1}\partial_v^{\beta_1}f,\partial_x^{\alpha_2}\partial_v^{\beta_2}g)\partial_x^\alpha\partial_v^\beta h\,{\rm d}v\\
         &\lesssim\sum_{m\le2}\Big\{\Big|\nabla_v\Big(\mu^\delta\partial_x^{\alpha_2}\partial_v^{\beta_2}g\Big)\Big|_{L_v^2}+\Big|w_{\ell,\beta}\partial_x^{\alpha_2}\partial_v^{\beta_2}g\Big|_{L_v^2}\Big\}\times\Big|w_{\ell,\beta}\partial_x^{\alpha_1}\partial_v^{\beta_1}f\Big|_{L_\nu^2}\Big|w_{\ell,\beta}\partial_x^{\alpha}\partial_v^{\beta}h\Big|_{L_\nu^2}\,.\notag
     \end{align}
\end{lemma}

Finally, we give a basic estimate for the nonlinear operator $\Gamma$.
\begin{lemma}[\cite{Xiao-Xiong-Zhao-2013-JDE}]\label{nonlinear}
    Set $w=w(v)=\langle v\rangle^{-\frac{\gamma}{2}}$ with $-3<\gamma<0$. For $l\ge0$, it holds that
    \begin{align*}
        |w^l\Gamma(g_1,g_2)|^2_{L_v^2}\lesssim\sum_{|\beta|\le 2}|w^{l-|\beta|}\partial_v^\beta g_1|^2_{L_\nu^2}|w^l g_2|^2_{L_\nu^2}\,,
    \end{align*}
    or
    \begin{align*}
        |w^l\Gamma(g_1,g_2)|^2_{L_v^2}\lesssim\sum_{|\beta|\le 2}|w^{l-|\beta|}\partial_v^\beta g_2|^2_{L_\nu^2}|w
        ^l g_1|^2_{L_\nu^2}\,.
    \end{align*}
\end{lemma}

\bigskip 
\noindent {\bf Acknowledgements:} 
The research of Yuanjie Lei was supported by the National Natural Science Foundation of China under contract 12171176. The research of Shuangqian Liu was supported by the National Natural Science Foundation of China under contracts 12325107 and 12471217. The research of Huijiang Zhao was supported by National Natural Science Foundation of China under contracts 12221001, 12371225 and 12571242.

\vspace{2mm}

\textbf{Conflict of interest.} The authors do not have any possible conflicts of interest.

\vspace{2mm}

\textbf{Data availability statement.}
 Data sharing is not applicable to this article as no data sets were generated or analyzed during the current study.

\bibliographystyle{plain}

\begin{thebibliography}{aaa}
















\bibitem{BCD-Book-2011} H. Bahouri, J. Chemin, R. Danchin,
{\it  Fourier Analysis and Nonlinear Partial Differential Equations.}
Grundlehren Math. Wiss, 343,
Springer, Heidelberg, 2011. 



\bibitem{CIP-1994}  C. Cercignani, R. Illner,  M. Pulvirenti, 
{\it The Mathematical Theory of Dilute Gases}. 
Applied  Mathematical Sciences, 106,
Springer-Verlag, New York, 1994. 

\bibitem{Dr-IM-2000} R.  Danchin,  
Global existence in critical spaces for compressible Navier-Stokes equations.
{\it Invent. Math.} {\bf 141 (3)} (2000)  579--614.

\bibitem{Deguchi-2024-MathAnn}  N. Deguchi,  
On the stability of stationary compressible Navier-Stokes flows in 3D.
{\it Math. Ann.} {\bf 390 (3)} (2024) 4361--4404.

\bibitem{Deguchi-2025} N. Deguchi, Stability of time-periodic solutions to the Navier-Stokes-Fourier system. arXiv:2601.00034.



\bibitem{D-2008-JDE} R. Duan,  On the Cauchy problem for the Boltzmann equation in the whole space: global existence and uniform stability in $L^2_{\xi}(H^N_x)$. 
{\it J. Differential Equations} {\bf 244 (12)} (2008) 3204--3234.
\bibitem{DLN-2026}R. Duan, Y. Lei, J. Ni, Global dynamics of the Boltzmann equation driven by a time-periodic source in $\mathbb{R}^3$. preprint, 2026.

\bibitem{Duan-Lei-Yang-Zhao-CMP-2017}R. Duan, Y. Lei, T. Yang, H. Zhao, The Vlasov-Maxwell-Boltzmann system near Maxwellians in the whole space with very soft potentials. Comm. Math. Phys. {\bf 351} (2017), no.~1, 95--153.




\bibitem{Duan-Liu-2015-Acta.Sci}R. Duan, S. Liu, Time-periodic solutions of the Vlasov-Poisson-Fokker-Planck system. {\it Acta Math. Sci. Ser. B (Engl. Ed.)} {\bf 35} (2015), no. 4, 876--886. 

\bibitem{dlx-16}
R. Duan,  S. Liu, J. Xu, Global well-posedness in spatially critical Besov space for the Boltzmann equation. {\it Arch. Ration. Mech. Anal.} {\bf 220} (2016), no. 2, 711--745.


\bibitem{DN-2026}R. Duan, J. Ni, 
Three-dimensional time-periodic problem on the Boltzmann equation with external force. arXiv:2604.21339.
\bibitem{DUYZ-CMP-2008}R. Duan, S. Ukai, T. Yang, H. Zhao, Optimal decay estimates on the linearized Boltzmann equation with time dependent force and their applications. {\it Comm. Math. Phys. }{\bf 277} (2008), no. 1, 189--236.














\bibitem{Guo-ARMA-2003} Y. Guo,  
Classical solutions to the Boltzmann equation for molecules with an angular cutoff.
{\it Arch. Ration. Mech. Anal.} {\bf 169 (4)} (2003) 305--353.





\bibitem{Guo-IUMJ-2004} Y. Guo, The Boltzmann equation in the whole space. \textit{Indiana Univ. Math. J.} \textbf{53}(2004), 1081--1094.

\bibitem{Grad-1958} H. Grad, Principles of the kinetic theory of gases. {\em Handbuch der Physik (herausgegeben von S. Fl$\ddot{u}$gge), Bd. 12, Thermodynamik der Gase}. Springer-Verlag, Berlin, 1958, pp,205-294.








\bibitem{Liu-Yang-Yu-PD-2004} T. Liu, T. Yang, S. Yu, Energy method for the Boltzmann equation. {\it Physica D}
{\bf 188} (2004), 178--192.



 
\bibitem{Serrin-1959-ARMA}
J. Serrin, 
A note on the existence of periodic solutions of the Navier-Stokes equations.
{\it Arch. Rational Mech. Anal.} {\bf 3} (1959) 120--122.

\bibitem{Strain-Guo-2008-ARMA}R. Strain, Y. Guo, Exponential decay for soft potentials near Maxwellian. {\it Arch. Ration. Mech. Anal.} {\bf 187 (2)} (2008) 287--339.
\bibitem{Tusda-2016-ARMA}K. Tsuda, On the existence and stability of time periodic solution to the compressible Navier-Stokes equation on the whole space. {\it Arch. Ration. Mech. Anal. } {\bf 219} (2016), no. 2, 637--678.


\bibitem{Ukai-2006-DCDS}S. Ukai, Time-periodic solutions of the Boltzmann equation. {\it Discrete Contin. Dyn. Syst.} {\bf 14} (2006), no. 3, 579--596.
\bibitem{Ukai-Yang-2006-AA}S. Ukai, T. Yang, The Boltzmann equation in the space $L^2\cap L^\infty_\beta$: global and time-periodic solutions. {\it Anal. Appl. (Singap.)} {\bf 4} (2006), no. 3, 263--310.
\bibitem{Villani 2002} C. Villani, A review of mathematical topics in collisional kinetic theory. {\it in Handbook of Mathematical Fluid Dynamics,} North-Holland, Amsterdam, {\bf I} (2002) 71--305.

\bibitem{Xiao-Xiong-Zhao-2013-JDE}Q. Xiao, L. Xiong and H. Zhao, The Vlasov-Poisson-Boltzmann system with angular cutoff for soft potentials. {\it J. Differential Equations} {\bf 255} (2013), no.~6, 1196--1232.


\end{thebibliography}

\end{document}